\documentclass[6pt]{article}

\usepackage[
backend=biber,
style=alphabetic,
maxbibnames=99
]{biblatex}
\usepackage[x11names]{xcolor} 

\usepackage{amsmath}
\usepackage{amsfonts}
\usepackage{amssymb}
\usepackage{amsthm}
\usepackage{extarrows}
\usepackage{dsfont}
\usepackage{mathtools}
\usepackage{zref-clever}
\zcsetup{cap,nameinlink=false}                         % capitalise refs
\usepackage{tikz-cd}

\usepackage[x11names]{xcolor}
\usepackage{tikz}
\usetikzlibrary{calc,patterns,arrows,positioning,cd,calc,math,shapes, arrows.meta, decorations.pathmorphing, decorations.markings}
\usetikzlibrary{external} %%% Externalize tikz graphics
\usetikzlibrary{snakes}
\usetikzlibrary{decorations.pathreplacing,decorations.markings}
\tikzset{
	on each segment/.style={
		decorate,
		decoration={
			show path construction,
			moveto code={},
			lineto code={
				\path [#1]
				(\tikzinputsegmentfirst) -- (\tikzinputsegmentlast);
			},
			curveto code={
				\path [#1] (\tikzinputsegmentfirst)
				.. controls
				(\tikzinputsegmentsupporta) and (\tikzinputsegmentsupportb)
				..
				(\tikzinputsegmentlast);
			},
			closepath code={
				\path [#1]
				(\tikzinputsegmentfirst) -- (\tikzinputsegmentlast);
			},
		},
	},
	mid arrow/.style={postaction={decorate,decoration={
				markings,
				mark=at position .5 with {\arrow[#1]{stealth}}
	}}},
}

\tikzset{crossFilled/.pic={
		\tikzset{scale=0.1*#1}
		\fill[white] circle[radius=2];
		\draw[line width=2.5*#1](-1,-1)--(1,1)
		(-1,1)--(1,-1);
	}
}

\tikzset{
	styleArrow/.style={postaction={decorate},decoration={markings,mark=at position 0.7 with {\arrow{Stealth}}}},
	blackCircle/.style={fill=black,thick,radius=0.1,inner sep=0},
	mid arrow/.style={postaction={decorate,decoration={
				markings,
				mark=at position .5 with {\arrow[#1]{stealth}}
	}}},
	whiteCircle/.style={fill=white,thick,radius=0.1,inner sep=0},
	loop style/.style={
		styleArrow, 
	}
}

\tikzset{styleNode/.style={
		draw, ellipse, inner sep=1,
}}

\usepackage{tikz-cd}
\tikzcdset{scale cd/.style={every label/.append style={scale=#1},    cells={nodes={scale=#1}}}}

\usepackage{comment}
\usepackage{hyperref}
\usepackage{xparse}
\usepackage{graphicx}
\usepackage{amsmath}
\usepackage{amssymb}
\usepackage{subsupscripts}
\usepackage{braket}
\usepackage{amsthm}
\usepackage{xcolor}
\usepackage{subcaption}
\usepackage{bbm}

\makeatletter

\@addtoreset{equation}{section}
\makeatother

\usetikzlibrary{decorations.pathreplacing,decorations.markings}
\tikzset{
  on each segment/.style={
    decorate,
    decoration={
      show path construction,
      moveto code={},
      lineto code={
        \path [#1]
        (\tikzinputsegmentfirst) -- (\tikzinputsegmentlast);
      },
      curveto code={
        \path [#1] (\tikzinputsegmentfirst)
        .. controls
        (\tikzinputsegmentsupporta) and (\tikzinputsegmentsupportb)
        ..
        (\tikzinputsegmentlast);
      },
      closepath code={
        \path [#1]
        (\tikzinputsegmentfirst) -- (\tikzinputsegmentlast);
      },
    },
  },
  mid arrow/.style={postaction={decorate,decoration={
        markings,
        mark=at position .5 with {\arrow[#1]{stealth}}
      }}},
}

\tikzstyle{decision} = [diamond, draw, fill=blue!20, 
    text width=4.5em, text badly centered, node distance=3cm, inner sep=0pt]
\tikzstyle{block} = [rectangle, draw, fill=blue!20, 
    text width=8em, text centered, rounded corners, minimum height=2em]
\tikzstyle{line} = [draw, -latex']
\tikzstyle{cloud} = [draw, ellipse,fill=red!20, node distance=3cm,
    minimum height=2em, text centered]

\newif\ifqArrowIsLoop%
\qArrowIsLoopfalse%
\tikzset{/tikz/qArrow@pattern/.style={}}%
\tikzset{/tikz/qArrow@decor/.style={}}%
\tikzset{/tikz/qArrow label style/.style={below, inner sep=1pt}}% named style for label%
\tikzset{/tikz/qArrow arrow tip style/.style={-{Stealth}}}% named style for arrow tip%
\tikzset{/tikz/qArrowSetIsLoop/.code={%
  \edef\qArrowTmpFlag{#1}%
  \def\qArrowTrueString{true}%
  \ifx\qArrowTmpFlag\qArrowTrueString%
    \qArrowIsLooptrue%
  \else%
    \qArrowIsLoopfalse%
  \fi%
}}%
\tikzset{qArrow/.cd,%
  color/.initial=black,%
  width/.initial=0.4pt,%
  arrowtip/.code=\tikzset{/tikz/qArrow arrow tip style/.style={-{#1}}},%
  arrowtip=Stealth,%
  pattern/.is choice,%
  pattern/solid/.style  = { /tikz/qArrow@pattern/.style={} },%
  pattern/dashed/.style = { /tikz/qArrow@pattern/.style={dash pattern=on 3pt off 4pt} },%
  pattern/dotted/.style = { /tikz/qArrow@pattern/.style={dotted} },%
  xshift/.initial=0pt,%
  yshift/.initial=0pt,%
  label/.initial={},%
  label xshift/.initial=0pt,%
  label yshift/.initial=0pt,%
  labelstyle/.code = \tikzset{/tikz/qArrow label style/.style={#1}},%
  labelstyle={below, inner sep=1pt},%
  arrowpos/.initial=0.5,%
  loop shorten/.initial=2.0pt,%
  loop eps/.initial=0.02,%
  @conn/.initial={--},%
  connector/.is choice,%
  connector/straight/.style = {%
    /tikz/qArrow/@conn={--},%
    /tikz/qArrowSetIsLoop=false%
  },%
  connector/bend/.style     = {%
    /tikz/qArrow/@conn={to[bend left=\pgfkeysvalueof{/tikz/qArrow/bend}]},%
    /tikz/qArrowSetIsLoop=false%
  },%
  bend/.initial=25,%
  connector/outin/.style    = {%
    /tikz/qArrow/@conn={to[%
      out=\pgfkeysvalueof{/tikz/qArrow/out},%
      in=\pgfkeysvalueof{/tikz/qArrow/in},%
      looseness=\pgfkeysvalueof{/tikz/qArrow/looseness}%
    ]},%
    /tikz/qArrowSetIsLoop=false%
  },%
  out/.initial=20,%
  in/.initial=160,%
  looseness/.initial=1,%
  connector/controls/.style = {%
    /tikz/qArrow/@conn={.. controls%
      \pgfkeysvalueof{/tikz/qArrow/control 1}%
      and%
      \pgfkeysvalueof{/tikz/qArrow/control 2}%
    ..},%
    /tikz/qArrowSetIsLoop=false%
  },%
  control 1/.initial={(0,0)},%
  control 2/.initial={(0,0)},%
  loop out/.initial=80,%
  loop in/.initial=100,%
  loop depth/.initial=12mm,%
  loop width/.initial=6mm,%
  connector/loop/.style={%
    /tikz/qArrowSetIsLoop=true%
  },%
  decor/.is choice,%
  decor/none/.style   = { /tikz/qArrow@decor/.style={} },%
  decor/snake/.style  = { /tikz/qArrow@decor/.style={decorate, decoration={snake,%
      amplitude=\pgfkeysvalueof{/tikz/qArrow/amp},%
      segment length=\pgfkeysvalueof{/tikz/qArrow/seg}%
    }} },%
  decor/coil/.style   = { /tikz/qArrow@decor/.style={decorate, decoration={coil,%
      amplitude=\pgfkeysvalueof{/tikz/qArrow/amp},%
      segment length=\pgfkeysvalueof{/tikz/qArrow/seg}%
    }} },%
  decor/zigzag/.style = { /tikz/qArrow@decor/.style={decorate, decoration={zigzag,%
      amplitude=\pgfkeysvalueof{/tikz/qArrow/amp},%
      segment length=\pgfkeysvalueof{/tikz/qArrow/seg}%
    }} },%
  amp/.initial=1.2pt,%
  seg/.initial=6pt%
}%
\tikzset{/tikz/qArrow@stroke/.style={%
  draw=\pgfkeysvalueof{/tikz/qArrow/color},%
  line width=\pgfkeysvalueof{/tikz/qArrow/width},%
  /tikz/qArrow@pattern%
}}%
\tikzset{%
  qArrow/base/.style={%
    qArrow/connector=straight,%
    qArrow/decor=none,%
    qArrow/pattern=solid,%
    to path={%
      \pgfextra{%
        \path%
          coordinate (q@start2) at ($(\tikztostart)+(\pgfkeysvalueof{/tikz/qArrow/xshift},\pgfkeysvalueof{/tikz/qArrow/yshift})$)%
          coordinate (q@end2)   at ($(\tikztotarget)+(\pgfkeysvalueof{/tikz/qArrow/xshift},\pgfkeysvalueof{/tikz/qArrow/yshift})$);%
        \ifqArrowIsLoop%
          \pgfmathsetmacro{\qArrowPosA}{\pgfkeysvalueof{/tikz/qArrow/arrowpos}}%
          \pgfmathsetmacro{\qArrowPosB}{min(0.99,\qArrowPosA+\pgfkeysvalueof{/tikz/qArrow/loop eps})}%
          \path%
            coordinate (q@loopStart) at%
              ($(q@start2)+(\pgfkeysvalueof{/tikz/qArrow/loop out}:\pgfkeysvalueof{/tikz/qArrow/loop shorten})$)%
            coordinate (q@loopEnd) at%
              ($(q@start2)+(\pgfkeysvalueof{/tikz/qArrow/loop in}:\pgfkeysvalueof{/tikz/qArrow/loop shorten})$)%
            coordinate (q@loopCtrlA) at%
              ($(q@start2)+(-\pgfkeysvalueof{/tikz/qArrow/loop width},\pgfkeysvalueof{/tikz/qArrow/loop depth})$)%
            coordinate (q@loopCtrlB) at%
              ($(q@start2)+(\pgfkeysvalueof{/tikz/qArrow/loop width},\pgfkeysvalueof{/tikz/qArrow/loop depth})$);%
          \path%
          (q@loopEnd)%
          .. controls%
            (q@loopCtrlB)%
            and%
            (q@loopCtrlA)%
          ..%
          node[pos=\qArrowPosA,coordinate] (q@mid) {}%
          node[pos=\qArrowPosA,coordinate] (q@headA) {}%
          node[pos=\qArrowPosB,coordinate] (q@headB) {}%
          (q@loopStart);%
        \draw[qArrow@stroke,qArrow@decor]%
          (q@loopEnd)%
          .. controls%
            (q@loopCtrlB)%
            and%
            (q@loopCtrlA)%
          ..%
          (q@loopStart);%
          \draw[qArrow@stroke,qArrow@decor]%
            (q@loopStart)%
            .. controls%
              (q@loopCtrlA)%
              and%
              (q@loopCtrlB)%
            ..%
            (q@loopEnd);%
          \draw[qArrow@stroke,/tikz/qArrow arrow tip style] (q@headA) -- (q@headB);%
        \else%
          \path%
            (q@start2) \pgfkeysvalueof{/tikz/qArrow/@conn} (q@end2)%
            coordinate[pos=\pgfkeysvalueof{/tikz/qArrow/arrowpos}] (q@mid);%
          \draw[qArrow@stroke,qArrow@decor]%
            (q@start2) \pgfkeysvalueof{/tikz/qArrow/@conn} (q@end2);%
          \draw[opacity=0,line width=0pt,decorate,decoration={markings,%
            mark=at position \pgfkeysvalueof{/tikz/qArrow/arrowpos}%
              with {%
                \pgfinterruptpath%
                  \pgfsetstrokeopacity{1}%
                  \pgfsetfillopacity{1}%
                  \path[qArrow@stroke,/tikz/qArrow arrow tip style] (0pt,0pt) -- (1pt,0pt);%
                \endpgfinterruptpath%
              }%
          }]%
            (q@start2) \pgfkeysvalueof{/tikz/qArrow/@conn} (q@end2);%
        \fi%
        \pgfkeysgetvalue{/tikz/qArrow/label}{\qArrowTmpLabel}%
        \if\relax\detokenize\expandafter{\qArrowTmpLabel}\relax
        \else
          \node[at=(q@mid),
            text=\pgfkeysvalueof{/tikz/qArrow/color},
            qArrow label style,
            xshift=\pgfkeysvalueof{/tikz/qArrow/label xshift},
            yshift=\pgfkeysvalueof{/tikz/qArrow/label yshift}
      ]{\qArrowTmpLabel};
    \fi
      }%
    }%
  },%
  qArrow/.style={qArrow/base}%
}%

\newcommand{\DeclareQArrowStyle}[2]{%
  \tikzset{#1/.style={qArrow/base,#2}}%
}

\DeclareQArrowStyle{qArrowDashed}{qArrow/pattern=dashed}
\DeclareQArrowStyle{qArrowDotted}{qArrow/pattern=dotted}
\DeclareQArrowStyle{qArrowFinite}{}
\DeclareQArrowStyle{qArrowLoop}{qArrow/connector=loop}
\DeclareQArrowStyle{qArrowLoopDashedPos}{qArrow/connector=loop,qArrow/pattern=dashed}
\DeclareQArrowStyle{qArrowLoopDashedNeg}{qArrow/connector=loop,qArrow/pattern=dashed}

\DeclareQArrowStyle{qArrowAbove}{
  qArrow/yshift=2pt,              % shift whole arrow up
  qArrow/color=black,                 % red line + red label
  qArrow/arrowtip=Stealth,          % stealth tip
  qArrow/arrowpos=0.48,             % arrowhead at 60%
  qArrow/width=0.5pt,               % medium-thin (arbitrary choice)
  qArrow/labelstyle={above, inner sep=1pt},
  qArrow/label xshift=0pt,          % deliberately included (redundant)
  qArrow/label yshift=4pt           % deliberately included (redundant)
}

\DeclareQArrowStyle{qArrowBelow}{
  qArrow/yshift=-2pt,              % shift whole arrow down
  qArrow/color=black,                 % red line + red label
  qArrow/arrowtip=Stealth,          % stealth tip
  qArrow/arrowpos=0.48,             % arrowhead at 75%
  qArrow/width=0.5pt,               % medium-thin (arbitrary choice)
  qArrow/labelstyle={below, inner sep=1pt},
  qArrow/label xshift=-2pt,          % deliberately included (redundant)
  qArrow/label yshift=-4pt           % deliberately included (redundant)
}

\DeclareQArrowStyle{qArrowAboveOp}{
  qArrow/yshift=2pt,              % shift whole arrow up
  qArrow/color=black,                 % red line + red label
  qArrow/arrowtip=Stealth[reversed],          % stealth tip
  qArrow/arrowpos=0.48,             % arrowhead at 75%
  qArrow/width=0.5pt,               % medium-thin (arbitrary choice)
  qArrow/labelstyle={above, inner sep=1pt},
  qArrow/label xshift=0pt,          % deliberately included (redundant)
  qArrow/label yshift=4pt           % deliberately included (redundant)
}

\DeclareQArrowStyle{qArrowBelowOp}{
  qArrow/yshift=-2pt,              % shift whole arrow down
  qArrow/color=black,                 % red line + red label
  qArrow/arrowtip=Stealth[reversed],          % stealth tip
  qArrow/arrowpos=0.48,             % arrowhead at 60%
  qArrow/width=0.5pt,               % medium-thin (arbitrary choice)
  qArrow/labelstyle={below, inner sep=1pt},
  qArrow/label xshift=-2pt,          % deliberately included (redundant)
  qArrow/label yshift=-4pt           % deliberately included (redundant)
}

\DeclareQArrowStyle{qLoop}{
  qArrow/connector=loop,%
  qArrow/loop out=80,%
  qArrow/loop in=100,%
  qArrow/loop width=8mm,%
  qArrow/loop depth=18mm,%
  qArrow/width=0.5pt,%
  qArrow/labelstyle={above, inner sep=1pt},%
  qArrow/arrowpos=0.8,%
  qArrow/label xshift=5pt,%
  qArrow/label yshift=15pt,%
  qArrow/arrowtip={Stealth[length=1.8mm,width=1.8mm]}%
}

\DeclareQArrowStyle{qLoopBelow}{
  qArrow/connector=loop,%
  qArrow/loop out=260,%
  qArrow/loop in=280,%
  qArrow/loop width=8mm,%
  qArrow/loop depth=18mm,%
  qArrow/width=0.5pt,%
  qArrow/labelstyle={below, inner sep=1pt},%
  qArrow/arrowpos=0.8,%
  qArrow/label xshift=5pt,%
  qArrow/label yshift=-15pt,%
  qArrow/arrowtip={Stealth[length=1.8mm,width=1.8mm]}%
}

\DeclareQArrowStyle{qLoopDashed}{%
  qArrow/connector=loop,%
  qArrow/loop out=80,%
  qArrow/loop in=100,%
  qArrow/loop width=8mm,%
  qArrow/loop depth=18mm,%
  qArrow/pattern=dashed,%
  /tikz/qArrow@pattern/.style={dash pattern=on 3pt off 4pt},%
  qArrow/width=1pt,%
  qArrow/labelstyle={above, inner sep=1pt},%
  qArrow/arrowpos=0.8,%
  qArrow/label xshift=5pt,%
  qArrow/label yshift=15pt,%
  qArrow/arrowtip={Stealth[length=1.8mm,width=1.8mm]}%
}%

\DeclareQArrowStyle{qArrowFinite}{%
  qArrow/color=black,%
  qArrow/decor=snake,%
  qArrow/arrowtip={Stealth[length=1.8mm,width=1.8mm]},%
  qArrow/arrowpos=0.5,%
  qArrow/width=0.5pt,%
  qArrow/labelstyle={above, inner sep=1pt},%
  qArrow/label xshift=0pt,%
  qArrow/label yshift=4pt%
}%

\tikzset{styleNodeFr/.style={
		rectangle,
        draw,
        inner sep=2
	}
}

\tikzset{styleNode/.style={
		draw,
        ellipse,
        inner sep=1,
}}

\tikzset{qNodeUnfrozen/.style={
		draw,
        ellipse,
        inner sep=1,
        minimum width=25pt,
        minimum height=20pt,
        fill=white
}}

\tikzset{qNodeUnfrozenCircle/.style={
		draw,
        circle,
        inner sep=1,
        fill=white
}}

\tikzset{qNodeUnfrozenEllipse/.style={
		draw,
        ellipse,
        inner sep=1,
        fill=white
}}

\tikzset{qNodeFrozen/.style={
		rectangle,
        draw,
        inner sep=2,
        fill=white
	}
}

\tikzset{qNodeFrozenRectangle/.style={
		rectangle,
        draw,
        inner sep=2,
        fill=white%,
	}
}

\tikzset{qNodeFrozenSquare/.style={
		square,
        draw,
        inner sep=2,
        fill=white
	}
}

\tikzset{qNodeDiamond/.style={
		diamond,
        draw,
        inner sep=2,
        fill=white
	}
}

\DeclareMathOperator*{\rank}{\text{rank }}

\newcommand{\superp}[2]{\genfrac{}{}{0pt}{}{#1}{#2}}

\newcommand{\pmat}[1]{\begin{pmatrix}#1\end{pmatrix}}
\newcommand{\vb}[1]{\mathbf{#1}}

 \def\d{\delta}

 \def\Im{{\rm Im ~}}
 \def\p{\partial}
 
 \def\a{\alpha}
 \def\b{\beta}
 \def\g{\gamma}
 \def\d{\delta}
 \def\e{\varepsilon}
 \def\eps{\epsilon}

 \def\k{\kappa}
 \def\l{\lambda}

 \def\x{\xi}

 \def\s{\sigma}

 \def\G{\Gamma}
 \def\D{\Delta}

 \def\S{\Sigma}
 \def\L{\Lambda}
 
 \def\o{\omega }
 \def\U{\Upsilon}
\def\CA{{\mathcal{A}}}
\def\CB{{\mathcal{B}}}
\def\CC{{\mathcal{C}}} %conflict with Ethan's complex numbers.
\def\CD{{\mathcal{D}}}
\def\CE{{\mathcal{E}}}
\def\CF{{\mathcal{F}}}

\def\CH{{\mathcal{H}}}

\def\CK{{\mathcal{K}}}

\def\CN{{\mathcal{N}}}
\def\CO{{\mathcal{O}}}
\def\CP{{\mathcal{P}}}
\def\CQ{{\mathcal{Q}}}

\def\CW{{\mathcal{W}}}
\def\CX{{\mathcal{X}}}

\def\la{\left\langle}
\def\ra{\right\rangle}

\def\hf{\dfrac{1}{2}}

\def\implies{\quad\Rightarrow\quad}

\def\vphi{\varphi}

\def\tW{{\tilde{W}}}

\def\bY{\bar Y}

\def\qf{\mathfrak{q}}
\def\tf{\mathfrak{t}}
\def\pf{\mathfrak{p}}

\def\hq{\hat q}
\def\hp{\hat p}

\def\vac{\emptyset}
\def\res{\mathop{\text{Res}}}

\def\bZ{{\bar Z}}
\def\tZ{\tilde{Z}}

\def\mZ{\mathbb{Z}}

\def\mC{\mathbb{C}}

\def\bsn{\boldsymbol{n}}

\def\tW{\tilde{W}}
\def\gl{\mathfrak{gl}}
\def\sl{\mathfrak{sl}}

\def\tX{\tilde{X}}
\def\tY{\tilde{Y}}
\def\tPsi{\tilde{\Psi}}

\def\bsu{{\boldsymbol{u}}}

\def\KK{\mathsf{k}}
\def\CWqt{\mathcal{W}_{\mathfrak{q},\mathfrak{t}}}
\def\CFsub{\CF_{\CQ_N^\text{sub}}[\CWqt^\text{sub}(\sl(N))]}

\def\bsc{{\boldsymbol{c}}}

\newcommand{\normordleft}{{}:\!}
\newcommand{\normordright}{\!:{}}
\DeclareFontFamily{U}{mathx}{}															% used for widecheck command
\DeclareFontShape{U}{mathx}{m}{n}{<-> mathx10}{}										% used for widecheck command
\DeclareSymbolFont{mathx}{U}{mathx}{m}{n}												% used for widecheck command
\DeclareMathAccent{\widecheck}{0}{mathx}{"71}

\theoremstyle{plain}
\newtheorem{theorem}{Theorem}[section]

\newtheorem{lemma}[theorem]{Lemma}
\newtheorem{proposition}[theorem]{Proposition}

\newtheorem{definition}[theorem]{Definition}
\newtheorem{example}[theorem]{Example}
\newtheorem{remark}[theorem]{Remark}

\AddToHook{env/lemma/begin}{\zcsetup{countertype={theorem=lemma}}}
\AddToHook{env/definition/begin}{\zcsetup{countertype={theorem=definition}}}
\AddToHook{env/proposition/begin}{\zcsetup{countertype={theorem=proposition}}}
\AddToHook{env/corollary/begin}{\zcsetup{countertype={theorem=corollary}}}
\AddToHook{env/conjecture/begin}{\zcsetup{countertype={theorem=conjecture}}}
\AddToHook{env/remark/begin}{\zcsetup{countertype={theorem=remark}}}
\AddToHook{env/example/begin}{\zcsetup{countertype={theorem=example}}}

\numberwithin{equation}{section}

\makeatletter
\newbox\dottedarrow@box
\setbox\dottedarrow@box\hbox
  {%
    \begin{tikzpicture}
      \draw[dotted,->] (0,0) -- (1.5em,0);
    \end{tikzpicture}%
  }
\makeatother

\newcommand{\no}[1]{\mathopen{:} #1 \mathclose{:}}										% normal ordering i.e. normally ordered product (prevent := or =:)
\newcommand{\qasltwo}[1][\qq]{ U_{#1} \brac{\asltwo} }
\NewDocumentCommand{\qwgen}{O{i}}{W_{#1}}
\newcommand{\affine}[1]{\widehat{#1}}
\newcommand{\AKMA}[2]{\affine{\liea{#1}}_{#2}} 						         			% Kac-Moody algebras
\newcommand{\asltwo}{\AKMA{sl}{2}}

\newcommand{\fld}[1]{\mathbb{#1}}  														% for fields and related things
\newcommand{\VOA}[1]{\mathsf{#1}}    													% for VOAs
\newcommand{\liea}[1]{\mathfrak{#1}}  													% for Lie algebras
\newcommand{\ZZ}{\fld{Z}}
\renewcommand{\simeq}{\cong}              												% which one should we use???
\renewcommand{\ge}{\geqslant} 															% never use \geq or \geqslant, use \ge which is globally defined to be one or the other
\renewcommand{\le}{\leqslant} 															% same for \le and \leslant (and we should use \ne instead of \neq for consistency I guess)
\DeclareMathAccent{\widecheck}{0}{mathx}{"71}
\newcommand{\wun}{\mathds{1}}  															% identity field/operator (uses dsfont)
\DeclarePairedDelimiter{\brac}{\lparen}{\rparen}									    % use \brac for (...) and \brac* to automatically scale the ( and )
\DeclarePairedDelimiterX{\comm}[2]{\lbrack}{\rbrack}{#1 , #2}  							%commutators
\newcommand{\kdelta}[2]{\delta_{#1 , #2}}		 										% Kronecker delta
\newcommand{\ee}{ e}   																	% Eulers constant: ln e = 1
\newcommand{\alg}[1]{\mathfrak{#1}}  													% for Lie algebras
\newcommand{\SLA}[2]{\alg{#1}_{#2}}     					            	     		% Lie algebras like sl(2)

\newcommand{\sltwo}{\SLA{sl}{2}}

\DeclarePairedDelimiterXPP{\uaff}[2]{\VOA{V}^{#1}}{\lparen}{\rparen}{}{#2}            % universal affine VA
\DeclarePairedDelimiterXPP{\saff}[2]{\VOA{L}_{#1}}{\lparen}{\rparen}{}{#2}            % simple affine VA

\newcommand{\bgsymb}{\VOA{G}}
\newcommand{\lsymb}{\Pi}

\newcommand{\lvoa}{\lsymb}                         % the half-lattice vertex algebra

\newcommand{\usltwo}[1][\KK]{\uaff{\KK}{\sltwo}}

\DeclarePairedDelimiterX{\bilin}[2]{\langle}{\rangle}{#1 , #2} 							%bilinear form
\DeclareMathOperator{\spn}{span}														% span
\newcommand{\cspn}{\spn_{\CC}}															% complex span
\newcommand{\bgvoa}{\bgsymb}                       % bosonic ghosts VOA
\newcommand{\heis}{\VOA{H}}                        % Heisenberg VOA
\DeclarePairedDelimiterX{\setbar}[2]{\lbrace}{\rbrace}{#1 \,\delimsize\vert\,\mathopen{} #2}  % much better
\newcommand{\liesl}{\liea{sl}}															% Special Linear Lie Algebra
\newcommand{\wt}[1]{\widetilde{#1}}
\newcommand{\wtaa}{\wt{a}}
\newcommand{\wtb}{\wt{b}}
\newcommand{\wtc}{\wt{c}}
\newcommand{\wtd}{\wt{d}}
\newcommand{\wta}{\wt{\alpha}}

\newcommand{\cA}{\mathbf{A}}
\newcommand{\cGp}{\mathbf{G}^+}
\newcommand{\cGm}{\mathbf{G}^-}
\newcommand{\cL}{\mathbf{L}}
\newcommand{\cJ}{\mathbf{J}}
\newcommand{\cU}{\mathbf{U}}
\newcommand{\cT}{\mathbf{T}}
\newcommand{\cW}{\mathbf{W}}
\newcommand{\cG}{\mathbf{G}}

\DeclareRobustCommand{\SkipTocEntry}[5]{}

\usepackage{xpatch}
\makeatletter
\xpatchcmd{\@tocline}
          {\hfil\hbox to\@pnumwidth{\@tocpagenum{#7}}\par}
          {\ifnum#1<0\hfill\else\dotfill\fi\hbox to\@pnumwidth{\@tocpagenum{#7}}\par}
          {}{}
\makeatother

\usepackage{enumitem}
\DeclareRobustCommand{\backtotoc}[1]{%
  \hyperlink{toc}{#1}%
}

\let\oldsection\section
\let\oldsubsection\subsection

\RenewDocumentCommand{\section}{som}{%
  \IfBooleanTF{#1}{%
    \oldsection*{#3}%
  }{%
    \IfNoValueTF{#2}{%
      \oldsection[\texorpdfstring{#3}{#3}]{\backtotoc{#3}}%
    }{%
      \oldsection[\texorpdfstring{#2}{#2}]{\backtotoc{#3}}%
    }%
  }%
}

\RenewDocumentCommand{\subsection}{som}{%
  \IfBooleanTF{#1}{%
    \oldsubsection*{#3}%
  }{%
    \IfNoValueTF{#2}{%
      \oldsubsection[\texorpdfstring{#3}{#3}]{\backtotoc{#3}}%
    }{%
      \oldsubsection[\texorpdfstring{#2}{#2}]{\backtotoc{#3}}%
    }%
  }%
}

\begin{document}
\title{Free field realizations of deformed W-algebras and cluster algebras: subregular W-algebras and Inverse Quantum Hamiltonian Reduction}
\author{Mikhail Bershtein, Jean-Emile Bourgine and Ethan Fursman}
% \date{May 2026}

\begin{titlepage}

\begin{center}
{\huge Deformed W-algebras and chiralized cluster\\

seeds: subregular W-algebras and Inverse Quantum\\
\vskip .5cm

Hamiltonian Reduction}

\vskip 1cm
{\Large Mikhail Bershtein\footnote{mbersht@sissa.it}$^\ast$, Jean-Emile Bourgine\footnote{bourgine@simis.cn}$^\dagger$ and Ethan Fursman\footnote{efursman@student.unimelb.edu.au}$^\circ$}\\

\vskip 1cm
{\it $\ast$ Scuola Internazionale Superiore di Studi Avanzati (SISSA)}\\
{\it in Via Bonomea 265, Trieste, Italy}\\
{\it \&}\\
{\it INFN, Section of Trieste}\\
{\it via Valerio 2, 34127 Trieste, Italy}
\vskip 1cm
{\it $\dagger$ Center for Mathematics and Interdisciplinary Sciences}\\
{\it Fudan University, Shanghai 200433, China}\\
{\it \&}\\
{\it Shanghai Institute for Mathematics and Interdisciplinary Sciences (SIMIS)}\\
{\it Shanghai 200433, China}

\vskip 1cm
{\it $^\circ$ School of Mathematics and Statistics}\\
{\it University of Melbourne}\\
{\it Parkville, Victoria 3010, Australia}\\

\end{center}
\vfill
\begin{abstract}
The recently introduced formalism of chiral cluster seeds replaces quantum cluster variables with deformed vertex operators. In this framework, a decorated quiver associated with a seed encodes the operator product expansions of the corresponding vertex operators. This formalism is applied to several $(q,t)$-deformed W-algebras, including $\mathcal{W}_{\mathfrak{q},\mathfrak{t}}(\mathfrak{gl}(N|M))$, $U_q(\widehat{\mathfrak{sl}}_2)$, and the deformed Bershadsky--Polyakov algebra. In particular, it is shown that different free field realizations of the currents are related by mutations of the associated chiral cluster seed.

The second part of the paper introduces a $(q,t)$-deformation of the subregular W-algebras, denoted by $\mathcal{W}_{\mathfrak{q},\mathfrak{t}}^{\text{sub}}(\mathfrak{sl}(N))$. All free field realizations obtainable through seed mutations are described. An embedding of $\mathcal{W}_{\mathfrak{q},\mathfrak{t}}^{\text{sub}}(\mathfrak{sl}(N))$ into the free field realization of $\mathcal{W}_{\mathfrak{q},\mathfrak{t}}(\mathfrak{sl}(N))$ tensored with a rank-two Heisenberg algebra is constructed. This embedding may be viewed as a deformed analogue of inverse quantum Hamiltonian reduction. The relation between the subregular algebras and $\mathcal{W}_{\mathfrak{q},\mathfrak{t}}(\mathfrak{gl}(1|N))$ is also discussed.
\end{abstract}
\vfill
\end{titlepage}

\setcounter{footnote}{0}

\newpage
\tableofcontents

\newpage

\section{Introduction}\label{sec:Intro}
Deformed W-algebras were introduced in the 1990s as a bridge between quantum groups and conformal W-algebras \cite{AKSOvir96,AKOSwalg97,Feigin1996quantum,Frenkel1997}. 
These two-parameter $(\qf,\tf)$-deformations of W-algebras exhibit deep connections to symmetric functions, and also enter in the description of some quantum integrable systems like the ABF model \cite{Asai1996,Jimbo2000}. 
Recently, a renewed interest in these algebraic structures was sparked by the correspondences with 5D $\CN=1$ supersymmetric gauge theories \cite{Awata:2009ur,Nekrasov:2015wsu,Kimura:2015rgi,Haouzi:2023doo}, and by their connection to the representation theory of quantum toroidal algebras \cite{Feigin:2021jzm,FJMV2020,FJM2022,Bershtein2018}.

The representation theory of conformal W-algebras, seen as a special class of vertex algebras, is a very rich and very active topic of research. 
In comparison, the subject appears to be underdeveloped for their deformed counterparts. 
One of the major difficulties lies in the absence of a good definition of these algebras.
In the conformal setting, W-algebras admit several equivalent constructions, including quantum Hamiltonian reduction, screening charges, or coset realizations. For deformed W-algebras, however, these approaches are no longer manifestly equivalent, and each presents its own technical challenges. Much of the difficulty comes from the absence of a satisfactory deformation of the operator-state correspondence. 
% Indeed, in the conformal case, several equivalent definitions of W-algebras are known, for example quantum Hamiltonian reduction, screening charges, or coset construction. 
% In the deformed case, each of these definitions presents some difficulties. 
% The main difficulty lies in the deformation of the operator-state correspondence.

In the literature, deformed W-algebras are most commonly defined through \emph{free field realization}. More precisely, the algebra is defined as a subalgebra generated by certain currents acting on the Fock space of a Heisenberg algebra. 
The caveat is that a single deformed W-algebra often possesses several different looking free field realizations, and one has to show their equivalence. 
Sometimes this is possible due to the presence of an extended symmetry, such as some version of the quantum toroidal algebra mentioned above. 
For instance, the quantum toroidal \(\mathfrak{gl}(1)\) algebra provides a natural framework to study the deformation of Corner VOAs \cite{Gaiotto:2017euk,Harada:2021xnm,Harada:2020woh,Bershtein2018,FJMV2020}.

Another question in the approach above is how to identify the generators of deformed W-algebras. The traditional method is to impose the condition of commutativity with screening currents, exactly as in the conformal case. 
This was the original approach to deformed W-algebras developed in \cite{Frenkel1997}. 
More recently, the qq-character formalism has emerged as an important tool for the construction of these generators \cite{Nekrasov:2015wsu,Kimura:2015rgi,Feigin:2021jzm,FJMV2020}.

% Another question in the approach above is where to find reasonable generators. 
% The basic method is to impose the condition of commutativity with screening currents. 
% This is analogous to the conformal algebra case. 
% This was the original approach in \cite{Frenkel1997}. 
% Recently the powerful qq-character technique was developed in this context \cite{Nekrasov:2015wsu,Kimura:2015rgi,Feigin:2021jzm,FJMV2020}. 

A different approach to the problem of constructing generators was proposed recently in terms of chiral cluster seeds \cite{Bershtein2025}. In this framework, one starts from an (iced) quiver with frozen vertices of different types and with additional parameters (complex weights) assigned to arrows. 
Associated to this data are a Heisenberg algebra $\CA_\CQ$, a collection of vertex operators attached to the vertices of the quiver, and screening charges associated to unfrozen vertices. 
The arrows encode the algebraic relations (OPEs) between these vertex operators. 
Different free field realizations are then related by a chiral analogue of cluster mutations. From this perspective, the condition of commutativity with screening charges is replaced by the Laurent property after mutation.

% Finally, recently another (related) method was proposed \cite{Bershtein2025}. 
% Namely, one can realize free field realizations in terms of a chiral cluster seed. 
% The combinatorial data for such a seed is a (iced) quiver with frozen vertices of different types and, most importantly, additional parameters (complex weights) assigned to arrows. 
% Associated to this quiver are a Heisenberg algebra $\CA_\CQ$ and a set of vertex operators attached to each node. 
% Arrows of the quiver encode the algebraic relations between these vertex operators (OPEs). 
% In addition, a set of screening charges is also associated with unfrozen nodes. 
% The different free field realizations in this approach are now connected to chiral analogs of the mutations. 
% The condition of commutativity with screenings is replaced by the Laurent property after mutation.

We review this approach in Section~\ref{sec:Cluster}. In section ~\ref{sec:def W and free field}, we apply it to several known examples of deformed W-algebras, including $\CWqt(\gl(N|M))$, the quantum affine $\sl(2)$ algebra, and the $q$-deformed Bershadsky-Polyakov (BP) algebra \cite{Harada:2020woh,FJM2022}. In each case, the algebra admits several (but finitely many) chiral cluster seeds realizations (equivalently, several free field realizations), and we show that the generating currents are Laurent in all of them. In Figures~\ref{fig:introCWgl4},~\ref{fig:introCWgl33} we present examples of the corresponding quivers.

\begin{figure}[h]
	\centering
	\begin{tikzpicture}[scale=0.85, transform shape, font=\small]
		\def\xs{2.5}
		\def\ys{1}

		%Coordinates
		\coordinate (introgl0) at (0*\xs,0*\ys);
		\coordinate (introgl1) at (1*\xs,0*\ys);
		\coordinate (introgl2) at (2*\xs,0*\ys);
		\coordinate (introgl3) at (3*\xs,0*\ys);
		\coordinate (introgl4) at (4*\xs,0*\ys);

		%Horizontal arrows
		\draw[qArrowAbove](introgl0) to (introgl1);
		\draw[qArrowBelow, qArrow/label={$q_3$}](introgl1) to (introgl0);
		\draw[qArrowAbove](introgl1) to (introgl2);
		\draw[qArrowBelow, qArrow/label={$q_3$}](introgl2) to (introgl1);
		\draw[qArrowAbove](introgl2) to (introgl3);
		\draw[qArrowBelow, qArrow/label={$q_3$}](introgl3) to (introgl2);
		\draw[qArrowAbove](introgl3) to (introgl4);
		\draw[qArrowBelow, qArrow/label={$q_3$}](introgl4) to (introgl3);

		%Loops
		\draw[qLoopDashed, qArrow/label={$q_1$, `$-$'}](introgl0) to (introgl0);
		\draw[qLoop, qArrow/label={$q_1$}](introgl1) to (introgl1);
		\draw[qLoop, qArrow/label={$q_1$}](introgl2) to (introgl2);
		\draw[qLoop, qArrow/label={$q_1$}](introgl3) to (introgl3);
		\draw[qLoopDashed, qArrow/label={$q_1$, `$+$'}](introgl4) to (introgl4);

		%Nodes
		\node[qNodeFrozen] (introgl0) at (introgl0) {$X^{(-)}_0$};
		\node[qNodeUnfrozen] (introgl1) at (introgl1) {$X_1$};
		\node[qNodeUnfrozen] (introgl2) at (introgl2) {$X_2$};
		\node[qNodeUnfrozen] (introgl3) at (introgl3) {$X_3$};
		\node[qNodeFrozen] (introgl4) at (introgl4) {$X^{(+)}_4$};
	\end{tikzpicture}
	\caption{Decorated quiver for the free field realization of \(\CWqt(\gl(4))\).}
	\label{fig:introCWgl4}
\end{figure}
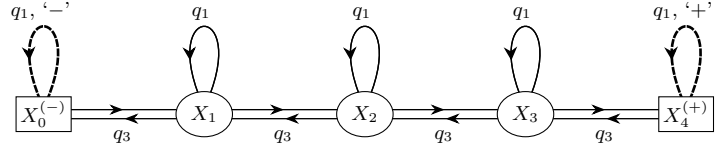

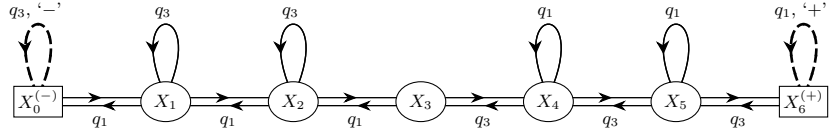
\begin{figure}[h]
	\centering
	\begin{tikzpicture}[scale=0.75, transform shape, font=\small]
		\def\xs{2.25}
		\def\ys{1}

		%Coordinates
		\coordinate (introgls0) at (0*\xs,0*\ys);
		\coordinate (introgls1) at (1*\xs,0*\ys);
		\coordinate (introgls2) at (2*\xs,0*\ys);
		\coordinate (introgls3) at (3*\xs,0*\ys);
		\coordinate (introgls4) at (4*\xs,0*\ys);
		\coordinate (introgls5) at (5*\xs,0*\ys);
		\coordinate (introgls6) at (6*\xs,0*\ys);

		%Horizontal arrows
		\draw[qArrowAbove](introgls0) to (introgls1);
		\draw[qArrowAbove](introgls1) to (introgls2);
		\draw[qArrowAbove](introgls2) to (introgls3);
		\draw[qArrowAbove](introgls3) to (introgls4);
		\draw[qArrowAbove](introgls4) to (introgls5);
		\draw[qArrowAbove](introgls5) to (introgls6);
		\draw[qArrowBelow, qArrow/label={$q_1$}](introgls1) to (introgls0);
		\draw[qArrowBelow, qArrow/label={$q_1$}](introgls2) to (introgls1);
		\draw[qArrowBelow, qArrow/label={$q_1$}](introgls3) to (introgls2);
		\draw[qArrowBelow, qArrow/label={$q_3$}](introgls4) to (introgls3);
		\draw[qArrowBelow, qArrow/label={$q_3$}](introgls5) to (introgls4);
		\draw[qArrowBelow, qArrow/label={$q_3$}](introgls6) to (introgls5);

		%Loops
		\draw[qLoopDashed, qArrow/label={$q_3$, `$-$'}](introgls0) to (introgls0);
		\draw[qLoop, qArrow/label={$q_3$}](introgls1) to (introgls1);
		\draw[qLoop, qArrow/label={$q_3$}](introgls2) to (introgls2);
		\draw[qLoop, qArrow/label={$q_1$}](introgls4) to (introgls4);
		\draw[qLoop, qArrow/label={$q_1$}](introgls5) to (introgls5);
		\draw[qLoopDashed, qArrow/label={$q_1$, `$+$'}](introgls6) to (introgls6);

		%Nodes
		\node[qNodeFrozen] (introgls0) at (introgls0) {$X^{(-)}_0$};
		\node[qNodeUnfrozen] (introgls1) at (introgls1) {$X_1$};
		\node[qNodeUnfrozen] (introgls2) at (introgls2) {$X_2$};
		\node[qNodeUnfrozen] (introgls3) at (introgls3) {$X_3$};
		\node[qNodeUnfrozen] (introgls4) at (introgls4) {$X_4$};
		\node[qNodeUnfrozen] (introgls5) at (introgls5) {$X_5$};
		\node[qNodeFrozen] (introgls6) at (introgls6) {$X^{(+)}_6$};
	\end{tikzpicture}
	\caption{Decorated quiver for the free field realization of \(\CWqt(\gl(3|3))\).}
	\label{fig:introCWgl33}
\end{figure}

In Section ~\ref{sec:subreg}, we turn to the family of deformed subregular \(W\)-algebras \(\CWqt^\text{sub}(\sl(N))\), also called Feigin-Semikhatov algebras, which includes the deformed \(\widehat{\mathfrak{sl}}_2\) and BP algebra as the cases \(N=2\) and \(N=3\), respectively. These algebras have been previously studied in \cite{Harada:2020woh} and \cite{FJM2022}, using different free field realizations. In particular, it was conjectured in \cite{FJM2022} that free field realizations of these algebras are labelled by the Dynkin diagrams of \(\mathfrak{gl}(1|N)\). In this section, we provide the explicit constructions of these realizations, and show their equivalence under chiral cluster mutations. Then, using a specific realization, we prove the existence of an embedding \(\CWqt^\text{sub}(\sl(N))\hookrightarrow \CWqt(\sl(N))\otimes \CH_2\), where \(\CH_2\) is a rank 2 Heisenberg algebra.
This embedding may be viewed as a deformation of the inverse quantum Hamiltonian reduction~\cite{Semikhatov1994,Adamovic2019,Fehily2024}. The description of these deformed W-algebras in the framework of chiral cluster seeds, and the resulting embedding theorem \ref{thm:IQQHR} are the main results of this paper.

For example, in the Fig.~\ref{fig:introSubregSl4} we present one of the quivers for \(\CWqt^\text{sub}(\sl(4))\). Its unfrozen part coincides with unfrozen part of the quiver of $\CWqt(\gl(1|4))$. This is a combinatorial projection of the relation between corresponding algebras suggested in \cite{Feigin:2021jzm}. Furthermore, this unfrozen part also contains the unfrozen part of the $\CWqt(\gl(4))$. This is a combinatorial projection of the deformed inverse quantum Hamiltonian reduction.

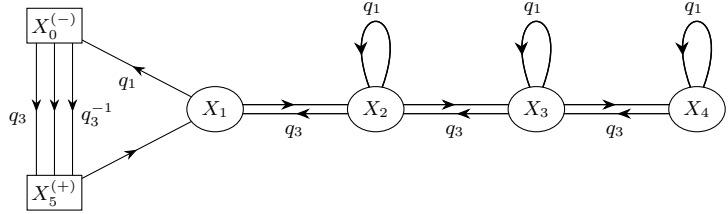
\begin{figure}[h]
	\centering
	\begin{tikzpicture}[scale=0.85, transform shape, font=\small]
		\def\xs{2.5}
		\def\ys{1.3}

		%Coordinates
		\coordinate (introsub0) at (0*\xs,1*\ys);
		\coordinate (introsub1) at (1*\xs,0*\ys);
		\coordinate (introsub2) at (2*\xs,0*\ys);
		\coordinate (introsub3) at (3*\xs,0*\ys);
		\coordinate (introsub4) at (4*\xs,0*\ys);
		\coordinate (introsub5) at (0*\xs,-1*\ys);

		%Frozen arrows
		\draw[qArrow, qArrow/xshift=-8pt, qArrow/label={$q_3$}, qArrow/label xshift=-8pt](introsub0) to (introsub5);
		\draw[qArrow](introsub0) to (introsub5);
		\draw[qArrow, qArrow/xshift=8pt, qArrow/label={$q_3^{-1}$}, qArrow/label xshift=11pt,
    qArrow/label yshift=4pt](introsub0) to (introsub5);
		\draw[qArrow, qArrow/label={$q_1$}, qArrow/label xshift=-3pt, qArrow/label yshift=-3pt](introsub1) to (introsub0);
		\draw[qArrow](introsub5) to (introsub1);

		%Double opposing arrows
		\draw[qArrowAbove](introsub1) to (introsub2);
		\draw[qArrowBelow, qArrow/label={$q_3$}](introsub2) to (introsub1);
		\draw[qArrowAbove](introsub2) to (introsub3);
		\draw[qArrowBelow, qArrow/label={$q_3$}](introsub3) to (introsub2);
		\draw[qArrowAbove](introsub3) to (introsub4);
		\draw[qArrowBelow, qArrow/label={$q_3$}](introsub4) to (introsub3);

		%Loops
		\draw[qLoop, qArrow/label={$q_1$}](introsub2) to (introsub2);
		\draw[qLoop, qArrow/label={$q_1$}](introsub3) to (introsub3);
		\draw[qLoop, qArrow/label={$q_1$}](introsub4) to (introsub4);

		%Nodes
		\node[qNodeFrozen] (introsub0) at (introsub0) {$X^{(-)}_0$};
		\node[qNodeUnfrozen] (introsub1) at (introsub1) {$X_1$};
		\node[qNodeUnfrozen] (introsub2) at (introsub2) {$X_2$};
		\node[qNodeUnfrozen] (introsub3) at (introsub3) {$X_3$};
		\node[qNodeUnfrozen] (introsub4) at (introsub4) {$X_4$};
		\node[qNodeFrozen] (introsub5) at (introsub5) {$X^{(+)}_5$};
	\end{tikzpicture}
	\caption{Decorated quiver for the deformed subregular algebra \(\CWqt^\text{sub}(\sl(4))\).}
	\label{fig:introSubregSl4}
\end{figure}

Finally, in Section~\ref{sec:conformal limit} we study the conformal limit of the algebra \(\CWqt^\text{sub}(\sl(N))\), its screening charges, and the inverse quantum Hamiltonian reduction.

\paragraph{Future directions.}
The present work focuses on chiral cluster seeds of finite type, namely those with finitely many seeds reachable by mutations. 
In addition, we only consider mutations at fermionic vertices, although mutations at vertices with a single loop are also known \cite{MishaLectures}.\footnote{In all the examples treated in this paper, unfrozen vertices have at most one loop.} The most interesting examples, however, are expected to be of infinite type. In that setting, the Laurent property can no longer be verified by checking finitely many seeds, suggesting the need for a chiral analogue of the  algebra of global functions. We hope to return to this problem in future work~\cite{BSS2026}. 

Another natural direction is the extension of the chiral cluster seed formalism to deformed W-algebras beyond the regular and subregular nilpotent cases considered here. Evidence for such a generalization comes from the case of zero nilpotent, i.e. free field realizations of the quantum affine \(\widehat{\mathfrak{sl}}(N)\) algebras, whose known free field realization \cite{AOS94} also arise from a chiral cluster seed \cite{BBS2026}. This suggests that chiral cluster seed realizations may exist for deformed W-algebras associated with any nilpotent orbit (at least in type A), possibly up to additional Heisenberg factors. 
Furthermore, one may expect a natural action of the shifted quantum toroidal algebras introduced recently \cite{FeiginJimboMukhin:2025extensions,FeiginJimboMukhin:2025affinization}.

More generally, chiral cluster mutations provide a natural mechanism for relating different free field realizations of the same algebra, and offers a useful tool for studying connections between distinct deformed W-algebras. In this paper, we use this idea to obtain a $(\qf,\tf)$-deformation of the inverse quantum Hamiltonian reduction for the subregular W-algebra. We expect the same method to apply in other cases as well. For example, for the zero nilpotent for the algebra \(\widehat{\mathfrak{sl}}(3)\), one can mutate the corresponding seed to the form shown in Fig.~\ref{fig:Uqsl3} where a subquiver associated with the subregular \(W\)-algebra $\CWqt^\text{sub}(\sl(3))$ becomes visible.

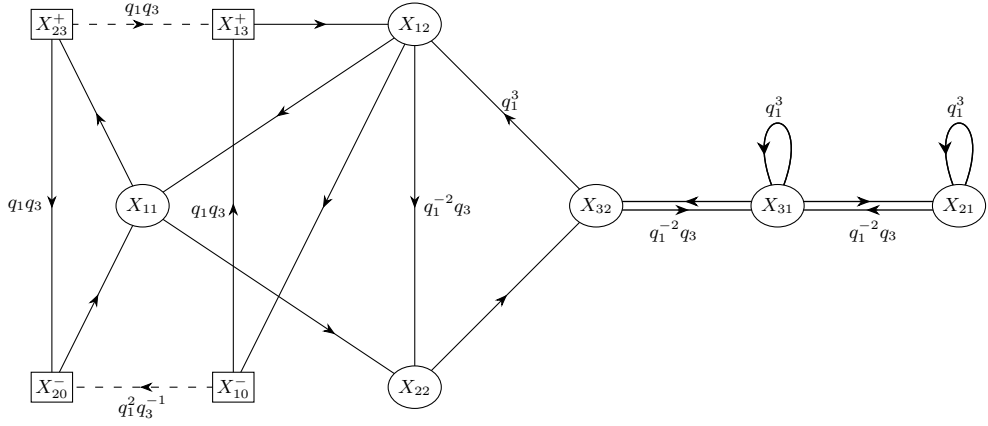
\begin{figure}
\begin{center}
\begin{tikzpicture}[scale=.8, transform shape, font=\small]

%------------------------------------------------------------------------------------
% Coordinates
%------------------------------------------------------------------------------------

\coordinate (X23) at (3,3);
\coordinate (X13) at (6,3);
\coordinate (X11) at (4.5,0);

\coordinate (X20) at (3,-3);
\coordinate (X10) at (6,-3);
\coordinate (X22) at (9,-3);
\coordinate (X12) at (9,3);

\coordinate (X32) at (12,0);
\coordinate (X31) at (15,0);
\coordinate (X21) at (18,0);

%------------------------------------------------------------------------------------%
%Frozen-node part
%------------------------------------------------------------------------------------%
\draw[qArrowDashed,qArrow/label={$q_1q_3$},qArrow/labelstyle={above,inner sep=1pt}, qArrow/label yshift=3pt](X23) to (X13);
\draw[qArrowDashed,qArrow/label={$q_1^2q_3^{-1}$},qArrow/labelstyle={below,inner sep=1pt}, qArrow/label yshift=-3pt](X10) to (X20);
\draw[qArrow](X13) to (X12);
\draw[qArrow](X12) to (X10);
\draw[qArrow,qArrow/label={$q_1q_3$},qArrow/labelstyle={below left,inner sep=1pt},qArrow/label xshift=-2pt,qArrow/label yshift=1pt,qArrow/arrowpos=0.5](X10) to (X13);
\draw[qArrow,qArrow/label={$q_1q_3$},qArrow/labelstyle={left,inner sep=1pt},qArrow/label xshift=-3pt,qArrow/arrowpos=0.5](X23) to (X20);
%------------------------------------------------------------------------------------%
%Unfrozen-node part
%------------------------------------------------------------------------------------%
\draw[qArrow](X12) to (X11);
\draw[qArrow](X20) to (X11);
\draw[qArrow](X11) to (X23);
\draw[qArrow, qArrow/arrowpos=0.7](X11) to (X22);
\draw[qArrow,qArrow/label={$q_1^{-2}q_3$},qArrow/labelstyle={right,inner sep=1pt},qArrow/label xshift=3pt](X12) to (X22);
\draw[qArrow](X22) to (X32);
\draw[qArrow,qArrow/label={$q_1^3$},qArrow/labelstyle={above left,inner sep=1pt},qArrow/label xshift=8pt, qArrow/label yshift=1pt](X32) to (X12);
%------------------------------------------------------------------------------------%
%Double arrows
%------------------------------------------------------------------------------------%
\draw[qArrowAboveOp,qArrow/arrowpos=0.55](X32) to (X31);
\draw[qArrowBelowOp,qArrow/label={$q_1^{-2}q_3$},qArrow/arrowpos=0.55](X31) to (X32);
\draw[qArrowAboveOp,qArrow/arrowpos=0.55](X21) to (X31);
\draw[qArrowBelowOp,qArrow/label={$q_1^{-2}q_3$},qArrow/arrowpos=0.55](X31) to (X21);
%------------------------------------------------------------------------------------%
%Loops
%------------------------------------------------------------------------------------%
\draw[qLoop,qArrow/label={$q_1^3$}](X31) to (X31);
\draw[qLoop,qArrow/label={$q_1^3$}](X21) to (X21);

%------------------------------------------------------------------------------------
% Nodes
%------------------------------------------------------------------------------------

\node[qNodeUnfrozen] (X12node) at (X12) {$X_{12}$};
\node[qNodeUnfrozen] (X11node) at (X11) {$X_{11}$};
\node[qNodeUnfrozen] (X22node) at (X22) {$X_{22}$};
\node[qNodeUnfrozen] (X32node) at (X32) {$X_{32}$};
\node[qNodeUnfrozen] (X31node) at (X31) {$X_{31}$};
\node[qNodeUnfrozen] (X21node) at (X21) {$X_{21}$};

\node[qNodeFrozen] (X23node) at (X23) {$X_{23}^{+}$};
\node[qNodeFrozen] (X13node) at (X13) {$X_{13}^{+}$};
\node[qNodeFrozen] (X20node) at (X20) {$X_{20}^{-}$};
\node[qNodeFrozen] (X10node) at (X10) {$X_{10}^{-}$};

\end{tikzpicture}
\end{center}
\caption{%
Decorated quiver expected to correspond to a free field realization of $U_q(\widehat{\sl(3)})$}
\label{fig:Uqsl3}
\end{figure}

% Another important question is the extension of the IQHR embeddings to more general algebras. 
% In this direction, the next step is to construct the embedding of the quantum affine $\sl(3)$ algebra in the deformed BP algebra extended by Heisenberg algebras. 
% To do this, we need to find a suitable free field realization of $U_q(\widehat{\sl(3)})$. 
% Starting from the decorated quiver of the rectangular realization [Misha-JE-Junishi], and performing a certain sequence of mutations, we find the quiver represented in Figure \ref{fig:Uqsl3}. 
% If well-defined, the corresponding free field realization associated is a good candidate for the IHQR. 
% Indeed, the tail of bosonic nodes suggests a relation to the deformed BP algebras (in the second realization), by removing the top nodes (i.e. all frozen nodes and $X_{11}$), and freezing the nodes $X_{12}$ and $X_{22}$. 
% We plan explore this further in a forthcoming publication. 

% \paragraph{Organisation} 
% This paper is organized as follows. 
% In the second section, we introduce the notion of chiral cluster seed, define mutations and screening charges. 
% In the third section, we apply this formalism to algebras $\CWqt(\gl(N))$, $\CWqt(\gl(N|M))$, $U_q(\widehat{\sl(2)})$ and deformed BP. 
% The fourth section is dedicated to the study of the algebra $\CFsub$. 
% Finally, we examine the conformal limit in the last section.

\paragraph{Notations} 
In this paper, $q_1,q_2,q_3$ denote three complex numbers such that $q_1^aq_2^bq_3^c=1$ implies $a=b=c$. 
Following \cite{FJMV2020}, we also introduce the notation for $c=1,2,3$,
\begin{equation}\label{def_sc_tc}
     s_c=q_c^{\frac12},\qquad t_c=s_c-s_c^{-1}.
\end{equation}
We will also denote $s_2=q$. 
In terms of $\CWqt$-algebras parameters, unless otherwise stated, we identify $(\qf,\tf,\pf)=(q_2,q_1^{-1},q_3^{-1})$.% introduce $\b$ such that $\tf=\qf^{\b}$.

\paragraph{Acknowledgements} 
The authors thank B.~Feigin, G.~Schrader, A.~Shapiro, J. Shiraishi for many enlightening discussions. M.B. is grateful to SIMIS for hospitality.

%%%%%%%%%%%%%%%%%%%%%%%%%%%%%%%%%%%%%%%%%%%%%%%%%%%%%%%%%%%%%%%%%%%%%%%%%%%%%%%%%%%%%%%%%%%%%%%%%%%%%%%%
%------------------------------------------------------------------------------------------------------%
%------------------------------------------------------------------------------------------------------%
%%%%%%%%%%%%%%%%%%%%%%%%%%%%%%%%%%%%%%%%%%%%%%%%%%%%%%%%%%%%%%%%%%%%%%%%%%%%%%%%%%%%%%%%%%%%%%%%%%%%%%%%
\section[Cluster Algebras]{Cluster Algebras}\label{sec:Cluster}
%%%%%%%%%%%%%%%%%%%%%%%%%%%%%%%%%%%%%%%%%%%%%%%%%%%%%%%%%%%%%%%%%%%%%%%%%%%%%%%%%%%%%%%%%%%%%%%%%%%%%%%%
This section develops the formalism of chiral cluster seeds introduced by one of the authors \cite{MishaLectures}. After a short review on quantum cluster algebras, we explain how to extend this framework by promoting quantum variables to vertex operators. Chiral cluster seeds will be represented as a certain decorated quiver. Each vertex of this quiver is associated to a vertex operator, and its arrows encode their algebraic properties. We also introduce the notion of screening charges, and define mutations at unfrozen vertices without loop.

\subsection[Quantum Cluster Algebras]{Quantum Cluster Algebras}\label{sec:Cluster_qCluster}
%%%%%%%%%%%%
We briefly review here the definition of cluster seeds and their mutations, and refer to \cite{FockGoncharov2009,Bershtein2025} for more information on quantum cluster algebras.

\begin{definition}\label{def:cluster}
	A \emph{cluster seed} \(\mathsf{s}\) consists of the following data \((I,I_{\mathrm{f}}, \epsilon, \mathbf{X})\), where 
	\begin{itemize}
		\item \(I\) is a finite set, \(I_{\mathrm{f}}\subset I\) is a frozen subset, and we set $I_u=I\setminus I_f$
		\item \(\epsilon\) is an anti-symmetric matrix with rows and columns labeled by \(I\). The matrix element \(\epsilon_{ij}\in \frac12 \mathbb{Z}\), \(\forall i,j \in I\) and, moreover  \(\epsilon_{ij}\in \mathbb{Z}\), if \( i \in I_u\) or \(j \in I_u\). 
		\item \(\mathbf{X}\) is a set of variables  \(\mathbf{X}=(X_i\mid i\in I)\) labeled by \(I\), and subject to the relations
		\begin{equation}\label{eq:Xi Xj}
			X_i X_j= q^{2\epsilon_{ij}} X_jX_i.
		\end{equation}
	\end{itemize}
    For any given seed we assign the quantum torus \(\mathbb{C}_q[\mathcal{X}_\mathsf{s}]\) that is the algebra with a basis of Laurent monomials on \(\mathbf{X}\) subject to the relations \eqref{eq:Xi Xj}.
\end{definition}

For a quantum torus algebra with basis $X_i$ satisfying \ref{eq:Xi Xj}, we define inductively the q-Weyl normal ordered product of variables $X_i$ by
\begin{equation}
    X_i :\prod_j X_j^{n_j}:=q^{\sum_j \e_{ij}n_j}:X_i \prod_j X_j^{n_j}:.
\end{equation}

It is convenient to represent these data graphically using quivers (i.e. oriented graphs). The set of vertices for the quiver is \(I\). Frozen vertices are associated with the subset \(I_{\mathrm{f}}\), they are depicted by squares. Unfrozen vertices, associated with $I_u$, are depicted by circles. If \(\epsilon_{ij}\in \mathbb{Z}_{\ge 0}\) we draw \(\epsilon_{ij}\) solid arrows from \(i\) to \(j\). If \(\epsilon_{ij}\in \mathbb{Z}_{\ge 0}+\frac{1}{2}\) we draw \((\epsilon_{ij}-\frac12)\) solid arrows between \(i\) and \(j\) and also one dashed arrow. Slightly abusing notations we will call vertices by their numbers \(i\in I\) and also by the corresponding variables \(X_i\). It is worth mentioning that the quiver associated with a cluster contains neither loops or 2-cycles. This property is no longer true in the chiralized setting.

\begin{definition}
    The quantum Weyl algebra $\mC_q[Q_i^{\pm1},P_i^{\pm1}]$ of rank $N$ is the algebra generated by the elements $\{Q_i^{\pm1},\ P_i^{\pm1}|i=1\cdots N\}$, with the relations
    \begin{equation}
        P_i Q_j=q^{\d_{i,j}}Q_jP_i,\qquad Q_iQ_i^{-1}=Q_i^{-1}Q_i=1,\qquad P_iP_i^{-1}=P_i^{-1}P_i=1.
    \end{equation}
    We call a \emph{choice of polarization} a homomorphism of algebras $\mathbb{C}_q[\mathcal{X}_\mathsf{s}]\to\mC_q[Q_i^{\pm1},P_i^{\pm1}]$.
\end{definition}

Note that the q-Weyl normal order relations for the variables $Q_i$ and $P_i$ read
\begin{equation}
    P_iQ_j=q^{\frac12\d_{i,j}} :P_iQ_j:,\qquad Q_iP_j=q^{-\frac12\d_{i,j}}:Q_iP_j:,\qquad Q_iQ_j=:Q_iQ_j:,\qquad P_iP_j=:P_iP_j:.
\end{equation}

% The \emph{q-Weyl normal-ordering} of a monomial in $\mC_q[Q_i^{\pm1},P_i^{\pm1}]$ is defined recursively by the relations
%     \begin{align}
%         &P_i^n:\prod_{j=1}^N P_j^{\a_j}Q_j^{\b_j}:=q^{n\b_i}:\prod_{j=1}^N P_j^{\a_j+n\d_{i,j}}Q_j^{\b_j}:,\quad
%             :\prod_{j=1}^N P_j^{\a_j}Q_j^{\b_j}:P_i^n=q^{-n\b_i}:\prod_{j=1}^N P_j^{\a_j+n\d_{i,j}}Q_j^{\b_j}:,\\
%         &Q_i^n:\prod_{j=1}^N P_j^{\a_j}Q_j^{\b_j}:=q^{-n\a_i}:\prod_{j=1}^N P_j^{\a_j}Q_j^{\b_j+n\d_{i,j}}:,\quad
%             :\prod_{j=1}^N P_j^{\a_j}Q_j^{\b_j}:Q_i^n=q^{n\a_i}:\prod_{j=1}^N P_j^{\a_j}Q_j^{\b_j+n\d_{i,j}}:.    
%     \end{align}
%     for $n,\a_i,\b_i\in\mZ$. We will write $AB::f$ for the normal-ordering relation $AB=f\normordleft AB\normordright $. 

\paragraph{Mutations} Mutations are transformation between seeds associated with unfrozen vertices.

\begin{definition} 
	Mutation at an unfrozen vertex \(k\in I_u\) is a transformation of seeds \(\mu_k\colon (I,I_{\mathrm{f}},\epsilon,\mathbf{X}) \rightarrow  (I,I_{\mathrm{f}},\tilde{\epsilon},\tilde{\mathbf{X}}) \) such that the matrix $\epsilon$ is transformed as 
	\begin{equation}\label{eq:mutation epsilon}
		\tilde{\epsilon}_{ij}=
		\begin{cases} 		
			-\epsilon_{ij}, \quad \text{ if } i=k  \text{ or } j=k  
			\\ 
			\epsilon_{ij}+\frac{\epsilon_{ik}|\epsilon_{kj}|-\epsilon_{jk}|\epsilon_{ki}|}{2},\ \text{ otherwise }
		\end{cases}	
	\end{equation}
	while the variables are transformed as 
	\begin{equation}\label{eq:mutation X}
		\tilde{X}_i=\begin{cases}
			X_k^{-1} \qquad&\text{if } k=i,
			\\  
			X_i \qquad &\text{if } \epsilon_{ik}=0, k\neq i
			\\
			X_{i}(1+q^{-1}X_k^{-1})^{-1}\cdot\ldots\cdot (1+q^{2\epsilon_{ik}+1}X_k^{-1})^{-1} \qquad &\text{if }  \epsilon_{ik}< 0,
			\\
			X_{i}(1+q^{-1}X_k)\cdot\ldots\cdot (1+q^{1-2\epsilon_{ik}}X_k) \qquad &\text{if } \epsilon_{ik}> 0.
		\end{cases}
	\end{equation}
\end{definition}

In terms of quivers, the mutation of the matrix \(\epsilon\) can be stated simply as the following algorithm:
\begin{enumerate}
	\item  Reverse the directions of all arrows incident to the vertex $k$.
	\item For each pair of arrows $i \to k$ and $k \to j$, draw an arrow $ j\to i$ (close 3-cycles).
	\item Delete pairs of arrows with opposite direction $i \to j$ and $j \to i$ (remove 2-cycles).
\end{enumerate}

% \paragraph{Paths and operators} Assume that we have sequence of vertices \(X_0,\dots,X_n\) such that for any \(1\leq i \leq n\) there is exactly one edge from \(X_{i-1}\) to \(X_{i}\) and there are no edges between \(X_i, X_j\) if \(|i-j|>1\). Then we denote the telescoping sum by 
% \begin{equation}\label{eq:def_tel_S}
% 	\Sigma\left[X_0,X_1,\cdots,X_{n-1}\right]={}\!\normordleft X_0\normordright+\normordleft X_0X_1\normordright+\normordleft X_0X_1X_2\normordright+\cdots+\normordleft X_0X_1\cdots X_{n-1}\normordright,
% \end{equation} 
% Under the assumptions above let \(\mathcal{P}\) be a path from \(X_0\) to \(X_n\) following the direction of the arrows of the quiver. We will use the notations
% \begin{equation}\label{eq:def_SP_PsiP}
% 	\Sigma[{\mathcal{P}}]=\Sigma[X_0,X_1,\cdots,X_{n-1}],\qquad \Psi[{\mathcal{P}}]={}\!\normordleft X_0X_1\cdots X_n\normordright.
% \end{equation} 

\subsection[Chiral cluster seeds]{Chiral cluster seeds}\label{sec:Cluster_qChiral}
%%%%%%%%%%%%
The cluster seed, and the corresponding quiver, provides us with a convenient way to represent the q-commutation relations of a set of quantum torus variables $X_i$. The main idea behind the chiralisation is to promote these variables to vertex operators $X_i(z)$, i.e. formal series of operators acting on a Fock space. The chiral cluster seed will be used to represent the algebraic relations satisfied by these operators. We will restrict ourselves to very specific vertex operators with exchange relations involving a certain type of rational functions defined in terms of parameters $\k_a$. These parameters will be associated with the arrows of the quiver.

In addition to the introduction of extra parameters, it is necessary to distinguish between three types of vertices in the quiver: unfrozen, frozen $+$, and frozen $-$. On top of the regular arrows with parameters $\k_a$, dashed arrows of type $\pm$ and parameters $\k_a^\pm$ can also be introduced between frozen vertices of the same type. These parameters are collectively denoted $\k$ in the following definition.

\begin{definition}\label{def:chiralCluster}
	A \emph{chiral cluster seed} \(\mathsf{s}^\text{ch}\) consists of the following data \((I,I^\pm_{\mathrm{f}}, \nu,\k,\mathbf{X}, \mathbf{X}^\text{ch})\), where 
	\begin{itemize}
		\item \(I\) is a finite set. \(I^\pm_{\mathrm{f}}\subset I\) are two frozen subsets with $I_f^+\cap I_f^-=\vac$. Denote $I_u=I\setminus(I_f^+\cup I_f^-)$.
        \item $\nu$ is a set of nonnegative integers containing integers $\nu_{ij}$ indexed by pairs $(i,j)\in I\times I$, and two sets of nonnegative integers $\nu_{ij}^\pm$ indexed by pairs $(i,j)\in I_f^\pm\times I_f^\pm$,
        \item $\k$ is a set of complex parameters associated with pairs of indices $(i,j)$ as follows:
            \begin{align}
                &\k_{ij}^{(s)},\qquad \forall (i,j)\in I\times I,\quad s=1\cdots \nu_{ij},\\
                &\k_{ij}^{(+,s)},\qquad \forall (i,j)\in I_f^+\times I_f^+,\quad s=1\cdots \nu^+_{ij},\\
                &\k_{ij}^{(-,s)},\qquad \forall (i,j)\in I_f^-\times I_f^-,\quad s=1\cdots \nu^-_{ij}.
            \end{align}
        \item \(\mathbf{X}\) is a set of variables  \(\mathbf{X}=(X_i\mid i\in I)\) labeled by \(I\), and subject to the relations \ref{eq:Xi Xj} where $\eps_{ij}$ is the anti-symmetric matrix defined as
            \begin{equation}
                \eps_{ij}=
                \begin{cases}
                0, & i=j,\\
                \hf(\nu_{ij}^\pm-\nu_{ji}^\pm)+\nu_{ij}-\nu_{ji}, & (i,j)\in I_f^\pm\times I_f^\pm,\\   
                \nu_{ij}-\nu_{ji}, & \text{else.}
                \end{cases}
            \end{equation} 
        \item \(\mathbf{X}^\text{ch}\) is an algebra with generators $\{x_{i,n},\ i\in I,\ n\in\mZ^\times\}$ and relations
        \begin{equation}\label{eq:com_xx}
            [x_{i,n},x_{j,m}]=\dfrac{\d_{n+m,0}}{n}C_{ij}^{[n]},\qquad (n>0),
        \end{equation}
        where $C_{ij}^{[n]}$ is the $n$th-Adams operation sending $q\to q^n$, $\k_a\to \k_a^n$ applied to the matrix
        \begin{equation}
            C_{ij}=(q-q^{-1})\times
            \begin{cases}
            (q-q^{-1})-\sum_{s=1}^{\nu_{ii}}(q \k_{ii}^{(s)}-q^{-1}(\k_{ii}^{(s)})^{-1}), & i=j\in I_u,\\
            \mp q^{\mp 1}-\sum_{s=1}^{\nu_{ii}}(q \k_{ii}^{(s)}-q^{-1}(\k_{ii}^{(s)})^{-1})\pm\sum_{s=1}^{\nu_{ii}^\pm}q^{\mp 1} (\k_{ii}^{(\pm,s)})^{\mp 1}, & i=j\in I_f^\pm,\\
            \sum_{s=1}^{\nu_{ij}}q^{-1}(\k_{ij}^{(s)})^{-1}-\sum_{s=1}^{\nu_{ji}}q \k_{ji}^{(s)}+\sum_{s=1}^{\nu_{ij}^+}q^{-1}(\k_{ij}^{(+,s)})^{-1}, & (i,j)\in I_f^+\times I_f^+,\ j\neq i,\\
            \sum_{s=1}^{\nu_{ij}}q^{-1}(\k_{ij}^{(s)})^{-1}-\sum_{s=1}^{\nu_{ji}}q \k_{ji}^{(s)}-\sum_{s=1}^{\nu_{ji}^-}q\k_{ji}^{(-,s)}, & (i,j)\in I_f^-\times I_f^-,\ j\neq i,\\
            \sum_{s=1}^{\nu_{ij}}q^{-1}(\k_{ij}^{(s)})^{-1}-\sum_{s=1}^{\nu_{ji}}q \k_{ji}^{(s)} & \text{else}.
            \end{cases}
        \end{equation}		
        \end{itemize}
 \end{definition}

It is interesting to note that, for unfrozen indices $i,j\in I_u$, the matrix $C$ satisfies the symmetry property $C_{ij}=C_{ji}^{[-1]}$. Note also that, in fact, the matrix $\eps_{ij}$ can be obtained from $C_{ij}$ in the limit $\hbar\to0$ of $q=e^\hbar$, $\k_{i,j}^{(s)}=e^{\hbar\a_{ij}^{(s)}}$ and $\k_{i,j}^{(s,\pm)}=q^{\hbar\a_{ij}^{(s,\pm)}}$, i.e.
\begin{equation}
    \eps_{ij}=\hf\lim_{\hbar\to 0}\dfrac{C_{ij}-C_{ji}}{q-q^{-1}}.
\end{equation}

\begin{definition}\label{def:UEAbos}
    To a chiral cluster seed $\mathsf{s}^\text{ch}$ we associate the algebra $\CA_{\mathsf{s}^\text{ch}}$ defined as the completion of the universal enveloping algebra of $\mathbf{X}\cup\mathbf{X}^\text{ch}$ generated by elements of the form
    \begin{equation}
        \sum_{k=1}^\infty p_k(\{x_{i,n},X_i\}) x_{i_{k},n_k},
    \end{equation}
    where $p_k(\{x_{i,n},X_i\})$ are polynomials in $x_{i,n}$ and $X_i$, and for all $N\geq0$, there are only finitely many $n_k\geq N$.
\end{definition}

\paragraph{Vertex operators} It is convenient to associate to each vertex $i\in I$ of a chiral cluster seed a formal series of operators called \textit{vertex operators}, and defined as
\begin{equation}\label{eqn:vertexOp}
    X_i(z)=X_i\exp\left(\sum_{n>0}z^{n}x_{i,-n}\right)\exp\left(\sum_{n>0}z^{-n}x_{i,n}\right)\in\CA_\CQ[\![z,z^{-1}]\!].
\end{equation}
The quantum variable $X_i$ is often called \textit{zero mode} of the vertex operator $X_i(z)$.

The normal ordered product of two such operators, denoted $:X_i(z)X_j(w):$, is defined by moving the operators $x_{i,n}$ with $n>0$ to the right, and applying the q-Weyl normal order on zero modes,
\begin{equation}
    :X_i(z)X_j(w):=:X_iX_j:\exp\left(\sum_{n>0}(z^{n}x_{i,-n}+w^nx_{j,-n})\right)\exp\left(\sum_{n>0}(z^{-n}x_{i,n}+w^{-n}x_{j,n})\right).
\end{equation}
 It is also convenient to introduce the formal series of the modes $x_{i,n}$, $x_i(z) = \sum_{n \neq 0} x_{i, n} z^{-n}$, and such that
\begin{equation}\label{eqn:xi(z)}
 X_i(z)=X_i:e^{x_i(z)}:.
\end{equation}

The OPE between vertex operators $X_i(z)$ and $X_j(w)$ is the factor relating the usual product of operators to the normal ordered one. It can be computed explicitly using the defining relations of the algebras $\mathbf{X}$ and $\mathbf{X}^\text{ch}$, 
\begin{equation}
    X_i(z)X_j(w)= q^{\epsilon_{ij}}\exp\left(\sum_{n>0}\frac{w^nz^{-n}}{n}C_{ij}^{[n]}\right):X_i(z)X_j(w):.
\end{equation}
For short, we will often use the following notation for the OPE of two vertex operators,
\begin{equation}
    X_i(z)X_j(w):: q^{\epsilon_{ij}}\exp\left(\sum_{n>0}\frac{w^nz^{-n}}{n}C_{ij}^{[n]}\right).
\end{equation}
Assuming the radial ordering $|w|<\!<|z|$, the series inside the exponential can be re-summed. From our assumption on the form of the matrix $C_{ij}$, the exponential produces a rational function of $w/z$. In this way, the chiral cluster seed encodes the rational OPEs of the set of q-deformed vertex operators $X_i(z)$ indexed by $i\in I$.

\paragraph{Free field realizations} In order to build representations of the algebra associated with the chiral cluster seed, it is convenient to embed the generators $(X_i,x_{i,n})$ in a simpler algebra, called here \textit{diagonal Heisenberg algebra}. In this way, we find the minimal representation of this algebra in terms of Heisenberg generators, eliminating relations between $x_{i,n}$ coming from null vectors of the matrices $C$ and $C^T$. We first recall some standard definition in the context of deformed vertex algebras.

\begin{definition}\label{def:diagHeis}
    The \emph{diagonal Heisenberg algebra} of rank $N$, denoted $\CH_N$, is the algebra generated by elements $\hq_i$ and $J_{i,n}$ with $i=1\cdots N$, $n\in\mZ$, and the relations
        \begin{equation}
            [J_{i,n},J_{j,m}]=n\d_{i,j}\d_{n+m,0},\quad [J_{i,n}, \hq_j]=\d_{i,j}\d_{n,0}.
        \end{equation}
        Note that operators $P_i=q^{J_{i,0}}$ and $Q_i=e^{\hq_i}$ form a quantum torus algebra $\mC_q[Q_i^{\pm1},P_i^{\pm1}]$ called \emph{zero modes subalgebra}.
\end{definition}

These algebras can be represented on Fock modules defined as follows.

\begin{definition}
     The Fock module $\CF_\bsu$ is a vector space indexed by $N$-tuples $\bsu\in\mC^N$, and built upon a vacuum vector $v_{\bsu}$, annihilated by $J_{i,n}$ for $i=1\cdots N$, $n>0$, and such that $J_{i,0}v_\bsu=u_iv_\bsu$, 
     \begin{equation}
        \CF_\bsu=\bigoplus_{\bsn\in\mZ^N}\mC[J_{i,-1},J_{i,-2},\cdots]v_{\bsu+\bsn},\qquad v_{\bsu+\bsn}=e^{\sum_{i=1}^Nn_i\hat q_i}v_\bsu.
     \end{equation}
\end{definition}

It is possible to define an action of the algebra $\CA_{\mathsf{s}^\text{ch}}$ associated to a chiral cluster seed $\mathsf{s}^\text{ch}$ on a Fock module $\CF_\bsu$ by embedding it into a diagonal Heisenberg algebra. We call such embeddings \textit{diagonal free field realizations}.

\begin{definition}
    A \emph{diagonal free field realization} is a homomorphism of algebras $\CA_{\mathsf{s}^\text{ch}}\to\CH_N$ consisting in 
    \begin{enumerate}
        \item a choice of polarization $\mathbb{C}_q[\mathcal{X}_\mathsf{s}]\to \mC_q[Q_i^{\pm1},P_i^{\pm1}]$ where $\mC_q[Q_i^{\pm1},P_i^{\pm1}]$ is the zero modes subalgebra of $\CH_N$,
        \item a choice of matrix factorization $C=B^+(B^-)^T$ into a product of two rectangular matrices $B^\pm$ of size $|I|\times N$ with $N=\rank C$. These matrices define the homomorphism of algebras\footnote{The matrix elements of $B^\pm$ depend on the parameters $q,\k_a$, and the Adams operation is well-defined. We restrict ourselves to linear maps obtained from $B^\pm$ by the $n$th Adams operation for convenience. It is possible to extend the definition with an infinite set of matrices such that $C^{[n]}=B_n^+(B_n^-)^T$, for $n\in\mZ^{\geq0}$, but it seems an unnecessary complication in the present context.}
        \begin{equation}
            x_{i,\pm n}=\dfrac1n\sum_{j=1}^N B_{i,j}^{\pm[n]} J_{j,\pm n},\qquad \text{for}\ n>0.
        \end{equation}
    \end{enumerate}
\end{definition}

Note that, for a general matrix factorization $C=B^+(B^-)^T$, $B^{\pm}$ are rectangular matrices of size $|I|\times N$, with $N\geq \rank C$. Our definition corresponds to the minimal factorization with $N=\rank C$. In this case, all matrix factorizations are isomorphic since two factorizations can be related by a transformation $B^+\to B^+ M$, $B^-\to B^- (M^{-1})^T$ for some $M\in GL(N)$. In physics terms, the rank $N$ is the minimal number of free bosons needed to realize the chiral cluster seed. A similar matrix factorization can be introduced for the zero modes.

% A minimal matrix factorization can also be introduced for the zero modes, taking $\eps_{ij}=(\rho\mu^T)_{ij}$ with rectangular matrices $\rho$ and $\mu$ of size $|I|\times r_\eps$ where $r_\eps=\rank \eps$. Each factorization is associated to a choice of polarization
% \begin{equation}
%     X_i=:\prod_{k=1}^{r_\eps}Q_k^{\mu_{ik}}P_k^{\rho_{ik}}:
% \end{equation}
% In general, $\rank \eps\neq\rank C$.

In summary, for every chiral cluster seed, it is possible to define a free field realization of vertex operators $X_i(z)$ acting on Fock modules. In the next section, we will consider embeddings of deformed W-algebra generators in the algebra $\CA_{\mathsf{s}^\text{ch}}$. By extension, we will call such embeddings \textit{free field realizations}.

\subsubsection{Decorated quiver}
The data of a chiral cluster seed can be encoded in a \emph{decorated quiver} with three types of vertices: frozen $(+)$, frozen $(-)$ and unfrozen, indexed respectively by the sets $I_f^+$, $I_f^-$ and $I_u=I\setminus(I_f^+\cup I_f^-)$. In this quiver, the set of arrows $A$ is decomposed into three subsets: dashed $+$ ($A_f^+$), dashed $-$ ($A_f^-$) and plain ($A_u=A\setminus(A_f^+\cup A_f^-)$). Dashed arrows of type $\pm$ are restricted to connect only frozen vertices of the same type $\pm$, while plain arrows can connect any vertex type. The number of plain arrows from vertex $i\in I$ to $j\in I$ is given by $\nu_{ij}\in\mZ^{\geq0}$. The number of dashed arrows of type $\pm$ from vertex $i\in I_f^\pm$ to $j\in I_f^\pm$ is given by $\nu_{ij}^\pm$. Each arrow $a\in A(i\to j)$ is decorated by a parameter $\k_a$ corresponding to $\k_{ij}^{(s)}$ for a plain arrow ($s$ index the arrow), and $\k_{ij}^{(\pm,s)}$ for a dashed arrow of type $\pm$.

From such a quiver, it is easy to reconstuct the matrices $C_{ij}$ and $\eps_{ij}$,
\begin{align}\label{eq:expre_cij}
    &C_{ij}=(q-q^{-1})\times
    \begin{cases}
        (q-q^{-1})-\sum_{\ell\in A_u(i\to i)}(q \k_\ell-q^{-1}\k_\ell^{-1}), & i=j\in I_u,\\
        \mp q^{\mp 1}-\sum_{\ell\in A_u(i\to i)}(q \k_\ell-q^{-1}\k_\ell^{-1})\pm\sum_{\ell\in A_\pm(i\to i)}q^{\mp 1} \k_\ell^{\mp 1}, & i=j\in I_f^\pm,\\
        \sum_{a\in A_u(i\to j)}q^{-1}\k_a^{-1}-\sum_{a\in A_u(j\to i)}q \k_a +\sum_{a\in A_+(i\to j)}q^{-1}\k_a^{-1} & (i,j)\in I_f^+\times I_f^+,\ j\neq i,\\
        \sum_{a\in A_u(i\to j)}q^{-1}\k_a^{-1}-\sum_{a\in A_u(j\to i)}q \k_a -\sum_{a\in A_-(j\to i)}q \k_a & (i,j)\in I_f^-\times I_f^-,\ j\neq i,\\
        \sum_{a\in A_u(i\to j)}q^{-1}\k_a^{-1}-\sum_{a\in A_u(j\to i)}q \k_a & \text{else},
    \end{cases}\\
        &\eps_{ij}=
    \begin{cases}
        0,& i=j,\\
        \sharp A_u(i\to j)-\sharp A_u(j\to i)+\hf(\sharp A_\pm(i\to j)-\sharp A_\pm(j\to i)), & (i,j)\in I_f^\pm\times I_f^\pm,\ j\neq i\\
        \sharp A_u(i\to j)-\sharp A_u(j\to i), & \text{else},
    \end{cases},\label{eq:expre_epsij}
\end{align}
where we denoted $A_x(i\to j)$ the set of arrows of type $x$ from $i$ to $j$, and $\sharp A_x(i\to j)$ the number of such arrows. We note that the contribution of two plain arrows joining the same vertices but going in opposite directions, and with parameters related by $\k_1\k_2=q^{-2}$, give no contribution to the matrices $C$ and $\eps$. Since we are mainly interested in the algebra $\CA_{\mathsf{s}^\text{ch}}$, seeds $\mathsf{s}^\text{ch}$ producing the same algebras should be identified. Thus, we will simply ignore such pair of arrows, and delete them from the quiver. In the same way, two loops at the same vertex with parameters satisfying the relation $\k_1\k_2=q^{-2}$ cancel each other. Finally, a dashed arrow $i\dashrightarrow j$ of parameter $\k_1$ together with a plain arrow $j\to i$ of parameter $\k_2$ with $\k_1\k_2=q^{-2}$ can be replaced by a dashed arrow $j\dashrightarrow i$ of parameter $\k_2$.

For a decorated quiver $\CQ$ that defines a chiral cluster seed $\mathsf{s}^\text{ch}$, we denote, for short, the algebra $\CA_\CQ=\CA_{\mathsf{s}^\text{ch}}$.

\paragraph{Gauge transformations} We call \textit{gauge transformation of parameter $\a\in\mC^\times$ at the vertex $i\in I_u$} the rescaling of the generators $x_{i,n}\to \a^n x_{i,n}$ for a chiral cluster seed. In the quiver presentation, this transformation corresponds to rescale the parameter of incoming arrows at the vertex $i$ as $\k_a\to\a\k_a$ and outgoing arrows as $\k_a\to\a^{-1}\k_a$. In terms of vertex operators, it corresponds to a rescaling of the spectral variable $X_i(z)\to X_i(z/\a)$. This gauge transformation does not modify the overall algebraic properties, and so two chiral cluster seeds related by gauge transformations should be considered equivalent.

\paragraph{OPEs} Despite the apparent complexity of the formula \ref{eq:expre_cij}, the OPE between vertex operators can be expressed in a simple manner by introducing the rational functions
\begin{equation}
\varphi_n(z,w)=\frac{q^n z-q^{-n}w}{z-w},
\end{equation}
satisfying the properties $\varphi_{-n}(z,w) = \varphi_n(w,z)$, $\varphi_n(z,\k w) = \varphi_n(\k^{-1}z,w)$, $\varphi_{-n}(z,w) = \varphi_n(z,q^{2n}w)^{-1}$. In the absence of loops, the variable at vertex $i$ will have the following self-OPE, depending on the type of vertex,\footnote{Note that an unfrozen vertex can be decomposed into two frozen vertices of different type with no arrows between them. In terms of vertex operators, we have the decomposition $X(z)=X_+(z)X_-(z)$ with $X_+(z)X_-(w)::1$ and $X_-(z)X_+(w)::1$.}
\begin{equation}
    X_i(z)X_i(w)::
    \begin{cases}
        \varphi_1(z,w)^{-1}\varphi_{-1}(z,w)^{-1}, & i\in I_u,\\
        q\varphi_1(z,w)^{-1}, & i\in I_f^+,\\
        q^{-1}\varphi_{-1}(z,w)^{-1}, & i\in I_f^-.
    \end{cases}
\end{equation}
Each loop $\ell\in A(i\to i)$ with label $\k_\ell$ will bring an extra factor that depends on the type of loop,
\begin{align}
    &\vphi_1(\k_\ell z,w)\vphi_{-1}(z,\k_\ell w),\quad \ell\in A_u(i\to i),\\
    &q^{-1}\vphi_1(\k_\ell z,w),\quad \ell\in A_+(i\to i),\\
    &q\vphi_{-1}(z,\k_\ell w),\quad \ell\in A_-(i\to i).
\end{align}

Arrows will be represented as follows (resp. plain and frozen $\pm$):

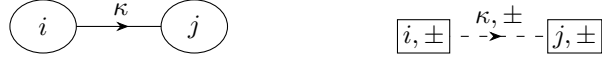
\begin{center}
\begin{tikzpicture}
	\def\xs{1}

	% Coordinates
	\coordinate (X) at (0,0);
	\coordinate (Y) at (2*\xs,0);

	% Arrows
	\draw[
		qArrow/base,
		qArrow/label={$\k$},
		qArrow/labelstyle={above, inner sep=1pt},
		qArrow/label xshift=-2pt,
		qArrow/label yshift=3pt,
        qArrow/arrowpos=0.55
	] (X) to (Y);

	% Nodes
	\node[qNodeUnfrozen] at (X) {$i$};
	\node[qNodeUnfrozen] at (Y) {$j$};
\end{tikzpicture}
\hspace{20mm}
\begin{tikzpicture}
	\def\xs{1}

	% Coordinates
	\coordinate (X) at (0,0);
	\coordinate (Y) at (2*\xs,0);

	% Arrows
	\draw[
		qArrow/base,
		qArrow/pattern=dashed,
		qArrow/label={$\k,\pm$},
		qArrow/labelstyle={above, inner sep=1pt},
		qArrow/label xshift=0pt,
		qArrow/label yshift=2pt
	] (X) to (Y);

	% Nodes
	\node[qNodeFrozen] at (X) {$i,\pm$};
	\node[qNodeFrozen] at (Y) {$j,\pm$};
\end{tikzpicture}
\end{center}

% \begin{center}
% \begin{tikzpicture} 
% 	\def\xs{1}
% 	\node[styleNode] (X) at (0,0) {$i$};	
% 	\node[styleNode] (Y) at (2*\xs,0) {$j$};
% 	\draw[styleArrow](X) to node[midway,above]{$\k$}  (Y) ;
% \end{tikzpicture}
% \hspace{20mm}
% \begin{tikzpicture} 
% 	\def\xs{1}
% 	\node[styleNodeFr] (X) at (0,0) {$i,\pm$};	
% 	\node[styleNodeFr] (Y) at (2*\xs,0) {$j,\pm$};
% 	\draw[styleArrow,dashed](X) to node[midway,above]{$\k,\pm$}  (Y) ;
% \end{tikzpicture}
% \end{center}
Each arrow $i\to j$ brings contributions to both OPEs $X_i(z)X_j(w)$ and $X_j(z)X_i(w)$, the precise factor depending on the type of arrow, 
\begin{align}
    X_i(z)X_j(w)&::\varphi_1(\k_a z,w),& X_j(z)X_i(w)&::\varphi_{-1}(z,\k_a w),& a&\in A_u(i\to j),\\ 
    X_i(z)X_j(w)&::q^{-1/2}\varphi_1(\k_a z,w), &X_j(z)X_i(w)&::q^{-1/2},& a&\in A_+(i\to j),\\ 
    X_i(z)X_j(w)&::q^{1/2}, &X_j(z)X_i(w)&::q^{1/2}\varphi_{-1}(z,\k_aw), &a&\in A_-(i\to j).
\end{align}
The total OPE is the product of the contributions of all arrows between $i\to j$ and $j\to i$.

\paragraph{Paths} In the next section, we will define embeddings of the deformed W-algebras in the chiral cluster seed algebra $\CA_\CQ$. The generating currents of the W-algebras will be sent to sums of monomials of vertex operators $X_i(z_i)$. A simple way to generate such sums is to consider products and telescoping sums of vertex operators along a certain path $\CP$ in the quiver.

Let \(X_0,\dots,X_n\) be a sequence of vertices in the decorated quiver, and $\CP:X_0\to X_1\cdots\to X_n$ a path following the directions of arrows $X_i\to X_{i+1}$ with parameters $\k_i$. We assume that vertices \(X_1,\dots,X_{n-1}\) are unfrozen. To such a path, we associate the telescoping sum $\S[\CP](z)$ defined as
\begin{multline}\label{eq:def_tel_S}
	\Sigma\left[\CP\right](z)={}\!\normordleft X_0(z)\normordright+\normordleft X_0(z)X_1(\k_1z)\normordright+\normordleft X_0(z)X_1(\k_1z)X_2(\k_1\k_2z)\normordright+\cdots \\ +\normordleft X_0(z)X_1(\k_1z)\cdots X_{n-1}(\k_1\cdots \k_{n-1}z)\normordright,
\end{multline} 
and the product of vertex operators
\begin{equation}\label{eq:def_SP_PsiP}
	\Psi[{\mathcal{P}}](z)={}\!\normordleft X_0(z)X_1(\k_1z)\cdots X_n(\k_1\cdots\k_n z)\normordright.
\end{equation} 
We will sometimes also denote these quantities using the explicit sequence of vertices, e.g. $\S[X_0,X_1,\cdots,X_{n-1}](z)$.

\subsubsection{Screening charges}\label{sec:screenings}
In this subsection, we restrict ourselves to chiral cluster seeds in which the unfrozen vertices carry at most one loop. To such vertices, it is possible to associate screening currents and screening charges. Unfrozen vertices $i\in I_u$ with no loop will be called \textit{fermionic vertices}, and will be associated with a single screening charge $Q_i$. Unfrozen vertices with a single loop will be called \textit{bosonic vertices}, and associated with two screening charges $Q_i^\pm$. We will show in the next section that these charges  (anti)commute with the currents generating the free field realization of deformed W-algebras.

\paragraph{Fermionic vertices} Let $i$ be an unfrozen vertex with no loops in the decorated quiver of a  chiral cluster seed. Taking a choice of polarization such that $X_i=-q^{-2J_{i,0}}$, with $[J_{i,0},\hq_j]=\d_{i,j}$, it is possible to rewrite the corresponding vertex operator $X_i(z)$ in terms of the fermionic fields $\Psi_i(z)$ and $\Psi_i^\ast(z)$,
\begin{equation}\label{eq:screen_ferm}
    X_i(z)=-:\Psi_i(q^{-1}z)\Psi_i^\ast(qz):, 
\end{equation}
with
\begin{align}\label{eq:bosonization}
    &\Psi_i(z)=:e^{\hq_i}z^{J_{i,0}}:\exp\left(-\sum_{n>0}z^{n}\dfrac{x_{i,-n}}{q^n-q^{-n}}\right)\exp\left(\sum_{n>0}z^{-n}\dfrac{x_{i,n}}{q^n-q^{-n}}\right),\\
    &\Psi_i^\ast(z)=:e^{-\hq_i}z^{-J_{i,0}}:\exp\left(\sum_{n>0}z^{n}\dfrac{x_{i,-n}}{q^n-q^{-n}}\right)\exp\left(-\sum_{n>0}z^{-n}\dfrac{x_{i,n}}{q^n-q^{-n}}\right).
\end{align}
The term \textit{fermionic} comes from the fact that these currents obey the anticommutation relations of a Dirac fermion,
\begin{equation}\label{fermion_anticom}
    \{\Psi_i(z),\Psi_i(w)\}=\{\Psi_i^\ast(z),\Psi_i^\ast(w)\}=0,\quad \{\Psi_i(z),\Psi_i^\ast(w)\}=\d(z/w).
\end{equation}
% It is also useful to recall the OPEs,
% \begin{align}
%     \begin{split}
%         &\Psi(z)\Psi(w)::(zw)^{-1/2}(z-w),\qquad \Psi^\ast(z)\Psi^\ast(w)::(zw)^{-1/2}(z-w),\\
%         &\Psi(z)\Psi^\ast(w)::\dfrac{(zw)^{1/2}}{z-w},\qquad \Psi^\ast(z)\Psi(w)::\dfrac{(zw)^{1/2}}{z-w}.
%     \end{split}
% \end{align}
The corresponding screening charge is defined by the integral\footnote{It is possible to consider more generally $X_i(z)=-C:\Psi_i(q^{-1}z)\Psi_i^\ast(qz):$, where the constant $C$ had been previously absorbed by a redefinition of zero modes $J_{i,0}\to J_{i,0}-\log(C)/(2\log q)$. In this case, the screening charge should be defined by
\begin{equation}
    Q_i=\oint z^{\frac{\log(C)}{2\log q}}\Psi_i^\ast(z)dz. 
\end{equation}}
\begin{equation}\label{def:Q_screening}
    Q_i=\oint \Psi_i^\ast(z)\dfrac{dz}{z}. 
\end{equation}
For this integral to be well-defined, the mode $J_{i,0}$ must act as an integer on the vacuum vector of the Fock module $\CF_\bsu$, which restricts the possible choices of weights $\bsu$.

We give below two examples in which we associate certain currents to a path in a decorated quiver, and show that they (anti)commute with the fermionic screening charge.

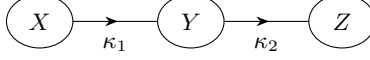
\begin{figure}
\begin{center}
\begin{tikzpicture}[scale=1, font=\small]

% Coordinates
\coordinate (x) at (0,0);
\coordinate (y) at (2,0);
\coordinate (z) at (4,0);

% Arrows
\draw[
	qArrow/base,
	qArrow/label={$\k_1$},
	qArrow/labelstyle={below, inner sep=1pt},
	qArrow/label xshift=0pt,
	qArrow/label yshift=-5pt
] (x) to (y);

\draw[
	qArrow/base,
	qArrow/label={$\k_2$},
	qArrow/labelstyle={below, inner sep=1pt},
	qArrow/label xshift=0pt,
	qArrow/label yshift=-5pt
] (y) to (z);

% Nodes
\node[qNodeUnfrozen] at (x) {$X$};
\node[qNodeUnfrozen] at (y) {$Y$};
\node[qNodeUnfrozen] at (z) {$Z$};

\end{tikzpicture}
\caption{Example of a simple (sub)quiver with three vertices and single arrows.}
\label{fig:XYZ}
\end{center}
\end{figure}
% \begin{figure}
% \begin{center}
% \begin{tikzpicture}[scale=1, font=\small]
% \node[styleNode] (x) at (0,0) {$X$};
% \node[styleNode] (y) at (2,0) {$Y$};
% \node[styleNode] (z) at (4,0) {$Z$};
% \draw[styleArrow](x) to[] node[midway,below]{$\k_1$}(y);
% \draw[styleArrow](y) to[] node[midway,below]{$\k_2$}(z);
% \end{tikzpicture}
% \caption{Example of a simple (sub)quiver with three vertices and single arrows.}
% \label{fig:XYZ}
% \end{center}
% \end{figure}

\begin{example}\label{Ex:XYZ}
As a first example, we consider the decorated quiver with three vertices represented in Figure \ref{fig:XYZ}. In this example, vertices $X$ and $Z$ can be attached to other vertices, or carry loops, only the nature of the vertex $Y$ and the single arrows between $X$ and $Y$, and $Y$ and $Z$ matter for this discussion. From the configuration of arrows, we deduce that the corresponding vertex operators can be written in the form
\begin{equation}\label{eq:XYZ}
    X(z)=V_X(z)\Psi(q\k_1 z),\qquad Y(z)=-:\Psi(q^{-1}z)\Psi^\ast(qz):,\qquad Z(z)=V_Z(z)\Psi^\ast(q^{-1}\k_2^{-1}z).
\end{equation}
where $V_X(z),\ V_Z(z)$ are vertex operators commuting with $\Psi(w)$ and $\Psi^\ast(w)$. Let's consider the quantities $W(z)=\S[X,Y](z)$ and $\chi(z)=\Psi[X,Y,Z](z)$ that may appear when considering a path $\cdots\to X\to Y\to Z\to\cdots$. Explicitly
\begin{equation}\label{eq:W_XYZ}
    W(z)=V_X(z)(\Psi(q\k_1 z)-\Psi(q^{-1}\k_1 z)),\qquad \chi(z)=-:V_X(z)V_Z(\k_1\k_2 z):.
\end{equation}
The quantity $\chi(z)$ does not depend on the fermionic currents $\Psi(z)$ and $\Psi^\ast(z)$, and so it trivially commutes with the associated screening charge $Q_Y$. Then, from the expression \ref{eq:W_XYZ} of $W(z)$, we deduce the anti-commutation relations
\begin{equation}
    \{\Psi^\ast(z),W(w)\}=V_X(w)\left(\d(q\k_1 w/z)-\d(q^{-1}\k_1 w/z)\right),
\end{equation}
which implies that $W(z)$ anticommutes with the screening charge $Q_Y$ defined in \ref{def:Q_screening}.
\end{example}

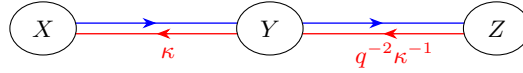
\begin{figure}[h]
\centering
\begin{tikzpicture}[scale=1, transform shape, font=\small]
\coordinate (X) at (0,0);
\coordinate (Y) at (3,0);
\coordinate (Z) at (6,0);

%Horizontal arrows
\draw[qArrowAboveOp, qArrow/color=blue, qArrow/arrowpos = 0.55](Y) to (X);
\draw[qArrowAboveOp, qArrow/color=blue, qArrow/arrowpos = 0.55](Z) to (Y);
\draw[qArrowBelowOp, qArrow/label={$\k$}, qArrow/color=red, qArrow/label xshift=0pt, qArrow/label yshift=-4pt, qArrow/arrowpos = 0.55](X) to[] (Y);
\draw[qArrowBelowOp, qArrow/label={$q^{-2}\k^{-1}$}, qArrow/color=red, qArrow/label xshift=0pt, qArrow/label yshift=-2pt, qArrow/arrowpos = 0.55](Y) to[] (Z);
%Nodes
\node[qNodeUnfrozen] (X) at (X) {$X$};
\node[qNodeUnfrozen] (Y) at (Y) {$Y$};
\node[qNodeUnfrozen] (Z) at (Z) {$Z$};
\end{tikzpicture}
	\caption{Example of a (sub)quiver with three vertices and double arrows.}\label{fig:XYZ_double}
\end{figure}

\begin{example}\label{Ex:XYZ_double} As a second example, we consider the decorated quiver with three vertices and double arrows represented in Figure \ref{fig:XYZ_double}. Again, for the sake of the argument, it doesn't matter if the vertices $X$ and $Z$ are attached to other vertices or carry loops. We note that, by gauging the vertices, the parameter of blue arrows can be fixed to one. On the other hand, we impose the requirement that the product of parameters for the red arrows equals to $q^{-2}$, and so there is only one independent parameter $\k$. In this configuration, the vertex operators can be written in the form
\begin{equation}\label{eq:XYZ_double}
    X(z)=V_X(z):\Psi(qz)\Psi^\ast(q^{-1}\k^{-1}z):,\qquad Y(z)=-:\Psi(q^{-1}z)\Psi^\ast(qz):,\qquad Z(z)=V_Z(z):\Psi(q^{-1}\k^{-1}z)\Psi^\ast(q^{-1}z):,
\end{equation}
where $V_X(z),\ V_Z(z)$ are again vertex operators commuting with $\Psi(w)$ and $\Psi^\ast(w)$.

Let's consider the following quantities that might be associated with the paths $\cdots\to X\to Y\to Z\to\cdots$ and $\cdots\to Z\to Y\to X\to\cdots$,
\begin{align}\label{eq:W_XYZ_double}
    W(z)&=\a X(z)+:X(z)Y(z):+:X(z)Y(z)Z(z):,& \chi(z)&=:X(z)Y(z)Z(z):,\\
    \bar W(z)&=Z(z)+\a:Z(z)Y(q^{-2}\k^{-1}z):+\a:Z(z)Y(q^{-2}\k^{-1}z)X(q^{-2}z):,& \bar\chi(z)&=:Z(z)Y(q^{-2}\k^{-1}z)X(q^{-2}z):.
\end{align}
Here, $W(z)$ (resp. $\bar W(z)$) can be seen as a deformation of the telescoping sum $\S[X,Y,Z]$ (resp. $\S[Z,Y,X]$) by the parameter $\a$, while $\chi(z)$ (and $\bar\chi(z)$) correspond to $\Psi[X,Y,Z]$ (resp. $\Psi[Z,Y,X]$). In the parameterization \ref{eq:XYZ_double},
\begin{align}
    W(z)&=V_X(z):\Psi^\ast(q^{-1}\k^{-1}z)(\a\Psi(qz)-\Psi(q^{-1}z)):-:V_X(z)V_Z(z):,& \chi(z)&=-:V_X(z)V_Z(z):,\\
    \bar W(z)&=V_Z(z):\Psi^\ast(q^{-1}z)(\Psi(q^{-1}\k^{-1}z)-\a\Psi(q^{-3}\k^{-1}z)):-\a:V_Z(z)V_X(q^{-2}z):,& \bar\chi(z)&=-:V_Z(z)V_X(q^{-2}z):.
\end{align}
Clearly, both $\chi(z)$ and $\bar\chi(z)$ commute with the screening charge $Q_Y$. On the other hand, we have
\begin{align}
    % &[\Psi^\ast(z),X(w)]=(q\k^{1/2}-q^{-1}\k^{-1/2})\d(qwz^{-1})V_X(w)\Psi^\ast(q^{-1}\k^{-1}w),\\
    % &[\Psi^\ast(z),:X(w)Y(w):]=-(\k^{1/2}-\k^{-1/2})\d(q^{-1}wz^{-1})V_X(w)\Psi^\ast(q^{-1}\k^{-1}w),\\
    &[\Psi^\ast(z),W(w)]=V_X(w)\Psi^\ast(q^{-1}\k^{-1}w)\left[\a(q\k^{1/2}-q^{-1}\k^{-1/2})\d(qwz^{-1})-(\k^{1/2}-\k^{-1/2})\d(q^{-1}wz^{-1})\right],
\end{align}
and so the charge $Q_Y$ commutes with $W(z)$ provided that we choose 
\begin{equation}\label{eq:alpha}
    \a=\dfrac{\k^{1/2}-\k^{-1/2}}{q\k^{1/2}-q^{-1}\k^{-1/2}}.
\end{equation}
We arrive at the same conclusion for $\bar W(z)$, using
\begin{align}
    % &[\Psi^\ast(z),Z(w)]=-(\k^{1/2}-\k^{-1/2})\d(q\k z/w)V_Z(w)\Psi^\ast(q^{-1}w),\\
    % &[\Psi^\ast(z),:Z(w)Y(q^{-2}\k^{-1}w):]=(q\k^{1/2}-q^{-1}\k^{-1/2})\d(q^3\k z/w)V_Z(w)\Psi^\ast(q^{-1}w),\\
    &[\Psi^\ast(z),\bar W(w)]=V_Z(w)\Psi^\ast(q^{-1}w)\left(-(\k^{1/2}-\k^{-1/2})\d(q\k z/w)+\a (q\k^{1/2}-q^{-1}\k^{-1/2})\d(q^3\k z/w)\right).
\end{align}
\end{example}

\paragraph{Bosonic vertices}
Let's consider again a chiral cluster seed, and let $i$ be an unfrozen vertex with one loop of parameter $\k$. The self-OPE of the corresponding vertex operator suggests to introduce two \textit{bosonic screening current} $S_i^\pm(z)$, such that
\begin{equation}\label{eq:screen_bos}
    X_i(z)=:S_i^\pm(\s_\pm z)S_i^\pm(\s_\pm^{-1} z)^{-1}:,
\end{equation}
with $\s_+=q$ and $\s_-=\k^{1/2}$. It is convenient to drop the indices $\pm$ here, and also introduce $\bar\s=q\k^{\frac12}\s^{-1}$ (such that if $\s=\s_\pm$, then $\bar\s=\s_\mp$). Introducing the modes $s_{i,n}=x_{i,n}/(\s^n-\s^{-n})$, and the zero modes $(s_{i,0},S_i)$ such that $[s_{i,0},S_i]=2\frac{\log(\bar\s)}{\log(\s)}S_i$,\footnote{I.e. $z^{s_{i,0}}S_iz^{-s_{i,0}}=z^{2\frac{\log\bar\s}{\log\s}}S_i$, so that $S_iX_i=\bar\s^{4}X_iS_i$.} the screening current can be written in the form
\begin{equation}\label{eq:bosonization_Bos}
    S_i(z)=S_i z^{-s_{i,0}}:e^{-\sum_{n\neq0}z^{-n}s_{i,n}}:\implies X_i(z)=\s^{-2s_{i,0}}:e^{\sum_{n\neq0}z^{-n}(\s^n-\s^{-n})s_{i,n}}:.
\end{equation}
These vertex operators are characterized by the following OPE,
\begin{equation}
    X_i(z)S_i(w)::\bar\s^{-2}\dfrac{(z-\bar\s q\k^{1/2}w)(z-\bar\s q^{-1}\k^{-1/2}w)}{(z-\bar\s^{-1}q\k^{1/2}w)(z-\bar\s^{-1}q^{-1}\k^{-1/2}w)},\qquad     S_i(z)X_i(w)::\bar\s^{2}\dfrac{(z-\bar\s^{-1}q\k^{1/2}w)(z-\bar\s^{-1}q^{-1}\k^{-1/2}w)}{(z-\bar\s q\k^{1/2}w)(z-\bar\s q^{-1}\k^{-1/2}w)}.
\end{equation}
Finally, the associated screening charge is simply the integral of the corresponding screening current,
\begin{equation}
    Q_i^\pm=\oint S_i^\pm(z) \dfrac{dz}{z}.
\end{equation}

\begin{figure}[h]
\centering
\begin{tikzpicture}[scale=1, transform shape, font=\small]
\coordinate (X) at (0,0);
\coordinate (Y) at (3,0);
\coordinate (Z) at (6,0);

%Horizontal arrows
\draw[qArrowAboveOp, qArrow/color=blue, qArrow/arrowpos = 0.55](Y) to (X);
\draw[qArrowBelowOp, qArrow/label={$q^{-2}\k^{-1}$}, qArrow/color=red, qArrow/arrowpos = 0.55](X) to[] (Y);
\draw[qArrowAboveOp, qArrow/color=blue, qArrow/arrowpos = 0.55](Z) to (Y);
\draw[qArrowBelowOp, qArrow/label={$q^{-2}\k^{-1}$}, qArrow/color=red, qArrow/arrowpos = 0.55](Y) to[] (Z);
%Loops
\draw[qLoop, qArrow/label={$\k$}](Y) to (Y);
%Nodes
\node[qNodeUnfrozen] (X) at (X) {$X$};
\node[qNodeUnfrozen] (Y) at (Y) {$Y$};
\node[qNodeUnfrozen] (Z) at (Z) {$Z$};
\end{tikzpicture}
	\caption{%
		Typical configuration for a (sub)quiver involving a bosonic vertex}\label{fig:BosonicScreenings}
\end{figure}
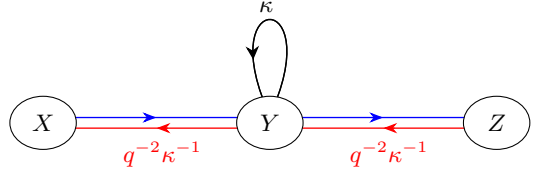

\begin{example}\label{Ex:BosonicScreenings}
As a typical example, we consider the chiral cluster seed associated with the decorated quiver represented in Figure \ref{fig:BosonicScreenings}. Up to a gauge transformation, it is possible to set the parameters of the blue arrows to one. On the other hand, we require that the product of parameters for arrows forming a 3-cycle is $\k_1\k_2\k_3=q^{-2}$. It fixes the parameters of the red arrows to $q^{-2}\k^{-1}$, where $\k$ is the parameter of the loop. Since we are only interested in the screening charges defined at the vertex $Y$, only the OPEs involving this vertex $Y$ will matter. As a consequence, our discussion does not depend on whether vertices $X$ and $Z$ are frozen or unfrozen, carry loops, or are related to other vertices, as long as they can be parameterized as follows,
\begin{equation}
    X(z)=V_X(z)Y^{-\frac12}:\exp\left(-\sum_{n\neq0}\frac{z^{-n} q^{-n}\k^{-n/2}y_n}{q^n\k^{n/2}+\k^{-n/2}q^{-n}}\right):,\qquad Z(z)=V_Z(z)Y^{-\frac12}:\exp\left(-\sum_{n\neq0}\frac{z^{-n} q^{n}\k^{n/2}y_n}{q^n\k^{n/2}+\k^{-n/2}q^{-n}}\right):,
\end{equation}
with $[Y(z),V_X(z)]=[Y(z),V_Z(z)]=0$, and
\begin{equation}
    Y(z)=Y:e^{\sum_{n\neq0}z^{-n}y_n}:,\qquad [y_n,y_m]=-\frac{\d_{n+m,0}}{n}(\s^n-\s^{-n})(\bar\s^n-\bar\s^{-n})(q^n\k^{n/2}+q^{-n}\k^{-n/2}),
\end{equation}
where we denoted $\s$ either $q$ or $\k^{1/2}$ depending on the choice of screening, and as before $\bar\s=q\k^{1/2}\s^{-1}$. The following property lists quantities commuting with the screening charges defined at the vertex $Y$.

\begin{proposition}\label{Prop:CommBosonicScreening}
    The currents
    \begin{equation}
        \Psi[X,Y,Z](z),\qquad :X(z)X(q^{-2}\k^{-1}z)Y(z):,\qquad \text{and}\qquad \S[X,Y](z),
    \end{equation}
    commute with the screening charges $Q_Y^\pm$.
\end{proposition}

\begin{proof}
    The proof for the first two currents follows from the fact that they do not depend on the modes $(y_n,Y)$,
    \begin{equation}
        \Psi[X,Y,Z](z)=:V_X(z)V_Z(z):,\qquad\text{and}\qquad  :X(z)X(q^{-2}\k^{-1}z)Y(z):=:V_X(z)V_X(q^{-2}\k^{-1}z):,
    \end{equation}
    and so they trivially commute with $Q_Y^\pm$. On the other hand, the proof for the third quantity is non-trivial. Let's pick one of the two screening currents $S^\pm(z)$, and drop the superscript $\pm$. Using the relations between the modes $(y_n,Y)$ and $(s_n,S)$, we can compute the following OPEs,
    \begin{align}
        X(z)S(w)&::\bar\s\dfrac{z-\bar\s^{-1}q^{-1}\k^{-1/2}w}{z-\bar\s q^{-1}\k^{-1/2}w},& S(w)X(z)&::\bar\s\dfrac{z-\bar\s^{-1}q^{-1}\k^{-1/2}w}{z-\bar\s q^{-1}\k^{-1/2}w},\\
        :X(z)Y(z):S(w)&::\bar\s^{-1}\dfrac{z-\bar\s q\k^{1/2}w}{z-\bar\s^{-1}q\k^{1/2}w}& \qquad S(w):X(z)Y(z):&::\bar\s^{-1}\dfrac{z-\bar\s q\k^{1/2}w}{z-\bar\s^{-1}q\k^{1/2}w}.
    \end{align}
    From this result, we deduce the commutation relations
    \begin{align}
        [X(z),S(w)]&=(\bar\s-\bar\s^{-1})\d(\s z/w):X(z)S(\s z):,\\
        [:X(z)Y(z):,S(w)]&=-(\bar\s-\bar\s^{-1})\d(z/(\s w)):X(z)S(\s z):, %=-(\bar\s-\bar\s^{-1})\d(z/(\s w)):X(z)Y(z)S(\s^{-1} z):
    \end{align}
    using $\s\bar\s=q\k^{1/2}$ and $:Y(z)S(\s^{-1} z):=S(\s z)$ in the second line. Thus, we have for $\S[X,Y](z)$,
    \begin{equation}
        [\S[X,Y](z),S(w)]=(\bar\s-\bar\s^{-1})\left(\d(\s z/w)-\d(z/(\s w))\right):X(z)S(\s z):\implies [\S[X,Y](z),Q_Y]=0.
    \end{equation}
\end{proof}
\end{example}

\subsubsection{Mutations}
The notion of mutations between cluster seeds is a key element in the definition of cluster algebras. In this subsection, we extend this notion to the fermionic vertices of chiral seeds. These mutations will be used in the next section to relate different free field realizations of deformed W-algebras.

The mutation of a chiral cluster seed at an unfrozen vertex $k\in I_u$ with no loop can be defined using the decorated quiver by this simple algorithm:
\begin{enumerate}
	\item Reverse the direction of all arrows incident to the vertex $k$. Replace the corresponding labels $\k\to q^{-1}\k^{-1}$. 
	\item For each pair of arrows $i \to k$ and $k \to j$ with labels resp. $(\k_1,\k_2)$ , draw an arrow $ j\to i$ with label $q^{-2}\k_1^{-1}\k_2^{-1}$ (complete the 3-cycles).
	\item Delete pairs of arrows between two vertices $(i,j)$ going in opposite directions $i \to j$ and $j \to i$ if their label satisfy the condition $\k_1\k_2=q^{-2}$ (remove 2-cycles). Similarly, delete pairs of loops if their labels satisfy this condition. Finally, if two frozen vertices of type $\pm$ are connected by a dashed arrow $i\to j$ of label $\k_1$, and a plain arrow $j\to i$ of label $\k_2$ with $\k_1\k_2=q^{-2}$, these two arrows should be replaced by a single dashed arrow of type $\pm$ from $j$ to $i$ with parameter $\k_2$.
\end{enumerate}
This mutation is involutive.

The new decorated quiver define matrices $\tilde{C}$ and $\tilde{\e}$ using the formulas \ref{eq:expre_cij} and \ref{eq:expre_epsij}, which in turn define the new chiral cluster seed. The new matrices are expressed in terms of the old ones as follows,
\begin{align}
    \begin{split}
    \tilde{C}_{ij}=&C_{ij}-\d_{i,k}(q-q^{-1})\left(\sum_{a\in A_u(k\to j)}(1+q^{-1})\k_a^{-1}-\sum_{a\in A_u(j\to k)}(1+q)\k_a)\right)\\
    &-\d_{k,j}(q-q^{-1})\left(\sum_{a\in A_u(i\to k)}(1+q^{-1})\k_a^{-1}-\sum_{a\in A_u(k\to i)}(1+q)\k_a\right)\\
    &+(q-q^{-1})\sum_{a\in A_u(i\to k)}\sum_{b\in A_u(k\to j)}(q\k_a\k_b)^{-1}-(q-q^{-1})\sum_{a\in A_u(j\to k)}\sum_{b\in A(k\to i)}q\k_a\k_b,
    \end{split}
    % C_{ij}^{[n]}\to &C_{ij}^{[n]}-\d_{i,k}\left(\sum_{a\in A_u(k\to j)}(1+q^{-n})\k_a^{-n}-\sum_{\superp{a\in A_u(j\to k)}{j\neq k}}(1+q^{n})\k_a^{n})\right)\\
    % &-\d_{k,j}\left(\sum_{\superp{a\in A_u(i\to k)}{i\neq k}}(1+q^{-n})\k_a^{-n}-\sum_{a\in A_u(k\to i)}(1+q^{n})\k_a^{n}\right)\\
    % &+\sum_{a\in A_u(i\to k)}\sum_{b\in A_u(k\to j)}q^{-n}(\k_a\k_b)^{-n}-\sum_{a\in A_u(j\to k)}\sum_{b\in A(k\to i)}q^n(\k_a\k_b)^n,\\
\end{align}
and $\tilde{\eps}_{ij}=\eps_{ij}-2\d_{i,k}\eps_{k,j}-2\d_{j,k}\e_{i,k}+\nu_{jk}\nu_{ki}-\nu_{ik}\nu_{kj}$ where $\nu_{ij}=\sharp A(i\to j)$ is the number of arrows from $i$ to $j$ (in these formulas, the sets of arrows refer to the old quiver). The first two lines in the formula for $\tilde{C}$ correspond to the first step of the algorithm, and the second line corresponds to the second step. The third step is the result of cancellations between summands.

To define the mutation on vertex operators, we need to introduce the fermionic fields $(\Psi(z),\Psi^\ast(z))$ associated with the mutation vertex. For definiteness, we denote this vertex $Y$, and let $i$ be any other vertex of the quiver. The OPEs suggest that the corresponding vertex operators can be parameterized as
\begin{equation}\label{eq:expr_Y_X_mutation}
    Y(z)=-:\Psi(q^{-1}z)\Psi^\ast(qz):,\qquad X_i(z):=:\prod_{a\in A(i\to Y)}\Psi(q\k_a z)\times\prod_{b\in A(Y\to i)}\Psi^\ast(q^{-1}\k_b^{-1}z): V_i(z),
\end{equation}
where the vertex operator $V_i(z)$ commutes with the fermionic fields. The mutation is defined by replacing the pair of fermionic fields $(\Psi(z),\Psi^\ast(z))$ with another pair $(\tPsi(z),\tPsi^\ast(z))$ obtained from the transformation
\begin{equation}\label{transfo_fermions}
    \tPsi^\ast(z)=\Psi(qz)-\Psi(q^{-1}z),\qquad \Psi^\ast(z)=\tPsi(q^{-1}z)-\tPsi(qz).
\end{equation}
This transformation preserves the anticommutation relations \ref{fermion_anticom}, and also implies that $V_i(z)$ commutes with $(\tPsi(z),\tPsi^\ast(z))$. The expression of the new vertex operators $\tilde{Y}(z)$ and $\tilde{X}_i(z)$ is given by the following proposition.

\begin{proposition}\label{prop:mutation}
    Let $Y(z)$ and $X_i(z)$ denote the vertex operators at the vertices $Y$ and $i\neq Y$, with the parametrization \ref{eq:expr_Y_X_mutation}. After mutation of the chiral cluster seed at the fermionic vertex $Y$, the new vertex operators at these vertices can be written the form
    \begin{equation}\label{eq:def_tX_tY}
    \tY(z)=-:\tPsi(q^{-1}z)\tPsi^\ast(qz):,\qquad \tX_i(z):=C_i:\prod_{a\in \tilde{A}(Y\to i)}\tPsi^\ast(q^{-1}\tilde{\k}_a^{-1}z)\times\prod_{b\in \tilde{A}(i\to Y)}(-\tPsi(q\tilde{\k}_b z)): V_i(z),
    \end{equation}
    where $\tilde{\k}_a=q^{-1}\k_a^{-1}$ are the parameters of the arrows after mutation, and $C_i$ is an arbitrary constant.\footnote{While the constants $C_i$ are not determined by OPEs, in all the applications discussed in this paper we find $C_i=1$.}

    %     After the mutation of a chiral cluster seed at a fermionic node $Y$, the vertex operators at the nodes $Y$ and $i\neq Y$ can be expressed in the form
    % \begin{equation}\label{eq:def_tX_tY}
    % \tY(z)=-:\tPsi(q^{-1}z)\tPsi^\ast(qz):,\qquad \tX_i(z):=C_i:\prod_{a\in \tilde{A}(Y\to i)}\tPsi^\ast(q^{-1}\tilde{\k}_a^{-1}z)\times\prod_{b\in \tilde{A}(i\to Y)}(-\tPsi(q\tilde{\k}_b z)): V_i(z),
    % \end{equation}
    % where $V_i(z)$ is the vertex operator appearing in the formula \ref{eq:expr_Y_X_mutation} before mutation, $\tilde{\k}_a=q^{-1}\k_a^{-1}$ are the parameters of the arrows after mutation, and $C_i$ are arbitrary constants.\footnote{While the constants $C_i$ are not determined by OPEs, in all the applications discussed in this paper we find $C_i=1$.}
\end{proposition}

\begin{proof}
    To prove this proposition, we need to show that the OPEs of vertex operators $\tY(z)$ and $\tX_i(z)$ defined by the equation \ref{eq:def_tX_tY} do indeed coincide with the ones encoded by the new chiral cluster seed. This is an easy check for the OPEs $X_i(z)Y(w)$ and $Y(z)X_i(w)$. The non-trivial part is to check that OPEs of any two vertices different from $Y$,i.e. $X_i(z)X_j(w)$. For this purpose, we compute the following ratio of OPEs,
    \begin{equation}
        \dfrac{\la\tX_i(z)\tX_j(w)\ra}{\la X_i(z)X_j(w)\ra}=\prod_{\superp{b\in \tilde{A}(i\to Y)}{a\in \tilde{A}(Y\to j)}}\dfrac{\la\tPsi(\k_b^{-1}z)\tPsi^\ast(\k_a w)\ra}{\la\Psi^\ast(q^{-1}\k_b^{-1}z)\Psi(q\k_aw)\ra}\times \prod_{\superp{a\in \tilde{A}(Y\to i)}{b\in \tilde{A}(j\to Y)}}\dfrac{\la\tPsi^\ast(\k_az)\tPsi(\k_b^{-1} w)\ra}{\la\Psi(q\k_a z)\Psi^\ast(q^{-1}\k_b^{-1}w)\ra}.
    \end{equation}
    This expression is obtained by noticing that, by mutation rules, $\tilde{A}(i\to Y)=A(Y\to i)$ and $\tilde{A}(Y\to i)=A(i\to Y)$ (for $i\neq Y$), and using the fact that factors coming from $V_iV_j$, $\Psi_i\Psi_j/\tPsi_i^\ast\tPsi_j^\ast$ and $\Psi^\ast_i\Psi_j^\ast/\tPsi_i\tPsi_j$ cancel each other. Evaluating these OPEs, we find
    \begin{equation}
        \dfrac{\la\tX_i(z)\tX_j(w)\ra}{\la X_i(z)X_j(w)\ra}=\prod_{\superp{b\in\tilde{A}(i\to Y)}{a\in\tilde{A}(Y\to j)}}\vphi_{-1}(z,\k_a\k_bw)\times \prod_{\superp{a\in \tilde{A}(Y\to i)}{b\in\tilde{A}(j\to Y)}}\vphi_1(\k_a\k_bz,w).
    \end{equation}
    The first product encodes the presence of the new arrows $j\to i$ with parameter $\k_a\k_b$, and the second product the new arrows $i\to j$ also with parameter $\k_a\k_b$. This is indeed the expected result from the mutation rules. Specializing this result to $i=j$, we observe that the self-OPE $X_i(z)X_i(w)$ contains factors corresponding to new loops appearing for every pair of arrows $(i\to Y,Y\to i)$, in agreement with the mutation rules, hence proving the claim.
\end{proof}

Proposition \ref{prop:mutation} will be used to study the property of currents under mutations. More precisely, it is useful in order to express currents defined in terms of vertex operators $X_i(z)$, using the new vertex operators $\tX_i(z)$. This procedure is explained in the following examples. 

% \begin{figure}
% \begin{center}
% \begin{tikzpicture}[scale=1, font=\small]
% \node[styleNode] (x) at (0,0) {$X$};
% \node[styleNode] (y) at (2,2) {$Y$};
% \node[styleNode] (z) at (4,0) {$Z$};
% \draw[styleArrow](x) to[] node[midway,left]{$\k_1$}(y);
% \draw[styleArrow](y) to[] node[midway,right]{$\k_2$}(z);
% \node at (5.5,1) {$\longleftrightarrow$};
% \node[styleNode] (xp) at (7,0) {$X$};
% \node[styleNode] (yp) at (9,2) {$Y$};
% \node[styleNode] (zp) at (11,0) {$Z$};
% \draw[styleArrow](yp) to[] node[midway,left]{$q^{-1}\k_1^{-1}$}(xp);
% \draw[styleArrow](zp) to[] node[midway,right]{$q^{-1}\k_2^{-1}$}(yp);
% \draw[styleArrow](xp) to[] node[midway,below]{$\k_1\k_2$} (zp);
% \end{tikzpicture}
% \caption{Example of a quiver mutation at the node $Y$.}
% \label{fig_mutation_affine}
% \end{center}
% \end{figure}
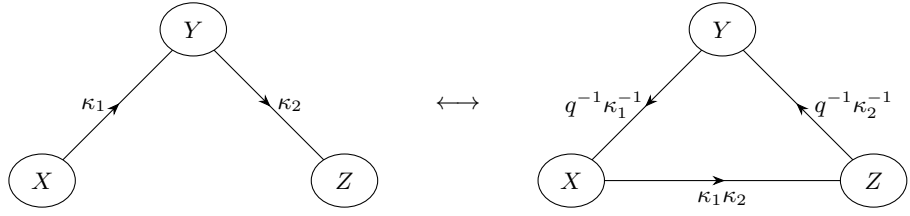
\begin{figure}
\begin{center}
\begin{tikzpicture}[scale=1, font=\small]

% Coordinates
\coordinate (x) at (0,0);
\coordinate (y) at (2,2);
\coordinate (z) at (4,0);

\coordinate (xp) at (7,0);
\coordinate (yp) at (9,2);
\coordinate (zp) at (11,0);

% Arrows
\draw[
	qArrow/base,
	qArrow/label={$\k_1$},
	qArrow/labelstyle={left, inner sep=1pt},
	qArrow/label xshift=-3pt,
	qArrow/label yshift=0pt
] (x) to (y);

\draw[
	qArrow/base,
	qArrow/label={$\k_2$},
	qArrow/labelstyle={right, inner sep=1pt},
	qArrow/label xshift=2pt,
	qArrow/label yshift=0pt
] (y) to (z);

\node at (5.5,1) {$\longleftrightarrow$};

\draw[
	qArrow/base,
	qArrow/label={$q^{-1}\k_1^{-1}$},
	qArrow/labelstyle={left, inner sep=1pt},
	qArrow/label xshift=0pt,
	qArrow/label yshift=0pt
] (yp) to (xp);

\draw[
	qArrow/base,
	qArrow/label={$q^{-1}\k_2^{-1}$},
	qArrow/labelstyle={right, inner sep=1pt},
	qArrow/label xshift=5pt,
	qArrow/label yshift=0pt
] (zp) to (yp);

\draw[
	qArrow/base,
	qArrow/label={$\k_1\k_2$},
	qArrow/labelstyle={below, inner sep=1pt},
	qArrow/label xshift=0pt,
	qArrow/label yshift=-3pt
] (xp) to (zp);

% Nodes
\node[qNodeUnfrozen] at (x) {$X$};
\node[qNodeUnfrozen] at (y) {$Y$};
\node[qNodeUnfrozen] at (z) {$Z$};

\node[qNodeUnfrozen] at (xp) {$X$};
\node[qNodeUnfrozen] at (yp) {$Y$};
\node[qNodeUnfrozen] at (zp) {$Z$};

\end{tikzpicture}
\caption{Example of a quiver mutation at the node $Y$.}
\label{fig_mutation_affine}
\end{center}
\end{figure}

\begin{example} As a first example of mutation, let's come back to the configuration of Example \ref{Ex:XYZ} (see Figure \ref{fig_mutation_affine} (left)). In equ. \ref{eq:XYZ}, the corresponding vertex operators have been parameterized in terms of the fermionic fields $(\Psi(z),\Psi^\ast(z))$ associated with the vertex $Y$. We consider the two currents $W(z)=\S[X,Y](z)$ and $\chi(z)=\Psi[X,Y,Z]$ introduced in this Example. Their expression in terms of fermionic fields has been given in equ. \ref{eq:XYZ}. Applying the transformation \ref{transfo_fermions}, we find the expressions
\begin{equation}
    W(z)=V_X(z)\tPsi^\ast(\k_1 z),\qquad \chi(z)=-:V_X(z)V_Z(\k_1\k_2 z):.
\end{equation}
These expressions can be identified with $\S[\tX](z)$, and $\Psi[\tX,\tZ](z)$ respectively, where the new vertex operators coincide with those given in Proposition \ref{prop:mutation}, namely
\begin{equation}
    \tX(z)=\tPsi^\ast(\k_1 z)V_X(z),\qquad \tZ(z)=-V_Z(z)\tPsi(\k_2^{-1}z).
\end{equation}
By Proposition \ref{prop:mutation}, the OPEs of these vertex operators coincide with the ones obtained from the mutated quiver represented in Figure \ref{fig_mutation_affine} (right).

Next, we consider the mutation from the quiver of Figure \ref{fig_mutation_affine} (right) to Figure \ref{fig_mutation_affine} (left). The vertex operators associated to the quiver on the right can be parameterized as
\begin{equation}
    X(z)=V_X(z)\Psi^\ast(\k_1 z),\qquad Y(z)=-:\Psi(q^{-1}z)\Psi^\ast(qz):,\qquad Z(z)=V_Z(z)\Psi(\k_2^{-1}z),
\end{equation}
with $V_X(z),\ V_Z(z)$ commuting with $(\Psi(w),\Psi^\ast(w))$. Let's consider the same currents, namely $W(z)=\S[X](z)$ and $\chi(z)=\Psi[X,Z](z)$. In this parameterization,
\begin{equation}
    W(z)=V_X(z)\Psi^\ast(\k_1 z),\qquad \chi(z)=-q^{-1}:V_X(z)V_Z(\k_1\k_2 z):.
\end{equation}
After the transformation \ref{transfo_fermions}, we find
\begin{equation}
    W(z)=V_X(z)(\tPsi(q^{-1}\k_1z)-\tPsi(q\k_1 z)),\qquad \chi(z)=:V_X(z)V_Z(\k_1\k_2 z):,
\end{equation}
which can be identified respectively with $\S[\tX,\tY](z)$ and $\Psi[\tX,\tY,\tZ](z)$, where the new vertex operators coincide with those given in Proposition \ref{prop:mutation}. We recover the original expression \ref{eq:W_XYZ} of the currents $W(z)$ and $\chi(z)$, thus showing that the mutation is indeed involutive on the currents.
% \begin{equation}
%     \tX(z)=-V_X(z)\tPsi(q\k_1 z),\qquad \tY(z)=-:\tPsi(q^{-1}z)\tPsi^\ast(qz):,\qquad \tZ(z)=V_Z(z)\tPsi^\ast(q^{-1}\k_2^{-1}z).
% \end{equation}

\end{example}

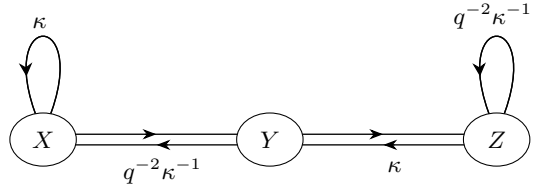
\begin{figure}[h]
\centering
\begin{tikzpicture}[scale=1, transform shape, font=\small]
\coordinate (X) at (0,0);
\coordinate (Y) at (3,0);
\coordinate (Z) at (6,0);

%Horizontal arrows
\draw[qArrowAboveOp, qArrow/arrowpos = 0.55](Y) to (X);
\draw[qArrowAboveOp, qArrow/arrowpos = 0.55](Z) to (Y);
\draw[qArrowBelowOp, qArrow/label={$q^{-2}\k^{-1}$}, qArrow/arrowpos = 0.55](X) to[] (Y);
\draw[qArrowBelowOp, qArrow/label={$\k$}, qArrow/label xshift=0pt, qArrow/label yshift=-5pt, qArrow/arrowpos = 0.55](Y) to[] (Z);
%Loops
\draw[qLoop, qArrow/label={$\k$}](X) to (X);
\draw[qLoop, qArrow/label={$q^{-2}\k^{-1}$}](Z) to (Z);
%Nodes
\node[qNodeUnfrozen] (X) at (X) {$X$};
\node[qNodeUnfrozen] (Y) at (Y) {$Y$};
\node[qNodeUnfrozen] (Z) at (Z) {$Z$};
\end{tikzpicture}
	\caption{Quiver obtained after mutation of the quiver of Figure \ref{fig:XYZ_double} at the vertex $Y$ (and gauging).}\label{fig:XYZ_double_mutated}
\end{figure}

\begin{example}\label{Ex:XYZ_double_mutation}
    When studying the mutation on currents, it is important to perform the fermionic transformation \ref{transfo_fermions} outside the normal order. To illustrate this point, let's consider the configuration of Example \ref{Ex:XYZ_double} associated with the decorated quiver of  Figure \ref{fig:XYZ_double}. After mutation at the vertex $Y$, and gauging $Y$ by $q\k$, we find the quiver represented in Figure \ref{fig:XYZ_double_mutated}. 

    Let's examine the behavior of the currents $W(z)$ and $\bar W(z)$ defined in equ. \ref{eq:W_XYZ_double}. Taking into account the non-trivial OPEs between fermionic currents, and the expression \ref{eq:alpha} of $\a$, the currents can be written in the form
    \begin{align}
        &W(z)=-(\k^{1/2}-\k^{-1/2})V_X(z)\Psi^\ast(q^{-1}\k^{-1}z)(\Psi(qz)-\Psi(q^{-1}z))-:V_X(z)V_Z(z):,\\
        &\bar W(z)=(\k^{1/2}-\k^{-1/2})V_Z(z)\Psi^\ast(q^{-1}z)(\Psi(q^{-1}\k^{-1}z)-\Psi(q^{-3}\k^{-1}z))-\a :V_Z(z)V_X(q^{-2}z):.
    \end{align}
    These expressions contain no normal ordering, and so it is possible to perform the substitution \ref{transfo_fermions},
    \begin{align}
        &W(z)=-(\k^{1/2}-\k^{-1/2})V_X(z)(-\tPsi(\k^{-1}z)+\tPsi(q^{-2}\k^{-1}z))\tPsi^\ast(z)-:V_X(z)V_Z(z):,\\
        &\bar W(z)=(\k^{1/2}-\k^{-1/2})V_Z(z)(-\tPsi(z)+\tPsi(q^{-2}z))\tPsi^\ast(q^{-2}\k^{-1}z)-\a :V_Z(z)V_X(q^{-2}z):.
    \end{align}
    Re-introducing the normal-order, the currents take the form
    \begin{align}
        &W(z)=\tilde{X}(z)+\a:\tilde{X}(z)\tilde{Y}(z):+:\tilde{X}(z)\tilde{Y}(z)\tilde{Z}(z):,\\
        &\bar W(z)=\a \tilde{Z}(z)+:\tilde{Z}(z)\tilde{Y}(\k z):+\a:\tilde{Z}(z)\tilde{Y}(\k z)\tilde{X}(q^{-2}z):.
    \end{align}
    where the new vertex operators agrees with the expressions of Proposition \ref{prop:mutation},\footnote{Note the shift of the argument of the vertex operator $\tY(z)$ due to gauging.} 
    \begin{equation}
        \tilde{X}(z)=-V_X(z):\tPsi(\k^{-1} z)\tPsi^\ast(z):,\qquad \tilde{Y}(q\k z)=-:\tPsi(q^{-1}z)\tPsi^\ast(qz):,\qquad \tilde{Z}(z)=-V_Z(z):\tPsi(z)\tPsi^\ast(q^{-2}\k^{-1}z):.
    \end{equation}
    Comparing with the original expressions \ref{eq:W_XYZ_double}, we observe that the constant $\a$ that deforms the telescoping sum moved in front of a different term as an effect of the treatment of normal orders.
\end{example}

\section{Deformed W-algebras and their free field realizations} \label{sec:def W and free field}

In this section, we apply the formalism of chiral cluster seeds introduced previously to study the free field realizations of several known families of deformed W-algebras. We show that certain free field realizations can be obtained as an algebra of currents embedded in a chiral cluster seed, and commuting with the screening charges associated with its vertices. We denote such free field realizations $\CF_{\CQ}[\CWqt]$, where $\CQ$ is the decorated quiver associated with the chiral cluster seed, and $\CWqt$ the corresponding deformed W-algebra. We also show that mutations can be used to relate different free field realizations of the same deformed W-algebra.

We first consider the case of the algebra $\CWqt(\gl(N))$ corresponding to the deformation of the principal W-algebras $W_N$ (extended by a Heisenberg algebra $\CH_1$). Then, we study the deformed W-algebras $\CWqt(\gl(N|M))$ associated with super Lie algebras $\sl(N|M)$. Their free field realizations can be obtained as representations of the quantum toroidal $\gl(1)$ algebra. We show how they are realized in our formalism, and how mutations relate free field realizations associated with different Dynkin diagrams of $\sl(N|M)$.

Finally, we study the quantum affine algebra of $\sl(2)$, and the $(\qf,\tf)$-deformation of the Bershadsky-Polyakov algebra. These two algebras can be seen as deformations of subregular W-algebras of rank one and two respectively. For each algebra, we introduce two different free field realizations. We show that these realizations are associated with different decorated quivers related by a mutation.

\subsection{Deformed W-algebras $\CWqt(\sl(N))$ and $\CWqt(\gl(N))$}
The deformed W-algebras $\CWqt(\sl(N))$ have been introduced as deformation of principal W-algebras $W_N$ \cite{AKOSwalg97,Feigin1996quantum}. They depend on two quantum parameters, often denoted $(\qf,\tf)$, and are defined in terms of the generating currents $W_i(z)$ ($i=1\cdots N-1$). Here, we will instead use the parameters $(q_1,q_2,q_3)=(\tf^{-1},\qf,\pf^{-1})$ with $\pf=\qf/\tf$. The current $W_1(z)$ is known to generate the full algebra, since higher currents can be obtained inductively from the algebraic relation
\begin{equation}\label{eqn:qRegWAlgRelationsFirstGen}
\begin{aligned}
f^{1j}\brac*{\frac{w}{z}}\,W_1(z)W_j(w)
-
W_j(w)W_1(z)\,f^{j1}\brac*{\frac{z}{w}}
&=
-\frac{(1-q_1)(1-q_2)}{1-q_3^{-1}}
\\
&\quad\times \brac*{ \delta\brac*{q_3^{-\frac{j+1}{2}}\frac{w}{z}}\,W_{j+1}\brac*{q_3^{-\frac12}w} - \delta\brac*{q_3^{\frac{j+1}{2}}\frac{w}{z}}\,W_{j+1}\brac*{q_3^{\frac12}w} },
\end{aligned}
\end{equation}
where $f^{1j}(x)$ is a special case of the functions $f^{kl}(x)$ defined in \cite{AKOSwalg97},
\begin{equation}\label{eq:def_fkl}
    f^{kl}(x)=\exp\left(\sum_{n>0}\dfrac{x^n}{n}(1-q_1^{-n})(1-q_2^{-n})q_3^{-n(k+l)/2}\dfrac{1-q_3^{n\min(k,l)}}{1-q_3^n}\dfrac{1-q_3^{n(\max(k,l)-N)}}{1-q_3^{-nN}}\right).
\end{equation}
This current plays a key role in our discussion.  

\paragraph{Free field realization} We recall the free field realization of the deformed algebra $\CWqt(\sl(N))$ proposed in \cite{AKOSwalg97}, and rewrite it in terms of the parameters $(q_1,q_2,q_3)$. It is also convenient to introduce $\b=-(\log q_1)/(\log q_2)$ and $\tau=\sqrt{-(\log q_1)(\log q_2)}$, such that $(q_1,q_2)=(\ee^{-\tau\sqrt{\beta}},\ee^{\tau/\sqrt{\beta}})$. In this realization, the currents $W_i(z)$ are written as sums of vertex operators $\L_i(z)$ with $i=1\cdots N$,
\begin{equation}\label{eqn:qRegWAlgGeneratorsViaLambda}
W_i\brac*{q_3^{(i-1)/2}z}
= \sum_{1\le j_1<\cdots<j_i\le N} \no{\Lambda_{j_1}(z)\Lambda_{j_2}(q_3 z)\cdots \Lambda_{j_i}\brac*{q_3^{i-1}z}}.
\end{equation}
These vertex operators can be decomposed into modes $\l_{i,n}$, 
\begin{equation}
\Lambda_i(z)
= q_3^{\,i-(N+1)/2} \ee^{\sum_{k>0} z^k \l_{i,-k}} \ee^{\sum_{k>0} z^{-k} \l_{i,k}} \ee^{\tau \l_{i,0}},
\end{equation}
satisfying the commutation relations
\begin{align}\label{eqn:comm_hk}
\comm{\l_{i,k}}{\l_{j,l}}
&=
-\frac{1}{k}\brac{1-q_1^k}\brac{1-q_2^k}\,\delta_{k+l}\,
\begin{cases}
\dfrac{1-q_3^k}{1-q_3^{-kN}}\,q_3^{-kN}, & i<j,\\[1ex]
\dfrac{1-q_3^{-k(N-1)}}{1-q_3^{-kN}}, & i=j,\\[1ex]
\dfrac{1-q_3^k}{1-q_3^{-kN}}, & i>j,
\end{cases}\\
\comm{\l_{i,n}}{\L_j}&= \brac*{\delta_{ij}-\frac{1}{N}}\kdelta{n}{0},
\end{align}
where the zero modes $\L_j$ only enter in the expression of the screening currents. These modes are also required to satisfy the following constraints, compatible with the previous commutation relations, that reduces the number of independent bosons from $N$ to $N-1$,
\begin{equation}\label{eqn:qRegWBosonConstraints}
\begin{aligned}
\sum_{i=1}^N q_3^{-ik}\,\l_{i,k}=0,
\qquad
\sum_{i=1}^N \L_i=0.
\end{aligned}
\end{equation}
In particular, it implies $W_N(q_3^{(N-1)/2}z)=:\L_1(z)\L_1(q_3z)\cdots\L_N(q_3^{N-1}z):=1$.

\paragraph{Chiral cluster seed} In order to fit in the framework of chiral cluster seeds, we need to consider the extension $\CWqt(\gl(N))=\CWqt(\sl(N))\otimes \CH_1(\gl(N))$ of the algebra $\CWqt(\sl(N))$ by an extra Heisenberg algebra of rank one. The extended algebra is generated by the elements $\b_k$ of $\CH_1(\gl(N))$, and the modes of dressed currents $\bar W_i(z)=\U_i(z)W_i(z)$, where $\U_i(z)\in\CH_1(\gl(N))$. It is convenient to normalize the generators $\b_k$ such that they satisfy the commutation relations
    \begin{equation}
    \comm{\beta_k}{\beta_l}
    = -\frac{1}{k} \brac*{1-q_1^k}\brac*{1-q_2^k}\brac*{1-q_3^k}\brac*{1-q_3^{N k}} \delta_{k+l}.
\end{equation}
We will use the following dressing factors for the spin one and $N-1$ currents respectively,
% \begin{equation}
%     U_i(q_3^{(i-1)/2}z)=:\prod_{j=1}^i U(q_3^{j-1}z):,\qquad U(z)=\exp\brac*{
% \sum_{k>0}
% \frac{q_3^{-N k} z^k}{q_3^{N k/2}-q_3^{-N k/2}}
% \,\beta_{-k}}
% \exp\brac*{
% \sum_{k>0}
% \frac{z^{-k}}{q_3^{N k/2}-q_3^{-N k/2}}
% \,\beta_k}.
% \end{equation}
\begin{align}
    \U_1(z)&=\exp\brac*{\sum_{k>0}\frac{q_3^{-N k/2} z^k}{q_3^{N k/2}-q_3^{-N k/2}}\,\beta_{-k}}\exp\brac*{\sum_{k>0}\frac{q_3^{-Nk/2}z^{-k}}{q_3^{N k/2}-q_3^{-N k/2}}\,\beta_k},\\
    \U_{N-1}(q_3^{(N+1)/2}z)&=\exp\brac*{-\sum_{k>0}\frac{q_3^{Nk/2}z^k}{q_3^{N k/2}-q_3^{-N k/2}}\,\beta_{-k}}\exp\brac*{-\sum_{k>0}\frac{q_3^{Nk/2}z^{-k}}{q_3^{N k/2}-q_3^{-N k/2}}\,\beta_k}.
\end{align}
These dressing factors play an important role in physics, e.g. in the context of the AGT correspondence \cite{AGT2009,Awata:2009ur,Awata:2011dc}, where they are often referred to as a \textit{$U(1)$ factors}. They also enter in the relation with the Fock representation of the quantum toroidal $\gl(1)$ algebra \cite{FHHSY2009,FJMV2020}. We will recall this relation in the next subsection in the more general context of deformed W-algebras $\CWqt(\gl(N|M))$.

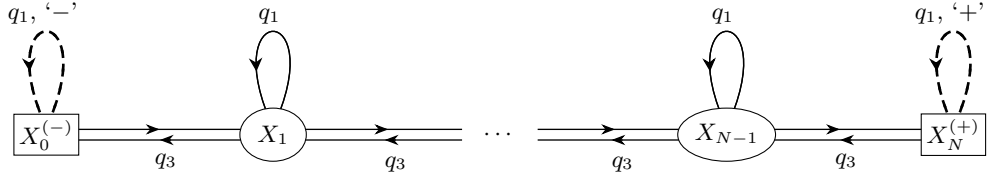
\begin{figure}[h]
\centering
\begin{tikzpicture}[scale=1, transform shape, font=\small]
\coordinate (Y0) at (0,0);
\coordinate (Y1) at (3,0);
\coordinate (ymidBefore) at (5.5,0);
\coordinate (ymid) at (6,0);
\coordinate (ymidAfter) at (6.5,0);
\coordinate (YNmin1) at (9,0);
\coordinate (YN) at (12,0);

%Horizontal arrows
\draw[qArrowAboveOp, qArrow/arrowpos = 0.55](Y1) to (Y0);
\draw[qArrowBelowOp, qArrow/label={$q_3$}, qArrow/arrowpos = 0.55](Y0) to[] (Y1);

\draw[qArrowAboveOp, qArrow/label xshift=0pt, qArrow/label yshift=0pt, qArrow/arrowpos = 0.45](ymidBefore) to (Y1);
\draw[qArrowBelowOp, qArrow/label={$q_3$}, qArrow/label xshift=0pt, qArrow/label yshift=-5pt, qArrow/arrowpos = 0.65](Y1) to[] (ymidBefore);

\draw[qArrowAboveOp, qArrow/label xshift=0pt, qArrow/label yshift=0pt, qArrow/arrowpos = 0.65](YNmin1) to (ymidAfter);
\draw[qArrowBelowOp, qArrow/label={$q_3$}, qArrow/label xshift=0pt, qArrow/label yshift=-5pt, qArrow/arrowpos = 0.45](ymidAfter) to[] (YNmin1);
\draw[qArrowAboveOp, qArrow/arrowpos = 0.55](YN) to (YNmin1);
\draw[qArrowBelowOp, qArrow/label={$q_3$}, qArrow/arrowpos = 0.55](YNmin1) to[] (YN);
%Loops
\draw[qLoopDashed, qArrow/label={$q_1$, `$-$'}](Y0) to (Y0);
\draw[qLoopDashed, qArrow/label={$q_1$, `$+$'}](YN) to (YN);
\draw[qLoop, qArrow/label={$q_1$}](Y1) to (Y1);
\draw[qLoop, qArrow/label={$q_1$}](YNmin1) to (YNmin1);
%Nodes
\node[qNodeFrozen] (Y0) at (Y0) {$X^{(-)}_0$};
\node[qNodeUnfrozen] (Y1) at (Y1) {$X_1$};
\node (ymid) at (ymid) {$\cdots$};
\node[qNodeUnfrozen] (YNmin1) at (YNmin1) {$X_{N-1}$};
\node[qNodeFrozen] (YN) at (YN) {$X_{N}^{(+)}$};
\end{tikzpicture}
	\caption{%
		Quiver for the free field realization of $\CWqt(\gl(N))$.
	\label{fig:qRegWAlgQuiver}}
\end{figure}

Let's consider the chiral cluster seed associated with the decorated quiver represented in Figure \ref{fig:qRegWAlgQuiver}. We show below that it is naturally associated with the free field realization \ref{eqn:qRegWAlgGeneratorsViaLambda} of $\CWqt(\sl(N))$ dressed by the $U(1)$-factor.

\begin{proposition}
    Let $\CQ_\text{AKOS}$ denote the linear quiver of Figure \ref{fig:qRegWAlgQuiver}, and consider the paths $\CP_+:X_0\to X_1\to\cdots\to X_N$ and $\CP_-:X_N\to X_{N-1}\to\cdots\to X_0$. The following assignment defines a free field realization of the algebra $\CWqt(\gl(N))$,
    \begin{align}
        \bar W_1(z)&=\S[\CP_+](z),\qquad &:\U_1(z)\U_{N-1}(q_3^{(N+1)/2}z):&=\Psi[\CP_+](z),\\
        \bar W_{N-1}(q_3^{(N+1)/2}z)&=\S[\CP_-](z),\qquad &:\U_1(q_3^Nz)\U_{N-1}(q_3^{(N+1)/2}z):&=\Psi[\CP_-](z).
    \end{align}
    The subalgebra generated by these currents commutes with the bosonic screening charges attached to the unfrozen vertices $i=1\cdots N-1$. It will be denoted $\CF_{\CQ_\text{AKOS}}\left[\CWqt(\gl(N))\right]$.
\end{proposition}

\begin{proof} Starting from the free field realization of $\CWqt(\sl(N))$ and the dressing factors, we can rewrite the currents $\bar W_1(z)$ and $\bar W_{N-1}(z)$ in the form of telescoping sums with
\begin{equation}
    X_j(z)=
    \begin{cases}
        \U_1(z)\L_N(z), & j=0,\\
        :\L_{N-j}(z)\L_{N+1-j}(z)^{-1}:, & j=1\cdots N-1,\\
        \U_{N-1}(q_3^{(N+1)/2}z):\L_1(q_3z)^{-1}:, & j=N.
    \end{cases}
\end{equation}
It is then easy to check that the OPEs of these vertex operators coincide with the rational functions encoded in the decorated quiver. Here, the presence of the $U(1)$-factor is essential to find rational OPEs for the vertex operators $X_0(z)$ and $X_N(z)$.

We note that the vertex operators $X_i(z)$ for $i=1\cdots N-1$ are independent of the dressing factors. The associated screening currents, as defined in the previous section, coincide with those defined in \cite{AKOSwalg97} (up to a trivial rescaling of the spectral parameter). Thus, the corresponding screening charges commute with $\bar W_1(z)$ and $\bar W_{N-1}(z)$, and by consequence with all currents $\bar W_i(z)$.

Taking the products of vertex operators $\Psi[\CP_\pm](z)$, the dependence in the modes $\l_{i,n}$ disappears due to the constraint \ref{eqn:qRegWBosonConstraints}, and we recover the expressions given in the proposition. They commute trivially with the bosonic screening charges.
\end{proof}

\subsection{Deformed W-algebras $\CWqt(\gl(N|M))$}\label{sec:WglN}
Deformed W-algebras associated to Lie superalgebras have been introduced by Ding and Feigin in \cite{Ding1998}. Since then, they have been extensively studied in the literature using various methods, e.g. quantum toroidal algebras \cite{Bershtein2018,FJMV2020}, Corner VOA \cite{Harada:2021xnm}, qq-characters \cite{Feigin:2021jzm},... It is worth noting that quadratic relations among generating currents have also been obtained \cite{Kojima2021a,Kojima2021b}. In this paper, we focus on the $\gl$ version of the algebra, i.e. the extension $\CWqt(\gl(N|M))$ of $\CWqt(\sl(N|M))$ by a rank one Heisenberg algebra. This version is better suited to the chiral cluster formalism since OPEs are rational instead of being elliptic.

In \cite{Bershtein2018,FJMV2020}, deformed W-algebras $\CWqt(\gl(N|M))$ are defined through a correspondence with the quantum toroidal \(\gl(1)\) algebra. More specifically, free field realizations of $\CWqt(\gl(N|M))$ are obtained as the tensor products of Fock representations of the quantum toroidal $\gl(1)$ algebra. In this section, we take a similar approach, and review the construction below. Then, we reformulate it in the language of chiral cluster seeds, and show that the different free field realizations are obtained by mutations of the corresponding decorated quiver. 

% The deformed W-algebra $\CWqt(\sl(2|1))$ was introduced in \cite{Ding1998} using the screening currents approach. We call 
% The version for arbitrary \(N,M\), with an additional Heisenberg algebra (we call this version $\CWqt(\gl(N|M))$), was studied through its correspondence with the quantum toroidal \(\gl(1)\) algebra in \cite{Bershtein2018,FJMV2020}. More specifically, free field realizations of $\CWqt(\gl(N|M))$ have been obtained as the tensor products of Fock representations of the quantum toroidal $\gl(1)$ algebra.

% In this subsection, we review this construction and relate it to the chiral cluster seed formalism of the previous section. 
% In particular, we show how mutations of the quiver lead to different free field realizations of the same algebra $\CWqt(\gl(N|M))$.
% We use this (\(\gl\) rather than \(\sl\)) version of the algebra since it has rational rather than elliptic OPEs, and our chiral cluster construction is suited for that. 

% A quadratic relations among generators was found later by Kojima in \cite{Kojima2021a} and for . 
% This presentation was then extended in \cite{Kojima2021b} to deformed W-algebras $\CWqt(\sl(N|M))$ associated with super Lie-algebras $\sl(N|M)$. 

\paragraph{Quantum toroidal $\gl(1)$ construction}
The quantum toroidal $\gl(1)$ algebra $\CE$ is a quantum group with parameters $q_1,q_2\in\mC^\times$, generated by the modes of four generating currents $x^\pm(z)$ and $\psi^\pm(z)$, and a central elements $C$. These currents have the modes decomposition
\begin{equation}
    x^\pm(z)=\sum_{n\in\mZ}z^{-n}x_n^\pm,\qquad \psi^\pm(z)=\psi_0^\pm \exp\left(\pm\sum_{k>0}z^{\mp k}h_{\pm k}\right),
\end{equation}
and satisfy the following set of relations,
\begin{equation}
\label{def_DIM}
\begin{aligned}
&[\psi^\pm(z),\psi^\pm(w)]=0,\quad \psi^+(z)\psi^-(w)=\dfrac{g(C w/z)}{g(C^{-1}w/z)}\psi^-(w)\psi^+(z),\\
&\psi^+(z)x^\pm(w)=g(C^{\pm1/2}z/w)^{\pm1}x^\pm(w)\psi^+(z),\quad \psi^-(z)x^\pm(w)=g(C^{\mp1/2}z/w)^{\pm1}x^\pm(w)\psi^-(z),\\
&\prod_{i=1,2,3}(z-q_i^{\pm1}w)\ x^\pm(z)x^\pm(w)=\prod_{i=1,2,3}(z-q_i^{\mp1}w)\ x^\pm(w)x^\pm(z),\\
&[x^+(z),x^-(w)]=(1-q_1)(1-q_2)(1-q_3)\left(\d(C^{-1}z/w)\psi^+(C^{1/2}w)-\d(C z/w)\psi^-(C^{-1/2}w)\right),
\end{aligned}
\end{equation}
where $g(z)$ is the structure function
\begin{equation}\label{def_g}
g(z) = \prod_{\a=1,2,3}\dfrac{1-q_\a z}{1-q_\a^{-1}z}.
\end{equation}
It is noted that the zero modes $\psi_0^\pm=(C^\perp)^{\mp1/2}$ are also central. The subalgebra generated by the elements $\psi_{\pm k}^\pm$ and the central element $C$ will be denoted $\mathfrak{h}$. The algebra $\CE$ has the structure of a topological Hopf algebra with the Drinfeld coproduct defined as
\begin{align}
\begin{split}\label{Drinfeld_coproduct}
&\D(x^+(z))=x^+(z)\otimes 1+\psi^-(C_1^\frac12z)\otimes x^+(C_1z),\\
&\D(x^-(z))=x^-(C_2 z)\otimes \psi^+(C_2^{\frac12}z)+1\otimes x^-(z),\\
&\D(\psi^\pm(z))=\psi^\pm(C_2^{\pm\frac12}z)\otimes\psi^\pm(C_1^{\mp\frac12}z),\quad \D(h_k)=h_k\otimes C^{-|k|/2}+C^{|k|/2}\otimes h_k,\\
\end{split}
\end{align}
with $C_1=C\otimes 1$, $C_2=1\otimes C$ and $\D(C)=C_1C_2$. The expressions of antipode and counit will not be needed in this paper.

This algebra is known to have three different Fock representations $\rho_c$ indexed by a choice $c=1,2,3$ for the representation of the central element $\rho_c(C)=q_c^\frac12$. These are free field representations defined using a single Heisenberg algebra $\CH_1$ generated by modes $[J_k,J_l]=k\d_{k+l,0}$ (and no zero modes). We denote $\CF_c$ the corresponding Fock module on which the generators act. 
In fact, we only need the Fock representations with $c=1,3$, and for simplicity we restrict our definition to these two.\footnote{The representation $\rho_2$ is obtained by permuting the quantum group parameters.} 
In the representation $\rho_c$, the Drinfeld currents take the form,
\begin{align}
    &\rho_c(h_k)=-\dfrac{s_c^{-k/2}}{k}(1-q_2^k)(1-q_c^{k})J_k,\quad \rho_c(h_{-k})=-\dfrac{s_c^{-k/2}}{k}(1-q_c^{2k})(1-q_2^{-k}q_c^{-k})J_{-k},\quad (k>0),\\
    &\rho_c(x^+(z))=\g_c^+u\exp\left(\sum_{k>0}\dfrac{z^{k}}{k}(1-q_2^{-k}q_c^{-k})J_{-k}\right)\exp\left(-\sum_{k>0}\dfrac{z^{-k}}{k}(1-q_2^k)J_k\right),\\
    &\rho_c(x^-(z))=\g_c^-u^{-1}\exp\left(-\sum_{k>0}\dfrac{z^{k}}{k}s_c^{k}(1-q_2^{-k}q_c^{-k})J_{-k}\right)\exp\left(\sum_{k>0}\dfrac{z^{-k}}{k}s_c^{k}(1-q_2^k)J_k\right).
\end{align}
In order to reproduce the W-algebra currents, it is convenient to parameterize the constant factors in front of the currents as follows,
\begin{equation}
    \g_1^\pm=\mp t_1 ,\qquad \g_3^\pm=\pm t_3 ,
\end{equation}
and $u\in\mC^\times$ is an arbitrary weight.

Representations $\rho_c$ serve as building blocks for higher representations obtained using the coproduct. 
Let $\bsc\in\{1,3\}^{\otimes (N+M)}$ be a coloring vector that encodes a succession of Fock modules $\CF_\bsc=\CF_{c_1}\otimes\cdots\otimes\CF_{c_{N+M}}$. We define the associated representation
\begin{equation}
    \rho_\bsc=\left(\rho_{c_1}\otimes\rho_{c_2}\otimes \cdots\otimes \rho_{c_{N+M}}\right)\circ\left(\D\otimes1\otimes\cdots\otimes1\right)\circ\cdots\circ(\D\otimes 1)\circ\D.
\end{equation}
This representation also depends on a vector of complex weights $\bsu=(u_1,\cdots,u_{N+M})$.

Like the algebras $\CWqt(\gl(N))$, the deformed W-algebras $\CWqt(\gl(N|M))$ are generated by the dressed current $\bar W_1(z)=\U_1(z)W_1(z)$ and the modes of a Heisenberg algebra ($U(1)$-factor). 
The following theorem summarizes the known results regarding the construction of its free field realizations via the quantum toroidal $\gl(1)$ Fock modules.

\begin{theorem}[\cite{Bershtein2018,FJMV2020}]\label{thm:Wqt(gl(N|M))} 
    For any coloring vector \(\bsc\) consisting of \(N\) entries equal to \(1\) and \(M\) entries equal to \(3\), 
    the representation $\rho_\bsc$ provides a free field realization of the algebra $\CWqt(\gl(N|M))=\CWqt(\sl(N|M))\otimes\CH_1(\gl(N|M))$, such that
    \begin{itemize}
        \item $\CH_1(\gl(N|M))=\rho_\bsc(\mathfrak{h})$ is the image of the subalgebra $\mathfrak{h}$, 
        \item the current $W_1(z)\in\CWqt(\sl(N|M))$ is obtained as
            \begin{equation}
                \rho_\bsc(x^+(z))=\U_1(z)W_1(z),
            \end{equation}
            with $\U_1(z)\in\CH_1(\gl(N|M))$ and $[W_1(z),\CH_1(\gl(N|M))]=0$.
    \end{itemize}
\end{theorem}

Note that deformed W-algebras are defined in terms of free field realizations. 
The theorem above states that, for given \(N, M\), the deformed W-algebra does not depend on the choice of the coloring.

\begin{remark}
    Usually, the choices of the coloring vectors \(\bsc\) are identified with Dynkin diagrams of the super Lie algebra $\sl(N|M)$. 
    Recall that the latter is a linear graph with \(N+M-1\) nodes. The node \(i\) is bosonic if $c_{i+1}=c_i$ and fermionic if $c_{i+1}=4-c_i$. 
    
    For example, for \(N=4\), \(M=3\), the sequence \(1, 1, 3, 1, 1, 3, 3\) corresponds to the diagram represented in Figure \ref{fig:DynkinExample} (left). 
    Here we represent a bosonic node as a circle and a fermionic node by a crossed one. 
    \begin{figure}[h]
        \centering
        \begin{tikzpicture}[scale=1, transform shape]
            \tikzset{dynkinNode/.style={draw, circle, minimum size=3.5mm, inner sep=0pt, fill=white}}

            \coordinate (v1) at (0,0);
            \coordinate (v2) at (1.2,0);
            \coordinate (v3) at (2.4,0);
            \coordinate (v4) at (3.6,0);
            \coordinate (v5) at (4.8,0);
            \coordinate (v6) at (6,0);
            \coordinate (s1) at (8,0);
            \coordinate (s2) at (9.2,0);
            \coordinate (s3) at (10.4,0);
            \coordinate (s4) at (11.6,0);
            \coordinate (s5) at (12.8,0);
            \coordinate (s6) at (14,0);

            \draw (s1) -- (s2) -- (s3) -- (s4) -- (s5) -- (s6);
            \draw (v1) -- (v2) -- (v3) -- (v4) -- (v5) -- (v6);

            \node[dynkinNode] (s1n) at (s1) {};
            \node[dynkinNode] (s2n) at (s2) {};
            \node[dynkinNode] (s3n) at (s3) {};
            \node[dynkinNode] (s4n) at (s4) {};
            \node[dynkinNode] (s5n) at (s5) {};
            \node[dynkinNode] (s6n) at (s6) {};
            \node[dynkinNode] (v1n) at (v1) {};
            \node[dynkinNode] (v2n) at (v2) {};
            \node[dynkinNode] (v3n) at (v3) {};
            \node[dynkinNode] (v4n) at (v4) {};
            \node[dynkinNode] (v5n) at (v5) {};
            \node[dynkinNode] (v6n) at (v6) {};

            \draw (s4n.225) -- (s4n.45)
                  (s4n.135) -- (s4n.315);
            \draw (v2n.225) -- (v2n.45)
                  (v2n.135) -- (v2n.315);
            \draw (v3n.225) -- (v3n.45)
                  (v3n.135) -- (v3n.315);
            \draw (v5n.225) -- (v5n.45)
                  (v5n.135) -- (v5n.315);
        \end{tikzpicture}
        \caption{%
        The Dynkin diagram corresponding to the coloring \(1, 1, 3, 1, 1, 3, 3\) (left) and the standard Dynkin diagram (right).
        \label{fig:DynkinExample}}
    \end{figure}
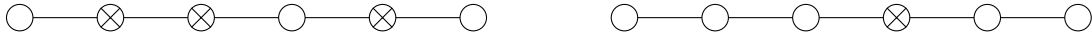
    % is associated a representation $\rho_\bsc$ of the quantum toroidal $\gl(1)$ algebra, with the coloring vector $\bsc$ obtained as follows: starting from $c_1=1$, take $c_{i+1}=c_i$ if the node $i$ is even, $c_{i+1}=4-c_i$ if the node $i$ is odd. 
\end{remark}

\begin{remark}
    We will call the coloring vector with $c_i=1$ for $i=1\cdots N$ and $c_i=3$ for $i=N+1\cdots N+M$ \emph{standard}. 
    The corresponding Dynkin diagram has first \(N-1\) bosonic nodes, then one fermionic node and \(M-1\) bosonic ones. 
    The corresponding representation $\rho_\bsc$ will be called \emph{standard free field realization} of $\CWqt(\gl(N|M))$, and denoted $\rho_{N|M}$.
    % The standard Dynkin diagram of $\sl(N|M)$ corresponds to the choice $c_i=1$ for $i=1\cdots N$ and $c_i=3$ for $i=N+1\cdots M$. By extension, 
\end{remark}

\paragraph{Chiral cluster seeds} In the following, we associate to any Dynkin diagram of $\sl(N|M)$ a decorated quiver. The corresponding chiral cluster seed can be used to encode the OPEs of the free field realizations constructed previously using Fock representations.
\begin{definition}\label{def:Q_bsc}
    For the coloring vector $\bsc$ we  associate the decorated quiver $\CQ_\bsc$ obtained as follows:
    \begin{itemize}
        \item To each vertex $i=1\cdots N+M-1$  associate an unfrozen vertex $X_i$, without loop if \(c_i=4-c_{i+1}\) , with a loop of parameter $q^{-2}q_{c_i}^{-1}$ if \(c_i=c_{i+1}\) ,
        \item Add a frozen vertex $X_0$ of type $-$ with a dashed loop of parameter $q_3$ and type $-$,
        \item Add a frozen vertex $X_{N+M}$ of type $+$ with a dashed loop of parameter $q^{-2}q_{c_{N+M}}^{-1}$ and type $+$,
        \item Add an arrow $X_{i}\to X_{i+1}$ with parameter $1$ for $i=0\cdots N+M-1$,
        \item Add an arrow $X_{i+1}\to X_i$ with parameter $q_{c_{i+1}}$ for $i=0\cdots N+M-1$.
    \end{itemize}
\end{definition}

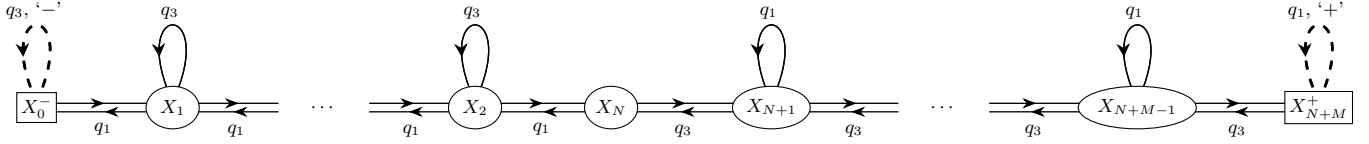
\begin{figure}[h]
\centering
\begin{tikzpicture}[scale=.8, transform shape, font=\small]
\coordinate (YN) at (-1.2,0);
\coordinate (YNm1) at (1.05,0);
\coordinate (ymidBefore) at (2.8,0);
\coordinate (ymid) at (3.55,0);
\coordinate (ymidAfter) at (4.3,0);
\coordinate (Y1) at (6.05,0);
\coordinate (Y0) at (8.3,0);
\coordinate (Y1b) at (10.95,0);
\coordinate (ymidBeforeb) at (13.05,0);
\coordinate (ymidb) at (13.80,0);
\coordinate (ymidAfterb) at (14.55,0);
\coordinate (YMbm1) at (17.00,0);
\coordinate (YMb) at (20.00,0);

%Horizontal arrows

% YN <--> YNm1
\draw[qArrowAbove](YN) to (YNm1);
\draw[qArrowBelow, qArrow/label={$q_1$}](YNm1) to[] (YN);

% YNm1 <--> ymidBefore
\draw[qArrowAbove, qArrow/label xshift=0pt, qArrow/label yshift=0pt, qArrow/arrowpos = 0.6](YNm1) to (ymidBefore);
\draw[qArrowBelow, qArrow/label={$q_1$}, qArrow/label xshift=0pt, qArrow/label yshift=-5pt, qArrow/arrowpos = 0.4](ymidBefore) to[] (YNm1);

% ymidAfter <--> Y1
\draw[qArrowAbove, qArrow/label xshift=0pt, qArrow/label yshift=0pt, qArrow/arrowpos = 0.4](ymidAfter) to (Y1);
\draw[qArrowBelow, qArrow/label={$q_1$}, qArrow/label xshift=0pt, qArrow/label yshift=-5pt, qArrow/arrowpos = 0.6](Y1) to[] (ymidAfter);

% Y1 <--> Y0
\draw[qArrowAbove](Y1) to (Y0);
\draw[qArrowBelow, qArrow/label={$q_1$}](Y0) to[] (Y1);

% Y0 <--> Y1b
\draw[qArrowAbove, qArrow/arrowpos = 0.45](Y0) to (Y1b);
\draw[qArrowBelow, qArrow/label={$q_3$},  qArrow/label xshift=0pt, qArrow/label yshift=-5pt, qArrow/arrowpos = 0.55](Y1b) to[] (Y0);

% Y1b <--> ymidBeforeb
\draw[qArrowAbove, qArrow/label xshift=0pt, qArrow/label yshift=0pt, qArrow/arrowpos = 0.65](Y1b) to (ymidBeforeb);
\draw[qArrowBelow, qArrow/label={$q_3$}, qArrow/label xshift=0pt, qArrow/label yshift=-5pt, qArrow/arrowpos = 0.35](ymidBeforeb) to[] (Y1b);

% ymidAfterb <--> YMbm1
\draw[qArrowAbove, qArrow/label xshift=0pt, qArrow/label yshift=0pt, qArrow/arrowpos = 0.3](ymidAfterb) to (YMbm1);
\draw[qArrowBelow, qArrow/label={$q_3$}, qArrow/label xshift=0pt, qArrow/label yshift=-5pt, qArrow/arrowpos = 0.7](YMbm1) to[] (ymidAfterb);

% YMbm1 <--> YMb
\draw[qArrowAbove, qArrow/label xshift=0pt, qArrow/label yshift=-5pt, qArrow/arrowpos = 0.55](YMbm1) to (YMb);
\draw[qArrowBelow, qArrow/label={$q_3$}, qArrow/label xshift=0pt, qArrow/label yshift=-5pt, qArrow/arrowpos = 0.45](YMb) to[] (YMbm1);

%Loops
\draw[qLoopDashed, qArrow/label={$q_3$, `$-$'}](YN) to (YN);
\draw[qLoop, qArrow/label={$q_3$}](YNm1) to (YNm1);
\draw[qLoop, qArrow/label={$q_3$}](Y1) to (Y1);
\draw[qLoop, qArrow/label={$q_1$}](Y1b) to (Y1b);
\draw[qLoop, qArrow/label={$q_1$}](YMbm1) to (YMbm1);
\draw[qLoopDashed, qArrow/label={$q_1$, `$+$'}](YMb) to (YMb);

%Nodes
\node[qNodeFrozen] (YN) at (YN) {$X_0^-$};
\node[qNodeUnfrozen] (YNm1) at (YNm1) {$X_1$};
\node (ymid) at (ymid) {$\cdots$};
\node[qNodeUnfrozen] (Y1) at (Y1) {$X_{2}$};
\node[qNodeUnfrozen] (Y0) at (Y0) {$X_{N}$};
\node[qNodeUnfrozen] (Y1b) at (Y1b) {$X_{N+1}$};
\node (ymidb) at (ymidb) {$\cdots$};
\node[qNodeUnfrozen] (YMbm1) at (YMbm1) {$X_{N+M-1}$};
\node[qNodeFrozen] (YMb) at (YMb) {$X_{N+M}^+$};
\end{tikzpicture}
	\caption{%
		Quiver associated with the standard realization of $\CWqt(\gl(N|M))$. 
    }\label{fig:qRegWAlgQuiverStd}
\end{figure}

Note that bosonic (resp. fermionic) vertices in the decorated quiver correspond to bosonic (resp. fermionic) nodes of the Dynkin diagram.

\begin{example}
    The quiver $\CQ_{N|M}$ associated with the standard Dynkin diagram of $\sl(N|M)$ is represented in Figure \ref{fig:qRegWAlgQuiverStd}. Note that when $N=0$, we recover the quiver associated with the W-algebras $\CWqt(\gl(M))$ in the previous subsection. On the other hand, when $M=0$, we find the quiver corresponding to the W-algebras $\CWqt(\gl(N))$ with the role of $q_1$ and $q_3$ exchanged in the identification of parameters, i.e. $(\qf,\tf)=(q_2,q_3^{-1})$.
\end{example}

In order to define the free field realizations, we need to introduce twisted telescoping sums.

\begin{definition}
    Let $\bsc$ be a coloring vector of length $N+M$, and $\CP$ a path in a decorated quiver that goes through $N+M+1$ vertices. The twisted telescoping sum $\S_\a^{(\bsc)}[\CP](z)$ (resp. $\bar\S_\a^{(\bsc)}[\CP](z)$) is obtained from the untwisted one $\S[\CP](z)$ by multiplying the $i$th term (resp. $(N+M+1-i)$th term) by the factor $\a=-t_1/t_3$ iff $c_i=1$.  
\end{definition}

\begin{example}
    Let $\CP:X_0\to X_1\cdots \to X_{N+M}$, for $\bsc=(1,\cdots,1,3,\cdots,3)$ associated with the standard Dynkin diagram, we have
    \begin{align}
        \begin{split}
            \S_\a^{(\bsc)}[\CP]&=\a \S[X_0,X_1,\cdots X_{N-1}](z)+:X_0(z)X_1(\k_1z)\cdots X_{N-1}(\k_1\cdots\k_{N-1} z)\S[X_{N},\cdots,X_{N+M-1}](\k_1\cdots\k_{N}z):,\\
            \bar\S_\a^{(\bsc)}[\CP]&=\S[X_0,X_1,\cdots X_{M-1}](z)+\a :X_0(z)X_1(\k_1z)\cdots X_{M-1}(\k_1\cdots\k_{M-1} z)\S[X_{M},\cdots,X_{N+M-1}](\k_1\cdots\k_{M}z):.
        \end{split}
    \end{align}
\end{example}

The following theorem states that the free field realizations of $\CWqt(\gl(N|M))$ obtained from the quantum toroidal $\gl(1)$ algebra fit in the framework of chiral cluster seeds.

\begin{theorem}
    Let $\CQ_\bsc$ be the decorated quiver defined in Definition \ref{def:Q_bsc}, and $\CP_\pm$ be the paths
    \begin{equation}
        \CP_+:X_0\to X_1\to\cdots\to X_{N+M},\qquad \CP_-:X_{N+M}\to X_{N+M-1}\cdots\to X_0.
    \end{equation}
    Let $\CF_{\CQ_\bsc}[\CWqt(\gl(N|M))]$ be the subalgebra of $\CA_{\CQ_\bsc}$ generated by the twisted telescopic sums $\S_\a^{(\bsc)}[\CP_+](z)$, $\bar\S_\a^{(\bsc)}[\CP_-](z)$ and the products $\Psi[\CP_\pm](z)$. Then, we have the following results:
    \begin{itemize}
        \item The algebra $\CF_{\CQ_\bsc}[\CWqt(\gl(N|M))]$ commutes with the screening charges associated with the unfrozen vertices $i=1\cdots N+M-1$ of the decorated quiver.
        \item The image of the algebra $\CE$ under the representation $\rho_\bsc$ is isomorphic to $\CF_{\CQ_\bsc}[\CWqt(\gl(N|M))]$, with the Drinfeld currents identified as
        \begin{align}
            &\rho_\bsc(x^+(z))=t_3\S_\a^{(\bsc)}[\CP_+](z),\qquad &\rho_\bsc(\psi^-(\G^\frac12z))=\Psi[\CP_+](z),\\
            &\rho_\bsc(x^-(\G z))=-t_3\bar\S_\a^{(\bsc)}[\CP_-](z),\qquad &\rho_\bsc(\psi^+(\G^\frac32z))=\Psi[\CP_-](z),
        \end{align}
        with $\G=\prod_{i=1}^{N+M}s_{c_i}$.
        \item The mutation of the decorated quiver at a fermionic vertex $X_i$, followed by gauging the vertex $\tX_i$ by $qq_{c_i}$, produces the decorated linear quiver associated with the Dynkin diagram of the super Lie algebra $\sl(N|M)$ obtained from the previous one by a reflection of odd root $\a_i$. The paths $\CP_\pm$ remain invariant under mutation, and the Drinfeld currents in the representation $\rho_{\tilde{\bsc}}$ associated with the new Dynkin diagram are obtained from $\S_\a^{(\tilde{\bsc})}[\CP_+](z)$, $\bar\S_\a^{(\tilde{\bsc})}[\CP_-](z)$ and $\Psi[\CP_\pm](z)$ using the same formulas as before.
    \end{itemize}
\end{theorem}

\begin{proof}
    The first statement, i.e. the (anti)commutation relations with screening charges, follows from the discussion of Examples \ref{Ex:XYZ_double} and \ref{Ex:BosonicScreenings}.

    The second statement is shown by rewriting the image of the Drinfeld currents $x^\pm(z)$ in terms of twisted telescoping sums, and $\psi^\pm(z)$ in terms of normal-ordered product of vertex operators. This rewriting follows from the coproduct construction of these representations, and the properties of Fock representations
    \begin{equation}
        \psi^+(s_c^{-\frac12}z)=(\g_c^+\g_c^-)^{-1}:x^+(z)x^-(s_c^{-1}z):,\qquad \psi^-(s_c^\frac12 z)=(\g_c^+\g_c^-)^{-1}:x^+(z)x^-(s_cz):.
    \end{equation}
    It leads to identify
    \begin{equation}
        X_j(z)=
        \begin{cases}
            (\g_{c_1}^+)^{-1}\rho_{c_1}(x^+(z))(\otimes1)^{N+M-1},& j=0,\\
            (\g_{c_j}^-\g_{c_{j+1}}^+)^{-1}(1\otimes)^{j-1}\rho_{c_j}(x^-(\G_j z))\otimes \rho_{c_{j+1}}(x^+(\G_jz))(\otimes1)^{N+M-1-j},& j=1\cdots N+M-1,\\
            (\g_{c_{N+M}}^-)^{-1}(1\otimes)^{N+M-1}\rho_{c_{N+M}}(x^-(\G_{N+M}z)), & j=N+M,
        \end{cases}
    \end{equation}
    with $\G_j=\prod_{k=1}^j s_{c_k}$.   Then, it is possible to check that the OPEs of these vertex operators coincide with the OPEs of the vertex operators associated with the vertices of the decorated quiver.
    
    Finally, due to the local nature of mutations, and the linear structure of the quiver, it is enough to prove the last statement for the example of $\gl(2|1)$ which is treated below. 
\end{proof}
 
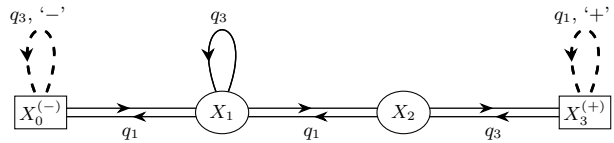
\begin{figure}[h]
\centering
\begin{tikzpicture}[scale=.8, transform shape, font=\small]
\coordinate (Y0) at (0,0);
\coordinate (Y1) at (3,0);
\coordinate (Y2) at (6,0);
\coordinate (Y3) at (9,0);

%Horizontal arrows
\draw[qArrowAbove](Y0) to (Y1);
\draw[qArrowAbove](Y1) to (Y2);
\draw[qArrowAbove](Y2) to (Y3);
\draw[qArrowBelow, qArrow/label={$q_3$}](Y3) to[] (Y2);
\draw[qArrowBelow, qArrow/label={$q_1$}](Y2) to[] (Y1);
\draw[qArrowBelow, qArrow/label={$q_1$}](Y1) to[] (Y0);
% %Loops
\draw[qLoopDashed, qArrow/label={$q_3$, `$-$'}](Y0) to (Y0);
\draw[qLoop, qArrow/label={$q_3$}](Y1) to (Y1);
% \draw[qLoop, qArrow/label={$q_3$}](Y2) to (Y2);
\draw[qLoopDashed, qArrow/label={$q_1$, `$+$'}](Y3) to (Y3);
%Nodes
\node[qNodeFrozen] (Y0) at (Y0) {$X_0^{(-)}$};
\node[qNodeUnfrozen] (Y1) at (Y1) {$X_1$};
\node[qNodeUnfrozen] (Y2) at (Y2) {$X_2$};
\node[qNodeFrozen] (Y3) at (Y3) {$X_3^{(+)}$};
\end{tikzpicture}
	\caption{%
		Quiver for the standard realization of $\CWqt(\gl(2|1))$. 
    }\label{fig:CWsl21_1st}
\end{figure}

\begin{example}[$\CWqt(\gl(2|1))$] Let's consider the case of the W-algebra $\CWqt(\gl(2|1))$. The standard Dynkin diagram has one even node, and one odd node. Labeling the even node by $1$, and the odd node by $2$, the construction of Definition \ref{def:Q_bsc} leads to the decorated quiver represented in Figure \ref{fig:CWsl21_1st}, and associated with the coloring vector $\bsc=(1,1,3)$. Let's consider the following currents,
\begin{align}
    &\S_\a^{(\bsc)}[\CP_+](z)=\a X_0(z)+\a:X_0(z)X_1(z):+:X_0(z)X_1(z)X_2(z):,\\
    &\Psi[\CP_+](z)=:X_0(z)X_1(z)X_2(z)X_3(z):,\\
    &\bar\S_\a^{(\bsc)}[\CP_-](z)=X_3(z)+\a :X_3(z)X_2(q_3z):+\a:X_3(z)X_2(q_3z)X_1(q_1q_3z):,\\
    &\Psi[\CP_-](z)=:X_3(z)X_2(q_3z)X_1(q_1q_3z)X_0(q_1^2q_3z):,
\end{align}
and study their behavior under mutation of the quiver at vertex $2$. For this purpose, we introduce the fermionic fields at this vertex,
\begin{equation}
    X_2(z)=-:\Psi(q^{-1}z)\Psi^\ast(qz):,
\end{equation}
and express the other vertex operators in the form
\begin{equation}
    X_0(z)=V_0(z),\qquad X_1(z)=V_1(z):\Psi(qz)\Psi^\ast(q^{-1}q_1^{-1}z):,\qquad X_3(z)=V_3(z):\Psi(qq_3z)\Psi^\ast(q^{-1}z):,
\end{equation}
where vertex operators $V_0(z)$, $V_1(z)$ and $V_3(z)$ commute with $\Psi(w),\Psi^\ast(w)$. After mutation, and gauging the vertex $\tX_2$ by $qq_1$, we find after careful manipulation of the normal-orderings (as in Example \ref{Ex:XYZ_double_mutation}),
\begin{align}
    &\S_\a^{(\bsc)}[\CP_+]=\a V_0(z)+:V_0(z)V_1(z)\tPsi^\ast(z)(-\tPsi(q_1^{-1}z)+\a\tPsi(q_3z)):,\\
    &\Psi[\CP_+]=-:V_0(z)V_1(z)V_3(z):,\\
    &\bar\S_\a^{(\bsc)}[\CP_-]=V_3(z):\tPsi^\ast(q_3z)(-\a\tPsi(z)+\tPsi(q^{-2}z)):-\a :V_3(z)V_1(q_1q_3z):,\\
    &\Psi[\CP_-]=-:V_3(z)V_1(q_1q_3z)V_0(q_1^2q_3z):.
\end{align}
These currents can be rewritten as
\begin{align}
    \S_\a^{(\bsc)}[\CP_+](z)&=\S_\a^{(\tilde{\bsc})}[\tX_0,\tX_1,\tX_2](z),\qquad & \Psi[\CP_+](z)&=\Psi[\tX_0,\tX_1,\tX_2,\tX_3](z),\\
    \bar\S_\a^{(\bsc)}[\CP_-](z)&=\bar\S_\a^{(\tilde{\bsc})}[\tX_3,\tX_2,\tX_1](z),\qquad &\Psi[\CP_-](z)&=\Psi[\tX_3,\tX_2,\tX_1,\tX_0](z),
\end{align}
with $\tilde{\bsc}=(1,3,1)$, and the vertex operators
\begin{align}
    &\tX_0(z)=V_0(z),\qquad &\tX_1(z)&=-V_1(z):\tPsi^\ast(z)\tPsi(q_1^{-1}z):,\\
    &\tX_2(qq_1z)=:\tPsi(q^{-1}z)\tPsi^\ast(qz):,\qquad &\tX_3(z)&=-V_3(z):\tPsi(z)\tPsi^\ast(q_3z):.
\end{align}
By Proposition \ref{prop:mutation}, the OPEs of these operators coincide with those encoded in the mutated quiver represented in Figure \ref{fig:CWsl21_2nd}, and associated with the Dynkin diagram with two odd nodes.

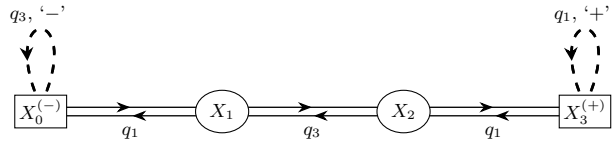
\begin{figure}[h]
\centering
\begin{tikzpicture}[scale=.8, transform shape, font=\small]
\coordinate (Y0) at (0,0);
\coordinate (Y1) at (3,0);
\coordinate (Y2) at (6,0);
\coordinate (Y3) at (9,0);

%Horizontal arrows
\draw[qArrowAbove](Y0) to (Y1);
\draw[qArrowAbove](Y1) to (Y2);
\draw[qArrowAbove](Y2) to (Y3);
\draw[qArrowBelow, qArrow/label={$q_1$}](Y3) to[] (Y2);
\draw[qArrowBelow, qArrow/label={$q_3$}](Y2) to[] (Y1);
\draw[qArrowBelow, qArrow/label={$q_1$}](Y1) to[] (Y0);
% %Loops
\draw[qLoopDashed, qArrow/label={$q_3$, `$-$'}](Y0) to (Y0);
% \draw[qLoop, qArrow/label={$q_3$}](Y1) to (Y1);
% \draw[qLoop, qArrow/label={$q_3$}](Y2) to (Y2);
\draw[qLoopDashed, qArrow/label={$q_1$, `$+$'}](Y3) to (Y3);
%Nodes
\node[qNodeFrozen] (Y0) at (Y0) {$X_0^{(-)}$};
\node[qNodeUnfrozen] (Y1) at (Y1) {$X_1$};
\node[qNodeUnfrozen] (Y2) at (Y2) {$X_2$};
\node[qNodeFrozen] (Y3) at (Y3) {$X_3^{(+)}$};
\end{tikzpicture}
	\caption{%
		Quiver for the second realization of $\CWqt(\gl(2|1))$. 
    }\label{fig:CWsl21_2nd}
\end{figure}

Performing a second mutation on the same vertex $2$, we come back to the previous quiver (Figure \ref{fig:CWsl21_1st}) and the original expression of the currents. Instead, performing the mutation at the vertex $1$, and gauging $\tX_1$ by $qq_1$, we arrive at a third realization corresponding to the quiver of Figure \ref{fig:CWsl21_3rd} and associated with the coloring vector $\bsc=(3,1,1)$. In fact, this realization is equivalent to the first one using the homomorphism of the quantum toroidal algebra $\CWqt(\gl(2|1))\simeq\CWqt(\gl(1|2))$. This exhausts the possibilities to generate new quivers by mutations.

\begin{figure}[h]
\centering
\begin{tikzpicture}[scale=.8, transform shape, font=\small]
\coordinate (Y0) at (0,0);
\coordinate (Y1) at (3,0);
\coordinate (Y2) at (6,0);
\coordinate (Y3) at (9,0);

%Horizontal arrows
\draw[qArrowAbove](Y0) to (Y1);
\draw[qArrowAbove](Y1) to (Y2);
\draw[qArrowAbove](Y2) to (Y3);
\draw[qArrowBelow, qArrow/label={$q_1$}](Y3) to[] (Y2);
\draw[qArrowBelow, qArrow/label={$q_1$}](Y2) to[] (Y1);
\draw[qArrowBelow, qArrow/label={$q_3$}](Y1) to[] (Y0);
% %Loops
\draw[qLoopDashed, qArrow/label={$q_1$, `$-$'}](Y0) to (Y0);
% \draw[qLoop, qArrow/label={$q_3$}](Y1) to (Y1);
\draw[qLoop, qArrow/label={$q_3$}](Y2) to (Y2);
\draw[qLoopDashed, qArrow/label={$q_3$, `$+$'}](Y3) to (Y3);
%Nodes
\node[qNodeFrozen] (Y0) at (Y0) {$X_0^{(-)}$};
\node[qNodeUnfrozen] (Y1) at (Y1) {$X_1$};
\node[qNodeUnfrozen] (Y2) at (Y2) {$X_2$};
\node[qNodeFrozen] (Y3) at (Y3) {$X_3^{(+)}$};
\end{tikzpicture}
	\caption{%
		Quiver for the third realization of $\CWqt(\gl(2|1))$. 
    }\label{fig:CWsl21_3rd}
\end{figure}
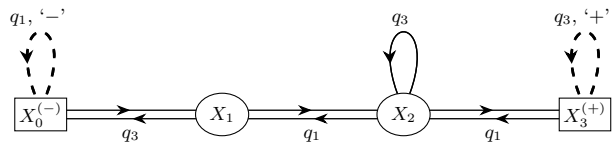
\end{example}

\paragraph{Embedding of $\CWqt(\gl(1|N))$ in  $\CWqt(\gl(N))\otimes\CH_1$} Comparing the decorated quivers associated with $\CWqt(\gl(N))$ and $\CWqt(\gl(1|N))$ in the standard realization, we observe that the former can be obtained from the latter by removing a frozen vertex, and turning the fermionic vertex into a new frozen vertex. This observation is a consequence of the following well-known result.

\begin{proposition}\label{prop:slN1_red}
There is an embedding of deformed W-algebras
\begin{equation}
     \CWqt(\gl(1|N))\hookrightarrow\CWqt(\gl(N))\otimes\CH_1,
\end{equation}
which restricts to the embedding of $U(1)$ factors $\CH_1(\gl(1|N))\hookrightarrow\CH_1(\gl(N))\otimes\CH_1$, and such that the generating current $W_1^{\sl(1|N)}(z)$ of the $\CWqt(\sl(1|N))$ subalgebra is mapped as follows
\begin{equation}\label{rel_Wsl1N}
    W_1^{\sl(1|N)}(z)\to A(z)W_1^{\sl(N)}(s_1z)+B(z),\qquad\text{with}\qquad [A(z),W_1^{\sl(N)}(w)]=[B(z),W_1^{\sl(N)}(w)]=0.
\end{equation}
With this embedding, the algebra $\CWqt(\gl(1|N))$ is considered in the standard free field realization $\CQ_{1|N}$, and $\CWqt(\gl(N))$ in the AKOS free field realization. 
\end{proposition}

\begin{proof} This proposition is a direct consequence of the construction using quantum toroidal $\gl(1)$ Fock representations. Indeed, the free field realization $\CF_{\CQ_{1|N}}[\CWqt(\gl(1|N))]$ is obtained from the representation $\rho_{1|N}$ associated with the standard Dynkin diagram, and acting on the tensor product $\CF_1\otimes (\CF_3)^{\otimes N}$ of Fock modules. By construction, the currents are expressed using the generators of a rank $N+1$ diagonal Heisenberg algebra obtained from the coproduct construction as $J_{i,n}=(1\otimes)^{i-1}J_n(\otimes1)^{N+1-i}$, with $i=1\cdots N+1$, and $n\in\mZ^\times$. The Heisenberg algebra $\CH_1$ in the proposition will be identified with image of the algebra generated by $J_{1,n}$. 

On the other hand, the free field realization $\CF_{\CQ_\text{AKOS}}[\CWqt(\gl(N))]$ coincides with the one obtained from the representation $\rho_{0|N}$ of the quantum toroidal $\gl(1)$ algebra, acting on the Fock space $(\CF_3)^{\otimes N}$. Its currents are expressed in terms of the generators of a rank $N$ diagonal Heisenberg algebra which can be identified with the image of the generators $J_{i,n}$ for $i=2\cdots N+1$. In the following, the embedding corresponding to this identification of diagonal Heisenberg algebras will be omitted.

The identification of the image of the $U(1)$ factor follows from their realization as the image of the Cartan subalgebra $\mathfrak{h}$ of the quantum toroidal $\gl(1)$ algebra,
\begin{equation}
    \CH_1(\gl(1|N))=\rho_{1|N}(\mathfrak{h}),\qquad \CH_1(\gl(N))=\rho_{0|N}(\mathfrak{h}).
\end{equation}
The Cartan subalgebra is generated by the modes $h_k$ and the central elements $C$ and $C^\perp=\psi_0^-=(\psi_0^+)^{-1}$. The coproduct satisfies the property
\begin{equation}
    \D^N(h_k)=h_k\otimes C^{-|k|/2}\otimes\cdots\otimes C^{-|k|/2}+C^{|k|/2}\otimes\D^{N-1}(h_k),
\end{equation}
from which we deduce that
\begin{equation}
    \rho_{1|N}(h_k)=s_3^{-N|k|/2}\rho(h_k)(\otimes1)^N+s_1^{|k|/2}\otimes \rho_{0|N}(h_k),
\end{equation}
with $\rho(h_k)(\otimes1)^N\propto J_{1,k}$. Similarly, for the central elements $\D^N(C)=C\otimes \D^{N-1}(C)$ and $\D^N(C^\perp)=C^\perp\otimes \D^{N-1}(C^\perp)$. It shows that $\CH_1(\gl(1|N))\hookrightarrow\CH_1(\gl(N))\otimes\CH_1$.

The second part of the proposition also follows from a property of the coproduct, namely
\begin{equation}
    \D^N(x^+(z))=x^+(z)(\otimes 1)^N+\psi^-(C_1^\frac12z)\otimes\D^{N-1}(x^+(C_1z)).
\end{equation}
Evaluating this expression in the representation $\rho_{1|N}$, we find the relation
\begin{align}
    \begin{split}
        &\rho_{1|N}(x^+(z))=A'(z)\rho_{0|N}(x^+(s_1z))+B'(z),\\
        \text{with}\quad &    A'(z)=\rho_1^{(1,0)}(\psi^-(s_1^\frac12 z))(\otimes1)^N\in\CH_1,
        \quad B'(z)=\rho_1^{(1,0)}(x^+(z))(\otimes1)^N\in\CH_1.
    \end{split}
\end{align}
To conclude, we need to take into account the $U(1)$-factors in the definition of the currents, they both belong to $\CH_1(\gl(N))\otimes\CH_1$ as shown previously. As a result, we find the relation \ref{rel_Wsl1N} with 
\begin{equation}
    A(z)=\U^{\gl(1|N)}(z)^{-1}A'(z)\U^{\gl(N)}(s_1z)\in \CH_1(\gl(N))\otimes\CH_1,\quad B(z)=\U^{\gl(1|N)}(z)^{-1}B'(z)\in \CH_1(\gl(N))\otimes\CH_1.
\end{equation}
\end{proof}

\subsection{Quantum affine $\sl(2)$ algebra}\label{sec:Uqsl2}
In the Drinfeld presentation, the quantum affine $\sl(2)$ algebra $U_q(\widehat{\sl(2)})$ of quantum parameter $q$ is generated by the modes of four currents $E(z)$, $F(z)$, and $\psi^\pm(z)$,
\begin{equation}
E(z)=\sum_{n\in\mZ} z^{-n}E_n,\qquad F(z)=\sum_{n\in\mZ}z^{-n}F_n,\qquad \psi^{\pm}(z) = \psi_0^{\pm} \exp\left(\dfrac1{q-q^{-1}}\sum_{n\in\mZ^{>0}} H_{\pm n} z^{\mp n}\right),
\end{equation}
and a central element $q^\KK$, satisfying the algebraic relations,
\begin{equation}\label{eq:def_Uqsl2}
\begin{aligned}
&\comm{\psi^{\pm}(z)}{\psi^{\pm}(w)} = 0,
\qquad
\psi^{+}(z)\psi^{-}(w) = \frac{G(q^{\KK}z/w)}{G(q^{-\KK}z/w)}\,\psi^{-}(w)\psi^{+}(z),
\qquad
\psi^{+}_{0}\psi^{-}_{0} = \psi^{-}_{0}\psi^{+}_{0} = 1,
\\
&\psi^\pm(z)E(w)=G(q^{\pm\KK/2}z/w)E(w)\psi^\pm(z),\qquad \psi^\pm(z)F(w)=G(q^{\mp\KK/2}z/w)^{-1}F(w)\psi^\pm(z),\\
&(z-q^{2}w)\,E(z)E(w) =(q^{2}z-w)\,E(w)E(z),\qquad (z-q^{-2}w)\,F(z)F(w) =(q^{-2}z-w)\,F(w)F(z),
\\
&(q-q^{-1})\comm{E(z)}{F(w)} = \delta(q^{-\KK}z/w)\psi^{+}(q^{\KK/2}w) -\delta(q^{\KK}z/w)\psi^{-}(q^{-\KK/2}w),
\end{aligned}
\end{equation}
where we have introduced the structure function
\begin{equation}
G(z)=\frac{q^{2}z-1}{z-q^{2}}.
\end{equation}
We assume that the quantum group parameter $q$ and the central element $q^\KK$ are generic. The quantum affine $\sl(2)$ algebra can be seen as the deformed W-algebra associated to the Kac-Moody algebra $\widehat{\sl(2)}_{\KK}$, and with deformation parameters $(\qf,\tf)=(q^2,q^{2\KK+4})$. We present below two different free field realizations of this algebra, and relate them to our formalism of chiral cluster seeds.

% [20] J. Shiraishi, Phys. Lett. A171 (1992) 243-248.
% [21] Univ. A. Matsuo, preprint “Free (1992); Field“Free Realization Field Realization of Quantum ofAffine q-Deformed Algebra Primary Uq (sl2 c )”, Fields Nagoya for Uq (sl2c )”, Nagoya Univ. preprint (1992).
% [22] A. Abada, A.H. Bougourzi adn M.A. El Gradechi, “A Deformation of the Wakimoto
% Construction”, preprint CRM-1829 (1992).
% [23] preprint K. Kimura, RIMS-910 “On Free (1992); Boson Representation of the Quantum Affine Algebra Uq (sl2 c )”,
% A.H. Bougourzi, “Uniqueness of the Bosonization of the Uq (sl(2)k ) Quantum Current Algebra”, preprint CRM-1852 (1993).

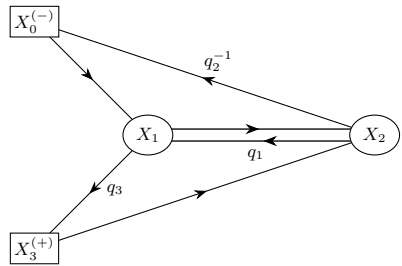
\begin{figure}[h]
\begin{center}
\begin{tikzpicture}[scale=0.75, transform shape, font=\small]
%Coordinates
\coordinate (x0) at (0,2);
\coordinate (x1) at (2,0);
\coordinate (x2) at (6,0);
\coordinate (x3) at (0,-2);
%Arrows
\draw[qArrow](x0) to[] (x1);
\draw[qArrow, qArrow/label={$q_3$}, qArrow/label xshift=12pt, qArrow/label yshift=5pt](x1) to[] (x3);
\draw[qArrow, qArrow/label={$q_2^{-1}$}, qArrow/label xshift=7pt, qArrow/label yshift=15pt](x2) to[] (x0);
\draw[qArrow](x3) to[] (x2);
%Double opposing arrows
\draw[qArrowAbove, qArrow/xshift=0pt, qArrow/yshift=-3pt, qArrow/label={$q_1$}, qArrow/label xshift=-5pt, qArrow/label yshift=-11pt](x2) to (x1);
\draw[qArrowBelow, qArrow/xshift=0pt, qArrow/yshift=3pt,qArrow/label={$\ $}, qArrow/label xshift=-8pt, qArrow/label yshift=-3pt](x1) to[] (x2);
%Nodes
\node[qNodeFrozen] (x0) at (x0) {$X^{(-)}_0$};
\node[qNodeUnfrozen] (x1) at (x1) {$X_1$};
\node[qNodeUnfrozen] (x2) at (x2) {$X_2$};
\node[qNodeFrozen] (x3) at (x3) {$X^{(+)}_3$};
\end{tikzpicture}
	\caption{%
		First free field realization of $\qasltwo$.
	}\label{fig:SL2Quiverrealization1}
\end{center}
\end{figure}

\paragraph{First realization} The first realization coincides with the one introduced in \cite{AOS94}, and for which both $E(z)$ and $F(z)$ are a sum of two vertex operators. The following proposition shows that this realization sits in the framework of chiral cluster seeds. More generally, the free field realizations $U_q(\sl(N))$ given in \cite{AOS94} all fit in this framework, this will be discussed in a forthcoming paper \cite{BBS2026}.

\begin{proposition}
    Let $\CQ_I$ denotes the decorated quiver represented in Figure \ref{fig:SL2Quiverrealization1}. Let $\CP_+:X_0\to X_1\to X_3$ and $\CP_-:X_3\to X_2\to X_0$ be two paths in this quiver. The following assignment defines a free field realization of the quantum affine $\sl(2)$ algebra at level $q^\KK=s_3/s_2$,\footnote{Explicitly parameters are identified as follows: $(q_1,q_2,q_3)\equiv(q^{-2\KK-4},q^2,q^{2\KK+2})$.}  
    \begin{equation}\label{qAffine_qSL2_realization1}
        \begin{aligned}
            (q-q^{-1})E(z) &= \S[\CP_+](z),& \psi^-(z) &= \Psi[\CP_+](q^{2-\KK/2}z),\\
            (q-q^{-1})F(z) &= -\S[\CP_-](q^{2+\KK}z),& \psi^+(z) &= \Psi[\CP_-](q^{\KK/2}z).
        \end{aligned}
    \end{equation}
    The subalgebra of the chiral cluster seed generated by the modes of $\S[\CP_\pm](z)$ and $\Psi[\CP_\pm](z)$ will be denoted $\CF_{\CQ_I}[U_q(\widehat{\sl(2)})]$. These currents (anti)commutes with the screening charges associated with the unfrozen vertices.
\end{proposition}

\begin{proof}
    The proof is a direct check that the algebraic relations \ref{eq:def_Uqsl2} are satisfied. The (anti)commutation with screening charges is shown using the same technique as in the previous examples.
\end{proof}

\begin{figure}[h]
\begin{center}
\begin{tikzpicture}[scale=0.75, transform shape, font=\small]
%Coordinates
\coordinate (x0) at (0,2);
\coordinate (x1) at (2,0);
\coordinate (x2) at (6,0);
\coordinate (x3) at (0,-2);
%Arrows
\draw[qArrow, qArrow/label={$q_3$}, qArrow/label xshift=-10pt, qArrow/label yshift=2pt](x0) to (x3);
\draw[qArrow, qArrow/label={$q_1$}, qArrow/label xshift=10pt, qArrow/label yshift=5pt](x1) to[] (x0);
\draw[qArrow](x3) to[] (x1);
%Double opposing arrows
\draw[qArrowAbove, qArrow/label={$q_3$}](x2) to (x1);
\draw[qArrowBelow](x1) to[] (x2);
%Loops
\draw[qLoop, qArrow/label={$q_1$}](x2) to (x2);
%Finite arrows
%
%Nodes
\node[qNodeFrozen] (x0) at (x0) {$X^{(-)}_0$};
\node[qNodeUnfrozen] (x1) at (x1) {$X_1$};
\node[qNodeUnfrozen] (x2) at (x2) {$X_2$};
\node[qNodeFrozen] (x3) at (x3) {$X^{(+)}_3$};
\end{tikzpicture}
	\caption{%
		Second realization of $\qasltwo$ obtained by mutating the quiver of Figure \ref{fig:SL2Quiverrealization1} at vertex $X_1$ and gauging.
	}\label{fig:SL2Quiverrealization2}
\end{center}
\end{figure}
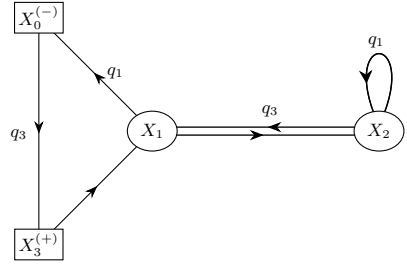

\paragraph{Second realization} Performing a mutation at the vertex $1$ of the decorated quiver of figure \ref{fig:SL2Quiverrealization1}, followed by gauging $X_1$ by $q^{-1}$ and $X_2$ by $q_1$, we find the decorated quiver represented in Figure \ref{fig:SL2Quiverrealization2}. It is associated with a second free field realization of $U_q(\widehat{\sl}(2))$, in which $E(z)$ is a single vertex operator, and $F(z)$ a sum of three vertex operators. Note that, by symmetry, a mutation at the vertex $2$ would give the same result, with the role of $E(z)$ and $F(z)$ exchanged (Chevalley involution). This second realization contains a bosonic vertex, and it will appear in the inverse Hamiltonian reduction in the next section.

\begin{proposition}
    Let $\CQ_{II}$ denote the decorated quiver represented in Figure \ref{fig:SL2Quiverrealization2}. Let $\CP_+:X_0\to X_3$ and $\CP_-:X_3\to X_1\to X_2\to X_1\to X_0$ be two paths in this quiver. The assignment \ref{qAffine_qSL2_realization1} defines a free field realization of the quantum affine $\sl(2)$ algebra at level $q^\KK=s_3/s_2$.
    % \begin{equation}\label{qAffine_qSL2_realization2}
    %     \begin{aligned}
    %         &z(q-q^{-1})E(z) = \S[\CP_+](z),\qquad \psi^-(q^{\KK/2}z) = q^{-1}\Psi[\CP_+](z),\\
    %         &z(q-q^{-1})F(q^{-\KK-4} z) = -\S[\CP_-](z),\qquad \psi^+(q^{-5\KK/2-2}z) =-q\Psi[\CP_-](z).
    %     \end{aligned}
    % \end{equation}
    
    The subalgebra of the chiral cluster seed generated by the modes of $\S[\CP_\pm](z)$ and $\Psi[\CP_\pm](z)$ will be denoted $\CF_{\CQ_{II}}[U_q(\widehat{\sl(2)})]$. These currents (anti)commutes with the screening charges associated with unfrozen vertices.

    This second realization is obtained from the first one by a mutation of the chiral cluster seed at the vertex $1$, followed by gauging $X_1$ by $qq_3$.
\end{proposition}

\begin{proof}
    The proof of the first statement is a direct check that the algebraic relations \ref{eq:def_Uqsl2} are satisfied. In the second statement, it easy to show that the two quivers are indeed related by mutation and gauging. The non-trivial part is the transformation of the currents $\S[\CP_\pm](z)$, it is obtained as a specialization of Theorem \ref{th:subreg_mutations} (see next section) for $N=2$ and $M=0,1$. Specifically, the realization I corresponds to $M=1$, and the realization II to $M=0$, and $\S[\CP_\pm](z)=G_M^\mp(z)$. On the other hand, the transformation of the currents $\Psi[\CP_\pm](z)$ is trivial as their expression is independent of the fermionic fields introduced for the mutation. 

\end{proof}

\subsection{Deformed Bershadsky-Polyakov algebra}\label{sec:qbp}
A $(\qf,\tf)$ deformation of the Bershadsky-Polyakov (BP) algebra \cite{Bershadsky1991,Polyakov1990} has been proposed using different construction methods in \cite{Harada:2020woh} and \cite{FJM2022}. Indeed, the algebra is constructed using the framework of Corner VOA in \cite{Harada:2020woh}, while it is obtained as an extension of $\CWqt(\gl(1|N))$ in \cite{FJM2022}. These two different approaches lead to different free field realizations, that we call respectively 'first' and 'second' realization below. We will show that these two realizations are associated to different decorated quiver in the chiral cluster seed formalism, and that they are related by a mutation.

The deformed BP algebra depends on two quantum parameters $q$ and $d$, they will be identified with our parameters $(q_1,q_2,q_3)=(q^{-2}d^2,q^2,d^{-2})$.\footnote{Note the parameters $(q_1,q_2,q_3)$ differs from the parameters of the quantum toroidal $\gl(1)$ algebra used in \cite{Harada:2020woh}, namely $(q_1,q_2,q_3)=(q^{-1}d,q^2,q^{-1}d^{-1})$.} It is defined by the generating currents $G^\pm(z)$, $T(z)$, $\Psi(z)$ and $\tPsi(z)$, satisfying a set of quadratic algebraic relations that can be found in \cite{Harada:2020woh}. In fact, the currents $G^+(z)$ and $G^-(z)$ generate the whole algebra, since other currents can be obtained from their commutator,
\begin{align}\label{eqn:qBPrelationGpGm}
    \begin{split}
    \comm{G^{+}(z)}{G^{-}(w)}
    = &- \, \delta \brac*{\frac{w}{z}} \frac{q(1 - q^{-2} d^{-2})}{(1 - d^{-2})(q - q^{-1})^{2}} \Psi(w) + \delta\brac*{\frac{d^{2}w}{z}} \frac{1}{(q - q^{-1})^{2}} T(dw) \\
    &- \delta\!\brac*{\frac{q^{2}d^{4}w}{z}} \frac{1}{(q - q^{-1})^{2}} \, \widetilde{\Psi}^{-}(q d^{2} w).
    \end{split}
\end{align}
In this section, we focus primarily on these two currents $G^\pm(z)$.
 
%  It can be seen as a deformation of the BP-algebra, with parameters $(\qf,\tf)=(q^2,q^2d^{-2})$, to which it reduces in the conformal limit $q,d\to1$ \cite{Harada:2020woh}. 

% A free field realization has been obtained in \cite{Harada:2020woh}, it corresponds to the first realization introduced below. A different free field realization has been obtained in \cite{FJM2022} as an extension of the deformed W-algebra $\CWqt(\gl(1|N))$, it corresponds to the second realization below. We show that both realizations can be associated with a chiral cluster seed, and that they are related by a single mutation.

\begin{figure}[h]
\centering
\begin{tikzpicture}[scale=0.75, transform shape, font=\small]
%Coordinates
\coordinate (x0) at (0,2);
\coordinate (x1) at (3,0);
\coordinate (x2) at (6,0);
\coordinate (x3) at (9,0);
\coordinate (x4) at (0,-2);
%Arrows
\draw[qArrow, qArrow/xshift=0pt] (x4) to (x0);
\draw[qArrow, qArrow/xshift=0pt] (x0) to (x1);
\draw[qArrow, qArrow/label={$q_2^{-1}$}, qArrow/label xshift=13pt, qArrow/label yshift=13pt](x2) to[] (x0);
\draw[qArrow, qArrow/label={$q_3$}, qArrow/label xshift=-3pt, qArrow/label yshift=10pt](x1) to[] (x4);
\draw[qArrow, qArrow/xshift=4pt] (x4) to (x2);
%Double opposing arrows
\draw[qArrowAbove](x1) to (x2);
\draw[qArrowBelow, qArrow/label={$q_1$}](x2) to[] (x1);
\draw[qArrowAbove](x2) to (x3);
\draw[qArrowBelow, qArrow/label={$q_3$}](x3) to[] (x2);
%Loops
% \draw[qLoop, qArrow/label={$q_1$}](x2) to (x2);
\draw[qLoop, qArrow/label={$q_1$}](x3) to (x3);
%Finite arrows
%
%Nodes
\node[qNodeFrozen] (x0) at (x0) {$X^{(-)}_0$};
\node[qNodeUnfrozen] (x1) at (x1) {$X_1$};
\node[qNodeUnfrozen] (x2) at (x2) {$X_2$};
\node[qNodeUnfrozen] (x3) at (x3) {$X_3$};
\node[qNodeFrozen] (x4) at (x4) {$X^{(+)}_4$};
\end{tikzpicture}
	\caption{%
		Quiver corresponding to the free field realization of $\CWqt^\text{sub}(\sl(3))$ given in \cite{Harada:2020woh}.
	}\label{fig:HaradaBPQuiver}
\end{figure}
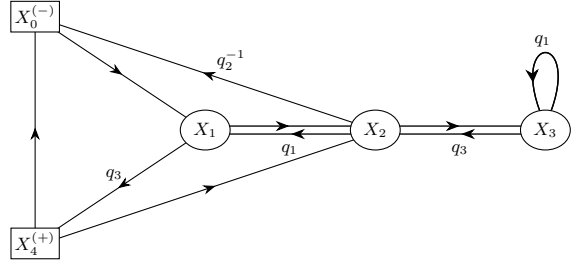

\paragraph{First realization} In the first free field realization, the current $G^+(z)$ is expressed as a sum of four vertex operators, and the current $G^-(z)$ is a sum of two. The following proposition shows that this realization sits in the framework of chiral cluster seeds.

\begin{proposition}
     Let $\CQ_{I}$ denotes the decorated quiver represented in Figure \ref{fig:HaradaBPQuiver}. Let $\CP_+:X_4\to X_2\to X_3\to X_2\to X_0$ and $\CP_-:X_0\to X_1\to X_4$ be two paths in this quiver. The following assignment defines a free field realization of the deformed BP algebra with parameters $(q,d)=(s_2,s_3^{-1})$,
    \begin{equation}\label{qBP_realization1}
    (q-q^{-1})G^+(z)=\S[\CP_+](z),\qquad (q-q^{-1})^2G^-(d^{-2}z)=\S[\CP_-](z).
    \end{equation}
    The subalgebra of $\CA_{\CQ_I}$ generated by the currents $\S[\CP_\pm](z)$ will be denoted $\CF_{\CQ_I}[\CWqt^\text{sub}(\sl(3))]$. The currents (anti)commutes with the screening charges associated with unfrozen vertices.
\end{proposition}

\begin{proof}
    The proof follows from the results of \cite{Harada:2020woh}, and the rewriting of generators $G^\pm(z)$ in the form of telescoping sums. It is then possible to check that the OPEs of these vertex operators coincide with the ones encoded in the decorated quiver. The (anti)commutation of the currents with the screening charges is shown in \cite{Harada:2020woh}.
\end{proof}
\begin{figure}[h]
\centering
\begin{tikzpicture}[scale=0.75, transform shape, font=\small]
%Coordinates
\coordinate (x0) at (0,2);
\coordinate (x1) at (3,0);
\coordinate (x2) at (6,0);
\coordinate (x3) at (9,0);
\coordinate (x4) at (0,-2);
%Arrows
\draw[qArrow, qArrow/xshift=-4pt, qArrow/label={$q_3$}, qArrow/label xshift=-10pt](x0) to (x4);
\draw[qArrow, qArrow/xshift=4pt] (x0) to (x4);
\draw[qArrow, qArrow/label={$q_1$}, qArrow/label xshift=10pt, qArrow/label yshift=10pt](x1) to[] (x0);
\draw[qArrow](x4) to[] (x1);
%Double opposing arrows
\draw[qArrowAbove](x1) to (x2);
\draw[qArrowBelow, qArrow/label={$q_3$}](x2) to[] (x1);
\draw[qArrowAbove](x2) to (x3);
\draw[qArrowBelow, qArrow/label={$q_3$}](x3) to[] (x2);
%Loops
\draw[qLoop, qArrow/label={$q_1$}](x2) to (x2);
\draw[qLoop, qArrow/label={$q_1$}](x3) to (x3);
%Finite arrows
%
%Nodes
\node[qNodeFrozen] (x0) at (x0) {$X^{(-)}_0$};
\node[qNodeUnfrozen] (x1) at (x1) {$X_1$};
\node[qNodeUnfrozen] (x2) at (x2) {$X_2$};
\node[qNodeUnfrozen] (x3) at (x3) {$X_3$};
\node[qNodeFrozen] (x4) at (x4) {$X^{(+)}_4$};
\end{tikzpicture}
	\caption{Quiver associated with the second realization of $\CWqt^\text{sub}(\sl(3))$.}\label{fig:OurBPQuiver}
\end{figure}
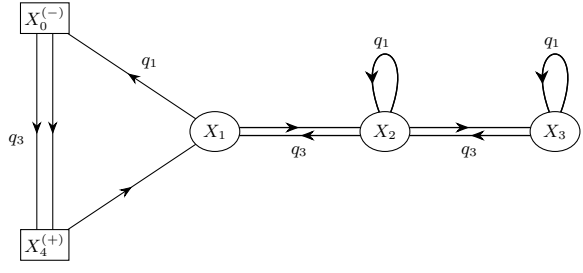

\paragraph{Second realization} In the second realization, $G^+(z)$ is a sum of eight vertex operators, while $G^-(z)$ is a single vertex operator. This realization coincides with the construction of the algebra as an extension of $\CWqt(\gl(1|3))$ \cite{FJM2022}. This observation follows from the comparison of the current $G^+(z)$ obtained here by mutation with the current derived from the q-character in \cite{FJM2022}. We note that the decorated quiver has two bosonic vertices in this realization. They produce the bosonic screening currents of the algebra $\CWqt(\sl(3))$ in the inverse Hamiltonian reduction discussed in the next section.

\begin{proposition}
    Let $\CQ_{II}$ denote the decorated quiver represented in Figure \ref{fig:OurBPQuiver}. Let $\CP_-:X_0\to X_4$ be the path in this quiver which follows the arrow of parameter $q_3$. The following assignment defines a free field realization of the deformed BP algebra,
    \begin{align}\label{qBP_realization2}
    \begin{split}
        (q-q^{-1})^2G^+(z)=&X_4(z)+:X_4(z)X_1(z):+:X_4(z)X_1(z)X_2(z)(1+X_3(z)+X_1(q_3z)):\\
            &+:X_4(z)X_1(z)X_2(z)X_3(z)X_1(q_3z)(1+X_2(q_3z)+X_2(q_3z)X_1(q_3^2z)):,\\
        (q-q^{-1})G^-(d^{-2}z)=&X_0(z).
    \end{split}
    \end{align}
    The subalgebra of $\CA_{\CQ_{II}}$ generated by these currents will be denoted $\CF_{\CQ_{II}}[\CWqt^\text{sub}(\sl(3))]$. The currents (anti)commute with the screening charges associated with unfrozen vertices.

    The second realization is obtained from the first one by a mutation of the chiral cluster seed at the vertex $1$, followed by gauging $X_1$ by $qq_3$.
\end{proposition}

\begin{remark}
    The other generators of the algebra can be expressed as follows,
    % \begin{equation}
    % \begin{aligned}
    % \Psi(d^{-2}z) = & (1-q_2)\dfrac{1-q_2^{-1}q_3}{1-q_3} \no{X_0(z) X_4(q_3 z)}\\
    % \widetilde{\Psi}^-( q^{-1}d^{-2} z) = &- (q-q^{-1})\no{X_4(z) X_1(z) X_2(z) X_3(z) X_1(q_3z) X_2(q_3 z) X_1(q_3^2 z) X_0(q_2^{-1} q_3 z)}\\
    % T(d^{-1}z) = & \dfrac1{q-q^{-1}} \no{X_0(z) X_4(z) \brac*{ \frac{(d^2-q^2)}{(1-d^2) q} \wun(z) + \Sigma \left[X_1(z), X_2(z), X_3(z)\right] } }.
    % \end{aligned}
    % \end{equation}
    \begin{align}
    \Psi(d^{-2}z) = & \no{X_0(z) X_4(q_3 z)}\\
    \widetilde{\Psi}^-( qz) = &- \no{X_4(z) X_1(z) X_2(z) X_3(z) X_1(q_3z) X_2(q_3 z) X_1(q_3^2 z) X_0(q_2^{-1} q_3 z)}\\
    T(d^{-1}z) = &-\no{X_0(z) X_4(z) \brac*{ \frac{(d^2-q^2)}{(1-d^2) q} \wun(z) + \Sigma \left[X_1(z), X_2(z), X_3(z)\right] } }.
    \end{align}
\end{remark}

\begin{proof}
    The first statement is proven by showing that the operators in the r.h.s. of the relation \ref{qBP_realization2} satisfy the algebraic relations of the deformed BP algebra given in \cite{Harada:2020woh}. The (anti)commutation with the screening charges will be shown in a broader context in the next section.

    The two decorated quivers are indeed related by a mutation and gauging, and the non-trivial part is the transformation of the currents $\S[\CP_\pm](z)$. It follows again from a specialization of Theorem \ref{th:subreg_mutations} (see next section), this time for $N=3$ and $M=0,1$. Specifically, the realization I corresponds to $M=1$, and the realization II to $M=0$, and $\S[\CP_\pm](z)=G_M^\pm(z)$. %The transformation of currents $\Psi[\CP_\pm](z)$ remains trivial as their expression is again independent of the fermionic fields. 

\end{proof}

\begin{figure}[h]
\centering
\begin{tikzpicture}[scale=0.7, transform shape]
\node (x0) at (0,0) {$V_4(z)\tPsi(q_1q_2z)$};
\node (x1) at (5,0) {$V_4(z)\tPsi(q_1z)$};
\node (x2) at (10,0) {$V_4(z)V_2(z)\tPsi(z)$};
\node (x3) at (15,0) {$V_4(z)V_2(z)V_3(z)\tPsi(z)$};
\node (x4) at (10,-3) {$V_4(z)V_2(z)\tPsi(q^{-2}z)$};
\node (x5) at (15,-3) {$V_4(z)V_2(z)V_3(z)\tPsi(q^{-2}z)$};
\node (x6) at (15,-6) {$V_4(z)V_2(z)V_3(z)V_2(q_3z)\tPsi(q_3z)$};
\node (x7) at (15,-9) {$V_4(z)V_2(z)V_3(z)V_2(q_3z)\tPsi(q_3q^{-2}z)$};
\draw[styleArrow](x0) to[] node[midway,above]{$\tX_1(z)$} (x1);
\draw[styleArrow](x1) to[] node[midway,above]{$\tX_2(z)$} (x2);
\draw[styleArrow](x2) to[] node[midway,above]{$\tX_3(z)$} (x3);
\draw[styleArrow](x2) to[] node[midway,left]{$\tX_1(q_3z)$} (x4);
\draw[styleArrow](x3) to[] node[midway,right]{$\tX_1(q_3z)$} (x5);
\draw[styleArrow](x4) to[] node[midway,above]{$\tX_3(z)$} (x5);
\draw[styleArrow](x5) to[] node[midway,right]{$\tX_2(q_3z)$} (x6);
\draw[styleArrow](x6) to[] node[midway,right]{$\tX_1(q_3^2z)$} (x7);
\end{tikzpicture}
	\caption{q-character structure of the current $G^+(z)$}\label{fig:qcharacterG+}
\end{figure}
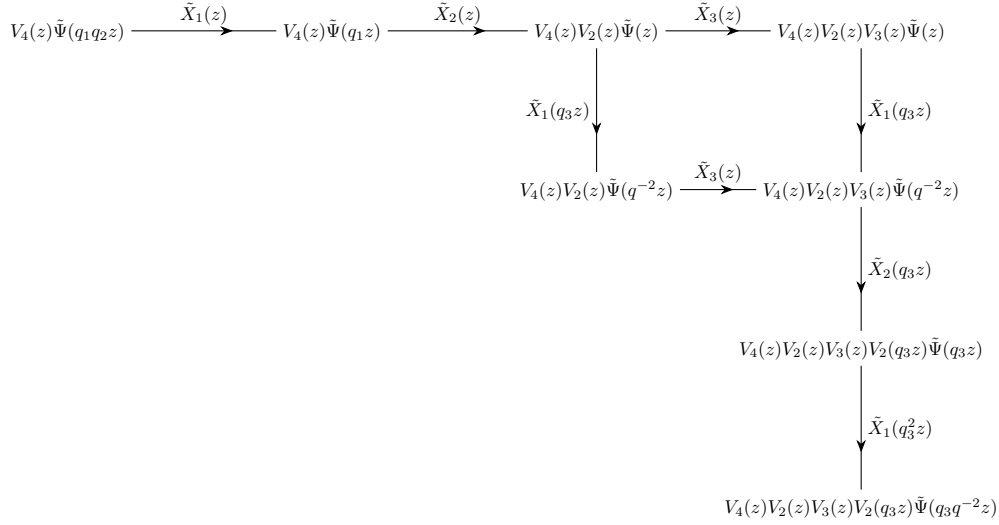

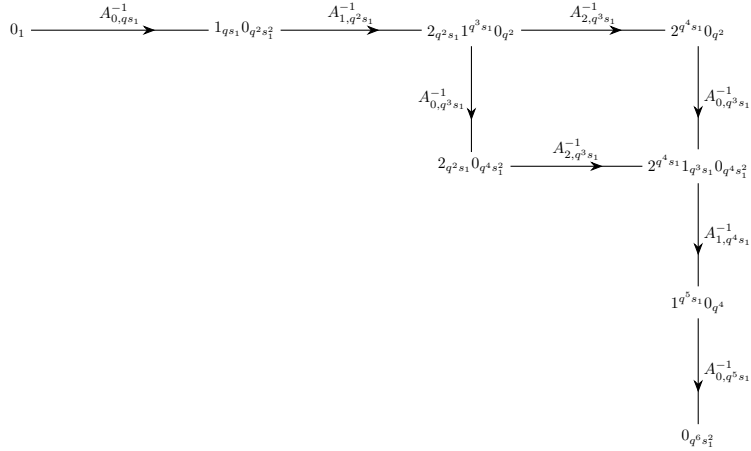
\begin{figure}[h]
\centering
\begin{tikzpicture}[scale=0.6, transform shape]
\node (x0) at (0,0) {$0_1$};
\node (x1) at (5,0) {$1_{qs_1}0_{q^2s_1^2}$};
\node (x2) at (10,0) {$2_{q^2s_1}1^{q^3s_1}0_{q^2}$};
\node (x3) at (10,-3) {$2_{q^2s_1}0_{q^4s_1^2}$};
\node (x4) at (15,0) {$2^{q^4s_1}0_{q^2}$};
\node (x5) at (15,-3) {$2^{q^4s_1}1_{q^3s_1}0_{q^4s_1^2}$};
\node (x6) at (15,-6) {$1^{q^5s_1}0_{q^4}$};
\node (x7) at (15,-9) {$0_{q^6s_1^2}$};
\draw[styleArrow](x0) to[] node[midway,above]{$A_{0,qs_1}^{-1}$} (x1);
\draw[styleArrow](x1) to[] node[midway,above]{$A_{1,q^2s_1}^{-1}$} (x2);
\draw[styleArrow](x2) to[] node[midway,left]{$A_{0,q^3s_1}^{-1}$} (x3);
\draw[styleArrow](x2) to[] node[midway,above]{$A_{2,q^3s_1}^{-1}$} (x4);
\draw[styleArrow](x3) to[] node[midway,above]{$A_{2,q^3s_1}^{-1}$} (x5);
\draw[styleArrow](x4) to[] node[midway,right]{$A_{0,q^3s_1}^{-1}$} (x5);
\draw[styleArrow](x5) to[] node[midway,right]{$A_{1,q^4s_1}^{-1}$} (x6);
\draw[styleArrow](x6) to[] node[midway,right]{$A_{0,q^5s_1}^{-1}$} (x7);
\end{tikzpicture}
	\caption{q-character of  $U_q(\widehat{\sl(3|1)})$ for the representation $\lambda = (1^3)$.}\label{fig:qcharacter}
\end{figure}

\begin{remark}
    We note that the generator $G^+(z)$ is no longer obtained as a telescoping sum in the second realization. In fact, its expression coincides with the formula obtained using the q-character of the algebra $U_q(\widehat{\gl(3|1)})$ in the representation $(1^3)$. To compare the two, we have represented the terms appearing in the expression of $G^+(z)$ in Figure \ref{fig:qcharacterG+}. Each node corresponds to a term in the expression, while the arrows are labeled by the factor needed to obtain the next term. We observe that some terms can be obtained in different ways, corresponding to different paths in the diagram.

    On the other hand, we have represented in Figure \ref{fig:qcharacter} the terms entering in the expression of the q-character of the representation $(1^3)$ of $U_q(\widehat{\gl(3|1)})$. This q-character has been computing following the method given in \cite{FJM2022}, and written using the notations of this paper. The $\ell$-weights of the q-character are associated with the node, while arrows carry the expression of the factor needed to obtain the next term, namely
    \begin{equation}
     A_{2,1}=2_{q,q^{-1}}1^1,\qquad A_{1,1}=2^1 1_{q,q^{-1}}0_{s_1}^{s_1^{-1}},\qquad A_{0,1}=1^1 0_{s_3}^{s_3^{-1}}.
    \end{equation}
    We observe that Figures \ref{fig:qcharacterG+} and \ref{fig:qcharacter} show the same diagram, and it is possible to identify the two expressions using the identification $X_i(z)=A_{i-1,q^i s_1}^{-1}$, and exchanging $q_2$ and $q_3$.
\end{remark}

\section{Free field realizations of deformed subregular W-algebras}
\label{sec:subreg}
In this section, we take the opposite approach of the previous one, namely we start from a certain chiral cluster seed, and a set of currents commuting with the corresponding screening charges. 
This defines an algebra that we call the deformed subregular W-algebra \(\CWqt^\text{sub}(\sl(N))\). 
At low rank, namely for \(N=2\) and \(N=3\), this algebra coincides with, respectively, the quantum affine $\sl(2)$ algebra and the deformed BP algebra. 

For each algebra \(\CWqt^\text{sub}(\sl(N))\), we define \(N+1\) free field realizations, and show that they are related by mutations. These realizations are in bijection with the Dynkin diagrams of \(\mathfrak{sl}(1|N)\), confirming a conjecture of \cite{FJM2022}. Then, using a particular free field realization, we define an embedding of \(\CWqt^\text{sub}(\sl(N))\) in $\CWqt(\sl(N))\otimes\CH_2$, which plays a role similar to the inverse Hamiltonian reduction, but in the deformed setting. Finally, we compare the free field realizations of $\CWqt^\text{sub}(\sl(N))$ and $\CWqt(\gl(1|N))$, and observe a relation between generating currents that had also been predicted in \cite{FJM2022}.

% and at the level of free field realizations. 
% In fact, we conjecture that this embedding can be lifted at the algebraic level $\CWqt^\text{sub}(\sl(N))\hookrightarrow \CWqt(\sl(N))\otimes\CH_2$, and this has been checked explicitly for $N=2$.

\subsection{Free field realization as a chiral cluster seed}
\subsubsection{Cluster seed}
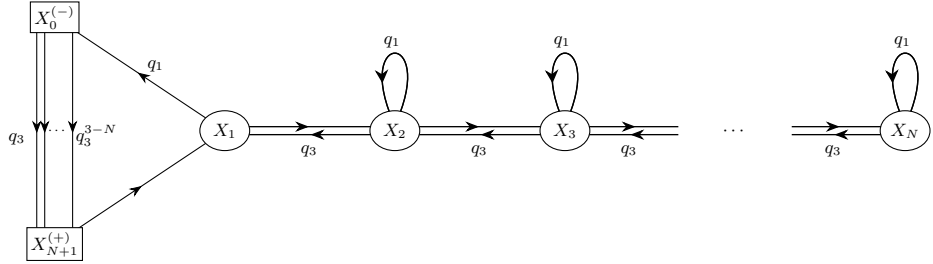
\begin{figure}[h]
\centering
\begin{tikzpicture}[scale=0.75, transform shape, font=\small]
%Coordinates
\coordinate (x0) at (0,2);
\coordinate (x1) at (3,0);
\coordinate (x2) at (6,0);
\coordinate (x3) at (9,0);
\coordinate (xmidBefore) at (11,0);
\coordinate (xmid) at (12,0);
\coordinate (xmidAfter) at (13,0);
\coordinate (xN) at (15,0);
\coordinate (xNp1) at (0,-2);

%Arrows
\draw[qArrow, qArrow/xshift=-9pt, qArrow/label={$q_3$}, qArrow/label xshift=-10pt] (x0) to (xNp1);
\draw[qArrow, qArrow/xshift=-5pt] (x0) to (xNp1);
\node[scale=0.8, xshift=2pt] at (0,0) {$\cdots$};
\draw[qArrow, qArrow/xshift=9pt, qArrow/label={$q_3^{3-N}$}, qArrow/label xshift=12pt, qArrow/label yshift=4pt] (x0) to (xNp1);

\draw[qArrow, qArrow/label={$q_1$}, qArrow/label xshift=8pt, qArrow/label yshift=8pt] (x1) to (x0);
\draw[qArrow] (xNp1) to (x1);

%Double opposing arrows
\draw[qArrowAbove] (x1) to (x2);
\draw[qArrowBelow, qArrow/label={$q_3$}] (x2) to (x1);

\draw[qArrowAbove] (x2) to (x3);
\draw[qArrowBelow, qArrow/label={$q_3$}] (x3) to (x2);

% Double arrows before the cdots
\draw[qArrowAbove, qArrow/arrowpos=0.6] (x3) to (xmidBefore);
\draw[qArrowBelow, qArrow/label={$q_3$}, qArrow/arrowpos=0.4] (xmidBefore) to (x3);

% Double arrows after the cdots
\draw[qArrowAbove, qArrow/arrowpos=0.4] (xmidAfter) to (xN);
\draw[qArrowBelow, qArrow/label={$q_3$}, qArrow/arrowpos=0.6] (xN) to (xmidAfter);

%Loops
\draw[qLoop, qArrow/label={$q_1$}] (x2) to (x2);
\draw[qLoop, qArrow/label={$q_1$}] (x3) to (x3);
\draw[qLoop, qArrow/label={$q_1$}] (xN) to (xN);

%Nodes
\node[qNodeFrozen] at (x0) {$X^{(-)}_0$};
\node[qNodeUnfrozen] at (x1) {$X_1$};
\node[qNodeUnfrozen] at (x2) {$X_2$};
\node[qNodeUnfrozen] at (x3) {$X_3$};
\node at (xmid) {$\cdots$};
\node[qNodeUnfrozen] at (xN) {$X_N$};
\node[qNodeFrozen] at (xNp1) {$X^{(+)}_{N+1}$};

\end{tikzpicture}
	\caption{Quiver associated with the algebra $\CWqt^\text{sub}(\sl(N))$.}
	\label{fig:FSQuiver}
\end{figure}

The following definition introduces the chiral cluster seed associated to the quiver represented on Figure \ref{fig:FSQuiver}.

\begin{definition}
    Let $\CQ_N^\text{sub}$ denote the quiver represented in Figure \ref{fig:FSQuiver} and defined as follows. The set of vertices consists of $N$ unfrozen vertices, with label $1,\cdots,N$, and two frozen vertices of type $-$ and $+$ respectively denoted $0^-$ and $(N+1)^+$. These vertices are connected by the following set of arrows,
    \begin{itemize}
        \item between each pair of vertices $(i,i+1)$ for $i=1\cdots N-1$, a pair of arrows in opposite direction, with the arrow $i\to(i+1)$ of parameter one and the arrow $(i+1)\to i$ of parameter $q_3$,
        \item at each vertex $i$ with $i=2\cdots N$, a single loop with parameter $q_1$,
        \item between the vertices $1$ and $0^-$, an arrow $1\to0^-$ with parameter $q_1$,
        \item between the vertices $(N+1)^+$ and $1$, an arrow $(N+1)^+\to 1$ with parameter $1$,
        \item between the frozen vertices $0^-$ and $(N+1)^+$, $N-1$ arrows $0^-\to(N+1)^+$ of parameters $q_3,1,q_3^{-1},\cdots q_3^{3-N}$.
    \end{itemize}
    This quiver defines a chiral cluster seed that we denote by the same symbol.
\end{definition}

It is instructive to discuss in details the properties of this chiral cluster seed, and the corresponding algebra $\CA_{\CQ_N^\text{sub}}$. 
Firstly, the matrix $\e$ associated with the quiver and encoding the OPEs of zero modes is
\begin{equation}
    \e=\pmat{0&-1&0&\cdots &0&N-1\\1&0&&\cdots &0&-1\\0&&&\cdots &&0\\\vdots&&&\ddots&&\vdots\\0&&&\cdots&&0\\1-N&1&0&\cdots&&0}.
\end{equation}
We observe that the quantum variables $X_i$ associated with unfrozen vertices $i=2\cdots N$ are Casimirs, i.e. they are in the center of the quantum torus algebra $\mC_q[\CX_{\mathsf{s}^\text{ch}}]$. 
Thus, only the vertices $0^-$, $1$ and $(N+1)^+$ carry non-trivial zero modes. 
The matrix $\e$ has rank two, and the quantum torus algebra has two pairs of generators, that we denote $(P_\a,Q_\a)$ for $\a=1,2$ (with $P_\a Q_\b=q^{\d_{\a,\b}}Q_\b P_\a$). 
In this section, we will use the choice of polarization given by 
\begin{equation}\label{eq:polarizaion_subreg}
    X_0=P_2P_1^{-2},\qquad X_1=Q_1,\qquad X_{N+1}=P_1^2Q_2^{2(N-1)}.
\end{equation}

Then, we examine the matrix $C$, which reads
\begin{equation}\label{eqn:CMatrix_Sub}
    C=(q-q^{-1})\pmat{q & -qq_1 & 0 & & \cdots & &0 & \a_N\\
    q^{-1}q_1^{-1} & q-q^{-1} & b &0 &\cdots &&0 & -q\\
    0 & -b^\vee & b^\vee-b & b &0 &\cdots &0 &0\\
    0 & 0 &-b^\vee & b^\vee-b & b &0&\cdots & 0\\
    & & & \ddots &\ddots &\ddots & & \\
    0 & & \cdots &  & -b^\vee & b^\vee-b & b & 0\\
    0 & & \cdots & & 0 & -b^\vee & b^\vee-b & 0\\
    -\a_N^\vee & q^{-1} & 0 & & \cdots & & 0 & -q^{-1}
    },
\end{equation}
where it has been convenient to introduce the following parameters,
\begin{equation}\label{eq:def_ab_aN}
b=q^{-1}(1-q_1^{-1}),\qquad \a_N=qq_1\dfrac{1-q_3^{N-1}}{1-q_3}, %a=q-q^{-1}-qq_1+q^{-1}q_1^{-1},\qquad
\end{equation}
and for a quantity $x$, we denote the Adams operation $x^{[-1]}=x^\vee$. This matrix has rank $N+1$, and so both $C$ and $C^T$ have null vectors. As a result, the center of the Heisenberg algebra generated by the modes $x_{i,n}$ is non-trivial.

\begin{proposition}\label{Prop:NullVector}
    When $N>2$, the following operators belong to the center of the algebra $\CA_{\CQ_N^\text{sub}}$,
    \begin{equation}\label{eq:vn_subreg}
        v_n=\b_0^{[n]}x_{0,n}+\dfrac{1-q_3^{-nN}}{1-q_3^{-n}}u_n+q_1^{-n}q_3^{-(N-1)n}\b_0^{[-n]}x_{N+1,n},\qquad v_{-n}=q_1^n q_3^{(N-1)n}x_{0,-n}+\dfrac{1-q_3^{Nn}}{1-q_3^n} u_{-n}+x_{N+1,-n},\qquad (n>0),
    \end{equation}
    where it has been convenient to introduce
    \begin{equation}\label{eq:def_u}
         \b_0=-q_1^{-1}\dfrac{1+q_1q_3-q_1-q_3^{1-N}}{1-q_3^{N-2}},\qquad u_{n}=\sum_{i=1}^N\dfrac{1-q_3^{- (N+1-i)n}}{1-q_3^{-Nn}} x_{i,n},\qquad (n\neq0).
    \end{equation}
    When $N=2$, the operators $v_n$ with $n>0$ should be replaced by $v_n=x_{0,n}+q_3^{-n}x_{3,n}$ while $v_{-n}$ is still given by the previous formula.
\end{proposition}

\begin{proof}
The proof is by a direct calculation.
% This is a direct consequence of the following vectors being null vectors of $C$ and $C^T$, i.e. $ C\vb{v}_-=C^T\vb{v}_+=\vb{0}$ for
% \begin{equation}
%     \vb{v}_{-}=\pmat{q_1q_3^{N-1}\\\frac{1-q_3^N}{1-q_3}\\\frac{1-q_3^{N-1}}{1-q_3}\\\vdots\\1\\1},\qquad
%     \vb{v}_+=\pmat{\b_0\\\frac{1-q_3^{-N}}{1-q_3^{-1}}\\\frac{1-q_3^{-(N-1)}}{1-q_3^{-1}}\\\vdots\\1\\q_1^{-1}q_3^{1-N}\b_0^\vee }.
% \end{equation}
% The operators in the center are obtained as $v_{\pm n}=(\vb{v}_\pm^{[n]})^T.\vb{x}_{\pm n}$ with $\vb{x}_n^T=\pmat{x_{0,n}&x_{1,n}&\cdots&x_{N+1,n}}$.
\end{proof}

In addition to the operators $v_{\pm n}$, there is another linear combination of generators that plays a notable role. It is the equivalent of the $U(1)$-factor for the algebra $\CFsub$, and will appear in the quotient defining the algebra $\CF_{\CQ_N^\text{sub}}[\CWqt^\text{sub}(\sl(N))']$ in section \ref{sec:subreg_sl} below.

\begin{proposition}\label{Prop:hn_subreg}
    When $N>2$, the following linear combinations of operators commute with the modes $x_{i,n}$ associated with all unfrozen vertices $i=1\cdots N$,
    \begin{equation}
        h_n=x_{0,n}+q_3^{-n} x_{N+1,n},\qquad (n\neq 0).
    \end{equation}
    When $N=2$, the operators $h_n$ with $n>0$ should be replaced by $h_n=q_2^nx_{0,n}+(1+q_3^{-n})u_n+x_{3,n}$, while $h_{-n}$ is still given by the formula above.
\end{proposition}

\begin{proof}
The proof is by a direct calculation.
%     It follows from the fact that these operators obtained as $h_{\pm n}=(\vb{h}_\pm^{[n]})^T.\vb{x}_{\pm n}$ with the following vectors
%     \begin{equation}
%     \vb{h}_\pm=\pmat{1\\0\\\vdots\\0\\q_3^{\mp1}}\implies C\vb{h}_-=(1-q_2^{-1})q_3\dfrac{1-q_3^{N-2}}{1-q_3}\pmat{q_1^{-1}q_3^{1-N}\b_{0}^{[-1]}\\0\\\vdots\\0\\-\b_0},\qquad  C^T\vb{h}_+=(1-q_2)\dfrac{1-q_3^{N-2}}{1-q_3}\pmat{q_3^{2-N}\\0\\\vdots\\0\\-q_1q_3}. 
% \end{equation}
% Note that the choice of vectors $\vb{h}_\pm$ is not unique as we can shift by $\vb{v}_\pm$ without changing the proposition. Our choice appears to be the most convenient one. 
\end{proof}

% \begin{equation}
%     \b_0=-q_1^{-1}\dfrac{1+q_1q_3-q_1-q_3^{1-N}}{1-q_3^{N-2}},\qquad \b_{N+1}=q_2q_3^{-1}\dfrac{1+q_1q_3-q_3-q_1q_3^N}{1-q_3^{N-2}}=q_2q_3^{2-N}\b_0^\vee.
% \end{equation}

\subsubsection{The algebra \(\CWqt^\text{sub}(\sl(N))\)}
\begin{definition}\label{def:Wsub_slN}
    Let $X_i(z)$ denote the vertex operators associated to the vertex $i$ in the chiral cluster seed with quiver $\CQ_N^\text{sub}$. Let $G^\pm(z)$ be the formal series of operators defined by $G^-(z)=X^-_0(z)$ and the Miura transform 
    \begin{equation}\label{eqn:GpSub}
        G^+(z)=:X_{N+1}^+(z)\left(1+X_1(z)X_2(z)\cdots X_N(z)T\right)\left(1+X_1(z)X_2(z)\cdots X_{N-1}(z)T\right)\cdots(1+X_1(z)T):1,
    \end{equation}
    where $T$ is the shift operator such that $Tz=q_3zT$, and $T.1=1$. 
    
    We define \(\CWqt^\text{sub}(\sl(N))\) to be the subalgebra of $\CA_{\CQ_N^\text{sub}}$ generated by the currents $G^\pm(z)$.    
\end{definition}

% \begin{conjecture} The algebra $\CFsub$ defines a free field realization of a deformed W-algebra $\CWqt^\text{sub}(\sl(N))$.
% \end{conjecture}

% The conjectured is proven for $N=2$ and $N=3$. 
For $N=2$, the algebra $\CWqt^\text{sub}(\sl(N))$ coincides with $U_q(\widehat{\sl(2)})$ and the currents $G^\pm(z)$ are identified with the free field representations of $E(z)$ and $F(z)$ in the second realization (see section \ref{sec:Uqsl2}). For $N=3$, the algebra $\CWqt^\text{sub}(\sl(N))$ coincides with the deformed Bershadsky-Polyakov algebra (see section \ref{sec:qbp}). %defined by Harada in \cite{Harada:2021xnm} 
% For $N>3$, it remains a conjecture due to the absence of a definition using generators and relations.

% In \cite{Harada:2021xnm}, a different free field realization of $\CWqt^\text{sub}(\sl(N))$ is proposed, using the relation to Corner VOA. For $N=2$ and $N=3$, it coincides with the first realization of resp. $U_q(\widehat{\sl(2)})$ and deformed BP discussed in the previous section. A different free field realization has been introduced in \cite{FJM2022}. {\color{red}These realizations are related to ours via mutations and coincide with the ones given in Section~\ref{ssec:other FF realiz}.}
% % For $N=4$, we have checked that this realization is related to ours by a quiver mutation at node $1$, and we expect it to be also the case for higher $N$. A different free field realization has been introduced in \cite{FJM2022}, it will be discussed in the section \ref{sec:FJM} below.

\paragraph{Algebraic relations}
While the full set of generators and relations is still unknown for algebras $\CWqt^\text{sub}(\sl(N))$, it is possible to derive some of the relations obeyed by the currents $G^\pm(z)$.

\begin{proposition}\label{Prop:commGpGm}
    The currents $G^+(z)$ and $G^-(z)$ satisfy the following algebraic relations
    \begin{align}\label{eq:relGpmGpm}
        &(z-q^{\mp 2}w)G^\pm(z)G^\pm(w)+(w-q^{\mp 2}z)G^\pm(w)G^\pm(z)=0,\\
        &[G^+(z),G^-(w)]=\sum_{n=1}^{N-1}\d(q_3^{2-n}w/z) \CK_n(w)+\d(q_2q_3^{2-N}w/z)\CK^+_{N}(z),\label{eq:commGpGm}
    \end{align}
    where $\CK_i(z)$ with $i=0\cdots N$ are higher generators. In particular,
    \begin{align}\label{eq:subreg_K0N}
    \begin{split}
        \CK_1(z)&=-(q-q^{-1})q^{N-2}\dfrac{(q_2^{-1}q_3;q_3)_{N-2}}{(q_3;q_3)_{N-2}}:X_0^-(z)X_{N+1}^+(q_3z):,\\
        \CK_N^+(z)&=(q-q^{-1}):X_0^-(q_2^{-1}q_3^{N-2}z)X_{N+1}^+(z)\prod_{i=1}^{N}\prod_{k=0}^{N-i}X_i(q_3^kz):,
    \end{split}
    \end{align}
    and for $N\geq3$,
    \begin{equation}\label{eq:subreg_K2}
        \CK_2(z)=-(q-q^{-1})q^{3-N}\dfrac{(q_2^{-1}q_3;q_3)_{N-3}}{(q_3;q_3)_{N-3}}:X_0^-(z)X_{N+1}^+(z)\left(q^{-1}\dfrac{1-q_3q_2}{1-q_3}+\S[X_1,X_2,\cdots,X_N](z)\right):.
    \end{equation}
\end{proposition}

\begin{proof}
    The proof is a bit lengthy; it can be found in Appendix \ref{AppB}.
\end{proof}

\begin{remark}
    The commutation relation \ref{eq:commGpGm} can be compared to Harada's formula (4.48) in \cite{Harada:2021xnm}. The latter can be rewritten in the following form,
    \begin{equation}
        [\CE_N(z),\CF_N(d^{N-4}w)]=\d(q^2d^{2N-2}w/z)\tilde{\CK}_N(w)+\sum_{i=0}^{N-2}\d(d^{2i-2}w/z)\tilde{\CK}_i(w).
    \end{equation}
    We observe that the argument of the delta functions match with ours under the identification $q^2\equiv q_2$ and $d^{-2}\equiv q_3$. The fact that this relation does not depend on the choice of quiver and free field realization supports the existence of the underlying algebra $\CWqt^\text{sub}(\sl(N))$ as an independent object.

    When comparing with the results of section \ref{sec:qbp}, we need to rescale the generating currents as follows,
    \begin{align}
        &G^+(z)=(q-q^{-1})G_\text{qBP}^+(z),\qquad G^-(z)=(q-q^{-1})^2G_\text{qBP}^-(zd^{-2}),\\
        &\CK_1(z)=(1-q_2)\dfrac{1-q_2^{-1}q_3}{1-q_3}\Psi_\text{qBP}(z),\qquad \CK_2(z)=(q-q^{-1})T_\text{qBP}(d^{-1}z),\qquad \CK_3^+(z)=-(q-q^{-1})\tPsi^-_\text{qBP}(q^{-1}z).
    \end{align}
    
    \end{remark}

\begin{remark}
    With the previous choice of operators $h_{\pm n}$ and $v_{\pm n}$, we have for $N>2$,
    \begin{align}
        \CK_1(z)&=-(q-q^{-1})q^{N-2}\dfrac{(q_2^{-1}q_3;q_3)_{N-2}}{(q_3;q_3)_{N-2}}:Q_2^{2(N-1)}P_2: \exp\left(\sum_{n>0}z^{n}h_{-n}\right)\exp\left(\sum_{n>0}z^{-n}h_n\right),\\
        \CK_N^+(z)&=(q-q^{-1}):Q_1^NQ_2^{2(N-1)}P_2:\exp\left(\sum_{n>0}z^{n}v_{-n}\right)\exp\left(\sum_{n>0}z^{-n}\left(v_n-(1+q_2^n)\frac{1-q_3^n}{1-q_3^{(N-2)n}}h_n\right)\right).
        % &:X_0^-(z)X_{N+1}^+(q_3z):=:Q_2P_2^{2(N-1)}: e^{\sum_{n>0}z^{n}h_{-n}^+}e^{\sum_{n>0}z^{-n}h_n^+},\\
        % &:X_{N+1}^+(z)X_0(q_1q_3^{N-1}z)\prod_{i=1}^N\prod_{k=1}^{N-i+1}X_i(q_3^{k-1}z):=:Q_1^NQ_2P_2^{2(N-1)}:e^{\sum_{n>0}z^{n}v_{-n}}e^{\sum_{n>0}z^{-n}\left(v_n+(1+q_2^n)\frac{1-q_3^n}{1-q_3^{(N-2)n}}h_n\right)}.
    \end{align}
    The fact that, upon setting $v_n\equiv0$, the current $\CK_N^+(z)$ has only positive modes justifies the superscript $+$. When $N=2$, we recover instead the expression of the currents $\psi^\pm(z)$ of the algebra $U_q(\widehat{\sl(2)})$,
        \begin{align}
        \CK_1(z)&=-(q-q^{-1}):Q_2^{2}P_2: \exp\left(\sum_{n>0}z^{n}h_{-n}\right)\exp\left(\sum_{n>0}z^{-n}v_n\right),\\
        \CK_N^+(z)&=(q-q^{-1}):Q_1^2Q_2^{2}P_2:\exp\left(\sum_{n>0}z^{n}v_{-n}\right)\exp\left(\sum_{n>0}z^{-n}h_n\right).
        % &:X_0^-(z)X_{N+1}^+(q_3z):=:Q_2P_2^{2(N-1)}: e^{\sum_{n>0}z^{n}h_{-n}^+}e^{\sum_{n>0}z^{-n}h_n^+},\\
        % &:X_{N+1}^+(z)X_0(q_1q_3^{N-1}z)\prod_{i=1}^N\prod_{k=1}^{N-i+1}X_i(q_3^{k-1}z):=:Q_1^NQ_2P_2^{2(N-1)}:e^{\sum_{n>0}z^{n}v_{-n}}e^{\sum_{n>0}z^{-n}\left(v_n+(1+q_2^n)\frac{1-q_3^n}{1-q_3^{(N-2)n}}h_n\right)}.
    \end{align}
\end{remark}

\paragraph{Screening charges} The following proposition shows that the algebra $\CWqt^\text{sub}(\sl(N))$ lies in the intersection of the kernel of the screening charge, a key property of W-algebras. 
Strictly speaking, generators $G^\pm(z)$ anticommute with the fermionic charge $Q_1$ attached to the vertex $1$, but it can be remedied by introducing a cocycle twist. 
Note also that we did not prove that $\CWqt^\text{sub}(\sl(N))$ actually coincides with the whole intersection of kernels.

\begin{proposition}[Screening charges]\label{Prop:screenings}
    Let $Q_1,\ Q_i^\pm$ denote the screening charges associated with the unfrozen vertices $1$ and $i=2\cdots N$, as defined in section \ref{sec:screenings}. The currents $G^\pm(z)$ (anti)commute with these charges, i.e.
    \begin{equation}
        \{G^\pm(z),Q_1\}=0,\qquad [G^\pm(z),Q_i^+]=[G^\pm(z),Q_i^-]=0.
    \end{equation}
\end{proposition}

\begin{proof} Let's start by examining the fermionic charge $Q_1$. For this purpose, we introduce the fermionic fields $(\Psi(z),\Psi^\ast(z))$ associated with the vertex $1$, and parameterize
\begin{align}
\begin{split}
    &X_0^-(z)=V_0(z)\Psi^\ast(qq_3z),\qquad X_1(z)=-:\Psi(q^{-1}z)\Psi^\ast(qz):,\qquad X_2(z)=V_2(z):\Psi(qq_3z)\Psi^\ast(q^{-1}z):,\\
    &X_i(z)=V_i(z),\qquad X_{N+1}^+(z)=V_{N+1}(z)\Psi(qz),
\end{split}
\end{align}
where $[V_i(z),\Psi(w)]=[V_i(z),\Psi^\ast(w)]=0$ and $i=3\cdots N$. We can rewrite the currents in this parameterization: $G^-(z)=\Psi^\ast(qq_3z)V_0(z)$, and\footnote{For the first $(N-1)$ factors, we can choose either $1$ or the factor proportional to $\Psi^\ast(qz)\Psi(qq_3z)T$. Choosing $n$ times the second factor gives an operator proportional to $\Psi^\ast(qz)\Psi(qq_3^nz)T^n$. Taking into account $\Psi(qz)$ in the prefactor, and the last factor $(1-\Psi^\ast(qz)\Psi(q^{-1}q_3z)T)$, we arrive at the formula given here.}
\begin{align}
\begin{split}
    G^+(z)=:V_{N+1}(z)\Psi(qz)&\left(1-\Psi^\ast(qz)\Psi(qq_3z)\prod_{i=2}^NV_i(z)T\right)\left(1-\Psi^\ast(qz)\Psi(qq_3z)\prod_{i=3}^NV_i(z)T\right)\\
    &\cdots(1-\Psi^\ast(qz)\Psi(qq_3z)V_2(z)T)(1-\Psi^\ast(qz)\Psi(q^{-1}z)T):1.
\end{split}
\end{align}
Obviously, $\{\Psi^\ast(z),G^-(w)\}=0$, and so $G^-(w)$ anticommutes with $Q_1$. For $G^+(z)$, we note that the current can be expanded as follows,
\begin{equation}
    G^+(z)=\sum_{n=0}^{N-1}\CO_n(z)\left(\Psi(qq_3^nz)-\Psi(q^{-1}q_3^nz)\right),
\end{equation}
where $\CO_n(z)$ is a sum of products of operators $V_i(z)$ that commutes with $\Psi(w)$ and $\Psi^\ast(w)$. As a result, we have
\begin{equation}
    \{\Psi^\ast(z),G^+(w)\}=\sum_{n=0}^{N-1}\CO_n(w)\left(\d(qq_3^nw/z)-\d(q^{-1}q_3^nw/z)\right),
\end{equation}
and so $G^+(z)$ also anticommutes with $Q_1$.

Next, we turn to bosonic screening charges. There is no arrow between the vertices $i=2\cdots N$ and the frozen vertex $0^-$, so the charges trivially commute with $G^-(z)$. Then, we examine the current $G^+(z)$, and consider the charges associated with the vertex $X_i(z)$. From the quiver, we observe that non-trivial OPEs exist only with vertices $j=i-1,i,i+1$ (or $j=N-1,N$ if $i=N$). In fact, these three vertices reproduce the configuration represented in Figure \ref{fig:BosonicScreenings}, and studied in Example \ref{Ex:BosonicScreenings}. The loop parameter is $\k=q_1$, and the vertex $Z$ is omitted if $i=N$. We can rewrite the current $G^+(z)$ in the following form,
\begin{equation}
    G^+(z)=:\CO_1(z,T)\left(1+X_1(z)\cdots X_i(z)T\right)\left(1+X_1(z)\cdots X_{i-1}(z)T\right)\CO_2(z,T)1:.
\end{equation}
In this formula, $\CO_2(z,T)$ is a sum of vertex operators containing neither $X_{i-1}(z)$, $X_i(z)$ or $X_{i+1}(z)$, and so it commutes with $X_i(z)$. On the other hand, $\CO_1(z,T)$ is a sum of vertex operators in which these three vertex operators enter only through the combination $:X_{i-1}(z)X_i(z)X_{i+1}(z):$. The Proposition \ref{Prop:CommBosonicScreening} with $\Psi[X_{i-1},X_i,X_{i+1}](z)$ implies that this operator also commutes with the screening charges $Q_i^\pm$. Let's expand the two middle factors, we find
\begin{equation}
    1+:A_i(z)(X_{i-1}(z)+X_{i-1}(z)X_i(z)):T+:A_i(z)A_i(q_3z)X_{i-1}(z)X_i(z)X_{i-1}(q_3z):T^2,
\end{equation}
where  $A_i(z)=:X_1(z)\cdots X_{i-2}(z):$ does not contain the vertex operators $X_{i-1}(z)$, $X_i(z)$ or $X_{i+1}(z)$. The second and third terms in this expression correspond respectively to the third and second quantities considered in Proposition \ref{Prop:CommBosonicScreening}. As a result, they both commute with the screening charges $Q_i^\pm$. Putting everything together, it shows that $G^+(z)$ commutes with bosonic charges $Q_i^\pm$.
\end{proof}

\subsubsection{Quotient by the $U(1)$ factor}\label{sec:subreg_sl}
Just as the algebra $\CWqt(\sl(N))$ is related to $\CWqt(\gl(N))$ by factoring out the Heisenberg algebra associated with the $U(1)$ factor, one can similarly define a quotient of $\CWqt^\text{sub}(\sl(N))$ by a rank one Heisenberg algebra $\CH_1^{\text{U(1)}}$. 
For $N=2$, it can be understood as a rescaling of the Drinfeld currents $E(z)$ and $F(z)$ generating $U_q(\widehat{\sl(2)})$ in such a way that the Cartan subalgebra generated by modes $H_n$ decouples. 
% We denote this quotient $\CWqt^\text{sub}(\sl(N))$. 

\begin{proposition}
    Consider the chiral cluster seed associated with the decorated quiver $\CQ_N^\text{sub}$, and let the currents $G^\pm(z)$, which generate the algebra $\CFsub$, be decomposed as follows.
    \begin{equation}
        G^\pm(z)=H^\pm(z)\bar G^\pm(z),\qquad\text{with}\qquad  H^\pm(z)\in\CH_1^{\text{U(1)}},\qquad [\bar G^\pm(z),\CH_1^{\text{U(1)}}]=0.
    \end{equation}
    We denote by $\CWqt^\text{sub}(\sl(N))'$ the algebra generated by the currents $\bar G^\pm(z)$. We have
    \begin{equation}
        \CWqt^\text{sub}(\sl(N))\simeq \CWqt^\text{sub}(\sl(N))' \otimes\CH_1^\text{U(1)}.
    \end{equation}    
\end{proposition}

\begin{proof}
To prove this statement, we introduce the decomposition
\begin{equation}
    v_n=k_n+\bar u_n,\qquad \bar u_n=\sum_{i=1}^N\dfrac{1-q_3^{-(N+1-i)n}}{1-q_3^{-n}}x_{i,n}=\dfrac{1-q_3^{-Nn}}{1-q_3^{-n}}u_n,
\end{equation}
and $k_m=\b_0^{[m]}x_{0,m}+q_1^{-m}q_3^{-(N-1)m}\b_{0}^{[-m]}x_{N+1,m}$, $k_{-m}=q_1^mq_3^{(N-1)m}x_{0,-m}+x_{N+1,-m}$ for $m>0$. Since $[h_n,\bar u_m]=[h_n,v_m]=0$, we have $[h_n,k_m]=0$. Recall that $[v_n,x_{i,m}]=0$ for $i=0\cdots N+1$ and so we can take $v_n=0$, i.e. $k_n=-\bar u_n$ is expressed in terms of $x_{i,n}$ with $i=1\cdots N$ (otherwise, we just replace $k_n\to k_n+v_n$ below). The transformation between $(h_n,k_n)$ and $(x_{0,n},x_{N+1,n})$ can be easily inverted,
\begin{align}
    &\pmat{x_{0,n}\\x_{N+1,n}}=\dfrac1{1-(q_1q_3^N)^{-n}}\pmat{q_1^{-n}q_3^{-(N-1)n}\b_0^{[-n]}&-\b_0^{[n]}\\-q_3^{-n}&1}\pmat{h_n\\k_n},\\
    &\pmat{x_{0,-n}\\x_{N+1,-n}}=\dfrac1{1-(q_1q_3^N)^n}\pmat{1&-q_1^nq_3^{(N-1)n}\\-q_3^n&1}\pmat{h_{-n}\\k_{-n}}.
\end{align}
Decomposing the operators $X_0^-(z)$ and $X_{N+1}^+(z)$ associated with the frozen vertices into $h_n$ and $k_n$ contributions, we can rewrite the currents in the form previously mentioned, with $\bar G^\pm(z)$ expressed solely in terms of modes $x_{i,n}$ for $i=1\cdots N$, and
\begin{equation}
    H^+(z)=Q_2^{2(N-1)}:e^{-\sum_{n\neq0}\frac{z^{-n}q_3^{-n}}{1-q_1^{-n}q_3^{-Nn}}h_n}:,\qquad H^-(z)=P_2 e^{\sum_{n>0}\frac{z^{n}}{1-q_1^nq_3^{Nn}}h_{-n}}e^{\sum_{n>0}\frac{z^{-n}q_1^{-n}q_3^{-(N-1)n}\b_0^{[-n]}}{1-q_1^{-n}q_3^{-Nn}}h_n}.
\end{equation}
Note that the zero modes part has also been factorized, with $\bar G^\pm(z)$ now depending only on a single pair of quantum torus variables, namely $Q_1,P_1$.
\end{proof}

\subsection{Other free field realizations} \label{ssec:other FF realiz}
In this section we define \(N+1\) free field realizations of the algebra $\CWqt^{\text{sub}}(\sl(N))$. 
Namely, for each \(M=0,\dots,N\), we define a decorated quiver $\CQ_N^{\text{sub}(M)}$ and a pair of currents $G_M^\pm(z)$. 
For \(M=0\) these are the quiver and currents defined above, and we show that quivers and currents associated with other values of \(M\) are obtainable by mutations. 

We observe that the unfrozen part of the quivers $\CQ_N^{\text{sub}(0)}$ and $\CQ_{1|N}$ (associated to the standard realization of the algebra $\CWqt(\gl(1|N))$) coincide. 
As a result, the \(N+1\) free field realizations $\CF_{\CQ^{\text{sub}(M)}}[\CWqt^\text{sub}(\sl(N))]$ are in one-to-one correspondence with Dynkin diagrams of $\sl(1|N)$. 
We will explore further the connection between these two algebras in subsection \ref{sec:FJM}. 

% \begin{figure}[!htbp]
%     \centering
%     \includegraphics[height=0.95\textheight,keepaspectratio]{./fig/SubRegN4all.pdf}
%     \caption{All subregular quivers for \(N=4\).}\label{fig:SubRegN4all}
% \end{figure}
% \clearpage

\begin{figure}[!htbp]
    \centering
\begin{tikzpicture}[scale=0.7, transform shape]
    \def\xs{3}
	\def\ys{2}
    \def\CQ{{\mathcal{Q}}}
	\tikzset{
		qSubArrow/.style={
			qArrow/color=black,
			qArrow/width=0.5pt,
			qArrow/arrowpos=0.5,
			qArrow/label={},
			qArrow/labelstyle={above, inner sep=1pt},
			qArrow/label xshift=0pt,
			qArrow/label yshift=0pt,
			qArrow/connector=straight,
			qArrow/decor=none,
			qArrow/pattern=solid
		},
		qSubLoop/.style={
			qArrow/connector=loop,
			qArrow/loop out=80,
			qArrow/loop in=100,
			qArrow/loop width=5mm,
			qArrow/loop depth=15mm,
			qArrow/width=0.5pt,
			qArrow/arrowpos=0.80,
			qArrow/labelstyle={left, inner sep=1pt},
			qArrow/label xshift=0pt,
			qArrow/label yshift=11pt
		},
		qFrozenCurve/.style={
			qSubArrow,
			qArrow/connector=controls,
			qArrow/arrowpos=0.50
		},
        qArrowBelowOp/.style={
    		qArrow/base,
    		qArrow/yshift=-3pt,
    		qArrow/color=black,
    		qArrow/arrowtip=Stealth[reversed],
    		qArrow/arrowpos=0.5,
    		qArrow/width=0.5pt,
    		qArrow/labelstyle={below, inner sep=1pt},
    		qArrow/label xshift=-2pt,
    		qArrow/label yshift=-4pt
	    },
        qArrowAbove/.style={
    		qArrow/base,
            qArrow/yshift=3pt,
            qArrow/color=black, 
            qArrow/arrowtip=Stealth,
            qArrow/arrowpos=0.5,
            qArrow/width=0.5pt,
            qArrow/labelstyle={above, inner sep=1pt},
            qArrow/label xshift=0pt,
            qArrow/label yshift=4pt
        }
	}
%Diagram 1--------------------------------------------------------------------------------------------------%
	\begin{scope}
		\node at (-1.35*\xs,0*\ys) {\normalsize $\CQ_{4}^{\mathrm{sub}(0)}$};

		% Coordinates
		\coordinate (x0) at (0*\xs,1*\ys);
		\coordinate (x1) at (1*\xs,0*\ys);
		\coordinate (x2) at (2*\xs,0*\ys);
		\coordinate (x3) at (3*\xs,0*\ys);
		\coordinate (x4) at (4*\xs,0*\ys);
		\coordinate (x5) at (0*\xs,-1*\ys);

		% Frozen-frozen arrows
		% \draw[qArrow/base, qSubArrow,
		%       qArrow/connector=bend,
		%       qArrow/bend=-42,
		%       qArrow/label={$q_3$},
		%       qArrow/labelstyle={left, inner sep=1pt},
		%       qArrow/label xshift=-28pt,
		%       qArrow/label yshift=0pt]
		% 	(x0) to (x5);
        \draw[qArrow/base, qSubArrow,
		      qArrow/xshift=8pt,
		      qArrow/label={$q_3$},
		      qArrow/labelstyle={left, inner sep=1pt},
		      qArrow/label xshift=-20pt,
		      qArrow/label yshift=0pt]
			(x0) to (x5);

		\draw[qArrow/base, qSubArrow]
			(x0) to (x5);

        \draw[qArrow/base, qSubArrow,
		      qArrow/xshift=-8pt,
		      qArrow/label={$q_3^{-1}$},
		      qArrow/labelstyle={right, inner sep=1pt},
		      qArrow/label xshift=18pt,
		      qArrow/label yshift=2pt]
			(x0) to (x5);

		% \draw[qArrow/base, qSubArrow,
		%       qArrow/connector=bend,
		%       qArrow/bend=42,
		%       qArrow/label={$q_3^{-1}$},
		%       qArrow/labelstyle={right, inner sep=1pt},
		%       qArrow/label xshift=26pt,
		%       qArrow/label yshift=0pt]
		% 	(x0) to (x5);

        % Frozen/unfrozen arrows
		\draw[qArrow/base,
		      qArrow/label={$q_1$},
		      qArrow/labelstyle={above right, inner sep=1pt},
		      qArrow/label xshift=2pt,
		      qArrow/label yshift=2pt]
			(x1) to (x0);

		\draw[qArrow/base]
			(x5) to (x1);

		% Double opposing arrows between unfrozen nodes
		\draw[qArrowAbove]
			(x1) to (x2);

		\draw[qArrowBelowOp,
		      qArrow/label={$q_3$}]
			(x1) to (x2);

		\draw[qArrowAbove]
			(x2) to (x3);

		\draw[qArrowBelowOp,
		      qArrow/label={$q_3$}]
			(x2) to (x3);

		\draw[qArrowAbove]
			(x3) to (x4);

		\draw[qArrowBelowOp,
		      qArrow/label={$q_3$}]
			(x3) to (x4);

		% Loops
		\draw[qArrow/base, qSubArrow, qSubLoop,
		      qArrow/label={$q_1$}]
			(x2) to (x2);

		\draw[qArrow/base, qSubArrow, qSubLoop,
		      qArrow/label={$q_1$}]
			(x3) to (x3);

		\draw[qArrow/base, qSubArrow, qSubLoop,
		      qArrow/label={$q_1$}]
			(x4) to (x4);

		% Nodes
		\node[qNodeFrozen] at (x0) {$X^{(-)}_0$};
		\node[qNodeUnfrozen] at (x1) {$X_1$};
		\node[qNodeUnfrozen] at (x2) {$X_2$};
		\node[qNodeUnfrozen] at (x3) {$X_3$};
		\node[qNodeUnfrozen] at (x4) {$X_4$};
		\node[qNodeFrozen] at (x5) {$X^{(+)}_5$};
	\end{scope}
%Diagram 2--------------------------------------------------------------------------------------------------%
	\begin{scope}[shift={(0,-3*\ys)}]
		\node at (-1.35*\xs,0*\ys) {\normalsize $\CQ_{4}^{\mathrm{sub}(1)}$};

		% Coordinates
		\coordinate (x0) at (0*\xs,1*\ys);
		\coordinate (x1) at (1*\xs,0*\ys);
		\coordinate (x2) at (2*\xs,0*\ys);
		\coordinate (x3) at (3*\xs,0*\ys);
		\coordinate (x4) at (4*\xs,0*\ys);
		\coordinate (x5) at (0*\xs,-1*\ys);

		% Frozen-frozen arrows
		\draw[qArrow/base, qSubArrow, qArrow/xshift=-5pt]
			(x0) to (x5);

		% \draw[qArrow/base, qSubArrow,
		%       qArrow/connector=bend,
		%       qArrow/bend=42,
		%       qArrow/label={$q_3^{-1}$},
		%       qArrow/labelstyle={right, inner sep=1pt},
		%       qArrow/label xshift=26pt,
		%       qArrow/label yshift=2pt]
		% 	(x0) to (x5);

        \draw[qArrow/base, qSubArrow,
		      qArrow/xshift=5pt,
		      qArrow/label={$q_3^{-1}$},
		      qArrow/labelstyle={right, inner sep=1pt},
		      qArrow/label xshift=5pt,
		      qArrow/label yshift=2pt]
			(x0) to (x5);

		% Frozen/unfrozen arrows
		\draw[qArrow/base]
			(x0) to (x1);

		\draw[qArrow/base,
		      qArrow/label={$q_3$},
		      qArrow/label xshift=2pt,
		      qArrow/label yshift=-2pt]
			(x1) to (x5);

        \draw[qArrow/base, qFrozenCurve,
		      qArrow/control 1={($(x2)!0.35!(x0)+(0,0.16*\ys)$)},
		      qArrow/control 2={($(x2)!0.65!(x0)+(0,0.16*\ys)$)},
		      qArrow/label={$q^{-2}$},
		      qArrow/label xshift=10pt,
		      qArrow/label yshift=2pt]
			(x2) to (x0);

		\draw[qArrow/base, qFrozenCurve,
		      qArrow/control 1={($(x5)!0.35!(x2)+(0,-0.16*\ys)$)},
		      qArrow/control 2={($(x5)!0.65!(x2)+(0,-0.16*\ys)$)}]
			(x5) to (x2);

		% Double opposing arrows between unfrozen nodes
		\draw[qArrowAbove]
			(x1) to (x2);

		\draw[qArrowBelowOp,
		      qArrow/label={$q_1$}]
			(x1) to (x2);

		\draw[qArrowAbove]
			(x2) to (x3);

		\draw[qArrowBelowOp,
		      qArrow/label={$q_3$}]
			(x2) to (x3);

		\draw[qArrowAbove]
			(x3) to (x4);

		\draw[qArrowBelowOp,
		      qArrow/label={$q_3$}]
			(x3) to (x4);

		% Loops
		\draw[qArrow/base, qSubArrow, qSubLoop,
		      qArrow/label={$q_1$}]
			(x3) to (x3);

		\draw[qArrow/base, qSubArrow, qSubLoop,
		      qArrow/label={$q_1$}]
			(x4) to (x4);

		% Nodes
		\node[qNodeFrozen] at (x0) {$X^{(-)}_0$};
		\node[qNodeUnfrozen] at (x1) {$X_1$};
		\node[qNodeUnfrozen] at (x2) {$X_2$};
		\node[qNodeUnfrozen] at (x3) {$X_3$};
		\node[qNodeUnfrozen] at (x4) {$X_4$};
		\node[qNodeFrozen] at (x5) {$X^{(+)}_5$};
	\end{scope}
%Diagram 3--------------------------------------------------------------------------------------------------%
	\begin{scope}[shift={(0,-6*\ys)}]
		\node at (-1.35*\xs,0*\ys) {\normalsize $\CQ_{4}^{\mathrm{sub}(2)}$};

		% Coordinates
		\coordinate (x0) at (0*\xs,1*\ys);
		\coordinate (x1) at (1*\xs,0*\ys);
		\coordinate (x2) at (2*\xs,0*\ys);
		\coordinate (x3) at (3*\xs,0*\ys);
		\coordinate (x4) at (4*\xs,0*\ys);
		\coordinate (x5) at (0*\xs,-1*\ys);

		% Frozen-frozen arrows
		\draw[qArrow/base,
		      qArrow/label={$q_3^{-1}$},
		      qArrow/label xshift=12pt,
		      qArrow/label yshift=10pt]
			(x0) to (x5);

		% Frozen/unfrozen arrows
		\draw[qArrow/base, qFrozenCurve,
		      qArrow/control 1={($(x5)!0.35!(x3)+(0,-0.3*\ys)$)},
		      qArrow/control 2={($(x5)!0.65!(x3)+(0,-0.9*\ys)$)}]
			(x5) to (x3);
        
        \draw[qArrow/base, qFrozenCurve,
		      qArrow/control 1={($(x3)!0.35!(x0)+(0,0.9*\ys)$)},
		      qArrow/control 2={($(x3)!0.65!(x0)+(0,0.3*\ys)$)},
		      qArrow/label={$q^{-2}q_3$},
		      qArrow/label xshift=10pt,
		      qArrow/label yshift=2pt]
			(x3) to (x0);

		\draw[qArrow/base, qFrozenCurve,
		      qArrow/control 1={($(x0)!0.35!(x2)+(0,0.2*\ys)$)},
		      qArrow/control 2={($(x0)!0.65!(x2)+(0,0.4*\ys)$)},
		      qArrow/label={$q_3^{-1}$},
		      qArrow/label xshift=10pt,
		      qArrow/label yshift=0pt,
            qArrow/arrowpos=0.6]
			(x0) to (x2);

        \draw[qArrow/base, qFrozenCurve,
		      qArrow/control 1={($(x2)!0.35!(x5)+(0,-0.4*\ys)$)},
		      qArrow/control 2={($(x2)!0.65!(x5)+(0,-0.2*\ys)$)},
		      qArrow/label={$q_3$},
		      qArrow/labelstyle={above, inner sep=1pt},
		      qArrow/label xshift=10pt,
		      qArrow/label yshift=-10pt,
            qArrow/arrowpos=0.4]
			(x2) to (x5);

		% Double opposing arrows between unfrozen nodes
		\draw[qArrowAbove]
			(x1) to (x2);

		\draw[qArrowBelowOp,
		      qArrow/label={$q_3$}]
			(x1) to (x2);

		\draw[qArrowAbove]
			(x2) to (x3);

		\draw[qArrowBelowOp,
		      qArrow/label={$q_1$}]
			(x2) to (x3);

		\draw[qArrowAbove]
			(x3) to (x4);

		\draw[qArrowBelowOp,
		      qArrow/label={$q_3$}]
			(x3) to (x4);

		% Loops
		\draw[qArrow/base, qSubArrow, qSubLoop,
		      qArrow/label={$q_1$}]
			(x1) to (x1);

		\draw[qArrow/base, qSubArrow, qSubLoop,
		      qArrow/label={$q_1$}]
			(x4) to (x4);

		% Nodes
		\node[qNodeFrozen] at (x0) {$X^{(-)}_0$};
		\node[qNodeUnfrozen] at (x1) {$X_1$};
		\node[qNodeUnfrozen] at (x2) {$X_2$};
		\node[qNodeUnfrozen] at (x3) {$X_3$};
		\node[qNodeUnfrozen] at (x4) {$X_4$};
		\node[qNodeFrozen] at (x5) {$X^{(+)}_5$};
	\end{scope}
%Diagram 4--------------------------------------------------------------------------------------------------%
	\begin{scope}[shift={(0,-9*\ys)}]
		\node at (-1.35*\xs,0*\ys) {\normalsize $\CQ_{4}^{\mathrm{sub}(3)}$};

		% Coordinates
		\coordinate (x0) at (0*\xs,1*\ys);
		\coordinate (x1) at (1*\xs,0*\ys);
		\coordinate (x2) at (2*\xs,0*\ys);
		\coordinate (x3) at (3*\xs,0*\ys);
		\coordinate (x4) at (4*\xs,0*\ys);
		\coordinate (x5) at (0*\xs,-1*\ys);

		% Frozen/unfrozen arrows
		\draw[qArrow/base, qFrozenCurve,
		      qArrow/control 1={($(x0)!0.35!(x3)+(0,0.25*\ys)$)},
		      qArrow/control 2={($(x0)!0.65!(x3)+(0,0.8*\ys)$)},
		      qArrow/label={$q_3^{-2}$},
		      qArrow/label xshift=15pt,
		      qArrow/label yshift=-2pt,
            qArrow/arrowpos=0.75]
			(x0) to (x3);

		\draw[qArrow/base, qFrozenCurve,
		      qArrow/control 1={($(x3)!0.35!(x5)+(0,-0.8*\ys)$)},
		      qArrow/control 2={($(x3)!0.65!(x5)+(0,-0.25*\ys)$)},
		      qArrow/label={$q_3$},
		      qArrow/label xshift=15pt,
		      qArrow/label yshift=-8pt,
            qArrow/arrowpos=0.25]
			(x3) to (x5);

        %\coordinate (x5down) at ($(x5)+(0,-0.1)$);
        %\coordinate (x0up) at ($(x0)+(0,0.1)$);
        
		\draw[qArrow/base, qFrozenCurve,
		      qArrow/control 1={($(x5)!0.35!(x4)+(0,-0.3*\ys)$)},
		      qArrow/control 2={($(x5)!0.65!(x4)+(0,-0.95*\ys)$)},
            qArrow/arrowpos=0.7]
			%(x5down) to (x4);
            (x5) to (x4);

		\draw[qArrow/base, qFrozenCurve,
		      qArrow/control 1={($(x4)!0.35!(x0)+(0,0.95*\ys)$)},
		      qArrow/control 2={($(x4)!0.65!(x0)+(0,0.3*\ys)$)},
		      qArrow/label={$q^{-2}q_3^2$},
		      qArrow/label xshift=22pt,
		      qArrow/label yshift=0pt,
            qArrow/arrowpos=0.3]
			(x4) to (x0);

		% Double opposing arrows between unfrozen nodes
		\draw[qArrowAbove]
			(x1) to (x2);

		\draw[qArrowBelowOp,
		      qArrow/label={$q_3$}]
			(x1) to (x2);

		\draw[qArrowAbove]
			(x2) to (x3);

		\draw[qArrowBelowOp,
		      qArrow/label={$q_3$}]
			(x2) to (x3);

		\draw[qArrowAbove]
			(x3) to (x4);

		\draw[qArrowBelowOp,
		      qArrow/label={$q_1$}]
			(x3) to (x4);

		% Loops
		\draw[qArrow/base, qSubArrow, qSubLoop,
		      qArrow/label={$q_1$}]
			(x1) to (x1);

		\draw[qArrow/base, qSubArrow, qSubLoop,
		      qArrow/label={$q_1$}]
			(x2) to (x2);

		% Nodes
		\node[qNodeFrozen] at (x0) {$X^{(-)}_0$};
		\node[qNodeUnfrozen] at (x1) {$X_1$};
		\node[qNodeUnfrozen] at (x2) {$X_2$};
		\node[qNodeUnfrozen] at (x3) {$X_3$};
		\node[qNodeUnfrozen] at (x4) {$X_4$};
		\node[qNodeFrozen] at (x5) {$X^{(+)}_5$};
	\end{scope}
%Diagram 5--------------------------------------------------------------------------------------------------%
	\begin{scope}[shift={(0,-12*\ys)}]
		\node at (-1.35*\xs,0*\ys) {\normalsize $\CQ_{4}^{\mathrm{sub}(4)}$};

		% Coordinates
		\coordinate (x0) at (0*\xs,1*\ys);
		\coordinate (x1) at (1*\xs,0*\ys);
		\coordinate (x2) at (2*\xs,0*\ys);
		\coordinate (x3) at (3*\xs,0*\ys);
		\coordinate (x4) at (4*\xs,0*\ys);
		\coordinate (x5) at (0*\xs,-1*\ys);

		% Frozen-frozen arrow
		\draw[qArrow/base, qSubArrow,
		      qArrow/label={$q^{-2}q_3^2$},
		      qArrow/label xshift=18pt,
		      qArrow/label yshift=-8pt]
			(x5) to (x0);

        \draw[qArrow/base, qFrozenCurve,
		      qArrow/control 1={($(x0)!0.35!(x4)+(0,0.3*\ys)$)},
		      qArrow/control 2={($(x0)!0.65!(x4)+(0,0.95*\ys)$)},
		      qArrow/label={$q_3^{-3}$},
		      qArrow/label xshift=5pt,
		      qArrow/label yshift=3pt]
			(x0) to (x4);

		\draw[qArrow/base, qFrozenCurve,
		      qArrow/control 1={($(x4)!0.35!(x5)+(0,-0.95*\ys)$)},
		      qArrow/control 2={($(x4)!0.65!(x5)+(0,-0.3*\ys)$)},
		      qArrow/label={$q_3$},
		      qArrow/label xshift=10pt,
		      qArrow/label yshift=-12pt]
			(x4) to (x5);

		% Double opposing arrows between unfrozen nodes
		\draw[qArrowAbove]
			(x1) to (x2);

		\draw[qArrowBelowOp,
		      qArrow/label={$q_3$}]
			(x1) to (x2);

		\draw[qArrowAbove]
			(x2) to (x3);

		\draw[qArrowBelowOp,
		      qArrow/label={$q_3$}]
			(x2) to (x3);

		\draw[qArrowAbove]
			(x3) to (x4);

		\draw[qArrowBelowOp,
		      qArrow/label={$q_3$}]
			(x3) to (x4);

		% Loops
		\draw[qArrow/base, qSubArrow, qSubLoop,
		      qArrow/label={$q_1$}]
			(x1) to (x1);

		\draw[qArrow/base, qSubArrow, qSubLoop,
		      qArrow/label={$q_1$}]
			(x2) to (x2);

		\draw[qArrow/base, qSubArrow, qSubLoop,
		      qArrow/label={$q_1$}]
			(x3) to (x3);

		% Nodes
		\node[qNodeFrozen] at (x0) {$X^{(-)}_0$};
		\node[qNodeUnfrozen] at (x1) {$X_1$};
		\node[qNodeUnfrozen] at (x2) {$X_2$};
		\node[qNodeUnfrozen] at (x3) {$X_3$};
		\node[qNodeUnfrozen] at (x4) {$X_4$};
		\node[qNodeFrozen] at (x5) {$X^{(+)}_5$};
	\end{scope}
\end{tikzpicture}
    \caption{All subregular quivers for \(N=4\).}\label{fig:SubRegN4all}
\end{figure}
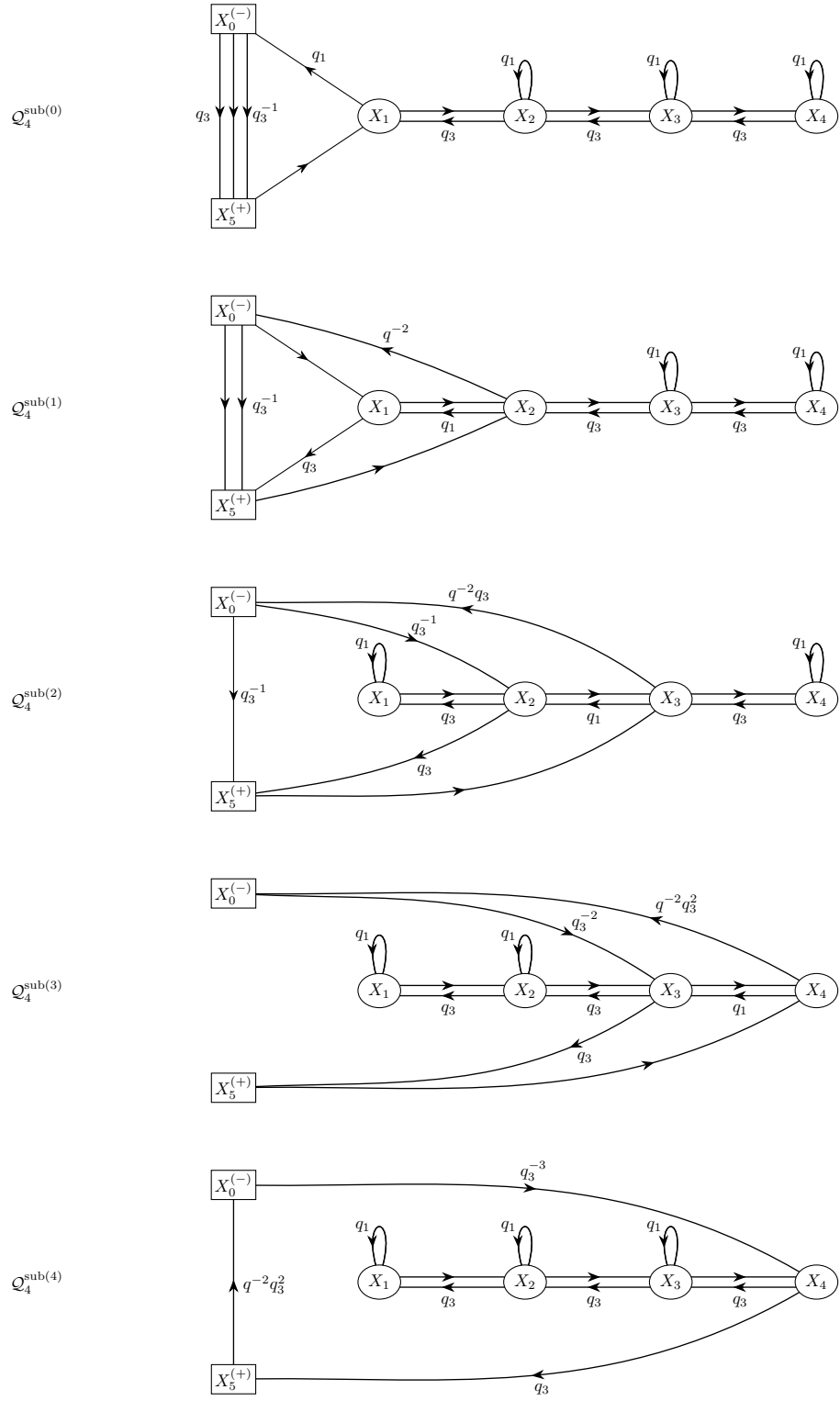
\clearpage

\begin{definition}
    Let $\CQ_N^{\text{sub}(0)}=\CQ_N^\text{sub}$. For $M=1\cdots N-1$, we define the quiver $\CQ_N^{\text{sub}(M)}$ as follows. The set of vertices consists of 
    \begin{itemize}
        \item two unfrozen vertices with no loops with index $i=M$ and $i=M+1$ respectively, 
        \item $N-2$ unfrozen vertices with a single loop of parameter $q_1$, indexed by $i=1\cdots M-1$ and $i=M+2\cdots N$,
        \item a frozen vertex of type $-$ with no loop indexed by $i=0$,
        \item a frozen vertex of type $+$ with no loop indexed by $i=N+1$.
    \end{itemize}
    The other arrows are given by
    \begin{itemize}
        \item between each pair of unfrozen vertices $(i,i+1)$ with $i=1\cdots M-1$ and $i=M+1\cdots N-1$, an arrow $i\to i+1$ of parameter one, and an opposite arrow $i+1\to i$ of parameter $q_3$,
        \item between the vertices $M$ and $M+1$, an arrow $M\to M+1$ of parameter $1$, and an opposite arrow of parameter $q_1$,
        \item between the frozen vertex $0$ and the vertex $M$, an arrow $0\to M$ of parameter $q_3^{1-M}$,
        \item between the frozen vertex $0$ and the vertex $M+1$, an arrow $M+1\to0$ of parameter $q^{-2}q_3^{M-1}$,
        \item between the frozen vertex $N+1$ and the vertex $M$, an arrow $M\to N+1$ of parameter $q_3$,
        \item between the frozen vertex $N+1$ and to the vertex $M+1$, an arrow $N+1\to M+1$ of parameter $1$,
        \item between the frozen vertices $0$ and $N+1$, $N-M-1$ plain arrows $0\to N+1$ of parameters $q_3^{1-M},q_3^{-M},\cdots,q_3^{3-N}$ (and there are no arrows if $M=N-1$).
    \end{itemize}
    Finally, we define the quiver $\CQ_N^{\text{sub}(N)}$ as follows. The set of vertices consists of
    \begin{itemize}
        \item an unfrozen vertex with no loop with index $i=N$, 
        \item $N-1$ unfrozen vertices with a single loop of parameter $q_1$, indexed by $i=1\cdots N-1$,
        \item a frozen vertex of type $-$ with no loop indexed by $i=0$,
        \item a frozen vertex of type $+$ with no loop indexed by $i=N+1$. 
    \end{itemize}
    The other arrows are given by
    \begin{itemize}
        \item between each pair of unfrozen vertices $(i,i+1)$ for $i=1\cdots N-1$, an arrow $i\to i+1$ of parameter one, and an opposite arrow $i+1\to i$ of parameter $q_3$,
        \item between the frozen vertex $0$ and the vertex $N$, an arrow $0\to N$ of parameter $q_3^{1-N}$,
        \item between the frozen vertex $N+1$ and the vertex $N$, an arrow $N\to N+1$ of parameter $q_3$,
        \item between the frozen vertices $0$ and $N+1$, a single plain arrow $N+1\to 0$ of parameter $q^{-2}q_3^{N-2}$.
    \end{itemize}
    These quivers have been represented in Figure \ref{fig:SubRegN4all} for $N=4$.
\end{definition}

\begin{definition}
    To a quiver $\CQ_N^{\text{sub}(M)}$ we associate the following currents. For $M\neq0,N$,
    \begin{align}
        G^+_M(z)&=:X_{N+1}^+(z)(1+X_{M+1}(z)\cdots X_N(z)T)(1+ X_{M+1}(z)\cdots X_{N-1}(z)T)\cdots (1+X_{M+1}(z)T)\cdot 1:\\
        G^-_M(z)&=:X_0^-(z)(1+X_1(z)X_2(q_3^{-1}z)\cdots X_M(q_3^{1-M}z)T)(1+X_2(q_3^{-1}z)\cdots X_M(q_3^{1-M}z)T)\cdots (1+X_M(q_3^{1-M}z)T)\cdot 1:,
    \end{align}
    for $M=0$, the currents $G_0^\pm(z)=G^\pm(z)$ defined in Definition \ref{def:Wsub_slN}, and for $M=N$, $G^+_N(z)=X_{N+1}^+(z)$ and 
    \begin{equation}
        G^-_N(z)=:X_0^-(z)(1+X_1(z)X_2(q_3^{-1}z)\cdots X_{N+1}(q_3^{-N}z)T)(1+X_2(q_3^{-1}z)\cdots X_{N+1}(q_3^{-N}z)T)\cdots (1+X_{N+1}(q_3^{-N}z)T)\cdot 1:.
    \end{equation}
    % We denote $\CF_{\CQ_N^{\text{sub}(M)}}[\CWqt^\text{sub}(\sl(N))]$ the corresponding free field realization.
\end{definition}
Here as above $T$ is the shift operator such that $Tz=q_3zT$, and $T.1=1$. 

% The following theorem states that algebras generated by \(G^-_M(z)\) for \(M=0,dots,N\) are related by mutations.

\begin{theorem}\label{th:subreg_mutations}
    % Free field realizations $\CF_{\CQ_N^{\text{sub}(M)}}[\CWqt^\text{sub}(\sl(N))]$ for different $M$ are related by mutations as follows:
    Chiral cluster seeds $\CQ_N^{\text{sub}(M)}$ and subalgebras in current algebras generated by \(G^\pm_M(z)\) are related by mutations as follows:
    \begin{itemize}
        \item For $M\neq0$, the mutation at the vertex $M$ of the quiver $\CQ_N^{\text{sub}(M)}$ (followed by gauging the vertex $M$ by $qq_3$) produces the quiver $\CQ_N^{\text{sub}(M-1)}$, and the currents $G^\pm_M(z)$ are mapped to the currents $G^\pm_{M-1}(z)$.
        \item For $M\neq N$, the mutation at the vertex $M+1$ of the quiver $\CQ_N^{\text{sub}(M)}$ (followed by gauging the vertex $M+1$ by $qq_1$) produces the quiver $\CQ_N^{\text{sub}(M+1)}$, and the currents $G^\pm_M(z)$ are mapped to the currents $G^\pm_{M+1}(z)$.
        % \item For $M\neq0$, the mutation at the vertex $M$ of the quiver $\CQ_N^{\text{sub}(M)}$ (followed by gauging the vertex $M$ by $qq_3$) produces the quiver  $\CQ_N^{\text{sub}(M-1)}$, and the currents $G^\pm_M(z)$ generating $\CF_{\CQ_N^{\text{sub}(M)}}[\CWqt^\text{sub}(\sl(N))]$ are mapped to the currents $G^\pm_{M-1}(z)$ generating $\CF_{\CQ_N^{\text{sub}(M-1)}}[\CWqt^\text{sub}(\sl(N))]$.
        % \item For $M\neq N$, the mutation at the vertex $M+1$ of the quiver $\CQ_N^{\text{sub}(M)}$ (followed by gauging the vertex $M+1$ by $qq_1$) produces the quiver  $\CQ_N^{\text{sub}(M+1)}$, and the currents $G^\pm_M(z)$ generating $\CF_{\CQ_N^{\text{sub}(M)}}[\CWqt^\text{sub}(\sl(N))]$ are mapped to the currents $G^\pm_{M+1}(z)$ generating $\CF_{\CQ_N^{\text{sub}(M+1)}}[\CWqt^\text{sub}(\sl(N))]$.
    \end{itemize}
\end{theorem}

\begin{proof}
    The proof can be found in Appendix \ref{AppD}.
\end{proof}

\begin{remark}
The free field realizations I and II of $U_q(\widehat{\sl(2)})$ and deformed BP considered in the previous section coincide with, respectively, the cases $M=1$ and $M=0$ and $N=2,3$ of this construction. In general, for arbitrary $N$, we expect the free field realization $\CQ^{\text{sub}(0)}$ to coincide with the construction given in \cite{FJM2022} as an extension of $\CWqt(\gl(1|N))$, and the realization $\CQ^{\text{sub}(1)}$ with Harada's construction \cite{Harada:2020woh}. A second free field realization is obtained in \cite[Sec. 3.2]{FJM2022} for $N$ even, based on the symmetric Dynkin diagram. It is expected to correspond to our realization for $M=N/2$.

% When $M=0$ and $M=N$, the quiver has only one vertex with no loop and so only one mutation is possible, it leads to the realizations with $M=1$ and $M=N-1$ respectively. For $M=0$, we have $G_0^\pm(z)=G^\pm(z)$.
\end{remark}

\begin{remark}
    In the conformal limit the \(N+1\) realizations of $\CWqt^{\text{sub}}(\sl(N))$ (depending on Dynkin diagram of \(\mathfrak{sl}(1|N)\)) should correspond to \(N+1\) realizations of $\CW^{\text{sub}}(\sl(N))$ found in \cite{FS04}. See Remark~\ref{rem:Gpm limit} below.
\end{remark}

\subsection{Deformed Inverse Quantum Hamiltonian Reduction}
In this subsection, we prove the embedding of the algebra \(\CWqt^\text{sub}(\sl(N))\) into the principal algebra $\CWqt(\sl(N))$ tensored with a rank two Heisenberg algebra $\CH_2$. 
This embedding is constructed in a particular free field realization of \(\CWqt^\text{sub}(\sl(N))\)  corresponding  to chiral cluster seed \(\CQ_N^{\text{sub}}=\CQ_N^{\text{sub}(0)}\).
% In this subsection, we prove the embedding of the algebra $\CFsub$ into the free field realization of the principal algebra $\CWqt(\sl(N))$ tensored with a rank two Heisenberg algebra $\CH_2$. We conjecture that this embedding can be lifted to the level of algebras, i.e. $\CWqt^{\text{sub}}(\gl(N))\hookrightarrow\CWqt(\sl(N))\otimes \CH_2$, and this has been checked explicitly for $N=2$. 
In the next section, we argue that this embedding is a $(\qf,\tf)$-deformed version of the inverse Hamiltonian reduction. In this respect, the Heisenberg algebra $\CH_2$ plays the role of the lattice algebra that extends the principal algebra.

\begin{theorem}\label{thm:IQQHR}
    There is an embedding of deformed W-algebras
    \begin{equation}
         \CWqt^\text{sub}(\sl(N))\hookrightarrow\CWqt(\sl(N))\otimes \CH_2,
    \end{equation}
    Here the algebra $\CWqt^\text{sub}(\sl(N))$ is given in the free field realization \(\CQ_N^{\text{sub}(0)}\), \(\CWqt(\sl(N))\) is given in its standard realization, and $\CH_2$ is a rank two Heisenberg algebra. In terms of generators we have 
    \begin{equation}
        G^-(z)\in\CH_2,\qquad G^+(z)=\sum_{k=0}^{N}A_k(z)W_k(q_3^{(k-1)/2}z),
    \end{equation}
    where $W_k(z)$ is the free field realization of the spin $k$ current of $\CWqt(\sl(N))$ (with $W_0(z)=W_N(z)=1$), and $A_k(z)\in \CH_2$. Under this embedding, higher generators satisfy $\CK_1(z),\CK_N^+(z)\in\CH_2$ and
    \begin{equation}\label{eqn:K2Embed}
        \CK_2(z)=B_0(z)+B_1(z)W_1(z),\qquad B_0(z),B_1(z)\in\CH_2.
    \end{equation}
    % In particular, $[W_k(z),A_l(w)]=[W_k(z),G^-(w)]=0$.
\end{theorem}

The current $\CK_2(z)$ is interpreted as the deformed analogue of the conformal vector $L(z)$.

% \begin{remark}
%    The algebra $\CH_2$ can be seen as a deformation of the free field realization of the $\b\g$-system. {\color{red}This deformation is not unique, as it depends on the rank $N$.} The rational limit of this system will be discussed in the next subsection. 
% \end{remark}

\begin{proof}
The key observation is that the current $G^+(z)$ can be written in the form of a Miura transform,
\begin{equation}
    G^+(z)=:X_{N+1}^+(z)(1+Y_1(z)T)(1+Y_2(z)T)\cdots(1+Y_N(z)T):1,
\end{equation}
involving the operators $Y_i(z)=:X_1(z)\cdots X_{N+1-i}(z):$ for $i=1\cdots N$. As a result, it can be decomposed into
\begin{equation}
    G^+(z)=\sum_{k=0}^NG_k^+(z),\quad G^+_k(z)=:X_{N+1}^+(z)\sum_{1\leq j_1<\cdots<j_k\leq N}:Y_{j_1}(z)Y_{j_2}(q_3z)\cdots Y_{j_k}(q_3^{k-1}z):,
\end{equation}
where each $G_k^+(z)$ corresponds to the term $O(T^k)$ of the Miura transform.

Next, we notice that the operators $Y_i(z)$ satisfy the OPEs,
\begin{align}\label{eqn:OPE_YY}
    &Y_i(z)Y_j(w)::
    \begin{cases}
        \dfrac{\vphi_1(q_3 z,w)}{\vphi_1(z,w)}, & i<j,\\
        (\vphi_1(z,w)\vphi_{-1}(z,w))^{-1}, & i=j,\\
        \dfrac{\vphi_{-1}(q_3^{-1}z,w)}{\vphi_{-1}(z,w)}, & i>j.
    \end{cases}
\end{align}
% \begin{equation}
%     Y_i(z)Y_j(w)::\vphi_1(q_1z,w)^{-1}\vphi_{-1}(q_1^{-1}z,w)^{-1}
%     \begin{cases}
%         \dfrac{\vphi_1(q_1z,w)}{\vphi_1(z,w)}, & i<j,\\
%         \dfrac{\vphi_1(q_1z,w)\vphi_{-1}(q_1^{-1}z,w)}{\vphi_1(z,w)\vphi_{-1}(z,w)}, & i=j,\\
%         \dfrac{\vphi_{-1}(q_1^{-1}z,w)}{\vphi_{-1}(z,w)}, & i>j.
%     \end{cases}
% \end{equation}
By comparison of these relations with those satisfied by the operators $\L_i(z)$ in section \ref{sec:WglN} (which can be deduced from the commutation relations \ref{eqn:comm_hk}), we find that the operators $Y_i(z)$ can be represented as the product $Y_i(z)\equiv U(z)\L_i(z)$ with
\begin{equation}
    U(z)=U :e^{\sum_{n\neq 0}z^{-n}u_n}:,\qquad [u_n,u_m]=-\dfrac{\d_{n+m}}{n}(1-q_2^n)(1-q_3^n)\dfrac{1-q_1^nq_3^{nN}}{1-q_3^{nN}},\quad (n>0),
\end{equation}
and $[U(z),\L_i(w)]=0$. In terms of modes, denoting $\l_{i,n}$ the modes of $\L_i(z)$, we find
\begin{equation}\label{eq:expr_un}
    y_{i,n}=\sum_{j=1}^{N+1-i}x_{j,n}=u_n+\l_{i,n}\implies u_n=\sum_{j=1}^N\dfrac{1-q_3^{-(N+1-j)n}}{1-q_3^{-Nn}}x_{j,n},
\end{equation}
where we used the fact that the sum of modes $q_3^{-in}\l_{i,n}$ is vanishing (see equ. \ref{eqn:qRegWBosonConstraints}). Note that the modes $u_n$ are precisely those introduced in Proposition \ref{Prop:NullVector}. As a consequence of this identification, we can write
\begin{equation}
    G^+_k(z)=:X_{N+1}^+(z)U_k(z)W_k(q_3^{(k-1)/2}z):,\qquad U_k(z)=:\prod_{j=1}^k U(q_3^{j-1}z):,
\end{equation}
where $W_k(z)$ are the spin $k$ currents of the algebra $\CWqt(\sl(N))$. Note that the zero modes also support this identification since the quantum variables $X_i$ for $i=2\cdots N$ are trivial. As a result, we find that the currents $W_k(z)$ do not carry any zero modes, as expected. The only non-trivial zero-mode is in the prefactor $U$, e.g. $U=Q_1$ under the choice of polarization \ref{eq:polarizaion_subreg}.
% it contains the zero modes $Q_1,Q_2,P_1,P_2$
Let $\CH_2$ be the algebra generated by the fields $X_0(z),U(z),X_{N+1}(z)$. By definition, $G^-(z)\in\CH_2$ and $A_k(z)=:X_{N+1}^+(z)U_k(z):\in\CH_2$. It remains to show that $[\CH_2,\CWqt(\sl(N))]=0$, and that $\CH_2$ is of rank two. 

Let's address the first point, we have 
\begin{equation}
    [u_n,\l_{i,m}]=[x_{0,n},\l_{i,m}]=[x_{N+1,n},\l_{i,m}]=0,\qquad i=1\cdots N,\quad n,m\in\mZ^\times.
\end{equation}
The first commutator vanishes by assumption in the decomposition $Y_i(z)=U(z)\L_i(z)$. The second and third commutators vanish as a consequence of the fact that both $x_{0,n}$ and $x_{N+1,n}$ commute with $x_{i,m}$ for $i=2\cdots N$. Let's focus on $x_{0,n}$, the argument is exactly the same for $x_{N+1,n}$. Inverting the relation between $\l_{i,n}$ and $x_{j,m}$, we find
\begin{equation}
    \l_{i,n}=-u_n+\sum_{j=1}^{N+1-i}x_{j,n}\implies [x_{0,n},\l_{i,m}]=-[x_{0,n},u_m]+[x_{0,n},x_{1,m}].
\end{equation}
Then, using the expression \ref{eq:expr_un} of $u_n$, we deduce that $[x_{0,n},u_m]=[x_{0,n},x_{1,m}]$ and so $[x_{0,n},\l_{i,m}]=0$.

Finally, let $C_\Pi$ be the deformed Cartan matrix encoding the commutation relations between the modes $(x_{0,n},u_n,x_{N+1,n})$. Explicitly,
\begin{equation}
    C_\Pi=(q-q^{-1})\pmat{q&-qq_1&\a_N\\q^{-1}q_1^{-1}&\g_N&-q\\-\a_N^\vee&q^{-1}&-q^{-1}},\qquad  \g_N=q(1-q_3)\dfrac{1-q_1q_3^N}{1-q_3^N}.
\end{equation}
It can be checked directly that this matrix is indeed of rank two, with the null vectors of $C_\Pi$ and $C_\Pi^T$ corresponding precisely to the expressions \ref{eq:vn_subreg} found for $v_n$ in the previous section.

The last part of the theorem follows from rewritting the higher generators in the form
\begin{align}\label{eqn:SubHigherKGen}
\begin{split}
    \CK_1(z)&=-(q-q^{-1})q^{N-2}\dfrac{(q_2^{-1}q_3;q_3)_{N-2}}{(q_3;q_3)_{N-2}}:G^-(z)A_0(q_3z):,\\
    \CK_2(z)&=-(q-q^{-1})q^{3-N}\dfrac{(q_2^{-1}q_3;q_3)_{N-3}}{(q_3;q_3)_{N-3}}:X_0^-(z)X_{N+1}^+(z)\left(q^{-1}\dfrac{1-q_3q_2}{1-q_3}+\sum_{j=1}^NY_j(z)\right):,\\
    &=-(q-q^{-1})q^{3-N}\dfrac{(q_2^{-1}q_3;q_3)_{N-3}}{(q_3;q_3)_{N-3}}\left(q^{-1}\dfrac{1-q_3q_2}{1-q_3}:G^-(z)A_0(z):+:G^-(z)A_1(z)\sum_{j=1}^N\L_j(z):\right),\\
    \CK_N^+(z)&=(q-q^{-1}):G^-(q_2^{-1}q_3^{N-2}z)X_{N+1}^+(z)\prod_{j=1}^{N}Y_j(q_3^{j-1}z):=(q-q^{-1}):G^-(q_2^{-1}q_3^{N-2}z)A_N(z):.
\end{split}
\end{align}
\end{proof}

\begin{remark}
    The OPEs of the dressing factors $A_k(z)$     can be expressed through the functions $f^{kl}(x)$ given in equation \ref{eq:def_fkl}, and entering in the quadratic relations of the algebra $\CWqt(\sl(N))$,
    \begin{equation}
        A_k(z)A_l(w)::
        \begin{cases}
            f^{kl}(q_3^{-(k-l)/2}w/z)\prod_{j=1}^{l-k-1}\frac{qz-q^{-1}q_3^jw}{z-q_3^jw},& k<l,\\
            f^{kk}(w/z)\dfrac{z-w}{qz-q^{-1}w},& k=l,\\
             f^{kl}(q_3^{-(k-l)/2}w/z)\dfrac{z-w}{qz-q^{-1}w}\prod_{j=1}^{k-l}\frac{q^{-1}z-qq_3^{1-j}w}{z-q_3^{1-j}w},& k>l.
        \end{cases}
    \end{equation}
%\cmt{Using these OPEs, it is possible to show that the Inverse Quantum Hamiltonian Reduction embedding extends to the level of algebras for $N=2$. Due to the nature of the embedding, it is sufficient to show that the relation $G^+G^+$ in \ref{eq:relGpmGpm} follows for the quadratic relations obeyed by the currents $W_i(z)$.}
\end{remark}

The Inverse Hamiltonian Reduction Theorem \ref{thm:IQQHR} can also be formulated for the quotient $\CWqt^{\text{sub}}(\sl(N))$.

\begin{theorem}
    There is an embedding of deformed W-algebras $\CWqt^{\text{sub}}(\sl(N))'\hookrightarrow\CWqt(\sl(N))\otimes \CH_1$, such that
    \begin{equation}
        \bar{G}^+(z)=\sum_{k=0}^{N}\bar{A}_k(z)W_k(q_3^{(k-1)/2}z),\qquad [W_k(z),\bar{A}_l(w)]=[W_k(z),\bar{G}^-(w)]=0,
    \end{equation}
    where $W_k(z)$ are the spin $k$ currents of $\CWqt(\sl(N))$ and $\bar{A}_k(z),\bar G^-(z)\in \CH_1$. 
\end{theorem}

\begin{proof}
    We have decomposed $\CH_2$ in Theorem \ref{thm:IQQHR} into $\CH_2\simeq \CH_1\otimes\CH_1^\text{U(1)}$ where $\CH_1$ is the algebra generated by $u_n$ and $Q_1,P_1$. Using the previous factorization, we have $G^-(z)=H^-(z)\bar G^-(z)$ and $A_k(z)=H^+(z)\bar A_k(z)$ with $\bar G^-(z),\bar A_k(z)\in\CH_1$.
\end{proof}

\subsection{Comparison with $\CWqt(\gl(1|N))$}\label{sec:FJM}
\begin{figure}[h]
\centering
\begin{tikzpicture}[scale=.8, transform shape, font=\small]
\coordinate (Y1) at (7,0);
\coordinate (Y0) at (10,0);
\coordinate (Y1b) at (13,0);
\coordinate (ymidBeforeb) at (14.5,0);
\coordinate (ymidb) at (15,0);
\coordinate (ymidAfterb) at (15.5,0);
\coordinate (YMbm1) at (17,0);
\coordinate (YMb) at (20,0);

%Horizontal arrows
\draw[qArrowAbove, qArrow/color=blue](Y1) to (Y0);
\draw[qArrowAbove, qArrow/color=blue](Y0) to (Y1b);
\draw[qArrowAbove, qArrow/color=blue, qArrow/label xshift=0pt, qArrow/label yshift=0pt, qArrow/arrowpos = 0.6](Y1b) to (ymidBeforeb);
\draw[qArrowAbove, qArrow/color=blue, qArrow/label xshift=0pt, qArrow/label yshift=0pt, qArrow/arrowpos = 0.4](ymidAfterb) to (YMbm1);
\draw[qArrowAbove, qArrow/color=blue](YMbm1) to (YMb);
\draw[qArrowBelow, qArrow/label={$q_3$}, qArrow/color=red](YMb) to[] (YMbm1);
\draw[qArrowBelow, qArrow/label={$q_3$}, qArrow/color=red, qArrow/label xshift=0pt, qArrow/label yshift=-4pt, qArrow/arrowpos = 0.6](YMbm1) to[] (ymidAfterb);
\draw[qArrowBelow, qArrow/label={$q_3$}, qArrow/color=red, qArrow/label xshift=0pt, qArrow/label yshift=-4pt, qArrow/arrowpos = 0.4](ymidBeforeb) to[] (Y1b);
\draw[qArrowBelow, qArrow/label={$q_3$}, qArrow/color=red](Y1b) to[] (Y0);
\draw[qArrowBelow, qArrow/label={$q_1$}, qArrow/color=red](Y0) to[] (Y1);
% %Loops
\draw[qLoopDashed, qArrow/label={$q_3$, $`-'$}](Y1) to (Y1);
\draw[qLoop, qArrow/label={$q_1$}](Y1b) to (Y1b);
\draw[qLoop, qArrow/label={$q_1$}](YMbm1) to (YMbm1);
\draw[qLoopDashed, qArrow/label={$q_1$, `$+$'}](YMb) to (YMb);
%Nodes
\node[qNodeFrozen] (Y1) at (Y1) {$\tilde{X}_{0}^-$};
\node[qNodeUnfrozen] (Y0) at (Y0) {$\tilde{X}_{1}$};
\node[qNodeUnfrozen] (Y1b) at (Y1b) {$\tilde{X}_2$};
\node (ymidb) at (ymidb) {$\cdots$};
\node[qNodeUnfrozen] (YMbm1) at (YMbm1) {$\tilde{X}_N$};
\node[qNodeFrozen] (YMb) at (YMb) {$\tilde{X}_{N+1}^{+}$};
\end{tikzpicture}
	\caption{%
		Quiver for the standard realization of $\CWqt(\gl(1|N))$. 
    }\label{fig:qRegWAlgQuiverN=1}
\end{figure}
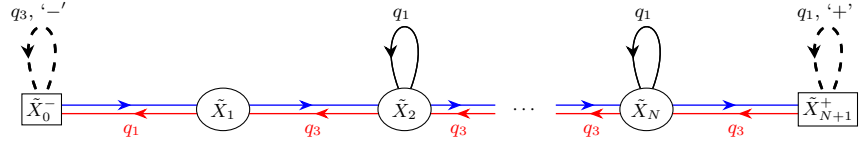

In \cite{FJM2022}, a free field realization for the deformed W-algebra $\CWqt^\text{sub}(\sl(N))$ is constructed as an extension of the free field realization of $\CWqt(\gl(1|N))$. In this construction, the current $G^+(z)$ is obtained using the qq-character method, i.e. promoting the q-character of $U_q(\widehat{\sl(1|N)})$ to a sum of vertex operators \cite{Kimura:2015rgi,FJMV2020}. The free field realization obtained in this way is expected to coincide with the one given here in section 4.1, and this was checked explicitly for $N=2,3,4$ by comparing the expression of the current $G^+(z)$ to the q-character.

In this subsection, we performed a detailed comparison of the algebras $\CWqt^\text{sub}(\sl(N))$ and $\CWqt(\gl(1|N))$ in free field realizations corresponding to chiral seeds $\CQ_N^\text{sub}$ and \(\CQ_{1|N}\) respectively. First, we produce an isomorphism of the algebras $\CA_\CQ$ associated with the chiral cluster seeds. Then, we show that the spin $N$ current of $\CWqt(\gl(1|N))$ and the product of generating currents $G^-(z)$, $G^+(z)$ are related by a simple factor.

\subsubsection{Comparing the algebras $\CA_\CQ$}
To compare chiral cluster seeds, we need to find a good basis of generators with respect to the matrices $C$ encoding commutation relations. In the case of $\CA_{\CQ_N^\text{sub}}$, the generators $(v_n,h_n,x_{i,n})$ with $i=1\cdots N$ provide such a good basis. Our first task is to define a similar basis for the algebra associated to the quiver $\CQ_{1|N}$ represented in Figure \ref{fig:qRegWAlgQuiverN=1}. To distinguish the matrices and operators associated with the two different chiral cluster seeds, we add a tilde on quantities pertaining to the algebra $\CWqt(\gl(1|N))$. Hence, the matrix $\tilde{C}$ entering in the definition of the seed is 
\begin{equation}
    \tilde{C}=(q-q^{-1})\pmat{c^\vee & c & 0 && \cdots && 0\\
    -c^\vee & q-q^{-1} & b & 0 &\cdots && 0\\
    0 & -b^\vee & b^\vee-b & b & 0 &\cdots &0\\
    & &\ddots & \ddots & \ddots &  &\\
    0 &\cdots &  &-b^\vee & b^\vee-b& b &0 \\
    0 &\cdots & &0 &-b^\vee & b^\vee-b& b \\
    0 & \cdots & && & -b^\vee & -b}
\end{equation}
where we denoted $c=q^{-1}(1-q_3^{-1})$, and $b$ is given in equ. \ref{eq:def_ab_aN}.
\begin{proposition}
    The operators% $\tilde{X}_i$ and
    \begin{equation}
        \tilde{v}_n=\sum_{i=0}^{N+1}\tilde{x}_{i,n},\qquad \tilde{v}_{-n}=q_1^nq_3^{Nn}\tilde{x}_{0,-n}+\sum_{i=1}^{N+1} q_3^{(N+1-i)n}\tilde{x}_{i,-n},
    \end{equation}
    are in the center of the algebra $\CA_{\CQ_{1|N}}$. Moreover, the operators
    \begin{equation}
        \tilde{h}_n=q_1^{-n}q_3^{-Nn}\tilde{x}_{0,n}+\sum_{i=1}^{N+1}q_3^{-(N+1-i)n}\tilde{x}_{i,n},\qquad \tilde{h}_{-n}=\sum_{i=0}^{N+1}\tilde{x}_{i,-n}
    \end{equation}
    commute with the modes associated with unfrozen vertices, i.e. $[\tilde{h}_n,\tilde{x}_{i,m}]=0$ for $i=1\cdots N$.
\end{proposition}

\begin{proof} The proof is by a direct calculation.

% These operators are obtained as $\tilde{h}_{\pm n}=(\vb{\tilde{h}}_\pm^{[n]})^T.\vb{\tilde{x}}_{\pm n}$, $\tilde{v}_{\pm n}=(\vb{\tilde{v}}_\pm^{[n]})^T.\vb{\tilde{x}}_{\pm n}$ from the vectors
% \begin{equation}
%     \vb{\tilde{x}}_n=\pmat{\tilde{x}_{0,n}\\\tilde{x}_{1,n}\\\vdots\\\tilde{x}_{N+1,n}},\qquad
%     \vb{\tilde{v}}_{-}=\pmat{q_1q_3^{N}\\q_3^N\\q_3^{N-1}\\\vdots\\q_3\\1},\qquad
%     \vb{\tilde{v}}_+=\pmat{1\\1\\\vdots\\1},\qquad
%     \vb{\tilde{h}}_-=\pmat{1\\1\\\vdots\\1},\qquad
%     \vb{\tilde{h}}_+=\pmat{(q_1q_3^{N})^{-1}\\q_3^{-N}\\q_3^{-(N-1)}\\\vdots\\q_3^{-1}\\1}.
% \end{equation}
% The commutation relations are a direct consequence of the properties of these vectors,
% \begin{equation}
%     \tilde{C}\vb{\tilde{v}}_-=\tilde{C}^T\vb{\tilde{v}}_+=\vb{0},\qquad \tilde{C}\vb{\tilde{h}}_-=(1-q_1)(1-q_2)(1-q_3)\pmat{-1\\0\\\vdots\\0\\1},\qquad  C^T\vb{\tilde{h}}_+=(1-q_1)(1-q_2)(1-q_3)\pmat{-(q_1q_3^N)^{-1}\\0\\\vdots\\0\\1}. 
% \end{equation}    
\end{proof}

\begin{remark}
    The operators $\tilde{v}_n$ can be set to zero, while the operators $\tilde{h}_n$ can be identified with the $U(1)$-factors. Indeed, the following products associated with paths following blue (resp. red) arrows read
    \begin{equation}
        \Psi[\tilde{X}_0^-,\tilde{X}_1,\cdots,\tilde{X}_{N+1}^+](z)=e^{\sum_{n>0}z^n \tilde{h}_{-n}}e^{\sum_{n>0}z^{-n}\tilde{v}_n},\qquad \Psi[\tilde{X}_{N+1}^+,\tilde{X}_N,\cdots,\tilde{X}_{0}^-](z)=e^{\sum_{n>0}z^n \tilde{v}_{-n}}e^{\sum_{n>0}z^{-n}\tilde{h}_n}.
    \end{equation}
\end{remark}

It is now possible to compare the algebras associated with the two chiral cluster seeds.
\begin{lemma}\label{lemma:Apm}
    There exists invertible matrices $A_\pm$ of size $(N+2)\times(N+2)$, defining an isomorphism of Heisenberg algebras
    \begin{equation}
        \tilde{x}_{i,\pm n}=\sum_{j=0}^{N+1}A_{\pm,i,j}^{[n]}x_{j,\pm n},
    \end{equation}
    such that the linear transformation
    \begin{itemize}
        \item is an algebra homomorphism, i.e. $\tilde{C}=A_+CA_-^T$,
        \item preserves the subquiver formed by unfrozen vertices, i.e. sends  $x_{i,n}\to\tilde{x}_{i,n}$ for $i=1\cdots N$,        
        \item maps central elements $v_{\pm n}$ to central elements $(r_\pm^{[n]})^{-1} \tilde{v}_{\pm n}$ for some parameters $r_\pm\in\mC$,
        \item maps operators $h_{\pm n}$ to $(s_\pm^{[n]})^{-1}\tilde{h}_{\pm n}$, for some parameters $s_\pm\in\mC$,
    \end{itemize}
    provided that
    \begin{equation}\label{eq:cond_lemma}
        s_+s_-=
        \begin{cases}
        \dfrac{(1-q_1)(1-q_3)}{q_1q_3^2(1+q_2)}, & N=2,\\
        -\dfrac{(1-q_1)(1-q_3)^2}{q_1q_3^2(1-q_3^{N-2})},& N>2.%s_+s_-=\dfrac{(1-q_1)(1-q_3)}{(1+q_2)q_1q_3^2}.
        \end{cases}
    \end{equation}
\end{lemma}

\begin{proof}
The proof of this lemma can be formulated as a simple linear algebra problem. It can be found in Appendix \ref{AppC}.
\end{proof}

\subsubsection{Comparison of the currents}
In this subsection, we denote the currents of spin $k\in\mZ^{>0}$ of the algebra $\CWqt(\gl(1|N))$ as $\tilde{W}_k(z)$. They have been obtained using the q-Miura transform in \cite{Harada:2021xnm},
\begin{equation}
    :\left(\sum_{k=0}^\infty\a_k:\prod_{j=1}^k\tilde{X}_0^-(q_3^{1-j}z):\ T^{-k}\right)(1+\tilde{Y}_1(z)T^{-1})(1+\tilde{Y}_2(z)T^{-1})\cdots (1+\tilde{Y}_N(z)T^{-1}):=\sum_{k=0}^\infty \tW_k(z) T^{-k},
\end{equation}
where $\tilde{Y}_i(z)=:\tilde{X}_0^-(z)\tilde{X}_1(z)\cdots \tilde X_{i}(z):$ for $i=1\cdots N$, and 
\begin{equation}
    \a_k=\prod_{j=1}^k\dfrac{s_3^{j-1}s_1^{-1}-s_3^{1-j}s_1}{s_3^j-s_3^{-j}}.
\end{equation}
In particular, the spin one current takes the form of a twisted telescoping sum corresponding to the coloring vector $\bsc=(1,3,\cdots,3),$
\begin{equation}
    \tW_1(z)=\S_{\a_1}^{(\bsc)}[\tX_0^-(z),\tX_1(z),\cdots,\tX_N(z)],\qquad \a_1=-t_1/t_3. %    \a_1\tilde{X}_0^-(z)+\sum_{i=1}^N:\tilde{X}_0^-(z)\tilde{X}_1(z)\cdots\tilde{X}_i(z):
\end{equation}
The following proposition relates the current of spin $N$ with the generator $G^+(z)$ in $\CWqt^{\text{sub}}(\sl(N))$.

\begin{proposition}\label{Prop:GpWN}
    Let $G^\pm(z)$ be the generating currents of $\CWqt^{\text{sub}}(\sl(N))$ in the free field realizations $\CQ_N^\text{sub}$, and $\tilde{W}_N(z)$ the spin $N$ current of $\CWqt(\gl(1|N))$ in the free field realization \(\CQ_{1|N}\). Under the identification of generators of Lemma \ref{lemma:Apm}, the currents are related by
    \begin{equation}
        G^-(q_3^{N-1}z)G^+(z)=H_1(z)\tilde{W}_N(q_3^{N-1}z),\qquad \CK_2(z)=H_2(z)\tilde{W}_1(z),
    \end{equation}
    where $H_i(z)$ for $i=1,2$ are $U(1)$ factors depending on the choice of parameters $s_\pm$.
\end{proposition}

In Proposition \ref{Prop:GpWN}, the relation between currents is established in the free field realizations $\CQ_N^\text{sub}$ and $\CQ_{1|N}$. By mutation, a similar result holds for any realization, provided we choose the quivers $\CQ_N^{\text{sub}(M)}$ and $\CQ_{1|N}^\bsc$ corresponding to the same underlying $\sl(1|N)$ Dynkin diagram.

\begin{remark}\label{rem:H}
    In general, $H_i(z)$ does not commute with $G^\pm(w)$ and $\tilde{W}_N(w)$. Under a specific choice of parameters $s_\pm$, the vertex operator $H_1(z)$ reduces to a half current. For instance, for $N>2$ taking\footnote{For $N=2$, we can take
    \begin{equation}\label{eq:spec_spm_N=2}
        s_-=q_3\dfrac{1-q_1}{1+q_3},\qquad s_+=q_1^{-1}q_3^{-3}\dfrac{1-q_3^{2}}{1+q_2},
    \end{equation}
    and find
    \begin{equation}
        H_1(z)=:Q_2^{2(N-1)}P_2:\exp\left(\sum_{n>0}z^{-n}\frac{1-q_3^{-2n}}{1+q_2^{n}}h_n\right).
    \end{equation}}
    \begin{equation}\label{eq:spec_spm}
        s_-=q_3^{N-1}\dfrac{(1-q_1)(1-q_3)}{(1-q_3^N)},\qquad s_+=-q_1^{-1}q_3^{-N-1}\dfrac{(1-q_3)(1-q_3^{N})}{1-q_3^{N-2}},
    \end{equation}
    we find
    \begin{equation}
        H_1(z)=:Q_2^{2(N-1)}P_2:\exp\left(\sum_{n>0}z^{-n}q_3^n\frac{1-q_3^{n}+q_3^{-n}-q_3^{-Nn}}{1-q_3^{(N-2)n}}h_n\right).
    \end{equation}
\end{remark}

\begin{proof}
Under the transformation of Lemma \ref{lemma:Apm}, the vertex operators at unfrozen vertices should be identified, i.e. $\tilde{X}_i(z)\equiv X_i(z)$ for $i=1\cdots N$ (in agreement with the zero modes algebra). Comparing the quiver, we note that the vertex operator at the vertex $\tilde{X}_0^-$ can be decomposed as follows,
\begin{equation}\label{eq:decomp_X0-}
    \tilde{X}_0^-(z)=A(z)X_{N+1}^+(z),\qquad A(z)X_{N+1}^+(w)::1, \qquad X_{N+1}^+(z)A(w)::1,%[A(z),\tilde{X}_{N+1}^+(w)]=0,
\end{equation}
with the OPEs,
\begin{equation}
    A(z)A(w)::q^{-1}\dfrac{\vphi_1(z,w)}{\vphi_{-1}(z,w)}\vphi_{1}(q_1z,w)^{-1},\qquad A(z)X_1(w)::\vphi_{-1}(q_1^{-1}z,w),\qquad X_1(z)A(w)::\vphi_1(q_1z,w),
\end{equation}
and trivial OPEs between $A(z)$ and $X_i(w)$ for $i=2\cdots N$. Expanding the q-Miura transform, we find
\begin{equation}
    \tW_k(q_3^{k-1}z)=\sum_{l=0}^{k}\a_{k-l}:\prod_{j=1}^{k-l}\tilde{X}_0^-(q_3^{k-j}z)\sum_{1\leq j_1<\cdots<j_l\leq N}\tilde{Y}_{j_1}(q_3^{l-1}z)\tilde{Y}_{j_2}(q_3^{l-2}z)\cdots\tilde{Y}_{j_l}(z):.
\end{equation}
The identification of unfrozen vertices implies that $\tilde{Y}_j(z)=\tilde{X}_0^-(z)Y_{N+1-j}(z)$, and so\footnote{We use
\begin{align}
    \begin{split}
        \sum_{1\leq j_1<\cdots<j_l\leq N}:Y_{N+1-j_1}(q_3^{l-1}z)Y_{N+1-j_2}(q_3^{l-2}z)\cdots Y_{N+1-j_l}(z):&=\sum_{1\leq j_l<\cdots<j_1\leq N}:Y_{j_1}(q_3^{l-1}z)Y_{j_2}(q_3^{l-2}z)\cdots Y_{j_l}(z):\\
        &=\sum_{1\leq j_1<\cdots<j_l\leq N}:Y_{j_1}(z)Y_{j_2}(q_3z)\cdots Y_{j_l}(q_3^{l-1}z):.
    \end{split}
\end{align}}
\begin{equation}
    \tW_k(q_3^{k-1}z)=:\prod_{j=1}^{k}\tilde{X}_0^-(q_3^{j-1}z)\left(\sum_{l=0}^{k}\a_{k-l}\sum_{1\leq j_1<\cdots<j_l\leq N}Y_{j_1}(z)Y_{j_2}(q_3z)\cdots Y_{j_l}(q_3^{l-1}z)\right):.
\end{equation}
Introducing the decomposition \ref{eq:decomp_X0-} of $\tilde{X}_0^-(z)$, we can rewrite this current as
\begin{align}
    &\tW_k(q_3^{k-1}z)=:B_k(z)\left(\sum_{l=0}^{k}\a_{k-l}\sum_{1\leq j_1<\cdots<j_l\leq N}X_{N+1}^+(z)Y_{j_1}(z)Y_{j_2}(q_3z)\cdots Y_{j_l}(q_3^{l-1}z)\right):,\\
    \text{with}\qquad &B_k(z)=:A(z)\prod_{j=2}^{k}\tilde{X}_0^-(q_3^{j-1}z):.
\end{align}
It remains to properly factorize the operator $B_k(z)$. For this purpose, we examine the OPEs,
\begin{equation}
    B_k(z)Y_j(w)::\vphi_1(q_3^k z,w)^{-1},\qquad B_k(z)X_{N+1}^+(w)::\prod_{j=2}^kq\vphi_1(q_3^{j-1}z,w)^{-1},
\end{equation}
and so,
\begin{align}
    \begin{split}
        &B_k(z)Z_l(w)::q^{k-1}\prod_{j=1}^{k-1}\vphi_1(q_3^jz,w)^{-1}\times\prod_{j=k-l+1}^k\vphi_1(q_3^jz,w)^{-1},\\
        \text{for}\quad &Z_l(z)=\sum_{1\leq j_1<\cdots<j_l\leq N}:X_{N+1}^+(z)Y_{j_1}(z)Y_{j_2}(q_3z)\cdots Y_{j_l}(q_3^{l-1}z):.    
    \end{split}
\end{align}
Taking the limit $w\to z$, we find
\begin{equation}
    B_k(z)Z_l(z)::q^{k-1}\dfrac{\a_{k-l}}{\a_k\a_{k-1}},
\end{equation}
and so
\begin{equation}
    \tilde{W}_k(q_3^{k-1}z)=q^{1-k}\a_k\a_{k-1}B_k(z)\left(\sum_{l=0}^{k}\sum_{1\leq j_1<\cdots<j_l\leq N}:X_{N+1}^+(z)Y_{j_1}(z)Y_{j_2}(q_3z)\cdots Y_{j_l}(q_3^{l-1}z):\right).
\end{equation}
On the other hand, the current $G^+(z)$ reads
\begin{equation}
    G^+(z)=\sum_{l=0}^{N}\sum_{1\leq j_1<\cdots<j_l\leq N}:X_{N+1}^+(z)Y_{j_1}(z)Y_{j_2}(q_3z)\cdots Y_{j_l}(q_3^{l-1}z):.
\end{equation}
Thus, we deduce the following relation between the currents of the two algebras,
\begin{equation}
    \tilde{W}_N(q_3^{N-1}z)=C(z)G^+(z),\qquad C(z)=q^{1-N}\a_N\a_{N-1}B_N(z).
\end{equation}

It remains to identify the current $C(z)$. First, we note that
\begin{equation}
    \tX_0^-(z)X_{N+1}^+(w)::q\vphi_1(z,w)^{-1}\implies B_k(z)X_{N+1}^+(w)::q^{k-1}\prod_{j=1}^{k-1}\vphi_1(q_3^jz,w)^{-1},
\end{equation}
from which we deduce the relation
\begin{equation}\label{eqn:tW_CG}
    C(z)X_{N+1}^+(z)=\a_N:\prod_{j=1}^N\tX_0^-(q_3^{j-1}z):.
\end{equation}
Let's introduce the modes decomposition of the vertex operator $C(z)$, the previous relation implies
\begin{equation}
    C(z)=C:e^{\sum_{n\neq 0}z^{-n}c_n}:\implies c_n=\frac{1-q_3^{-nN}}{1-q_3^{-n}}\tilde{x}_{0,n}-x_{N+1,n}.
\end{equation}
Let's assume $N>2$, and take $v_n=\tilde{v}_n=0$ for simplicity. Then we can write for $n>0$,\footnote{For $N=2$, the conclusion is the same, but we need to use instead
\begin{align}
    (1-q_1^{-n}q_3^{-2n})x_{0,n}&=q_3^{-n}\bar u_n-q_3^{-n}h_n,& (1-q_1^nq_3^{2n})x_{0,-n}&=q_3^n\bar u_{-n}+h_{-n},\\
    (1-q_1^{-n}q_3^{-2n})x_{3,n}&=-\bar u_n+h_n,&\qquad (1-q_1^nq_3^{2n})x_{3,-n}&=-\bar u_{-n}-q_1^nq_3^{n}h_n.
\end{align}}
\begin{align}
    (1-q_1^{-n}q_3^{-nN})x_{0,n}&=q_3^{-n}\bar u_n+\b_{N+1}^{[n]}h_n,& (1-q_1^nq_3^{nN})x_{0,-n}&=q_3^n\bar u_{-n}+h_{-n},\nonumber\\
    (1-q_1^{-n}q_3^{-nN})x_{N+1,n}&=-\bar u_n-\b_0^{[n]}h_n,&\qquad (1-q_1^nq_3^{nN})x_{N+1,-n}&=-\bar u_{-n}-q_1^nq_3^{(N-1)n}h_n,\nonumber\\
    (1-q_1^{-n}q_3^{-nN})\tilde{x}_{0,n}&=-(1-q_3^{-n})\bar u_n-\tilde{h}_n,& (1-q_1^nq_3^{nN})\tilde{x}_{0,-n}&=-(1-q_3^n)\bar u_{-n}+\tilde{h}_{-n},
\end{align}
with $\b_{N+1}=q_1^{-1}q_3^{1-N}\b_0^\vee$ and
\begin{equation}
    \bar u_n=\sum_{i=1}^N\frac{1-q_3^{-(N+1-i)n}}{1-q_3^{-n}}x_{i,n}=\sum_{i=1}^N\frac{1-q_3^{-(N+1-i)n}}{1-q_3^{-n}}\tilde{x}_{i,n}.
\end{equation}
We deduce
\begin{align}
    &(1-q_1^{-n}q_3^{-nN})(c_n-q_3^{-(N-1)n}x_{0,n})=(\b_0^{[n]}-q_3^{-(N-1)n}\b_{N+1}^{[n]})h_n-\dfrac{1-q_3^{-nN}}{1-q_3^{-n}}\tilde{h}_n,\\
    &(1-q_1^nq_3^{nN})(c_{-n}-q_3^{(N-1)n}x_{0,-n})=-q_3^{(N-1)n}(1-q_1^n)h_{-n}+\dfrac{1-q_3^{nN}}{1-q_3^n}\tilde{h}_{-n}.
\end{align}
Under identification of modes given in Lemma \ref{lemma:Apm}, we have $\tilde{h}_{\pm n}=s_\pm^{[n]}h_{\pm n}$, and so we have the decomposition
\begin{equation}
    C(z)=H_1(z)^{-1}X_0^-(q_3^{N-1}z),
\end{equation}
where $H_1(z)$ is expressed in terms of modes $h_n$ and $Q_2,P_2$ (in fact, $H_1\propto Q_2^{2(N-1)}P_2$). Moreover, under the specialization \ref{eq:spec_spm}, we find $c_{-n}=q_3^{(N-1)n}x_{0,-n}$ and
\begin{equation}
    c_n=q_3^{-(N-1)n}x_{0,n}+q_3^n\frac{1-q_3^{-n}-q_3^{-2n}+q_3^{-(N+1)n}}{1-q_3^{(N-2)n}}h_n,
\end{equation}
from which we deduce the expression of $H_1(z)$ given in remark~\ref{rem:H}.

Finally, we turn to the current $\CK_2(z)$ and compare the expression \ref{eq:subreg_K2} with the one of $\tW_1(z)$, assuming the identification of vertex operators at unfrozen vertices,
\begin{align}
    \CK_2(z)&=-(q-q^{-1})q^{3-N}\dfrac{(q_2^{-1}q_3;q_3)_{N-3}}{(q_3;q_3)_{N-3}}:X_0^-(z)X_{N+1}^+(z)\left(\a_1+\S[X_1,X_2,\cdots,X_N](z)\right):,\\
    \tW_1(z)&=:A(z)X_{N+1}^+(z)\left(\a_1+\S[X_1,X_2,\cdots,X_N](z)\right):.
\end{align}
It suggests to introduce the vertex operator $H_2(z)=\k_2 X_0^-(z)A(z)^{-1}$ where $\k_2$ is constant. Then, we have $[H_2(z),X_i(z)]=0$ for $i=1\cdots N$, and so it is possible to fix $\k_2$ such that $\CK_2(z)=H_2(z)\tW_1(z)$. For $N>2$, we can use the previous identification to determine the modes $h_{2,n}=x_{0,n}+x_{N+1,n}-\tilde{x}_{0,n}$ of $H_2(z)$, and show that they indeed correspond to a $U(1)$ factor, i.e. $h_{2,n}\propto h_n$.
\end{proof}

\section{Inverse quantum Hamiltonian reduction in the conformal limit} \label{sec:conformal limit}

The goal of this section is to analyse the results of the previous section in the conformal limit.
In \zcref{sec:BosGhost}, we consider the quiver associated to the free field realisation (FMS embedding) of the deformed bosonic ghost vertex algebra.
Although this is not a subregular W-algebra, it may still be interpreted as the $N=1$ case of the quiver $\CQ_N^\text{sub}$. We show that the currents $G^\pm(z)$ associated with this quiver reproduce the bosonic ghost fields $\b(z),\g(z)$ in the conformal limit.
In \zcref{sec:IQQHR_conformalLim}, we review how inverse reduction embeddings are constructed between principal and subregular W-algebras in the conformal setting, following \cite{FehSub23}.
Then, we show that the conformal limit of the deformed inverse reduction embedding from \zcref{thm:IQQHR} recovers the inverse reduction embeddings for the conformal W-algebras as given in, for instance \cite{Semikhatov1994,Adamovic2019, FehSub23}.

\subsection{Bosonic ghosts and FMS bosonisation}\label{sec:BosGhost}
In this subsection, we consider the q-deformation of the bosonic ghost vertex algebra $\bgvoa$. Recall that the vertex algebra $\bgvoa$ is strongly generated by the bosonic fields $\beta(z)$ and $\gamma(z)$ whose only non-regular operator product expansion is
\begin{equation}
    \beta(z) \gamma(w) \sim \frac{-\wun(w)}{z-w}.
\end{equation}
To compare with the q-deformed construction, we need to introduce the free field embedding of this vertex algebra into the so-called half-lattice algebra.

\paragraph{Half-lattice algebra.}
Consider the Heisenberg vertex algebra $\mathsf{H}_{\lvoa}$ associated to the vector space $\alg{h} = \cspn \set{ a, b }$ equipped with a symmetric bilinear form $\bilin{\cdot}{\cdot}$ so that $\bilin{a}{b}=0$ and $-\bilin{a}{a} = \bilin{b}{b} = \ell_N(\KK)$.
For our application to $\CW^\text{sub}(\sl(N))$, we set $\ell_N(\KK) = \frac{(N-1) \KK}{N} + N-2$.
It is standard to introduce the isotropic generators $c$, $d$ defined by
\begin{equation}
c = \frac{1}{\ell_N(\KK)} (-a + b)
\qquad \text{and} \qquad
d = a + b,
\end{equation}
and so $\bilin{c}{c} = \bilin{d}{d}=0$ and $\bilin{c}{d}=2$. In this subsection, we set $N=1$ and so $\ell_N(\KK)=-1$, while $N$ will be a positive integer in the next subsection.

Given $x \in \alg{h}$, we associate the Heisenberg modes $x_n$ such that $[x_n, y_m]=n \bilin{x}{y} \delta_{n+m, 0}$ $\forall x, y \in \alg{h}$. We denote the corresponding bosonic current by $x(z) = \sum_{n \in \ZZ} x_n z^{-n-1}$. For $x\in \alg{h}$, we also introduce the zero modes $\hat{x}$ such that for $y\in \alg{h}$,$[\hat{x}, y_n]=-\bilin{x}{y} \delta_{n, 0}$.
The lattice vertex algebra corresponding to $\mathsf{H}_{\lvoa}$ is generated by the fields $x(z)$ and $e^x(z)$, for $x \in \alg{h}$, where\footnote{In the literature, our notation for $e^x(z)$ is sometimes alternatively given as $e^{\oint x(z) dz}$.
Additionally, comparing with the notation introduced earlier in this paper, it should be emphasized that $:e^{x(z)}:$ and $e^{x}(z)$ correspond to different operators.}
\begin{equation}
 e^x(z) = e^{\hat{x}} z^{-x_0} e^{\sum_{n>0} \frac{z^n}{n} x_{-n}} e^{-\sum_{n>0} \frac{z^{-n}}{n} x_{n}}.
\end{equation}

The half-lattice vertex algebra $\lvoa$ is the subalgebra generated by the fields $c(z)$, $d(z)$ and $e^{nc}(z)$, for $n \in \ZZ$, whose non-regular operator product expansions are
\begin{equation}
  c(z)d(w) \sim \frac{2}{(z-w)^2} \quad \text{and} \quad d(z)e^{nc}(w) \sim \frac{2n e^{nc}(w)}{(z-w)}.
\end{equation}

In \cite{FMS86}, a realisation of $\bgvoa$ was given purely in terms of bosonic fields, known as Friedan-Marinec-Shenker (FMS) bosonisation.
This is described by vertex algebra embeddings $\Phi^{\mathrm{FMS}, i} \colon \bgvoa \hookrightarrow \lvoa$ given by
\begin{equation}
\Phi^{\mathrm{FMS}, 1} \colon
    \left\{
    \begin{aligned}
    &\beta(z) \mapsto \ee^{c}(z)
    \\
    &\gamma(z) \mapsto \no{\frac{1}{2} \brac*{ c(z) + d(z) }\ee^{-c}(z)}
    \end{aligned}
    \right.
\qquad
\Phi^{\mathrm{FMS}, 2} \colon
    \left\{
    \begin{aligned}
    &\beta(z) \mapsto \no{\frac{1}{2} \brac*{ c(z) - d(z) } e^{c}(z)}
    \\
    &\gamma(z) \mapsto \ee^{-c(z)}
    \end{aligned}
    \right.
\end{equation}
These admit an equivalent description in terms of screening charges, so if we introduce the currents
\begin{equation}\label{eqn:FMSscreening}
\Psi^{\mathrm{FMS}, 1}(z) = e^{+\frac{1}{2} c + \frac{1}{2} d}(z),
\qquad \Psi^{\mathrm{FMS}, 2}(z) = e^{-\frac{1}{2} c - \frac{1}{2} d}(z), 
\end{equation}
then
\begin{equation}
    \bgvoa \cong \Im \Phi^{\mathrm{FMS}, i} \cong \ker \brac*{ Q^{\mathrm{FMS}, i} },\qquad Q^{\mathrm{FMS}, i}=\oint dz \Psi^{\mathrm{FMS}, i}(z),\qquad i=1,2.
\end{equation}
We will consider only the first FMS embedding here.

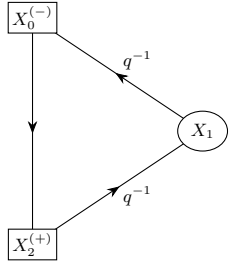
\begin{figure}[h]
\centering
\begin{tikzpicture}[scale=0.75, transform shape, font=\small]
%Coordinates
\coordinate (x0) at (0,2);
\coordinate (x1) at (3,0);
\coordinate (x2) at (0,-2);
%Arrows
\draw[qArrow] (x0) to (x2);
\draw[qArrow, qArrow/label={$q^{-1}$}, qArrow/label xshift=10pt, qArrow/label yshift=+13pt] (x1) to (x0);
\draw[qArrow, qArrow/label={$q^{-1}$}, qArrow/label xshift=10pt, qArrow/label yshift=+2pt](x2) to[] (x1);
%Nodes
\node[qNodeFrozen] (x0) at (x0) {$X^{(-)}_0$};
\node[qNodeUnfrozen] (x1) at (x1) {$X_1$};
\node[qNodeFrozen] (x2) at (x2) {$X^{(+)}_2$};
%final diagram
% \begin{scope}[xshift=8.5cm]
% %Coordinates
% \coordinate (x0) at (0,2);
% \coordinate (x1) at (3,0);
% \coordinate (x2) at (0,-2);
% %Arrows
% \draw[qArrow] (x0) to (x1);
% \draw[qArrow, qArrow/xshift=0pt](x1) to[] (x2);
% %Nodes
% \node[qNodeFrozen] (x0) at (x0) {$X^{(-)}_0$};
% \node[qNodeUnfrozen] (x1) at (x1) {$X_1$};
% \node[qNodeFrozen] (x2) at (x2) {$X^{(+)}_2$};
% \end{scope}
\end{tikzpicture}
	\caption{%
		Quiver corresponding to the deformed FMS realisation of $\bgvoa$.
	}\label{fig:BosonicGhostQuiver}
\end{figure}

\paragraph{Conformal limit of the deformed algebra}
Deformations of the bosonic ghost vertex algebra have appeared previously in the literature, for instance see \cite{Kon93, Harada:2020woh}.
These may be encoded by the quantum chiral cluster algebras corresponding to the quiver shown in \zcref{fig:BosonicGhostQuiver}.
The matrix $C$ associated with this quiver is
\begin{equation}
C=\left(q-q^{-1}\right)
\begin{pmatrix}
q & -1 & q^{-1}\\
1 & q-q^{-1} & -1\\
-q & 1 & -q^{-1}
\end{pmatrix}.
% \qquad
% C_2=(q-q^{-1})
% \begin{pmatrix}
% q & q^{-1} & 0\\
% -q & q-q^{-1} & q^{-1}\\
% 0 & -q & -q^{-1}
% \end{pmatrix}.
\end{equation}

The currents corresponding to the telescoping sums associated with the paths $X_0\to X_2$ and $X_2\to X_1\to X_0$ are
\begin{equation}\label{def:Gpm_N=1}
G^-(z) = \no{X_0(z)}, \quad G^+(z) = \no{X_2(z)} + \no{X_2(z)X_1(q^{-1}z)}.
\end{equation}
They coincide with the currents $G^\pm(z)$ introduced in the previous section, for the specific choice of $N=1$.
The following proposition claims that these two currents reduce to the $\b\g$ fields in the conformal limit.

\begin{proposition}\label{prop:bosGhost_CFTLim}
The FMS embedding of the bosonic ghost vertex algebra $\Im \Phi^{\text{FMS}, 1} \cong \bgvoa$ is recovered from the conformal limit of deformed W-algebra generated by the currents $G^\pm(z)$ defined in \ref{def:Gpm_N=1}, and associated with the quiver of \zcref{fig:BosonicGhostQuiver}. Specifically, the limit of these currents reproduces the $\b\g$ fields,
\begin{equation}
     \beta(z) = \lim\limits_{\hbar \to 0} G^-(z),\qquad \gamma(z) = -z^{-1}\lim\limits_{\hbar \to 0} \hbar^{-1} G^+(z).
\end{equation}
\end{proposition}

Before we prove the statement of this proposition, we need to say a few words about the definition of the conformal limit.
This limit is obtained by parameterizing the deformation parameter as $q=e^{\hbar/2}$, and sending $\hbar\to0$. It has been considered for various deformed W-algebras in \cite{AKSOvir96,AKOSwalg97,Awata2024}. It is usually obtained by introducing a $\hbar$-expansion of the generators, and expanding the algebraic relations.

Alternatively, it can also be defined with respect to a certain embedding inside a diagonal Heisenberg algebra. Then, it is possible to consider the limit of the operators' matrix elements  with respect to a fixed basis. In this context, it is possible to define the expansion of the currents,
\begin{equation}\label{def:exp_Xi}
X_i(z) = \sum_{s \in \ZZ^{\geq0}}  \hbar^sX_i^{(s)}(z),\qquad x_i(z) = \sum_{s \in \ZZ^{\geq0}} \hbar^s x_i^{(s)}(z),\qquad x^{(s)}_i(z) = \sum_{n \in \ZZ}z^{-n} x_{i, n}^{(s)} .
\end{equation}
Similarly, it is possible to take the limit of the zero modes using an embedding in the quantum Weyl algebra. Setting $Q_i=e^{\hat q_i}$, $P_i=q^{\hat p_i}=e^{\frac{\hbar}{2}\hat p_i}$, with $[\hat p_i,\hat q_j]=\d_{i,j}$, the expansion of the zero mode $X_i$ follows from the Taylor expansion of $P_i$. We will take this second approach here.

% Note that the modes $x_{i,n}^{(s)}$ satisfy the commutation relations
% \begin{equation}\label{eqn:confLim_approach2}
% [x_{i, n}^{(r)}, x_{j, m}^{(s)}] = \frac{1}{n} \left( C_{ij}^{[n]} \right)^{(r+s)} \d_{n+m, 0},\qquad\text{with}\qquad C_{ij}=\sum_{r=0}^\infty \hbar^r C_{ij}^{(r)}.
% \end{equation}

\begin{proof}
In order to produce the correct order in $\hbar$ for the modes $x_{i,n}$, we choose the following free field realization of the vertex operators $X_i(z)$ in terms of a rank two  diagonal Heisenberg algebra,

\begin{equation}
\begin{aligned}
X_0(z) &= :e^{-\hq_1+\hq_2} q^{\hat p_1}z^{-\hp_1-\hp_2}: 
  e^{\sum_{n>0}\frac{z^{n}}{n}\big[-q^{2n}J_{1,-n}+J_{2,-n}\big]}
  e^{\sum_{n>0}\frac{z^{-n}}{n}\big[-J_{1,n}-J_{2,n}\big]},
\\
X_1(z) &= -q^{-2 \hat p_1}
  e^{\sum_{n>0}\frac{z^{n}}{n}\big[(q^{n}-q^{-n})J_{2,-n}\big]}
  e^{\sum_{n>0}\frac{z^{-n}}{n}\big[(q^{n}-q^{-n})J_{2,n}\big]},
\\
X_2(z) &= :e^{\hq_1-\hq_2} q^{\hp_1}: z^{\hp_1+\hp_2}
  e^{\sum_{n>0}\frac{z^{n}}{n}\big[q^{-2n}J_{1,-n}-J_{2,-n}\big]}
  e^{\sum_{n>0}\frac{z^{-n}}{n}\big[J_{1,n}+J_{2,n}\big]}.
\end{aligned}
\end{equation}
With this realization, we have $v_n=0$ for the Casimir of Proposition \ref{Prop:NullVector}. We remark that, in order to recover the proper conformal limit, we need to normalize the vertex operators such that $X_1(z)$ has a minus sign, and the zero modes of $X_0(z)$ and $X_2(z)$ are inverse to each other at leading order in $\hbar$. We also included the possibility of inserting an extra power $z$, with the corresponding operator $\hp_1+\hp_2$ being a Casimir, and assuming that it acts as an integer on the Fock space.

Computing the first few terms of the $\hbar$-expansions of each $x_i(z)$, we identify the fields $c(z)$ and $d(z)$ as follows, 
\begin{equation}
\begin{aligned}
&e^{x_{0}^{(0)}(z)} = e^{c}(z) = e^{-x_{2}^{(0)}(z)}, & &x_1^{(0)}(z)=0,
& &x_1^{(1)}(z) -\hp_1= \frac{z}{2} (c(z) + d(z)),
\end{aligned}
\end{equation}
where in the last expression, the shift by $\hp_1$ comes from the expansion of the expansion of the zero mode $q^{-2\hp_1}$ of $X_1(z)$, and the extra factor $z$ in the r.h.s. from the fact that the fields $b(z),\ c(z)$ are expanded in powers of $z^{-n-1}$. 
Explicitly,
\begin{equation}
\begin{aligned}
&x_{0,n}^{(0)} = -\frac{1}{n}\big(J_{1,n}+J_{2,n}\big), & &x_{1,n}^{(0)} = 0, & &x_{2,n}^{(0)} = \frac{1}{n}\big(J_{1,n}+J_{2,n}\big),
\\
&x_{0,-n}^{(0)} = \frac{1}{n}\big(-J_{1,-n}+J_{2,-n}\big), & &x_{1,-n}^{(0)} = 0, & &x_{2,-n}^{(0)} = \frac{1}{n}\big(J_{1,-n}-J_{2,-n}\big),
\\
&x_{0,n}^{(1)} = 0, & &x_{1,n}^{(1)} = J_{2,n}, & &x_{2,n}^{(1)} = 0,
\\
&x_{0,-n}^{(1)} = -J_{1,-n}, & &x_{1,-n}^{(1)} = J_{2,-n}, & &x_{2,-n}^{(1)} = -J_{1,-n},
\end{aligned}
\end{equation}
from which we identify
\begin{equation}
\begin{pmatrix} c_{n}\\ d_{n}\end{pmatrix}=
\begin{pmatrix}
1 & 1\\
-1 & 1
\end{pmatrix}
\begin{pmatrix} J_{1,n}\\ J_{2,n}\end{pmatrix},
\qquad
\begin{pmatrix} c_{-n}\\ d_{-n}\end{pmatrix}=
\begin{pmatrix}
-1 &1\\
1 & 1
\end{pmatrix}
\begin{pmatrix} J_{1,-n}\\ J_{2,-n}\end{pmatrix},
\end{equation}
along with
\begin{equation}\label{eq:ghost_zm}
    \hat c=-\hq_1+\hq_2,\qquad \hat d=-\hq_1-\hq_2,\qquad c_0=-\hp_1-\hp_2,\qquad d_0=-\hp_1+\hp_2.
\end{equation}
\begin{comment}
\begin{equation}
\begin{aligned}
&x_{0,n}^{(0)} = -\frac{1}{n}c_{n}, & &x_{1,n}^{(0)} = 0, & &x_{2,n}^{(0)} = \frac{1}{n}c_{n},
\\
&x_{0,-n}^{(0)} = \frac{1}{n}c_{-n}, & &x_{1,-n}^{(0)} = 0, & &x_{2,-n}^{(0)} = -\frac{1}{n}c_{-n},
\\
&x_{0,n}^{(1)} = 0, & &x_{1,n}^{(1)} = \frac{n}{2}\big(c_{n}+d_{n}\big), & &x_{2,n}^{(1)} = 0,
\\
&x_{0,-n}^{(1)} = \frac{n}{2}\big(c_{-n}-d_{-n}\big), & &x_{1,-n}^{(1)} = \frac{n}{2}\big(c_{-n}+d_{-n}\big), & &x_{2,-n}^{(1)} = \frac{n}{2}\big(c_{-n}-d_{-n}\big),
\end{aligned}
\end{equation}
\end{comment}
With this identification, the screening current $\Psi_1^\ast(z)$ associated to the fermionic vertex $X_1(z)$ reduces to the FMS screening at order $\hbar^0$, i.e.  $\Psi_1^{*, (0)}(z) = e^{\frac{1}{2}(c+d)}(z) = \Psi^{\mathrm{FMS}, 1}(z)$ with $J_{1,0}=\hp_1$.

Finally, we consider the limit of the currents $G^\pm(z)$. From the previous identification, we have $G^{-, (0)}(z) = e^{c}(z)$, which is indeed the image of $\beta(z)$ under the FMS embedding. From
\begin{align}\label{eq:gammaTerms}
X^{(1)}_2(z) &= \no{  e^{\hq_1-\hq_2} z^{\hp_1+\hp_2} e^{x_2^{(0)}(z)}  \left(x_2^{(1)}(z)+\frac{\hat p_1}{2}\right) },\\
\brac*{ \no{X_2(z) X_1(q^{-1} z)} }^{(1)} &= -\no{ e^{\hq_1-\hq_2} z^{\hp_1+\hp_2}  e^{ x_2^{(0)}(z)}  \brac*{ x_2^{(1)}(z) + x_1^{(1)}(z)-\frac{\hp_1}{2} }},
\end{align}
and the fact that $G^{+(0)}(z)=0$, we deduce the conformal limit of $G^+(z)$,
\begin{equation}
-\lim_{\hbar \to 0} \hbar^{-1} G^+(z) = - X^{(1)}_2(z) - \brac*{ \no{X_2(z) X_1(q^{-1} z)} }^{(1)}  = \frac{z}{2} \no{ (c(z) + d(z)) e^{-c(z)} }
\end{equation}
which coincides with the image of $\gamma(z)$ under the FMS embedding.
\end{proof}

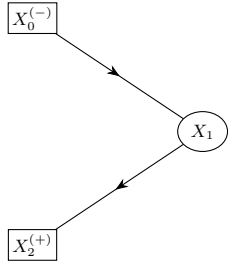
\begin{figure}[h]
\centering
\begin{tikzpicture}[scale=0.75, transform shape, font=\small]
%Coordinates
\coordinate (x0) at (0,2);
\coordinate (x1) at (3,0);
\coordinate (x2) at (0,-2);
%Arrows
\draw[qArrow] (x0) to (x1);
\draw[qArrow, qArrow/xshift=0pt](x1) to[] (x2);
%Nodes
\node[qNodeFrozen] (x0) at (x0) {$X^{(-)}_0$};
\node[qNodeUnfrozen] (x1) at (x1) {$X_1$};
\node[qNodeFrozen] (x2) at (x2) {$X^{(+)}_2$};
\end{tikzpicture}
	\caption{%
		Quiver corresponding to the deformed FMS realisation of $\bgvoa$, obtained by mutating the quiver given in \zcref{fig:BosonicGhostQuiver} at the vertex $X_1$.
	}\label{fig:BosonicGhostQuiver2}
\end{figure}

\begin{remark}
Upon mutation of the quiver \zcref{fig:BosonicGhostQuiver} at the vertex $X_1$, we find the quiver represented on Figure \ref{fig:BosonicGhostQuiver2}. It produces another deformation the deformed bosonic ghosts. In the conformal limit, we also recover the bosonic ghost fields $\b(z)$ and $\g(z)$ from the currents
\begin{equation}
    G^+(z) = - \left( \no{\widetilde{X}_0(z)} + \no{\widetilde{X}_0(z)\widetilde{X}_1(z)} \right), \qquad G^-(z) = \no{\widetilde{X}_2(z)}.
\end{equation}
\end{remark}

\subsection{Inverse reduction in the conformal limit}\label{sec:IQQHR_conformalLim}

In this subsection, we first recall the construction of the inverse reduction embedding, following \cite{FehSub23}. Then, we consider the conformal limit of the algebra $\CWqt^\text{sub}(\sl(N))$, and show that it reproduces the free field realization of the subregular W-algebra. Finally, we compare the deformed inverse reduction of Theorem \ref{thm:IQQHR} with its conformal version. 

\subsubsection{Inverse quantum Hamiltonian reduction}

Let $\heis_{\alg{h}}$ denote the rank $N-1$ Heisenberg vertex algebra strongly generated by $\tilde{\a}_i(z)$, for $i=1, \dots, N-1$, such that
\begin{equation}\label{rel_Heisenberg}
\tilde{\a}_i(z) \tilde{\a}_j(w) \sim \frac{\bilin{\tilde{\a}_i}{\tilde{\a}_j}}{(z-w)^2}
\qquad \text{where} \qquad \bilin{\tilde{\a}_i}{\tilde{\a}_j} = (N+\KK)A^{\liesl(N)}_{ij},
\end{equation}
and where $A^{\liesl(N)}$ is the Cartan matrix of $\liesl(N)$.

For non-critical level $\KK \neq -N$, the regular conformal W-algebra $\CW^{\KK}(\liesl(N))$ admits a free field embedding into $\heis_{\alg{h}}$.
If $\KK$ is generic, then this embedding is given by the intersection of the kernel of screening charges $Q_i^{\mathrm{reg}}$, for $i=1, \dots, N-1$, where $Q_i^{\mathrm{reg}} = \oint S_i^{\mathrm{reg}}(z) dz$ is the residue of the screening current
\begin{equation}
S_i^{\mathrm{reg}}(z) = e^{-\frac{1}{N+\KK} \tilde{\a}_i}(z),
\end{equation}
and so \cite{FF90}
\begin{equation}
    \CW^{\KK}(\liesl(N)) \cong \bigcap_{i=1}^{N-1} \ker Q_i^{\mathrm{reg}} \subset \heis_{\alg{h}}.
\end{equation}

Similarly, for $\KK \neq -N$, the conformal subregular W-algebra $\CW^{\KK, \mathrm{sub}}(\liesl(N))$ admits a free field embedding into $\heis_{\alg{h}} \otimes \bgvoa$.
For generic $\KK$, we define
\begin{equation}
S_i^{\mathrm{sub}}(z) =
    \left\{
    \begin{aligned}
        &\no{\beta(z) \ee^{-\frac{1}{N+\KK} \alpha_1}(z)} & &i=1,
        \\
        &\ee^{-\frac{1}{N+\KK} \alpha_i}(z) & &i=2, \dots, N-1
    \end{aligned}
    \right. ,
\end{equation}
where $\a_i(z)$ are the generators of a second rank $N-1$ Heisenberg vertex algebra with the same relations \ref{rel_Heisenberg}. Setting $Q_i^{\mathrm{sub}} = \oint S_i^{\mathrm{sub}}(z) dz$, we have \cite{Genra2016}
\begin{equation}
\CW^{\KK, \mathrm{sub}}(\liesl(N)) \cong \brac*{\ker Q_1^{\mathrm{sub}} } \cap \brac*{ \bigcap_{i=2}^{N-1} \ker Q_i^{\mathrm{sub}}  } \subset \heis_{\alg{h}} \otimes \bgvoa.
\end{equation}

To obtain an inverse reduction embedding, the key idea is to find a basis for the free field algebra $\heis_{\alg{h}} \otimes \lvoa$ such that we can identify some subset of the subregular W-algebra screening charges $\set{S_i^{\mathrm{sub}}}_{i=1, \dots, N-1}$ with the regular W-algebra screening charges $\set{S_i^{\mathrm{reg}}}_{i=1, \dots, N-1}$.
In particular, the screening charges we identify with the regular W-algebra screenings must act trivially on $\lvoa$.

Once this has been done, we may directly compare the free field realisations of these W-algebras.
To achieve this, one must FMS bosonise the free field realisation of $\CW^{\KK, \mathrm{sub}}(\liesl(N))$, which replaces the screening $S_1^{\mathrm{sub}}(z)$ with its image under the FMS embedding
\begin{equation}
S_1^{\mathrm{sub}}(z) = \no{\beta(z) \ee^{-\frac{1}{N+\KK} \alpha_1}(z)}
\quad \Rightarrow \quad
S_1^{\mathrm{sub}}(z) = \ee^{-\frac{1}{N+\KK} \a_1 + c}(z).
\end{equation}
The expense of doing so is including another screening charge, namely $Q^{\mathrm{FMS}, 1}$, so we have
\begin{equation}\label{eqn:BosWak_Sub_Screen}
\CW^{\KK, \mathrm{sub}}(\liesl(N)) \cong \brac*{\ker Q_1^{\mathrm{sub}} } \cap \brac*{ \bigcap_{i=2}^{N-1} \ker Q_i^{\mathrm{sub}}  } \cap Q^{\mathrm{FMS}, 1} \subset \heis_{\alg{h}} \otimes \lvoa,
\end{equation}
where $Q^\text{sub}_1$ is the screening charge associated to the current ${S}_1^\text{sub}(z)$. This gives a \emph{bosonised} free field embedding $\CW^{\KK, \mathrm{sub}}(\liesl(N)) \hookrightarrow \heis_{\alg{h}} \otimes \lvoa$.

Each screening charge $Q_i^{\mathrm{reg}}$ acts on $\heis_{\alg{h}}$ while the subregular screening charges act on $\heis_{\alg{h}} \otimes \lvoa$.
To ensure the screening charges for each W-algebra act upon the same free field algebra, each screening charge $Q_i^{\mathrm{reg}}$ is extended to a screening charge $Q_i^{\mathrm{reg}}$ which acts trivially on $\lvoa$.

Finally, we note that $Q_1^{\mathrm{sub}}$ act non-trivially on both $\heis_{\alg{h}}$ and $\lvoa$.
Consequently, we must find a change of basis for $\heis_{\alg{h}} \otimes \lvoa$ such that this screening acts trivially on $\lvoa$ and we can identify 
\begin{equation}
S_i^{\mathrm{sub}}(z) = S_i^{\mathrm{reg}}(z) \ \text{for } i=1, \dots, N-1.
\end{equation}
It was shown in \cite{FehSub23} that there is a choice of generators satisfying these conditions, given by
\begin{equation}\label{eqn:tildification}
\left\{
\begin{aligned}
& \widetilde{\alpha}_1(z) = \alpha_1(z)-(N+\KK) c(z),
\\[8pt]
&\widetilde{\alpha}_i(z) = \alpha_i(z) \quad \text{for } i=2, \dots, N-1,
\\[8pt]
&\widetilde{c}(z) = c(z),
\\[2pt]
&\widetilde{d}(z) = d(z) - \frac{N-1}{N}(N+\KK)c(z) + 2 \omega_1(z),
\end{aligned}
\right.
\qquad \text{where }\qquad \omega_1(z) = \frac{1}{N} \sum_{i=1}^{N-1} (N-i) \alpha_i(z).
\end{equation}
The key point here is that one may consider the screening charges $Q_i^{\mathrm{sub}}$ acting on the free field algebra $\heis_{\alg{h}} \otimes \lvoa$ which is generated by $\alpha_i(z)$, $c(z)$, $d(z)$, $e^{n c}(z)$, where $i=1, \dots, N-1$ and $n \in \ZZ$.
Alternatively one may consider the screenings $Q_1^{\mathrm{sub}}$, $Q_2^{\mathrm{sub}}$, ..., $Q_N^{\mathrm{sub}}$ acting on $\heis_{\alg{h}} \otimes \lvoa$ generated by $\wta_i(z)$, $\wtc(z)$, $\wtd(z)$, $e^{n \wtc}(z)$, as defined in \zcref{eqn:tildification}.
While both free field algebras are isomorphic and give free field realisations of $\CW^{\KK, \mathrm{sub}}(\liesl(N))$, the later also give an inverse reduction embedding, while the former does not.

\begin{lemma}[\cite{FehSub23}]
Fix the free field algebra $\heis_{\alg{h}} \otimes \lvoa$ generated by $\wta_i(z), \wtc(z), \wtd(z), e^{\wtc}(z), e^{-\wtc}(z)$, where $i=1, \dots, N-1$, along with the screening charges $Q_1^{\mathrm{sub}}$, $Q_2^{\mathrm{sub}}$, ..., $Q_N^{\mathrm{sub}}$.
The intersection of the kernel of these screening charges define a free field realisation of $\CW^{\KK, \mathrm{sub}}(\liesl(N))$ which is also an inverse reduction embedding.
\end{lemma}

To state the inverse reduction embedding, we use the standard notation $\cG^\pm(z),\cJ(z),\cL(z),\cU_i(z)$ for the currents generating the subregular W-algebra $\CW^{\KK,\text{sub}}(\sl(N))$, and $\cW_j(z)$ for the spin $j$ currents generating the regular W-algebra $\CW^{\KK}(\sl(N))$. The inverse reduction embedding can be found in \cite{FehSub23} and reads\footnote{The embedding from \cite{FehSub23} has been twisted by a conjugation automorphism, so it may be more easily compared with the embedding in \zcref{thm:IQQHR}.}
\begin{equation}\label{eqn:ZacEmbed}
\begin{aligned}
&\cL(z) \mapsto \cT(z) + t(z), \qquad
\cJ(z) \mapsto -b(z), \qquad
\cGm(z) \mapsto -\ee^{c}(z), \qquad
\cGp(z) \mapsto \cW_N(z) e^{-c}(z) + \sum_{j=0}^{N-1} \cW_j(z) :\cA_j(z)e^{-c}(z):,
\end{aligned}
\end{equation}
where $t(z)=\frac{1}{2} \no{\wtc(z)\wtd(z)} - \frac{1}{2}N\partial \wtaa(z) + \frac{1}{2}(N-2)\partial \wtb(z)$ is the stress-energy tensor of the half lattice, and $\cA_j(z)$ are certain fields defined in terms of the fields $c(z)$, $d(z)$ and their derivatives. 

\subsubsection{Conformal limit of the deformed subregular W-algebra}
In this subsection, we study the limit of the algebra $\CWqt^\text{sub}(\sl(N))$. We first identify the fields $b(z),\ c(z),\ \a_i(z)$ in the limit of the vertex operators $X_i(z)$ associated to the subregular quiver $\CQ_N^\text{sub}$. Then, we show that the limit of the screening currents associated to unfrozen vertices coincide with the screening currents used to define the algebra $\CW^{\KK,\text{sub}}(\sl(N))$.

\begin{lemma}\label{Lemma:id_conf_limit}
    In the conformal limit $q_1=e^{-(N+\KK)\hbar}$ and $q_2=e^{\hbar}$, with $\hbar\to0$, the fields $c(z)$, $d(z)$ and $\tilde{\a}_i(z)$ for $i=1,\dots, N-1$ can be identified with the first terms of the $\hbar$-expansion of the vertex operators associated to the quiver $\CQ_N^\text{sub}$ as follows,
    \begin{align}
        X_0^-(z)&=e^c(z)+O(\hbar),\\
        X_1(z)&=-1+\hbar\frac{z}{2}(c(z)+d(z))+O(\hbar^2),\label{eq:X1}\\
        X_i(z)&=1+\hbar z\tilde{\a}_{i-1}(z)+O(\hbar^2),\qquad i=2,\cdots,N,\\
        X_{N+1}^+(z)&=e^{-c}(z)+O(\hbar),
    \end{align}
\end{lemma}

\begin{remark}
    For this lemma, we set the Casimirs to $v_n=0$ in Proposition \ref{Prop:NullVector}. We also need to take a specific choice of polarization for the zero modes,
    \begin{equation}\label{def:zm_Xi}
        X_0=:e^{-\hat q_1+\hat q_2}q^{-(N-1)\hp_2}z^{-\hp_1-\hp_2}:,\qquad X_1=-q^{-2\hp_1},\qquad X_{N+1}=:e^{\hq_1-\hq_2}q^{-(N-1)\hp_2}z^{\hp_1+\hp_2}:,
    \end{equation}
    and $X_i=q^{2x_{i,0}}$ for $i=2,\cdots,N$. Note that the operator appearing in the power of $z$, and $x_{i,0}$ are Casimirs.
    
    We also note that, in order for the screening currents associated with unfrozen vertices to be well-defined in the conformal limit, the corresponding field must have a vanishing term at order $O(\hbar^0)$, i.e. $x_i^{(0)}(z)=0$ for $i=1\cdots N$. This condition is not required for frozen vertices, and in general we have $x_0^{(0)}(z)\neq0$ and $x_{N+1}^{(0)}(z)\neq0$.
\end{remark}

\begin{proof}
We need to check that the fields $c(z)$, $d(z)$ and $\tilde{\a}_i(z)$ obtained in this way have the correct pairing. This can be deduced from the $\hbar$-expansion of the matrix $C$ encoding commutation relation of the modes $x_{i,n}$. This matrix can be found in equ. \eqref{eqn:CMatrix_Sub}, and the first few terms of the $\hbar$-expansion $C_{ij}=\sum_{r=0}^\infty \hbar^r C_{ij}^{(r)}$ are $C^{(0)} = 0$,
\begin{equation}\label{eqn:CMats}
C^{(1)} =
\begin{pmatrix}
1 & -1 & \mathbf 0^{T} & N-1\\[3pt]
1 & 0 & \mathbf 0^{T} & -2\\[3pt]
\mathbf 0 & \mathbf 0 & \mathbf 0_{(N-1)\times(N-1)} & \mathbf 0\\[3pt]
-(N-1) & 1 & \mathbf 0^{T} & -1
\end{pmatrix}, \qquad
C^{(2)} = \hf
\begin{pmatrix}
1 & 2(N+\KK)-1 & \mathbf 0^{T} & \gamma_N\\[3pt]
2(N+\KK)-1 & 2 & -2(N+\KK)\,\mathbf{e}_1^{T} & -1\\[3pt]
\mathbf 0 & -2(N+\KK)\mathbf{e}_1 & 2(N+\KK)A^{\mathfrak{sl}(N)} & \mathbf 0\\[3pt]
\gamma_N & -1 & \mathbf 0^{T} & 1
\end{pmatrix},
\end{equation}
where $\gamma_N = -(N-1)(-(N-4)(N+\KK)+N-3)$. 

First, we note that this expansion is compatible with our requirement that $x_{i,n}^{(0)}=0$ for unfrozen vertices $i=1,\cdots, N$. Then, we note that the submatrix of $C^{(2)}$ formed by the rows and columns $3, \dots, N+1$ coincides with the Cartan matrix $A^{\sl(N)}$ of $\sl(N)$ multiplied by a factor $(N+\KK)$. We deduce $\comm{x_{i+1, n}^{(1)}}{x_{j+1, m}^{(1)}}=(N+\KK)A_{ij}^{\sl(N)}\d_{n+m, 0}$ for $i,j=1,\dots, N-1$, in agreement with the pairing $\bilin{\tilde{\a}_i}{\tilde{\a}_j}$.

Next, we consider the fermionic vertex, and introduce the decomposition $x_1^{(1)}(z)= \wtaa(z) - \widetilde{\omega}_1(z)$, where $\widetilde{\omega}_1(z)$ is the projection onto $\heis_{\alg{h}}$, and $\wtaa(z)$ the orthogonal projection onto $\heis_{\lvoa}$. From the expression \ref{eqn:CMats} of $C^{(2)}$, we deduce $[x_{1,n}^{(1)},x_{i,n}^{(1)}]=-(N+\KK)\d_{i,2}\d_{n+m,0}/n$ for $2=1\cdots N$, and so $\bilin{\widetilde{\omega}_1}{\wta_j} = (N+\KK)\delta_{j,1}$ for $j=1\cdots N-1$. This condition determines the expression of the field
\begin{equation}
    \widetilde{\omega}_1(z) = \sum_{i=1}^{N-1} \frac{N-i}{N} \wta_i(z),
\end{equation}
It implies $\bilin{\widetilde{\omega}_1}{\widetilde{\omega}_1} = \frac{1}{N}(N-1)(N+\KK)$. From $C^{(2)}$, we find the bilinear pairing $\bilin{x^{(1)}_1}{x^{(1)}_1}=1$ and so we deduce $\bilin{\wtaa}{\wtaa} = -\ell_N(\KK)$. These pairings coincide with the pairing of the fields $\tilde{\o}_1(z)$, $\tilde{a}(z)=(\tilde{d}(z)-\ell_N(\KK)\tilde{c}(z))/2$, and $\tilde{\a}_i(z)$ obtained in the transformation \ref{eqn:tildification}. We note that $\frac{1}{2} (c(z)+d(z)) = \wtaa(z) - \widetilde{\omega}_1(z)$, which gives the expression \ref{eq:X1}.

At this stage, we have only identified the linear combination $c(z)+d(z)$ of the half-lattice fields. In order to determine them properly, we need to examine the frozen vertices. Comparing with the formulas \ref{eqn:ZacEmbed} of inverse Hamiltonian reduction embedding leads to requiring that the current $G^-(z)$ reduces to the field $e^{c}(z)$ at the order $\hbar^0$, i.e. $G^{-,(0)}(z)=e^{c}(z)$ and so $x_{0,n}^{(0)}=-n c_n$. Using the $\hbar$-expansion of $C$, it is easy to check that identification is consistent with the pairings of the fields. Inspecting the zero entries of $C^{(1)}$, it is clear that the bosons $\tilde{\a}_i(z)$ commute with $\wtc(z)$ and $\wtd(z)$.

To properly identify the fields, we also need to study their zero modes. With the choice of polarization \ref{def:zm_Xi}, the $\hbar$-expansion of $X_0^-(z)$, $X_1(z)$ and $X_{N+1}(z)$ leads to the identification \ref{eq:ghost_zm} of the zero modes for the fields $c(z)$ and $d(z)$. In addition, we also have $\tilde{\a}_{i,0}=x_{i+1,0}$ by expanding $X_i(z)$ for $i=1\cdots N-1$.

Finally, assuming $v_n=0$, we find by expanding at first order the expressions $\ref{eq:vn_subreg}$ that $x_{N+1,n}^{(0)}=-x_{0,n}^{(0)}$, and so, at leading in $\hbar$ we find $X_{N+1}^{+,(0)}=e^{-c}(z)$. We find again agreement between pairings and commutation relations.
\end{proof}

\begin{proposition}\label{Prop:limit_scr}
Under the identification of the fields given in Lemma \ref{Lemma:id_conf_limit}, the screening currents $\Psi^{\ast}_1(z)$, $S_i^-(z)$ resp. associated to the unfrozen vertices $X_1$, $X_i$ with $i=2,\cdots, N$, reproduce the screening currents of $\CW^{\KK,\text{sub}}(\sl(N))$ in the free field realization associated to the inverse Hamiltonian reduction,
\begin{equation}
    \Psi^{\ast, (0)}_1(z)=\Psi^{\mathrm{FMS}}(z),\qquad S^{-, (0)}_i(z) = S_{i-1}^{\mathrm{sub}}(z),\qquad i=2, \dots, N.
\end{equation}
\end{proposition}

\begin{proof}
    The limit of the screening currents is obvious from the identification. The zero modes of the bosonic currents correspond to $S_{i+1}^-=e^{-\frac{\hat{\tilde{\a}}_i}{N+\KK}}$ and $s_{i+1,0}^-=\tilde{\a}_{i,0}/(N+\KK)$.
\end{proof}

\begin{remark}
    The screenings $S^{+, (0)}_{i+1}(z)$ can be referred to as Feigin-Frenkel dual screenings to $S^{-, (0)}_{i+1}(z)$, explicitly
    \begin{equation}
        S^{-, (0)}_{i+1}(z) = e^{-\frac{1}{N+\KK} \wta_i}(z) \qquad \text{and} \qquad S^{+, (0)}_{i+1}(z) = e^{\wta_i}(z)
        \qquad \text{for } i=1, \dots, N-1.
    \end{equation}
\end{remark}

\begin{remark}\label{rem:Gpm limit}
    The conformal limit of the current $G^-(z)$ is
    \begin{equation}
        G^{-,(0)}(z)=e^{c}(z),
    \end{equation}
    in agreement with the formula \ref{eqn:ZacEmbed} obtained in the inverse Hamiltonian reduction. In addition, the expression of $G^+(z)$ reduces to a Miura transform in the differential form,
    \begin{align}
        &G^+(z)=\hbar^N z^N:e^{-c}(z)\left((N-1+\KK)\p_z+y_1(z)\right)\left((N-1+\KK)\p_z+y_2(z)\right)\cdots \left((N-1+\KK)\p_z+y_N(z)\right):+O(\hbar^{N+1}),\\
        &\text{with}\qquad y_{N+1-i}(z)=\dfrac{(N-1+\KK)(i-1)}{z}+\hf(c(z)+d(z))+\sum_{j=1}^{i-1}\tilde{\a}_{j}(z),
    \end{align}
    that should be identified with the similar expression for the subregular algebra \cite{FS04}. Furthermore, identifying $\CK_1^{(1)}(z)\sim J(z)$, $\CK_2^{(2)}(z)\sim T(z)$, we observe that their expression is also compatible with the formulas \ref{eqn:ZacEmbed}. 
\end{remark}

It is well known that the conformal limit of the deformed W-algebra $\CWqt(\sl(N))$ is the regular W-algebra $\CW^{\KK}(\sl(N))$, and the limit of the screening current $S_i^\pm(z)$ reproduces the screening currents of the conformal algebra \cite{Feigin1996quantum}. The identification of the screening in Proposition \ref{Prop:limit_scr} shows that the limit of the deformed inverse Hamiltonian reduction of Theorem \ref{thm:IQQHR} indeed reproduces the inverse Hamiltonian reduction \cite{FehSub23}.

\appendix

\section{Proof of the Proposition 4.6}\label{AppB}
\subsection{Algebraic relation $[G^+,G^-]$}
To compute the commutator $[G^+,G^-]$,  it is convenient to rewrite the current $G^+(z)$ in a form that singles out the vertex operators $X_1(z)$ and $X_{N+1}^+(z)$ that have a non-trivial OPE with $X_0^-(z)$, i.e.
\begin{equation}
    G^+(z)=\sum_{n=0}^N G_n^+(z),\qquad G_n^+(z)=:X_{N+1}^+(z)\CO_n(z)\prod_{k=0}^{n-1}X_1(q_3^k z):
\end{equation}
with $[\CO_n(z),X_0^-(w)]=0$. In particular, 
\begin{equation}
    G_0(z)=X_{N+1}^+(z),\qquad G_1(z)=:X_{N+1}^+(z)\S[X_1,X_2,\cdots,X_N](z),\qquad G_N(z)=:X_{N+1}^+(z)\prod_{i=1}^{N}\prod_{k=0}^{N-i}X_i(q_3^kz):.
\end{equation}

Then, we can compute the following OPEs,
\begin{equation}
    G^+_n(z)G^-(w)::
    \begin{cases}
        \prod_{k=n}^{N-2}\vphi_{-1}(q_3^{k-1}z,w), & n=0\cdots N-2,\\
        1, & n=N-1,\\
        \vphi_{-1}(q_3^{N-2}z,w)^{-1},& n=N,
    \end{cases}
\end{equation}
and the same OPE for $G^-(w)G^+_n(z)$. The r.h.s. are rational functions with simple poles at $z=q_3^{-k}w$ for $k=n-1\cdots N-3$ ($n=0\cdots N-2$), and $z=q_2q_3^{2-N}w$ ($n=N$). Decomposing these fractions over their poles, and using the relation
\begin{equation}
    \left[\dfrac1{z-\a w}\right]_{|z|>\!>|w|}-        \left[\dfrac1{z-\a w}\right]_{|z|<\!<|w|}=z^{-1}\d(\a w/z),
\end{equation}
we deduce the commutation relations
\begin{equation}
    [G_n^+(z),G^-(w)]=
    \begin{cases}
        \sum_{k=n-1}^{N-3}\d(q_3^{-k}w/z)\tilde{\CO}_{n,k}(z), & n=0\cdots N-2,\\
        1, & n=N-1,\\
        \d(wq_2q_3^{2-N}/z)\tilde{\CO}_N(z),& n=N.
    \end{cases}
\end{equation}
where $\tilde{\CO}_{n,k}(z)$ are the operators obtained from the normal-ordered products multiplied by a residue. Taking the sum over $n$, we reproduce the second equation \ref{eq:commGpGm} in the proposition.

We also notice that the pole at $z=q_3w$ appears only in the OPE $G_0^+(z)G^-(w)$, and the pole at $z=q_2q_3^{2-N}$ in the OPE $G_N^+(z)G_0^-(w)$. Simarly, the pole at $z=w$ appears only in the OPEs $G_0^+(z)G^-(w)$ and $G_1^+(z)G^-(w)$. It is possible to compute the corresponding residue, and it leads to the expressions \ref{eq:subreg_K0N} and \ref{eq:subreg_K2}.

\subsection{Algebraic relation $G^\pm G^\pm$}
The relation for $G^-(z)$ directly follows from the OPE of the frozen vertices, i.e.
\begin{equation}
    X_0^-(z)X_0^-(w)::\dfrac{z-w}{z-q^2w}.
\end{equation}
On the other hand, the relation for $G^+(z)$ is more difficult to demonstrate. We first recall the definition of the operators $Y_k(z)=:X_1(z)X_2(z)\cdots X_{N+1-k}(z):$, and the OPEs,
\begin{equation}\label{eqn:OPE_XY}
    X_{N+1}^+(z)X_{N+1}^+(w)::q\vphi_1(z,w)^{-1},\qquad X_{N+1}^+(z)Y_i(z)::\vphi_1(z,w),\qquad Y_i(z)X_{N+1}^+(w)::\vphi_{-1}(z,w),
\end{equation}
while the OPE $Y_i(z)Y_j(w)$ can be found in \ref{eqn:OPE_YY}.

The proof will be done by induction. For this purpose, we introduce 
\begin{equation}
    G^{(k)}(z)=:X_{N+1}^+(z)(1+Y_{N+1-k}(z)T)(1+Y_{N+2-k}(z)T)\cdots (1+Y_N(z)T).1:.
\end{equation}
Expanding the first factor, we find the inductive relation
\begin{equation}\label{eqn:induction_Gk}
    G^{(k)}(z)=G^{(k-1)}(z)+:A^{(k)}(z)G^{(k-1)}(q_3z):,\qquad\text{with}\qquad \qquad A^{(k)}(z)=:X_{N+1}^+(z)X_{N+1}^+(q_3 z)^{-1}Y_{N+1-k}(z):.
\end{equation}
We will need the following OPEs,
\begin{align}
\begin{split}
    &A^{(k)}(z)A^{(k)}(w)::\dfrac{\vphi_1(q_3^{-1}z,w)}{\vphi_{-1}(q_3^{-1}z,w)}\vphi_1(z,w)^{-2},\qquad A^{(k)}(z)Y_i(w)::1,\qquad Y_i(z)A^{(k)}(w)::1,\qquad(\text{for }i>N+1-k),\\
    &A^{(k)}(z)X_{N+1}^+(w)::\dfrac{\vphi_{-1}(z,w)}{\vphi_1(z,w)}\vphi_1(q_3z,w),\qquad X_{N+1}^+(z)A^{(k)}(w)::\vphi_1(q_3^{-1}z,w),
\end{split}
\end{align}
from which we deduce
\begin{equation}\label{eqn:OPE_AkGl}
        A^{(k)}(z)G^{(l)}(w)::\dfrac{\vphi_{-1}(z,w)}{\vphi_1(z,w)}\vphi_1(q_3z,w),\qquad G^{(l)}(z)A^{(k)}(w)::\vphi_1(q_3^{-1}z,w),\qquad(\text{for }l<k).
\end{equation}
% As a result, we can establish the exchange relation for $l<k$,
% \begin{equation}\label{eqn:rel_AkGl}
%     (z-q^{-2}w)A^{(k)}(z)G^{(l)}(w)-(z-q^2w)\dfrac{z-q_3^{-1}q^{-2}w}{z-q_3^{-1}q^2w}G^{(l)}(w)A^{(k)}(z)=\k\d(q_3 z/w):A^{(k)}(z)G^{(l)}(q_3z):
% \end{equation}
% with $\k=s_1t_1t_2$.

Noticing that $G^+(z)=G^{(N)}(z)$, the relation $G^+G^+$ in \ref{eq:commGpGm} is a consequence of the following lemma:
\begin{lemma} The currents $G^{(k)}(z)$ satisfy the algebraic relation
\begin{equation}\label{eqn:rel_GkGk}
    (z-q^{-2}w)G^{(k)}(z)G^{(k)}(w)+(w-q^{-2}z)G^{(k)}(w)G^{(k)}(z)=0,
\end{equation}
and $G(z)^2=0$.
\end{lemma}

\begin{proof}
    The proof is by induction, exploiting the relation \ref{eqn:induction_Gk}. Let's first consider $G^{(1)}(z)=:X_{N+1}^+(z)(1+X_1(z)):$, and compute the normal ordering
    \begin{align}
        \begin{split}
            G^{(1)}(z)G^{(1)}(w)=\dfrac{z-w}{z-q^{-2}w}&\Big(:X_{N+1}^+(z)X_{N+1}^+(w):+q^{-1}\dfrac{z-q^{2}w}{z-w}:X_{N+1}^+(z)X_1(z)X_{N+1}^+(w):\\
            &+q\dfrac{z-q^{-2}w}{z-w}:X_{N+1}^+(z)X_{N+1}^+(w)X_1(w):+:X_{N+1}^+(z)X_1(z)X_{N+1}^+(w)X_1(w):\Big)
        \end{split}
    \end{align}
    and so
    \begin{align}
        \begin{split}
            (z-q^{-2}w)G^{(1)}(z)G^{(1)}(w)&=(z-w):X_{N+1}^+(z)X_{N+1}^+(w):+(q^{-1}z-qw):X_{N+1}^+(z)X_1(z)X_{N+1}^+(w):\\
            &+(qz-q^{-1}w):X_{N+1}^+(z)X_{N+1}^+(w)X_1(w):+(z-w):X_{N+1}^+(z)X_1(z)X_{N+1}^+(w)X_1(w):.
        \end{split}
    \end{align}
    Symmetrizing this expression with respect to $z$ and $w$, we observe that the relation \ref{eqn:rel_GkGk} is indeed satisfied for $k=1$. Specializing this relation to $z=w$, we find that the r.h.s. vanishes, and so $G^{(1)}(z)^2=0$.

    Before proving the induction step, let's examine the consequence of the lemma. Taking the product of two currents, it can be written in the form of a sum of normal-ordered exponentials,
    \begin{equation}
        (z-q^{-2}w)G^{(k)}(z)G^{(k)}(w)=\sum_{\a=1}^{M_k}f_\a^{(k)}(z,w)\CO_\a^{(k)}(z,w),
    \end{equation}
    where we have assumed that $\CO_\a^{(k)}(z,w)\neq\CO_\b^{(k)}(z,w)$ if $\a\neq\b$.\footnote{For instance, we can choose for $k=1$,
    \begin{align}
        &f_1^{(1)}(z,w)=z-w,& \CO_1^{(1)}(z,w)&=:X_{N+1}^+(z)X_{N+1}^+(w):,\\
        &f_2^{(1)}(z,w)=q^{-1}z-qw,&\qquad \CO_2^{(1)}(z,w)&=:X_{N+1}^+(z)X_1(z)X_{N+1}^+(w):,\\
        &f_3^{(1)}(z,w)=qz-q^{-1}w,& \CO_3^{(1)}(z,w)&=:X_{N+1}^+(z)X_{N+1}^+(w)X_1(w):,\\
        &f_4^{(1)}(z,w)=z-w,& \CO_4^{(1)}(z,w)&=:X_{N+1}^+(z)X_1(z)X_{N+1}^+(w)X_1(w):.
    \end{align}}
    The lemma implies that $f_\a^{(k)}(z,w)$ is a polynomial in $z$ and $w$, and
    \begin{equation}\label{eqn:sum_fCO}
        \sum_{\a=1}^{M_k}f_\a^{(k)}(z,w)\CO_\a^{(k)}(z,w)+\sum_{\a=1}^{M_k}f_\a^{(k)}(w,z)\CO_\a^{(k)}(w,z)=0.
    \end{equation}
    Thus, there must exist a permutation $\s\in S_{M_k}$ such that
    \begin{equation}
        f_{\a}^{(k)}(z,w)+f_{\s(\a)}^{(k)}(w,z)=0,\qquad \CO_{\a}^{(k)}(z,w)=\CO_{\s(\a)}^{(k)}(w,z).
    \end{equation}
    In addition, the condition $G^{(k)}(z)^2=0$ implies\footnote{Beware that it does not imply $f_\a^{(k)}(z,z)=0$ since there can be some $\b\neq\a$ such that $\CO_\a^{(k)}(z,z)=\CO_\b^{(k)}(z,z)$ (see the previous example).}
    \begin{equation}\label{eqn:cond_COa_zz}
        \sum_{\a=1}^{M_{k}}f_\a^{(k)}(z,z)\CO_\a^{(k)}(z,z)=0.
    \end{equation}

    To prove the induction step, we use the recursion relation \ref{eqn:induction_Gk} and decompose the product
    \begin{equation}
        (z-q^{-2}w)G^{(k)}(z)G^{(k)}(w)=\CA_k(z,w)+\CB_k(z,w)+\CC_k(z,w)+\CD_k(z,w),
    \end{equation}
    with
    \begin{align}
        \begin{split}
            &\CA_k(z,w)=(z-q^{-2}w)G^{(k-1)}(z)G^{(k-1)}(w),\\
            &\CB_k(z,w)=(z-q^{-2}w)G^{(k-1)}(z):A^{(k)}(w)G^{(k-1)}(q_3w):,\\
            &\CC_k(z,w)=(z-q^{-2}w):A^{(k)}(z)G^{(k-1)}(q_3z):G^{(k-1)}(w),\\
            &\CD_k(z,w)=(z-q^{-2}w):A^{(k)}(z)G^{(k-1)}(q_3z)::A^{(k)}(w)G^{(k-1)}(q_3w):.
        \end{split}
    \end{align}
    By induction, $\CA_k(z,w)+\CA_k(w,z)=0$, and $\CA_k(z,z)=0$. After normal-ordering, the other terms have the following expressions,
    \begin{align}
        \begin{split}
            \CB_k(z,w)&=q\dfrac{z-q^{-2}w}{z-q_3w}\sum_{\a=1}^{M_{k-1}}f_\a^{(k-1)}(z,q_3w):\CO_\a^{(k-1)}(z,q_3w)A^{(k)}(w):,\\
            \CC_k(z,w)&=q^{-1}\dfrac{z-q^{2}w}{q_3z-w}\sum_{\a=1}^{M_{k-1}}f_\a^{(k-1)}(q_3z,w):\CO_\a^{(k-1)}(q_3z,w)A^{(k)}(z):,\\
            \CD_k(z,w)&=q_3^{-1}\sum_{\a=1}^{M_{k-1}}f_\a^{(k-1)}(q_3z,q_3w):\CO_\a^{(k-1)}(q_3z,q_3w)A^{(k)}(z)A^{(k)}(w):.
        \end{split}
    \end{align}
    Let's consider the last term, namely $\CD_k(z,w)$. Since there are no pole, we can directly consider the sum $\CD_k(z,w)+\CD_k(w,z)$ which gives
    \begin{align}
        q_3^{-1}:\left(\sum_{\a=1}^{M_{k-1}}f_\a^{(k-1)}(q_3z,q_3w)\CO_\a^{(k-1)}(q_3z,q_3w)+\sum_{\a=1}^{M_{k-1}}f_a^{(k-1)}(q_3w,q_3z)\CO_\a^{(k-1)}(q_3z,q_3w)\right)A^{(k)}(z)A^{(k)}(w):=0,    
    \end{align}
    since the quantity inside the parenthesis vanishes as a consequence of the relation \ref{eqn:sum_fCO}. We also note that, using the relation \ref{eqn:cond_COa_zz},
    \begin{equation}
        \CD_k(z,z)=q_3^{-1}:\left(\sum_{\a=1}^{M_{k-1}}f_\a^{(k-1)}(q_3z,q_3z)\CO_\a^{(k-1)}(q_3z,q_3z)\right)A^{(k)}(z)A^{(k)}(z):=0.
    \end{equation}
    
    The treatment of the terms $\CB_k(z,w)$ and $\CC_k(z,w)$ is more difficult due to the potential pole at $z=q_3^{\pm1}w$. However, we can check that this pole has no residue, since
    \begin{align}
        \begin{split}
            \res_{z=q_3w}\CB_k(z,w)&=qw(q_3-q^{-2}):\left(\sum_{\a=1}^{M_{k-1}}f_\a^{(k-1)}(q_3w,q_3w)\CO_\a^{(k-1)}(q_3w,q_3w)\right)A^{(k)}(w):=0,\\
            \res_{z=q_3^{-1}w}\CC_k(z,w)&=q^{-1}q_3^{-1}w(q_3^{-1}-q^2):\left(\sum_{\a=1}^{M_{k-1}}f_\a^{(k-1)}(w,w)\CO_\a^{(k-1)}(w,w)\right)A^{(k)}(q_3^{-1}w):=0,\\
        \end{split}
    \end{align}
    by \ref{eqn:cond_COa_zz}. Thus, the sum $\CB_k(z,w)+\CC_k(w,z)$ does not produce a delta function, it reduces to
    \begin{align}
        q\dfrac{z-q^{-2}w}{z-q_3w}:\left(\sum_{\a=1}^{M_{k-1}}f_\a^{(k-1)}(z,q_3w)\CO_\a^{(k-1)}(z,q_3w)+\sum_{\a=1}^{M_{k-1}}f_\a^{(k-1)}(q_3w,z)\CO_\a^{(k-1)}(q_3w,z)\right)A^{(k)}(w):=0.
    \end{align}
    In the same way, we show that $\CB_k(w,z)+\CC_k(z,w)=0$. It shows that $G^{(k)}(z)$ satisfies the relation \ref{eqn:rel_GkGk}. Finally, computing 
        \begin{align}
        \begin{split}
            \CB_k(z,z)&=\dfrac{q-q^{-1}}{1-q_3}:\left(\sum_{\a=1}^{M_{k-1}}f_\a^{(k-1)}(z,q_3z)\CO_\a^{(k-1)}(z,q_3z)\right)A^{(k)}(z):,\\
            \CC_k(z,z)&=\dfrac{q-q^{-1}}{1-q_3}:\left(\sum_{\a=1}^{M_{k-1}}f_\a^{(k-1)}(q_3z,z)\CO_\a^{(k-1)}(q_3z,z)\right)A^{(k)}(z):,
        \end{split}
    \end{align}
    we show that $\CB_k(z,z)+\CC_k(z,z)=0$ by \ref{eqn:sum_fCO}. Combining with the fact that $\CA_k(z,z)=\CD_k(z,z)=0$ has been shown previously, we conclude that $G^{(k)}(z)^2=0$.
\end{proof}

\section{Proof of Lemma \ref{lemma:Apm}}\label{AppC}
To prove the lemma, it is convenient to introduce vectors $\vb{a}_\pm\in\mC^{N+1}$ to represent operators $a_{\pm n}$ obtained as linear combination of the generators $x_{i,n}$, i.e. such that $a_{\pm n}=\vb{a}_\pm^{[n]}.\vb{x}_{\pm}$ with $\vb{x}_{\pm}^T=\pmat{x_{0,\pm n} & x_{1,\pm n} & \cdots & x_{N+1,\pm n}}$. In this way, the commutator $[a_n,b_m]$ of two generators coincides with the matrix element
\begin{equation}
    [a_n,b_m]=\dfrac{\d_{n+m,0}}{n}(\vb{a}_+^{[n]})^TC^{[n]}\vb{b}_-^{[n]},\qquad (n>0).
\end{equation}
The vectors corresponding to $v_{\pm n}$ and $h_{\pm n}$ in $\CA_{\CQ_N^\text{sub}}$ are
\begin{equation}
    \vb{v}_{-}=\pmat{q_1q_3^{N-1}\\\frac{1-q_3^N}{1-q_3}\\\frac{1-q_3^{N-1}}{1-q_3}\\\vdots\\1\\1},\qquad
    \vb{v}_+=\pmat{\b_0\\\frac{1-q_3^{-N}}{1-q_3^{-1}}\\\frac{1-q_3^{-(N-1)}}{1-q_3^{-1}}\\\vdots\\1\\q_1^{-1}q_3^{1-N}\b_0^\vee },\qquad
    \vb{h}_\pm=\pmat{1\\0\\\vdots\\0\\q_3^{\mp1}},
\end{equation}
and the vectors corresponding to $\tilde{v}_{\pm n}$ and $\tilde{h}_{\pm n}$ in $\CA_{\CQ_{1|N}}$ are
\begin{equation}
    \vb{\tilde{v}}_{-}=\pmat{q_1q_3^{N}\\q_3^N\\q_3^{N-1}\\\vdots\\q_3\\1},\qquad
    \vb{\tilde{v}}_+=\pmat{1\\1\\\vdots\\1},\qquad
    \vb{\tilde{h}}_-=\pmat{1\\1\\\vdots\\1},\qquad
    \vb{\tilde{h}}_+=\pmat{(q_1q_3^{N})^{-1}\\q_3^{-N}\\q_3^{-(N-1)}\\\vdots\\q_3^{-1}\\1}.
\end{equation}
In this way, the lemma \ref{lemma:Apm} takes the form of a simple linear algebra problem, and the four conditions imposed are
\begin{equation}
    (1)\ \tilde{C}=A_+CA_-^T,\qquad (2)\ A_\pm\vb{e}_i=\vb{e}_i,\qquad (3)\ A_\pm^T\vb{\tilde{v}}_\pm=r_\pm \vb{v}_{\pm},\qquad (4)\  A_\pm^T\tilde{\vb{h}}_\pm=s_\pm \vb{h}_{\pm},
\end{equation}
for some parameters $r_\pm,s_\pm\in\mC$, and $i=1\cdots N$. Here $\vb{e}_i$ with $i=0\cdots N+1$ denote the canonical basis vectors. 

First, let's show the existence of $A_\pm$. The second requirement imposes that the row decomposition of these two matrices is of the form
\begin{equation}
    A_\pm=\pmat{\vb{a}_0^\pm\\\vb{e}_1\\\vdots\\\vb{e}_N\\\vb{a}_{N+1}^\pm},
\end{equation}
The third and fourth conditions then impose the following equations on the unknown vectors $\vb{a}_0^\pm$, $\vb{a}_{N+1}^\pm$,
\begin{align}
\begin{split}
&\vb{a}_0^++\sum_{i=1}^N\vb{e}_i+\vb{a}_{N+1}^+=r_+\vb{v}_+,\qquad q_1q_3^N\vb{a}_0^-+\sum_{i=1}^N q_3^{N+1-i}\vb{e}_i+\vb{a}_{N+1}^-=r_-\vb{v}_-,\\
&(q_1q_3^N)^{-1}\vb{a}_0^++\sum_{i=1}^N q_3^{i-(N+1)}\vb{e}_i+\vb{a}_{N+1}^+=s_+\vb{h}_+,\qquad \vb{a}_0^-+\sum_{i=1}^N\vb{e}_i+\vb{a}_{N+1}^-=s_-\vb{h}_-.
\end{split}
\end{align}
These equations can be easily solved, and introducing the previous expressions of $\vb{v}_\pm$ and $\vb{h}_\pm$, we find
\begin{align}
    \begin{split}
        &(1-q_1q_3^N)\vb{a}_{N+1}^+=-(1-q_1q_3^N)\vb{e}+(r_+\b_0-q_1q_3^Ns_+)\vb{e}_0+(r_+-q_1q_3^N(1-q_3^{-1}))\bar{\vb{u}}_++(q_1^{-1}q_3^{1-N}\b_0^\vee r_+-q_1q_3^{N-1}s_+)\vb{e}_{N+1},\\
        &(1-(q_1q_3^N)^{-1})\vb{a}_0^+=(r_+\b_0-s_+)\vb{e}_0+(r_+-1+q_3^{-1})\bar{\vb{u}}_++(q_1^{-1}q_3^{1-N}\b_0^\vee r_+-q_3^{-1}s_+)\vb{e}_{N+1}\\
        &(1-(q_1q_3^N)^{-1})\vb{a}_{N+1}^-=-(1-(q_1q_3^N)^{-1})\vb{e}+(s_--q_3^{-1}r_-)\vb{e}_0-(q_1q_3^N)^{-1}(r_-+1-q_3)\bar{\vb{u}}_--(q_3s_-+(q_1q_3^N)^{-1}r_-)\vb{e}_{N+1}\\
        &(1-q_1q_3^N)\vb{a}_0^-=(s_--q_1q_3^{N-1}r_-)\vb{e}_0-(r_-+1-q_3)\bar{\vb{u}}_--(r_-+q_3s_-)\vb{e}_{N+1},
    \end{split}
\end{align}
with
\begin{equation}
    \vb{e}=\sum_{i=1}^N\vb{e}_i,\qquad \bar{\vb{u}}_\pm=\sum_{i=1}^N\dfrac{1-q_3^{\mp(N+1-i)}}{1-q_3^{\mp 1}}\vb{e}_i.
\end{equation}
The solution is unique for a given choice of the four parameters $r_\pm$ and $s_\pm$.

It remains to show that the linear transformation is a homorphism of algebra, i.e. that it preserves the commutation relations between generators. Since both $v_{\pm n}$ and $\tilde{v}_\pm$ commute with every generator, the check is trivial for them. Then, the transformation obviously preserves the commutation relations between $x_{i,n}$ and $x_{j,m}$ since the map is trivial for these. Next, we can verify that the commutator $[h_n,h_m]$ is preserved, which corresponds to impose \begin{equation}\label{eqn:cond_hpm}
    \vb{\tilde{h}}_+^T\tilde{C}\vb{\tilde{h}}_-= \vb{\tilde{h}}_+^TA_+ CA_-^T\vb{\tilde{h}}_-=s_+s_-\vb{h}_+^TC\vb{h}_-.
\end{equation}
By a direct calculation, we show that
\begin{equation}
    \vb{h}_+^TC\vb{h}_-=q_3(1-q_2^{-1})\dfrac{1-q_3^{N-2}}{1-q_3}(1-(q_1q_3^N)^{-1}),\qquad \vb{\tilde{h}}_+^T\tilde{C}\vb{\tilde{h}}_-=(1-q_1)(1-q_2)(1-q_3)(1-(q_1q_3^N)^{-1}).
\end{equation}
Taking the ratio, we find that the condition \ref{eq:cond_lemma} indeed implies the equation \ref{eqn:cond_hpm}. Finally, we need to check that the commutation relation $[h_n,x_{i,m}]=0$ for $i=1\cdots N$ is preserved. This can be checked by comparing
\begin{align}
    &\vb{e}^T_i\tilde{C}\vb{\tilde{h}}_-=0,\qquad\longleftrightarrow\qquad \vb{e}_i^TA_+CA_-^T\vb{\tilde{h}}_-=s_-\vb{e}_i^TC\vb{h}_-=0,\\
    &\vb{\tilde{h}}_+^T\tilde{C}\vb{e}_i=0,\qquad\longleftrightarrow\qquad \vb{\tilde{h}}_+^TA_+CA_-^T\vb{e}_i=s_+\vb{h}_+^TC\vb{e}_i=0.
\end{align}

\section{Proof of theorem \ref{th:subreg_mutations}}\label{AppD}
We start with the mutation at the vertex $M$, and parameterize the vertex operators accordingly,
\begin{align}\label{eq:param_mut_M}
    X_0^-(z)&=\Psi(qq_3^{1-M}z)V_0(z),&\qquad X_{M-1}(z)&=:\Psi(qz)\Psi^\ast(q^{-1}q_3^{-1}z):V_{M-1}(z),\nonumber\\
    X_M(z)&=-:\Psi(q^{-1}z)\Psi^\ast(qz):,& X_{M+1}(z)&=:\Psi(q^{-1}q_3^{-1}z)\Psi^\ast(q^{-1}z):V_{M+1}(z),\nonumber\\
    X_{N+1}^+(z)&=\Psi^\ast(q^{-1}q_3^{-1}z)V_{N+1}(z),
\end{align}
and $X_i(z)=V_i(z)$ for all other vertices, imposing $[V_i(z),\Psi(w)]=[V_i(z),\Psi^\ast(w)]=0$. Considering first the current $G_M^+(z)$, it is convenient to rewrite it in the form
\begin{equation}
    G_M^+(z)=:X_N^+(z)(1+Y_1^{(M)}(z)T)\cdots (1+Y_{N-M}^{(M)}(z)T)\cdot 1:,\qquad Y_i^{(M)}(z)=:X_{M+1}(z)X_{M+2}(z)\cdots X_{N+1-i}(z):
\end{equation}
Using the parameterization \ref{eq:param_mut_M}, we can factor out the dependence in the fermionic fields,
\begin{equation}
    Y_i^{(M)}(z)=:\Psi(q^{-1}q_3^{-1}z)\Psi^\ast(q^{-1}z):\bar Y_i^{(M)}(z),\qquad \bar Y_i^{(M)}(z)=:V_{M+1}(z)\cdots V_{N+1-i}(z):.
\end{equation}
Expanding the Miura transform, the current takes the form
\begin{align}
    \begin{split}
        G_M^+(z)&=\sum_{k=0}^{N-M}\sum_{1\leq j_1<\cdots<j_k\leq N-M}:X_{N+1}^+(z)Y_{j_1}^{(M)}(z)\cdots Y_{j_k}^{(M)}(q_3^{k-1}z)T^k\cdot1:\\
        &=\sum_{k=0}^{N-M}\sum_{1\leq j_1<\cdots<j_k\leq N-M}\Psi^\ast(q^{-1}q_3^{k-1}z):V_{N+1}^+(z)\bY_{j_1}^{(M)}(z)\cdots \bY_{j_k}^{(M)}(q_3^{k-1}z)T^k\cdot1:.
    \end{split}
\end{align}
After mutation and gauging, the expression of the vertex operators is given by Proposition \ref{prop:mutation},
\begin{align}
    \tX_0^-(z)&=\tPsi^\ast(q_3^{1-M}z)V_0(z),&\qquad \tX_{M-1}(z)&=-:\tPsi(q_3^{-1}z)\tPsi^\ast(z):V_{M-1}(z),\nonumber\\
    \tX_M(qq_3z)&=-:\tPsi(q^{-1}z)\tPsi^\ast(qz):,& \tX_{M+1}(z)&=-:\tPsi(z)\tPsi^\ast(q^{-2}q_3^{-1}z):V_{M+1}(z),\nonumber\\
    \tX_{N+1}^+(z)&=-\tPsi(q_3^{-1}z)V_{N+1}(z),
\end{align}
and the current $G_M^+(z)$ becomes
\begin{align}
    \begin{split}
        G_M^+(z)&=\sum_{k=0}^{N-M}\sum_{1\leq j_1<\cdots<j_k\leq N-M}(\tPsi(q^{-2}q_3^{k-1}z)-\tPsi(q_3^{k-1}z)):V_{N+1}^+(z)\bY_{j_1}^{(M)}(z)\cdots \bY_{j_k}^{(M)}(q_3^{k-1}z)T^k\cdot1:,\\
        &=-:\sum_{k=0}^{N-M}\sum_{1\leq j_1<\cdots<j_k\leq N-M}\tPsi(q_3^{k-1}z)V_{N+1}^+(z)\bY_{j_1}^{(M)}(z)\cdots \bY_{j_k}^{(M)}(q_3^{k-1}z)T^k(1-\tPsi^\ast(q_3^{-1}z)\tPsi(q^{-2}q_3^{-1}z))\cdot1:.
    \end{split}
\end{align}
Introducing $\tY_i^{(M-1)}(z)=:\tX_M(z)\tX_{M+1}(z)\cdots\tX_{N+1-i}(z):$ for $i=1\cdots N-M+1$, we notice that for $i\neq N-M+1$ we have $\tY_i^{(M-1)}(z)=:\tPsi^\ast(q_3^{-1}z)\tPsi(z):\bY_i^{(M)}(z)$. Using this remark, we can rewrite the current as
\begin{align}
    \begin{split}
        G_M^+(z)&=:\sum_{k=0}^{N-M}\tX_{N+1}^+(z)\sum_{1\leq j_1<\cdots<j_k\leq N-M}\tY_{j_1}^{(M-1)}(z)\cdots \tY_{j_k}^{(M-1)}(q_3^{k-1}z)T^k(1+\tX_{M}(z)T)\cdot1:\\
        &=:\tX_{N+1}^+(z)(1+\tY_1^{(M-1)}(z)(z)T)(1+\tY_2^{(M-1)}(z)T)\cdots(1+\tY_{N-M+1}^{(M-1)}(z)T)\cdot 1:=\tilde{G}^+_{M-1}(z).
    \end{split}
\end{align}

The procedure is the same for the current $G_M^-(z)$. First, we rewrite in the form
\begin{equation}
    G_M^-(z)=:X_0^-(z)(1+Z_1^{(M)}(z)T)\cdots (1+Z_M^{(M)}(z)T)\cdot1:,\qquad Z_i^{(M)}(z)=:X_M(q_3^{1-M}z)X_{M-1}(q_3^{2-M}z)\cdots X_{M+1-i}(q_3^{i-M}z):
\end{equation}
and, using again the parameterization \ref{eq:param_mut_M}, factor out the fermionic fields for $i<M$,
\begin{equation}
    Z_i^{(M)}(z)=-:\Psi(qq_3^{2-M}z)\Psi^\ast(qq_3^{1-M}z):\bar Z_i^{(M)}(z),\qquad \bZ_i^{(M)}(z)=:V_{M-1}(q_3^{2-M}z)\cdots V_{M+1-i}(q_3^{i-M}z):.
\end{equation}
As a result, the current takes the following form after expanding the Miura transform (keeping the last factor as it is),
\begin{align}
    \begin{split}
        G_M^-(z)&=\sum_{k=0}^{M-1}\sum_{1\leq j_1<\cdots<j_k\leq M-1}:X_0^-(z)Z_{j_1}^{(M)}(z)\cdots Z_{j_k}^{(M)}(q_3^{k-1}z)T^k(1+X_M(q_3^{1-M}z)T)\cdot1:\\
        &=\sum_{k=0}^{M-1}\sum_{1\leq j_1<\cdots<j_k\leq M-1}(-1)^k:\Psi(qq_3^{k+1-M}z)V_0(z)\bZ_{j_1}^{(M)}(z)\cdots \bZ_{j_k}^{(M)}(q_3^{k-1}z)T^k(1-\Psi(q^{-1}q_3^{1-M}z)\Psi^\ast(qq_3^{1-M}z))\cdot1:,\\
        &=\sum_{k=0}^{M-1}\sum_{1\leq j_1<\cdots<j_k\leq M-1}(-1)^k:V_0(z)\bZ_{j_1}^{(M)}(z)\cdots \bZ_{j_k}^{(M)}(q_3^{k-1}z)T^k(\Psi(qq_3^{1-M}z)-\Psi(q^{-1}q_3^{1-M}z))\cdot1:.
    \end{split}
\end{align}
After mutation, introducing the quantities $\tZ_i^{(M-1)}(z)=:\tX_{M-1}(q_3^{2-M}z)\tX_{M-2}(q_3^{3-M}z)\cdots \tX_{M-i}(q_3^{i-M+1}z):$ for $i=1\cdots M-1$, and noticing that $\tZ_i^{(M-1)}=-:\tPsi(q_3^{1-M}z)\tPsi^\ast(q_3^{2-M}z):\bZ_i^{(M)}(z)$, we find
\begin{align}
    \begin{split}
        G_M^-(z)&=\sum_{k=0}^{M-1}\sum_{1\leq j_1<\cdots<j_k\leq M-1}(-1)^k:V_0(z)\bZ_{j_1}^{(M)}(z)\cdots \bZ_{j_k}^{(M)}(q_3^{k-1}z)T^k\tPsi^\ast(q_3^{1-M}z)\cdot1:,\\
        &=\sum_{k=0}^{M-1}\sum_{1\leq j_1<\cdots<j_k\leq M-1}:\tX_0^-(z)\tZ_{j_1}^{(M-1)}(z)\cdots \tZ_{j_k}^{(M-1)}(q_3^{k-1}z)T^k\cdot1:=\tilde{G}_{M-1}^-(z).
    \end{split}
\end{align}

The mutation at the vertex $M+1$ is treated in the same way. We use the following parameterization of vertex operators,
\begin{align}\label{eq:param_mut_M+1}
    X_0^-(z)&=\Psi^\ast(qq_3^{1-M}z)V_0(z),&\qquad X_{M}(z)&=:\Psi(qz)\Psi^\ast(qq_3z):V_{M}(z),\nonumber\\
    X_{M+1}(z)&=-:\Psi(q^{-1}z)\Psi^\ast(qz):,& X_{M+2}(z)&=:\Psi(qq_3z)\Psi^\ast(q^{-1}z):V_{M+2}(z),\nonumber\\
    X_{N+1}^+(z)&=\Psi(qz)V_{N+1}(z),
\end{align}
and $X_i(z)=V_i(z)$ for all other vertices, with $V_i(z)$ commuting with the fermionic field $(\Psi(w),\Psi^\ast(w))$. With this parameterization, we have the decomposition for $i<N-M$,
\begin{equation}
    Y_i^{(M)}(z)=-:\Psi(qq_3z)\Psi^\ast(qz):\bY_i^{(M)}(z),\qquad \bar Y_i^{(M)}(z)=:V_{M+2}(z)\cdots V_{N+1-i}(z):.
\end{equation}
Partially expanding the Miura transform, we find
\begin{align}
    \begin{split}
        G_M^+(z)&=\sum_{k=0}^{N-M-1}\sum_{1\leq j_1<\cdots<j_k\leq N-M-1}:X_{N+1}^+(z)Y_{j_1}^{(M)}(z)\cdots Y_{j_k}^{(M)}(q_3^{k-1}z)T^k(1+X_{M+1}(z)T)\cdot1:\\
        &=\sum_{k=0}^{N-M-1}\sum_{1\leq j_1<\cdots<j_k\leq N-M-1}(-1)^k:V_{N+1}(z)\bY_{j_1}^{(M)}(z)\cdots \bY_{j_k}^{(M)}(q_3^{k-1}z)T^k(\Psi(qz)-\Psi(q^{-1}z))\cdot1:.
    \end{split}
\end{align}
After mutation and gauging, the expression of the vertex operators is again given Proposition \ref{prop:mutation},
\begin{align}
    \tX_0^-(z)&=-\tPsi(q^2q_3^{1-M}z)V_0(z),&\qquad \tX_{M}(z)&=-:\tPsi(q^2q_3z)\tPsi^\ast(z):V_{M}(z),\nonumber\\
    \tX_{M+1}(q^{-1}q_3^{-1}z)&=-:\tPsi(q^{-1}z)\tPsi^\ast(qz):,& \tX_{M+2}(z)&=-:\tPsi(z)\tPsi^\ast(q_3z):V_{M+2}(z),\nonumber\\
    \tX_{N+1}^+(z)&=\tPsi^\ast(z)V_{N+1}(z),
\end{align}
and, introducing $\tY_i^{(M+1)}(z)=:\tX_{M+2}(z)\cdots\tX_{N+1-i}(z):$ for $i=1\cdots N-M-1$, we have $\tY_i^{(M+1)}(z)=-:\tPsi(z)\tPsi^\ast(q_3z):\bY_i^{(M)}(z)$. We deduce the transformation of the current,
\begin{align}
    \begin{split}
        G_M^+(z)&=\sum_{k=0}^{N-M-1}\sum_{1\leq j_1<\cdots<j_k\leq N-M-1}(-1)^k:V_{N+1}(z)\bY_{j_1}^{(M)}(z)\cdots \bY_{j_k}^{(M)}(q_3^{k-1}z)T^k\tPsi^\ast(z)\cdot1:\\
        &=\sum_{k=0}^{N-M-1}\sum_{1\leq j_1<\cdots<j_k\leq N-M-1}:\tX_{N+1}^+(z)\tY_{j_1}^{(M)}(z)\cdots \tY_{j_k}^{(M)}(q_3^{k-1}z)T^k\cdot1:=\tilde{G}_{M+1}^+(z).
    \end{split}
\end{align}

Finally, we consider the current $G_M^-(z)$, and decomposing
\begin{equation}
    Z_i^{(M)}(z)=:\Psi(qq_3^{1-M}z)\Psi^\ast(qq_3^{2-M}z):\bZ_i^{(M)}(z),\qquad \bZ_i^{(M)}(z)=:V_M(q_3^{1-M}z)\cdots V_{M+1-i}(q_3^{i-M}z):,
\end{equation}
we write
\begin{align}
    \begin{split}
        G_M^-(z)&=\sum_{k=0}^M\sum_{1\leq j_1<\cdots<j_k\leq M}:X_0^-(z)Z_{j_1}^{(M)}(z)\cdots Z_{j_k}^{(M)}(q_3^{k-1}z)T^k\cdot1:\\ 
        &=\sum_{k=0}^M\sum_{1\leq j_1<\cdots<j_k\leq M}:V_0(z)\bZ_{j_1}^{(M)}(z)\cdots \bZ_{j_k}^{(M)}(q_3^{k-1}z)T^k\Psi^\ast(qq_3^{1-M}z)\cdot1:.
    \end{split}
\end{align}
After mutation, introducing $\tZ_i^{(M+1)}=:\tX_{M+1}(q_3^{-M}z)\tX_M(q_3^{1-M}z)\cdots\tX_{M+2-i}(q_3^{i-M-1}z):$ for $i=1\cdots M+1$, and noting that for $i<M+1$ we have $\tZ_i^{(M+1)}(z)=:\tPsi(q^2q_3^{2-M}z)\tPsi^\ast(q^2q_3^{1-M}z):\bZ_i^{(M)}(z)$, we find
\begin{align}
    \begin{split}
         G_M^-(z)&=\sum_{k=0}^M\sum_{1\leq j_1<\cdots<j_k\leq M}:V_0(z)\bZ_{j_1}^{(M)}(z)\cdots \bZ_{j_k}^{(M)}(q_3^{k-1}z)T^k(\tPsi(q_3^{1-M}z)-\tPsi(q^2q_3^{1-M}z))\cdot1:\\
         &=-\sum_{k=0}^M\sum_{1\leq j_1<\cdots<j_k\leq M}:\tPsi(q^2q_3^{k+1-M}z)V_0(z)\bZ_{j_1}^{(M)}(z)\cdots \bZ_{j_k}^{(M)}(q_3^{k-1}z)T^k(1-\tPsi(q_3^{1-M}z)\tPsi^\ast(q^2q_3^{1-M}z)T)\cdot1:\\
        &=\sum_{k=0}^M\sum_{1\leq j_1<\cdots<j_k\leq M}:\tX_0^-(z)\tZ_{j_1}^{(M+1)}(z)\cdots \tZ_{j_k}^{(M+1)}(q_3^{k-1}z)T^k(1+\tZ_{M+1}^{(M+1)}(z)T)\cdot1:.
    \end{split}
\end{align}
The last line coincides indeed with the expansion of the Miura transform for $\tilde{G}_{M+1}^-(z)$.

\addcontentsline{toc}{section}{\refname}  
\printbibliography

\end{document}